\documentclass[12pt,a4paper]{article}
\usepackage{paper}
\usepackage[a4paper,left=30mm,top=30mm,right=30mm,bottom=30mm]{geometry}

\allowdisplaybreaks[4]   

\title{A theory of traces and the divergence theorem}
\author{Moritz Schönherr, Friedemann Schuricht\\
  \small TU Dresden - Fakultät Mathematik}  
\date{}

\begin{document}

\maketitle

\vspace{-5mm}

\begin{center}
  Dedicated in memoriam and gratitude to \\ 
  {\it Eberhard Zeidler} 
\end{center}

\bigskip

\begin{abstract}   
We introduce a general approach to traces that we consider 
as linear continuous functionals on some function space $X$
where we focus on the special choices 
$\cL^\infty(\edom)$ and $\cW^{1,\infty}(\edom)$. This leads to an integral
calculus for the computation of the precise representative of an integrable
function and of the trace of a Sobolev or BV function. For integrable vector
fields with distributional divergence being a measure, we also obtain 
Gauss-Green formulas on arbitrary Borel sets. It turns out that a second
boundary integral is needed in general. The advantage of the integral calculus
is that neither a normal field nor a trace function on the boundary is needed.
The Gauss-Green formulas are also available for Sobolev and BV functions.
Finally, for any open set the existence of a weak
solution of a boundary value problem containing the $p$-Laplace operator is
shown as application of the trace theory.
\end{abstract}

\bigskip

\renewcommand{\contentsname}{{\large Contents}}

{\small

\tableofcontents

}

\newpage

\section{Introduction}
\label{int}

The divergence theorem belongs to the most important tools 
in mathematical analysis and continuum physics and 
goes back to Gauss (1813), Ostrogradskii (1826), and Green (1828)
(cf. Stolze \cite{stolze} for a brief history).
It connects a volume integral with an integral over the
bounding surface. For a regular bounded open set
$\dom\subset\R^n$ and a regular vector field 
$F:\dom\to\R^n$ it says that
\begin{equation} \label{int-e1}
  \I{\dom}{\divv F}{\lem} = \I{\bd{\dom}}{F\cdot\nu^\dom}{\ham^{n-1}} \,
\end{equation}
where $\lem$ is the Lebesgue measure, $\ham^{n-1}$ the $(n-1)$-dimensional
Hausdorff measure, and $\nu^\dom$ the outer unit normal to $\dom$. 
In physics it combines volume and surface phenomena and is indispensable
for fundamental balance and conservation laws. 
If we apply \reff{int-e1} to a product 
$fF$ with some regular $f:\dom\to\R$, then we obtain the
Gauss-Green formula, also called integration by parts formula,
\begin{equation} \label{int-e2}
  \I{\dom}{\big(f\divv F+F\cdot Df\big)}{\lem} = 
  \I{\bd{\dom}}{f F\cdot\nu^\dom}{\ham^{n-1}} \,.
\end{equation}
This formula is 
the basis for the definition of weak derivatives and, 
according to Eberhard~Zeidler, it is
``{\it the key to the modern theory of partial
differential equations and to modern calculus of variations}''
(cf. \cite[p.~119]{zeidler_108}). It is of course desirable to have 
such an essential formula available for a very large class of sets $\dom$ and
functions $f$ and~$F$.  
But it turns out that the surface integral leads to crucial limitations. 
On the one hand we need sets $\dom$ having some normal field on its boundary
and on the other hand we need functions $F$, $f$ on $\dom$ that possess a
pointwise trace function on $\bd\dom$. In practice we need a balance in
regularity for these ingredients. This leads to the typical cases where, 
roughly speaking, either $F$, $f$ have to be smooth if $\dom$ has merely finite
perimeter (that are the most general sets with some kind of normal field 
on the boundary)
or we can allow that $fF$ has merely weak derivatives while
$\dom$ has to have Lipschitz boundary. In continuum mechanics this situation
prevents that we can use \reff{int-e1} on sets $\dom$ where concentrations
occur on $\bd\dom$ or on sets $\dom$ not having a normal field on
$\bd\dom$. More precisely, we are unable to compute the flux through the
boundary $\bd\dom$ for such sets. 
However, nature ``knows'' what happens or, with the words of Albert
Einstein (cf. \cite{infeld}): {\it ``God does not care about our mathematical
difficulties; He integrates empirically.''}
Therefore the derivation of more general versions of \reff{int-e2}
is an important task. But, without intending completeness, let us first 
sketch some developments going beyond smoothness. 

We start with the simple observation that the regularity of the product
$fF$ is essential for the availability of \reff{int-e2}. This means that 
a weak regularity of one factor requires a stronger regularity 
of the other factor. But also a bad boundary
$\bd\dom$ requires a more regular product $fF$ than a good one and it turns out
that a Lipschitz boundary is apparently always a good boundary. Moreover,
in cases where $\dom$ is not open, one needs some ambient open set $\edom$
containing $\dom$ such that $f$ and $F$ with their derivatives can be
reasonably defined on $\edom$.   

In the theory of partial differential equations the Gauss-Green formula is 
closely related to the treatment of boundary value problems where, during the
last century,  Sobolev functions $f\in\cW^{k,p}(\dom)$ became more and more
important. But these functions are merely defined  
$\lem$-a.e. on $\dom$ and their restriction to $\bd\dom$ doesn't make sense.
Thus the derivation of a trace operator 
$T:\cW^{k,p}(\dom)\to\cL^1(\bd\dom,\hm)$, that assigns reasonable values on the
boundary, is an essential requirement. 
Since Sobolev functions are not just integrable but also have an integrable
weak derivative, the limit of averages near the boundary exists $\hm$-a.e. on
$\bd\dom$ and provides a trace operator 
if e.g. $\dom$ is a bounded open set with Lipschitz boundary. 
It turns out that this is also
true for BV functions $f\in\cB\cV(\dom)$, 
that played an increasing role during the last decades. 
This way one obtains \reff{int-e2} for Sobolev and BV functions $f$, for 
Lipschitz continuous vector fields $F\in\op{Lip}(\ol\dom)$, and for open 
bounded $\dom$ with Lipschitz boundary where the second term on
the left hand side becomes $\I{\dom}{F}{Df}$ for BV functions
(cf. \cite[p.~168]{pfeffer}).

The balance in regularity for the product $fF$ is nicely worked out by 
Ancelotti~\cite{anzelotti} for bounded open $\dom$ with Lipschitz boundary. 
He treats the pairing of certain BV functions $f$
with bounded vector fields $F\in\cL^\infty(\dom)$
where the distributional divergence $\divv F$ either belongs to
$\cL^p(\dom)$ or it is just a Radon measure.
In this generality the product $F\cdot Df$ in \reff{int-e2} 
has to be interpreted as a measure.
Moreover one cannot assign a pointwise trace function to $F$ on $\bd\dom$, but
merely the normal component $F\cdot\nu^\dom$ has a so-called normal trace 
$\hm$-a.e. on $\bd\dom$ (cf. also Kawohl-Schuricht 
\cite[p.~540]{kawohl-schuricht}
for a variant that is relevant for the treatment of the 1-Laplace operator). 

The development of geometric measure theory lead to a different
substantial improvement,
since it allowed the extension of \reff{int-e2} to the large class of $\dom$
having merely finite perimeter.
Such sets have a measure theoretic outer unit normal $\nu^\dom$ 
for $\hm$-a.e. point on the measure theoretic boundary $\mbd\dom$ 
(that is contained in the topological boundary $\bd\dom$) and \reff{int-e2}
becomes 
\begin{equation}  \label{int-e3}
  \I{\dom}{\big(f\divv F+F\cdot Df\big)}{\lem} = 
  \I{\mbd{\dom}}{f F\cdot\nu^\dom}{\ham^{n-1}} \,.
\end{equation}
As price for the weak regularity of $\dom$ one had to require initially 
$fF\in\op{Lip}_c(\R^n,\R^n)$, i.e. Lipschitz continuous with compact support 
(cf. De Giorgi \cite{de-giorgi}, Federer \cite{federer-45}, \cite{federer-58}).
These results have then been extended to bounded vector fields 
of bounded variation, 
i.e. belonging to $\cB\cV(\dom,\R^n)$ 
(cf. Vol'pert \cite{volpert-67}, \cite{volpert-85}).
Pairings of bounded $F$ where $\divv F$ is a Radon measure with bounded 
BV functions~$f$ on sets of finite perimeter are treated  
by Crasta-De Cicco \cite{crasta-cicco-anz},
\cite{crasta-cicco-ext} and Crasta-De Cicco-Malusa \cite{crasta-cicco-malusa}.

Later on it became more and more common to consider the left hand
side in \reff{int-e2} as linear functional in $f$ or in $F$. This is 
an obvious idea if one has in mind the capability
of linear functionals in modern analysis (duality, weak derivatives etc.).
Let us now focus on the case where $F$ is a general vector field 
with weak regularity as needed in applications and let $f$ play the role of a
more regular test function, a situation that is mostly met in the literature
(it is a simple task to transfer the results to the opposite case).
It turned out that, with $\divv F$ taken in the
distributional sense and with $1\le p\le\infty$, 
\begin{equation*}
  \cD\cM^p(\edom)=\{F\in\cL^p(\edom,\R^n)\mid \divv F 
  \tx{ is a Radon measure} \} 
\end{equation*}
is a reasonable space for the selection of the vector field $F$.
Thus we consider the left hand side in
\reff{int-e2}, adapted to such general $F$ and for some $\dom\subset\edom$, 
as linear functional 
\begin{equation} \label{int-e4}
  T_F(f) = \I{\dom}{f}{\divv F} + \I{\dom}{F\cdot Df}{\lem}
  \qmq{for} f\in X\,
\end{equation}
where we basically find the following choices for $X$ in the literature 
\begin{equation*}
  X_1=\op{Lip}(\edom)\,, \;
  X_2=\op{Lip}_{\rm loc}(\edom)\,, \;
  X_3=C^1(\edom)\,, \;
  X_4=\{ f\in C(\edom)\mid Df\in\cL^{p'}(\edom)\} \,.
\end{equation*}
All choices have in common that, for
suitable $\dom\subset\edom$, we somehow have $X\subset C(\ol\dom)$. Moreover it
is an important observation that $T_F(f)$ in fact merely depends on the
values of $f$ on $\bd\dom$
(cf. \v Silhav\'y \cite[p.~22]{silhavy_2005}, 
\cite[p.~448]{silhavy_divergence_2009}, Chen-Torres-Ziemer
\cite[p.~254]{chen_gauss-green_2009}). Therefore most investigations are
devoted to the natural question for which $F$ the functional $T_F$ is related
to a Radon measure on $\bd\dom$ or on $\mbd\dom$.  

The strongest results are obtained for 
$F\in\cD\cM^\infty(\edom)$ and for sets of finite perimeter
$\dom\csubset\edom$. Here we also use the measure theoretic 
interior $\mint\dom$ (that differs from $\dom$ merely by an
$\lem$-null set). 
For such $F$ one generally has   
that $\divv F$ is absolutely continuous with respect to $\hm$
(cf. \v Silhav\'y \cite[p.~21]{silhavy_2005}, Chen-Torres 
\cite[p.~250]{chen_divergence-measure_2005}).
This implies that $T_F$ is always a Radon measure on $\mbd\dom$. 
If $(\divv F)(\mbd\dom)=0$ and if $F$ is a precise representative such that
$\hm$-a.e. $x\in\mbd\dom$ is a 
Lebesgue point, then one has the slightly modified version of \reff{int-e3}
that
\begin{equation}  \label{int-e3a}
  \I{\mint\dom}{f}{\divv F} + \I{\dom}{F\cdot Df}{\lem} = 
  \I{\mbd{\dom}}{f F\cdot\nu^\dom}{\ham^{n-1}} \,.
\end{equation}
(cf. Degiovanni \cite[p.~212]{degiovanni_cauchy_1999}). One also has 
\reff{int-e3a} for continuous $F$ 
(cf. \v Silhav\'y \cite[p.~85]{silhavy_2008},
Chen-Torres-Ziemer \cite[p.~284]{chen_gauss-green_2009},  
Comi-Payne \cite[p.~198]{comi-payne}). 
Otherwise one has to use some approximation to get a normal trace function 
$t^{\rm int}\in\cL^\infty(\mbd\dom,\hm)$ such that
\begin{equation} \label{int-e5}
  \I{\mint\dom}{f}{\divv F} + \I{\dom}{F\cdot Df}{\lem} = 
  \I{\mbd{\dom}}{f t^{\rm int}}{\ham^{n-1}} 
\end{equation}
(cf. \v Shilha\'y \cite[p.~25]{silhavy_2005} where 
\reff{int-e3} is applied to smooth approximations of $F$ and 
Chen-Torres \cite[p.~252]{chen_divergence-measure_2005} where
approximations of $\dom$ are used to obtain $t^{\rm int}$ as weak$^*$ 
limit of Radon measures; 
cf. also Comi-Payne \cite[p.~194, 200]{comi-payne}, 
Comi-Torres \cite{comi-torres}, \v Silhav\'y \cite[p.~6]{silhavy_2019}). 
Since $F$ can have jumps across the boundary
such that $(\divv F)(\mbd\dom)\ne 0$, it is reasonable to apply the previous 
result to the measure theoretic exterior $\mext\dom$ of $\dom$. This readily 
gives some $t^{\rm ext}\in\cL^\infty(\mbd\dom,\hm)$, 
that differs from $t^{\rm int}$ in general, such that 
\begin{equation*}
   \I{\mint\dom\cup\mbd\dom}{f}{\divv F} + \I{\dom}{F\cdot Df}{\lem} = 
  \I{\mbd{\dom}}{f t^{\rm ext}}{\ham^{n-1}} 
\end{equation*}
(cf. Chen-Torres-Ziemer \cite[p.~275, 281]{chen_gauss-green_2009}, 
Comi-Payne \cite[p.~194, 200]{comi-payne}). The extension $\hat F$ of
$F\in\cD\cM^\infty(\edom)$ with zero belongs to $\cD\cM^\infty(\R^n)$ if
$\edom$ is open and bounded and satisfies 
$\hm(\bd\dom\setminus\mext\dom)=0$
(cf. Chen-Li-Torres \cite[p.~242]{chen_2020}; cf. also Chen-Torres
\cite[p.~258]{chen_divergence-measure_2005}, Chen-Torres-Ziemer
\cite[p.~288]{chen_gauss-green_2009}, Comi-Payne \cite[p.~209]{comi-payne}). 
In this case one can readily get the previous results for $\dom$ with finite
perimeter that do not need to be compactly contained in $\edom$ by applying the
previous results to $\hat F$ on some larger $\edom$
(cf. Chen-Li-Torres \cite[p.~248]{chen_2020}). In applications to shock waves
or cracks it is also desirable to have a Gauss-Green formula with boundary
integral on the complete topological boundary, i.e. where the inner part of the
boundary is also taken into account. This can easily be derived from 
\reff{int-e5} by moving the part of the first integral on the left hand
side over $\bd\dom\cap\mint\dom$ to the right hand side. 
If $\dom$ is open, then solely the integral over
$\dom$ remains on the left hand side 
(cf. Chen-Li-Torres \cite[p.~248]{chen_2020}
and Remark~\ref{nm-s10c} below).
Let us also refer to Leonardi-Saracco \cite{leonardi-saracco}
for some special result in $\R^2$.

For general vector fields $F\in\cD\cM^1(\edom)$
and sets $\dom\csubset\edom$ with finite perimeter 
we have \reff{int-e3a} if $(\divv F)(\mbd\dom)=0$ and if 
$F$ is a precise representative such that both $\hm$-a.e. $x\in\mbd\dom$ is a
Lebesgue point and $F$ is $\hm$-integrable on $\mbd\dom$
(cf. Degiovanni \cite[p.~212]{degiovanni_cauchy_1999}). 
We also have \reff{int-e3a}
if the vector field $F$ is continuous 
(cf. \v Silhav\'y \cite[p.~85]{silhavy_2008},
Chen-Torres-Ziemer \cite[p.~284]{chen_gauss-green_2009},  
Comi-Payne \cite[p.~198]{comi-payne}; cf. also 
Chen-Comi-Torres \cite[p.~131]{chen-comi}). 
For general open and closed sets $\dom$ there are merely some
approximation results  
as stated in Proposition~\ref{dt-s1a} below
(cf. Schuricht \cite[p.~534]{schuricht_new_2007}, 
Schuricht \cite[p.~189]{schuricht_quaderni}, 
\v Silhav\'y \cite[p.~449]{silhavy_divergence_2009}, 
Chen-Comi-Torres \cite[p.~117-123]{chen-comi}). It turns out, however,
that even for $\dom$ with finite perimeter, 
$T_F$ needs not to be a Radon measure 
on the boundary of $\dom$ 
(cf. \v Silhav\'y \cite[p.~449]{silhavy_divergence_2009} for an example).   
In the general case we find several sufficient conditions for $T_F$ to be a
Radon measure on the boundary 
(cf. \v Silhav\'y \cite[p.~449]{silhavy_divergence_2009},
\v Silhav\'y \cite[p.~84]{silhavy_2008}, Chen-Comi-Torres 
\cite[p.~127-129]{chen-comi}). Conditions for the existence of an integrable
density are given in \v Silhav\'y \cite[p.~26]{silhavy_2005}.
Let us still mention that many results of \v Silhav\'y are worked out for the 
more general case where $F\in\cD\cM(\edom)$ is a vector valued
Radon measure such that the distributional divergence is also a Radon measure.  
We do not treat this generality in the present paper. But, due to the product
rule that $\tf F$ also belongs to $\cD\cM(\edom)$ for any 
$\tf\in\cW^{1,\infty}(\edom)$, it might be possibly to extend essential results
to that generality (cf. \v Silhav\'y 
\cite[p.~448]{silhavy_divergence_2009}). 

Let us come back to the linear functional $T_F$ and the underlying space $X$. 
As already mentioned, in previous investigations $X$ was somehow
always contained in $C(\ol\dom)$. Thus it was naturally
to ask how far $T_F$ is continuous on $C(\ol\dom)$ 
and, if this is the case, it is representable 
by a Radon measure on $\bd\dom$ as an element of the dual space. 
Since $T_F(f)$ merely depends on $f_{|\bd\dom}$, this seems to be an obvious 
strategy. The drawback is that not all $T_F$ can be represented by a
Radon measure on $\bd\dom$ and that there are rarely results beyond 
$\dom$ with finite perimeter.

Therefore let us discuss a different choice for $X$. If one looks at the
right hand side in \reff{int-e4}, the optimal pairing of
$F\in\cD\cM^1(\edom)$ seems to be with $f\in\woinf{\edom}$. In this case one
trivially gets that $T_F$ is a continuous linear functional on
$X=\woinf{\edom}$. Therefore a representation of $T_F$ as
dual element of $\woinf{\edom}$ is possible for all $F\in\cD\cM^1(\edom)$ and
all Borel sets $\dom\subset\edom$. But the question is how 
far this is reasonable. The dual of $\cW^{1,\infty}(\edom)$ 
can be identified with a product of dual spaces $\cL^\infty(\edom)^*$ and
it is known that  
$\cL^\infty(\edom)^*$ consists of $\cL^1(\edom)$ supplemented by certain finitely
additive measures. Since the integration theory for finitely additive
measures is commonly considered as 
not very powerful, 
one could hope that the functionals $T_F$ are in fact related to measures with
an $\cL^1$-integrable density. Notice that, in the simple case of a smooth $F$,
we can trivially identify $T_F$ with $(\divv F,F)\in\cL^1(\edom)^{n+1}$ 
as an element of the dual of $\woinf{\edom}$ (cf. \reff{int-e4}). 
But this is not what we want to get and it
gives no improvement. We are rather looking for some representation on or at
least near the boundary of $\dom$. 
Therefore let us choose $\delta>0$, let $\chi_\delta\in C^1_c(\edom)$ be
supported on the $\delta$-neighborhood $(\bd\dom)_\delta$ of $\bd\dom$ and
let $\chi_\delta=1$ on $(\bd\dom)_{\frac{\delta}{2}}$.
Then it is a standard result that
\begin{equation*}
  T_F(f) = T_F(\chi_\delta f) = T_{\chi_\delta F}(f)  \qmz{for all $f$\,.}
\end{equation*}
This motivates to replace $T_F$ with the functional 
$T_F^\delta=T_{\chi_\delta F}$ that merely
considers values of $f$ near $\bd\dom$. 
This way we get that $T_F$ can be localized near $\bd\dom$ and, similar as
above, we can identify $T_F^\delta$ with 
$(\divv (\chi_\delta F),\chi_\delta F)\in\cL^1(\edom)^{n+1}$ as dual element 
of $\cW^{1,\infty}(\edom)$. But such a representation always
depends on $\delta>0$
and it turns out that we cannot remove that dependence in general 
with dual elements belonging to $\cL^1(\edom)$. 

Hence let us briefly overcome
our preconception about finitely additive measures and let us take a brief
look at it. In continuum mechanics it is a simple observation that
contact interactions are naturally related to finite additivity
(cf. Schuricht \cite[p.~512]{schuricht_new_2007}). 
Even more, it seems that finite additivity is characterizing for short-range
phenomena (cf. Schuricht \cite{schuricht_quaderni}). Let us illuminate this by
a simple example. The density of a set $A\subset\R^n$ at point $x\in\R^n$
is an important standard tool in geometric measure theory
and it is given by  
\begin{equation*}
  \dens_xA = \lim_{\delta\downarrow 0}
  \frac{\lem(A\cap B_\delta(x))}{\lem(B_\delta(x))} \,
\end{equation*}
whenever the limit exists and where $B_\delta(x)$ is the open ball with 
radius $\delta>0$ centered at $x$.
Let us fix $x$ and consider $A\to\dens_xA$ as set function.
We readily see that it is additive for disjoint sets. 
Now we fix some $y\ne x$ and a sequence $B_k=B_{\delta_k}(y)$ of open balls 
with $\delta_k\uparrow |y-x|$. Then 
\begin{equation*}
  \dens_xB_k = 0 \qmq{for all $k$,} \text{but} \quad
  \dens_x\Big(\bigcup_{k\in\N} B_k\Big) = \frac{1}{2}\,,
\end{equation*}
which obviously prevents $\sigma$-additivity. We readily observe that we
merely have to know the intersection of $A$ with an arbitrarily small
neighborhood of $x$ for the determination of $\dens_xA$. But notice that
$\dens_x\{x\}=0$. Hence, roughly speaking, $\dens_x$ lives near $x$ but it is
not supported at the point $x$. By a simple Hahn-Banach argument $\dens_x$ can
be extended, though not uniquely, to a finitely additive measure on all Borel
sets. It turns out that this example is typical for a certain class of finitely 
additive measures $\me$. They can be characterized by the fact that there is  
a decreasing sequence of sets $A_k$ such that
\begin{equation*}
  \lem(A_k)\to 0 \qmq{and} |\me|(A_k^c)=0 \zmz{for all} k\in\N
\end{equation*}
where $|\me|$ is the total variation of $\me$ and $A_k^c$ denotes the
complement of $A_k$. 
But how far can that be useful for traces of functions. 
First we readily observe some similarity of traces to the measure $\dens_x$ so
far that the computation of the trace $\tilde f(x)$ of e.g. a Sobolev function
$f$ requires the values of $f$ merely in an arbitrarily small neighborhood of
$x$. For a more precise look we observe 
that one can define an integral for the finitely
additive measure $\dens_x$ similar to the usual integral for $\sigma$-additive
measures by, roughly speaking, just replacing ``convergence a.e.'' with 
``convergence in measure'' in the definition. 
Then, for any $f\in L^1(\dom)$ with $\dom$ open
and for any Lebesgue point $x\in\dom$, we obtain that $f$ is
$\dens_x$-integrable with 
\begin{equation*}
  \I{\dom}{f}{\dens_x} = f(x) \,.
\end{equation*}
Moreover, let $f\in\cW^{1,1}(\dom)$ be a Sobolev function on some open $\dom$
with Lipschitz boundary and let $\dens_x^\dom$ be a measure similar to
$\dens_x$ but, for $x\in\bd\dom$, giving the density within $\dom$
(cf. \reff{pm-e11} below). 
Then, for a trace function $\tilde f\in L^1(\bd\dom,\hm)$ of $f$, we have 
that
\begin{equation*}
  \I{\dom}{f}{\dens_x} = \tilde f(x) \qmq{for $\hm$-a.e. $x\in\bd\dom$\,}
\end{equation*}
(cf. Proposition~\ref{pm-prop7} below). Hence we can compute the
trace $\tilde f(x)$ by an integral over $\dom$ instead of the usual limit of
mean values. But this means that integrals for finitely additive measures
not only provide an integral calculus for traces but also 
for the precise representative of an $\cL^1$-function. These observations
suggest that finitely additive measures are a
very convenient and natural tool for the treatment of traces.

In the present paper we thus develop a general theory of traces that relies 
on the dual of $\cL^\infty(\edom)$. More precisely, 
we understand traces as linear continuous functionals and we focus on the
special case of functionals on $\cL^\infty(\edom)$ or $\cW^{1,\infty}(\edom)$. 
This way we derive Gauss-Green formulas for any $F\in\cD\cM^1(\edom)$ on any
Borel set $\dom\subset\edom$. In the most general case we get for each
$\delta>0$ the existence of a scalar measure $\lF\in\cL^\infty(\edom)^*$ and a
vector measure $\mF\in\cL^\infty(\edom,\R^n)^*$ such that
\begin{equation*}
  \df{T_F}{\tf} = \I{\dom}{\tf}{\divv F} + \I{\dom}{F\cdot D\tf}{\lem} = 
  \I{(\bd\dom)_\delta}{\tf}{\lF} + \I{(\bd\dom)_\delta}{D\tf}{\mF} 
\end{equation*}
for all $\tf\in\cW^{1,\infty}(\edom)$ (cf. Proposition~\ref{dt-s1}). 
If the functional $T_F$ is finite in
some sense (cf. \reff{dm-monster}), then we can remove the dependence 
on $\delta$ and we write
\begin{equation} \label{int-e11}
  \df{T_F}{\tf} = \I{\dom}{\tf}{\divv F} + \I{\dom}{F\cdot D\tf}{\lem} = 
  \sI{\bd\dom}{\tf}{\lF} + \sI{\bd\dom}{D\tf}{\mF} 
\end{equation}  
where $\sIsymb_{\bd\dom}$ means that we have to integrate over any small
neighborhood of $\bd\dom$. In this case we typically have that 
the measures $\lF$, $\mF$ are merely finitely additive. 
This Gauss-Green formula precisely accounts for boundary points belonging to
$\dom$, which is important in the case of concentrations of $\divv F$ on
$\bd\dom$. This way we can exactly treat any Borel set $\dom$ with 
$\Int\dom\subset\dom\subset\ol\dom$. We also show that the
second term on the right hand side cannot be neglected in general
(e.g. if $T_F$ restricted to $\tf\in C(\ol\dom)$ doesn't correspond to a 
Radon measure). If $T_F$ can be
extended to a linear continuous functional on $\cL^\infty(\edom)$ (i.e. if 
$T_F$ is continuous with respect to the $\cL^\infty$-norm), then we can choose
$\mF=0$ (cf. Proposition~\ref{dt-s2}). However the choice of $\lF$, $\mF$ is
not unique in general. We provide examples where $\mF=0$ in \reff{int-e11}
is possible but, alternatively, also $\lF=0$ with some nontrivial $\mF$
can be chosen. If $\mF=0$ and $\dom\csubset\edom$, then all 
$\tf\in\cW^{1,\infty}(\edom)$ belong to $C(\ol\dom)$ and we can identify 
$\lF$ with a Radon measure supported on $\bd\dom$. This way we somehow recover
previous results. But notice that we typically have to integrate over $\dom$
near $\bd\dom$ and not over $\bd\dom$. This raises the question how far a 
Radon measure or a trace function on $\bd\dom$ is really needed. 

Though the measures $\lF$ and $\mF$ are not unique in general, they have to be 
linear in $F$ as a whole and, of course, some explicite dependence of the
``boundary integrals'' on $F$ would be desirable. 
This turns out to be possible for $\dom$ with finite perimeter if we can 
choose $\mF=0$ and if $F$ is appropriate (e.g. essentially bounded).
However we do not intend to get
a Radon measure on the measure theoretic boundary in this case. Therefore we
replace the usually used pointwise normal field $\nu^\dom$ by some 
finitely additive normal measure $\nu^\dom$ that might be constructed e.g. 
by means of the signed distance function of $\bd\dom$. 
We provide several choices for $\nu^\dom$, e.g. related to an open $\dom$ or
to a closed $\dom$. If $F$ is $\nu^\dom$-integrable and satisfies some
compatibility condition (cf. \reff{nm-s5-3}), then we have that
\begin{equation} \label{int-e12}
  \df{T_F}{\tf} = \I{\dom}{\tf}{\divv F} + \I{\dom}{F\cdot D\tf}{\lem} = 
  \sI{\bd\dom}{\tf F}{\nu^\dom} \,.
\end{equation}
If $\dom$ is open, then $\nu^\dom$ only uses values of $\tf F$ 
near $\bd\dom$ inside $\dom$ and if it is closed, then only  
values outside of $\dom$ are used. In \reff{int-e12} one can even 
include a certain weight function $\chi$ for $\divv F$ 
on $\bd\dom$ (cf. \reff{nm-s7-1}). 
Since $\nu^\dom$ belongs to $\cL^\infty(\edom,\R^n)^*$, all essentially
bounded $F\in\cD\cM^\infty(\edom)$ are trivially $\nu^\dom$-integrable and they
also satisfy the additional compatibility condition. In addition, 
\reff{int-e12} is also applicable to certain unbounded vector fields $F$.
For normal measures based on the signed distance function of $\bd\dom$ we get
even more structural information for the boundary term. 
If e.g. $\dom$ is open and satisfies some mild perimeter bound 
(cf. \reff{nm-int-1}), we can use a suitable normal measure to derive the 
more explicite form
\begin{equation*}
  \df{T_F}{\tf} = \I{\dom}{\tf}{\divv F} + \I{\dom}{F\cdot D\tf}{\lem} = 
  \sI{\bd\dom}{\tf F\cdot\nu^\dom}{\dens_{\bd\dom}^{\rm int}} 
\end{equation*}
where $\nu^\dom$ is the normal field on $\dom$ given by the gradient of the
signed distance function for $\bd\dom$ and $\dens_{\bd\dom}^{\rm int}$ is, 
roughly speaking, a
finitely additive extension of the Radon measure $\reme{\hm}{\mbd\dom}$.
This version of the Gauss-Green formula is quite close to the usual form as
integral on $\mbd\dom$. The advantage here is that we do not need an
explicit trace function of $F$ on $\bd\dom$, since the knowledge of $F$ 
for $\lem$-a.e. point on a
neighborhood of $\bd\dom$ is sufficient to compute the integral.  
Moreover, in contrast to the measure theoretic boundary, the formula uses the
topological boundary $\bd\dom$ that takes into account also  
inner parts of the boundary. This is desirable for the treatment of cracks
and shocks. Notice that in the case where $\dom=\edom$ is open, the functions 
in $\cW^{1,\infty}(\edom)$ belong to $C(\dom)$ but they do not need to be in
$C(\ol\dom)$. Therefore it is possible to 
change $\tf$ independently on both sides of some inner crack or shock which 
allows a precise description of the situation on each side separately. 

The results can be easily transferred to Gauss-Green formulas for
Sobolev and BV functions $f$ with test functions
$\tf\in\cW^{1,\infty}(\edom,\R^n)$. For open $\dom$ with Lipschitz boundary  
we supplement the classical Gauss-Green formula \reff{int-e2} with 
\begin{equation} \label{int-e14}
  \I{\dom}{\big(f\divv\tf+\tf\cdot Df\big)}{\lem} = 
  \sI{\bd\dom}{f\tf}{\nu^{\rm int}} 
\:=\: 
  \sI{\bd\dom}{f\tf\cdot\nu^\dom}{\dens^{\rm int}_{\bd\dom}}
\end{equation}
where $\nu^{\rm int}$ is some normal
measure and $\dens_{\bd\dom}^{\rm int}$ is a measure as above. 
Notice that, similar as for vector fields, we do not need an explicit
trace function for $f$ on the boundary. One could say that its computation 
is somehow incorporated into the integral. 
As application of the trace theory we finally show
for some boundary value problem containing the $p$-Laplace operator
that there exists a weak solution for any bounded open set $\dom$.

Let us now briefly sketch how the paper is organized. 
In Section~\ref{pam} we start with some rough introduction to the 
integration theory for finitely additive measures, since it is 
not so well known. For that we summarize material that is widely spread
around in the book of Bhaskara Rao \& Bhaskara Rao \cite{rao}
and that cannot be found somewhere else in that compact form.
But notice that this survey is far away from providing all results of the
integration theory that are needed for the subsequent investigation.
This material is supplemented by some typical examples and by some new
results that are used later. Let us still add some warning: 
{\it We use the notion measure for any finitely additive measure while
$\sigma$-additivity is indicated by the notion $\sigma$-measure. 
Moreover, we orient our terminology to that typically used in measure theory
and, that way, we substantially deviate from the terminology used in
the underlying book \cite{rao}.} 

Section~\ref{tt} presents a general approach to traces on arbitrary sets. 
In Section~\ref{tt-gt} we first define traces as certain linear continuous 
functionals. Then we 
provide a simple but important class of trace functionals 
over $\cL^\infty(\edom)$. They are needed for later use, but initially 
they also serve for illustration. In Section~\ref{tt-tvf} we show that
$T_F$ from \reff{int-e4} is a trace functional on $\bd\dom$ over
$\cW^{1,\infty}(\edom)$ for any $F\in\cD\cM^1(\edom)$ and any Borel set 
$\dom\subset\edom$. We also get an analogous 
result for Sobolev functions in $\cW^{1,1}(\edom)$ and for BV functions in
$\cB\cV(\edom)$. Section~\ref{tt-r} is devoted to the representation of such
traces by means of measures that are ``living'' near $\bd\dom$. Here we
distinguish three variants of generality called (G), (L), and (C). 
Theorem~\ref{prop:dual} and several necessary and sufficient
conditions for certain special cases are the basis for the subsequent
Gauss-Green formulas. Some examples illuminate the spirit behind the
three variants. 

Section~\ref{dt} presents several divergence theorems or, more generally,
several Gauss-Green formulas. In Section~\ref{sec:ub} we start with
Theorem~\ref{dt-s1} that provides the Gauss-Green formula for any   
$F\in\cD\cM^1(\edom)$ and any Borel set $\dom\subset\edom$ and also covers
some special cases. Later other special cases are considered. 
Typical examples show the variability and applicability of the results.
The special case of normal measures is considered in Section~\ref{nm}.
The definition and construction of normal measures is followed
by some general integrability condition. This leads to Gauss-Green formulas
where $F$ and partially also the normal field $\nu^\dom$ are explicitly
contained in the boundary term and where, in addition, some weight on
$\bd\dom$ can be included. Several examples illustrate the variety of
applications. The relation of the new results to previous results from the
literature and some comprehensive discussion conclude that section. 
In Section~\ref{sf} we briefly transfer the former results  
for vector fields to Sobolev and BV functions for completeness, but also for
the convenience of the reader, since it might be not completely
straightforward to do that. 
In addition we study a Sobolev function on a set $\dom$ of finite perimeter
where the trace functional is not related to a Radon measure on $\bd\dom$.
Here it turns out that both boundary integrals are needed for a general 
Gauss-Green formula and we explicitly compute the related measures. 
We also show \reff{int-e14} for a bounded open $\dom$ with Lipschitz boundary.
Finally, for any bounded open set $\dom$ the existence of a weak solution 
for a general boundary value problem is shown.

Summarizing we can say that finitely additive measures appear to be a natural
tool for short range phenomena like traces. Though the underlying integration
theory was already known for many decades, the bridge to a relevant 
application seemed to be hidden.  
Therefore we hope that the new results can somewhat contribute to wake up this
``sleeping beauty''.   
The key observation for our investigation
was to consider the density $\dens_xA$ as a set function and to realize that 
the related integral gives the precise representative $\lem$-a.e.
The presented results can certainly be extended to localized spaces 
$\cD\cM_{\rm loc}(\edom)$, $\cB\cV_{\rm loc}(\edom)$ etc. by using compactly
supported test functions $\tf$. But we refrain from formulating the results 
in that generality in order to avoid unnecessary technicalities 
for this new approach. Notice that some of the results can already be found in 
\cite{schonherr_diss}, \cite{schonherr_gauss_2017}, 
and \cite{schonherr_pure_2017}.
Finally we wish to express our deep gratitude to Eberhard Zeidler and Stuart
Antman for their inestimable support and 
the profound scientific stimulation to 
the second author.

\medskip

\textit{Notation}.
For real numbers we use $\ol\R=[-\infty,+\infty]$ and
$\R_{\ge0}=[0,\infty)$. For $a\in\R^n$ the $p$-norm is $|a|_p$ and 
$|a|=|a|_2$. 
We denote by
$\al$ an algebra of subsets of a set $\edom$, by $\sal$ a $\sigma$-algebra on 
$\edom$, by $\bor{\edom}$ the Borel subsets of $\edom$, 
and by $\cP(\edom)$ the power set of $\edom$. We write $\als^c$ 
for the complement of a set $\als$ and $\chi_\als$ for its characteristic
function. For $\dom\subset\R^n$ we use $\op{int}\dom$, $\op{ext}\dom$,
$\bd\dom$, and $\cl\Omega$ for its topological interior, exterior, boundary,
and closure, respectively. The corresponding measure theoretic quantities are
denoted by $\mint\dom$, $\mext\dom$, $\mbd\dom$ while $\rbd\dom$ stands for
the reduced boundary. The signed distance function $\distf\dom\!:\R^n\to\R$
for $\dom$ ($\ne\emptyset,\R^n$) is given by
$\dist{\dom}{x}=\pm\inf_{y\in\bd\dom}|x-y|$ if
$x\,\substack{\not\in\\\in}\,\dom$. Then we define the 
open $\delta$-neighborhood of $\dom\subset\R^n$ for all $\delta\in\R$
by $\dom_\delta=\{x\in\R^n\mid \dist{\dom}{x}<\delta\}$ and 
$\dom\csubset U$ indicates that $\dom$ is compactly contained in $U$. 
The open ball of radius $r$ centered at $x$ is $B_r(x)$ and we set
$B_r^A(x):= B_r(x)\cap A$. By $\normal{\dom}$ we denote either the usual
outward unit normal of $\dom$ on the boundary or the normal field given
by the gradient $D\op{dist}_\dom$. The Lebesgue measure on $\R^n$ is 
$\lem$ and the $k$-dimensional Hausdorff measure is $\ham^k$.
We write $\op{Per}(\dom)$ for the perimeter of $\dom$.
We typically use $\me$, $\lme$ for a (finitely additive) measure
and $\sigma$ for a $\sigma$-additive measure. 
$\me^\pm$ stands for the positive or negative part, 
$|\me|=\me^++\me^-$ for the total variation, $|\me|^*$ for the associated 
outer measure, $\me_p$ for the pure part, $\me_c$ for the 
$\sigma$-additive part, and $\reme{\me}{\als}$ for the restriction to 
the set $\als$. By $\me\wac\lme$ ($\ac$, $\sing$, $\ssing$)
we indicate that $\me$ is weakly
absolutely continuous with respect to $\lme$ (absolutely continuous,
singular, strongly singular). $\baa$, $\caa$, and $\paa$ stands for 
the space of bounded measures on the algebra $\al$ of subsets of $\dom$
that are additive, countably additive, or pure.
$\baaw{\lme}\subset\baa$ is the subset of measures $\me$ with 
$\me\wac\lme$. The notion Radon measure is merely used for $\sigma$-measures 
in the usual sense and $\cM(\dom)$ stands for the set of Radon measures. 
For the support of a $\sigma$-measure we write $\supp\sigma$.
By $\baaw{\lme}^m$ and $\cM(\dom)^m$ we mean vector-valued measures. 
Let $L^p(\dom,\al,\mu)$ denote the space of $p$-integrable
functions on $\dom$ with respect to $\me$ and let $\lpam{p}$ be the
corresponding set of equivalence classes. We write 
$\Lpbm{p}$ for $L^p(\dom,\bor{\dom},\me)$ and $L^p(\dom)$ for 
$L^p(\dom,\bor{\dom},\lem)$ and we use $p'$ for the Hölder conjugate of $p$.  
Let $C(\dom)$, $C(\ol\dom)$, and $C_c(\dom)$ 
denote the spaces of continuous functions on $\dom$, with continuous extension
on $\ol\dom$, and with compact support. With 
$C^k(\dom)$, $C^k(\ol\dom)$, and $C^k_c(\dom)$ we denote the corresponding
spaces with continuous derivatives up to order $k$. In particular 
$C_c^\infty(\dom)$ is the usual space of test functions. 
$\op{Lip}(\dom)$ denotes the Lipschitz continuous functions on $\dom$.
$W^{k,p}(\dom)$ stands for the Sobolev space of $p$-integrable functions with
$p$-integrable weak derivatives up to order $k$ and 
$BV(\dom)$ stands for the set of functions of bounded variation.
$\cW^{k,p}(\dom)$ and $\cB\cV(\dom)$ are the corresponding spaces of
equivalence classes. $\cW^{k,p}_0(\dom)$ is the 
completion of $C_c^\infty(\dom)$ within $\cW^{k,p}(\dom)$.
By $C(\dom,\R^m)$, $L^p(\dom,\R^m)$ etc. we mean functions mapping into $\R^m$.
If $f\in X$ according to the context, then $\|f\|$ means
$\|f\|_X$ (several important norms used can be found at the beginning of
Section~\ref{tt-tvf}). 
For the precise representative of an $\cL^p$-function $f$ we write $\pr{f}$.
Let $(f)_{x,r}=\mI{B_r(x)}{f}{\lem}$ where $\mIsymb$ is the mean value
integral (formally $\mI{M}{f}{\lem}=0$ if $\lem(M)=0$). 
We take $f_{|\als}$ for the restriction of $f$ to $\als$ and $\supp f$ for
the support of $f$. By $\eta_\eps$ we mean the symmetric standard mollifier
supported on $B_\eps(0)$.

\section{Preliminaries about measures}
\label{pam}

\subsection{Measures and integration}
\label{pm}

Let us first provide some material about (finitely
additive) measures, as needed for our analysis, that is mostly taken
from Bhaskara Rao \& Bhaskara Rao \cite{rao},
Schönherr \cite{schonherr_diss}, and 
Schönherr-Schuricht \cite{schonherr_gauss_2017}, 
\cite{schonherr_pure_2017} 
(notice that some terminology in \cite{rao} differs).

Let $\dom$ be a set and let $\al$ be an \textit{algebra} on $\dom$, i.e.
a collection of subsets of $\dom$ containing 
$\emptyset$, $\Omega$ and being closed under complements and finite
unions and intersections.
At variance with common usage we call a
set function $\me:\al\to\ol\R$ a \textit{measure} on $(\dom,\al)$
if it is finitely additive, i.e. 
\begin{equation*}
  \me\bigg(\bigcup_{k=1}^m \als_k\bigg) = \sum \limits_{k=1}^m \me(\als_k) 
\end{equation*}
for all pairwise disjoint $\als_k\in\al$.
We call $\me$ a
\textit{$\sigma$-measure} on $(\dom,\al)$ if it is $\sigma$-additive, i.e.
\begin{equation*}
  \me\bigg(\bigcup_{k=1}^\infty \als_k\bigg) = \sum_{k=1}^\infty \me(\als_k) 
\end{equation*}
for any sequence $\{\als_k\}$ of pairwise disjoint $\als_k\subset\al$ 
with $\bigcup_{k=1}^\infty \als_k \in \al$
(usually we denote general measures by $\me$, $\lme$ and $\sigma$-measures
by $\sigma$). Notice that a measure cannot attain both values $\pm\infty$
and that always $\me(\emptyset)=0$ if $\me$ is finite somewhere.
We say that $\me$ is \textit{positive} if $\me(\als)\ge 0$ for all
$\als\in\al$. The \textit{positive} and \textit{negative part}
$\me^\pm:\al\to\ol\R_{\ge0}$ of measure $\me$ given by 
\begin{equation*}
  \me^\pm(\als) := \sup \big\{ \pm\me(\bals) \mid
                   \bals\subset\als,\; \bals\in\al \big\}
\end{equation*}
and the \textit{total variation} $|\me|:\al\to\ol\R_{\ge0}$ of $\me$  
\begin{equation*}
  |\me|:=\me^++\me^- \,
\end{equation*}
are positive measures (cf. \cite[p.~53, 85]{rao}).
A measure $\me$ is \textit{pure} if one has
for any $\sigma$-measure $\sme:\al\to\R$ that
\begin{equation*}
  0 \leq \sme \leq |\me| \qmq{implies} \sme=0 
\end{equation*}
(cf. \cite[p.~240]{rao}).
Thus a (nontrivial) pure measure cannot be extended to a
\mbox{$\sigma$-measure} on $\al$, but notice that a non $\sigma$-additive
measure 
need not be pure. With $\me$ also $\me^\pm$ and $|\me|$ are pure.
The \textit{outer measure} $\me^*:\cP(\dom)\to[0,\infty]$ of a
positive measure $\me$ given by
\begin{equation*}
  \me^*(\als) := \inf_{\substack{\als\subset\bals\\\bals\in\al}}\:\me(\bals) \, 
\end{equation*}
is finitely subadditive but not necessarily $\sigma$-subadditive
(cf. \cite[p.~86]{rao}).  
$\als\subset\dom$ is a \textit{null set} if $|\me|^*(\als)=0$.
A measure~$\me$ is \textit{bounded} if 
\begin{equation*}
  \sup_{\als \in \al} |\me(\als)| < \infty  
\end{equation*}
and it is \textit{bounded above} (\textit{below}) if $\me^+$ ($\me^-$) is
bounded. 
We have $\me=\me^+-\me^-$ if $\me^+$ or $\me^-$ is bounded. Let us set
\begin{eqnarray*}
  \baa &:= &
  \{\me:\al\to\R \mid \text{$\me$ is a bounded measure}\}\,, \\ 
  \paa &:= &
  \{\me\in\baa \mid \me \tx{ is pure}\}\,,  \\
  \caa &:= &
  \{\sme\in\baa \mid \sme \tx{ is $\sigma$-additive}\}\,
\end{eqnarray*}
where $\baa$ is a Banach space with $\|\me\|:=|\me|(\dom)$
(cf. \cite[p.~44]{rao}). 
We call $\me\in\baa^m$ also \textit{vector measure} if $m>1$. 
As {\it total variation} of $\me\in\baa^m$ for $\als\in\al$ we define
\begin{equation}\label{pm-e3}
  |\me|(\als) := \sup 
  \Big\{\sum_{j=1}^k |\me(\als_j)|\:\Big|\: \als_j\in\al
  \text{ pairwise disjoint,}\: \bigcup_{j=1}^k \als_j=\als\Big\} \,.
\end{equation}
Then we have $|\me|\in\baa$ (argue as in the proof of Theorem 1.6 in
\cite{ambrosio}) 
and for $m=1$ this coincides with the previous definition (cf. 
\cite[p.~46]{rao}). Obviously $|\me|(\dom)$ is a norm on  
$\me\in\baa^m$ that we use as standard norm.

Measures $\me$,~$\lme$ on $(\dom,\al)$ are called \textit{singular}
($\me\sing\lme$) if
for all $\eps>0$ there is some $\als\in\al$ with
\begin{equation*}
  |\me|(\als)<\eps \qmq{and} |\lme|(\als^c)<\eps \,
\end{equation*}
and they are called \textit{strongly singular} ($\me\ssing\lme$)
if there is some $\als\in\al$ with 
\begin{equation*}
  |\me|(\als)=0=|\lme|(\als^c)\,.
\end{equation*}
While strong singularity implies singularity, equivalence is met for
$\sigma$-measures on a $\sigma$-algebra (cf. \cite[p.~165]{rao}).
Singularity also means orthogonality, i.e.
\mbox{$|\me|\wedge|\lme|=0$}, on the lattice of 
measures (cf. \cite[p.~52, 166]{rao}).
For general measures $\me$ we have
\begin{equation*}
  \me^+\sing\me^- \qmz{if $\me^+$ or $\me^-$ is bounded} 
\end{equation*}
(cf. \cite[p.~53]{rao}). Moreover,
\begin{equation*}
  \me_c\sing\me_p \qmq{if} \me_c\in\caa,\z \me_p\in\paa 
\end{equation*}
(cf. \cite[p.~240]{rao}).
As a consequence of Riesz's decomposition theorem for lattices
we have an important singular
decomposition of bounded measures (cf. \cite[p.~241]{rao}). 
\begin{proposition} \label{pm-s1}
  Let $\al$ be an algebra on $\dom$ and let $\me \in \baa$.
  Then there is a unique decomposition 
\begin{equation*}
  \me = \me_c + \me_p  \qmq{with} \me_c\in\caa\,, \; \me_p \in \paa
\end{equation*}
where we call $\me_c$ {\em $\sigma$-additive part} and $\me_p$ {\em pure part}
of $\me$.
\end{proposition}

Many examples of pure measures found in the literature are 
just measures on $\dom=\N$. 
Moreover they are either defined on a very small algebra
or they rely on some construction not allowing an explicit computation
of the measure even on simple sets
(cf. \cite[p.~246]{rao}, 
\cite[p.~57 ff.]{yosida_finitely_1951}).
Let us provide some simple but typical examples of pure measures.
\begin{example}\label{ex:puremn}   
For $\dom=\N$ and the algebra
$\al=\{\als\subset\dom\mid \als \tx{ or } \als^c \tx{ finite}\}$
we readily get a measure by  
\begin{equation*}
  \me(\als):= \bigg\{
  \begin{array}{ll} 0 & \tx{if $\als$ is finite,} \\
                    1 & \tx{if $\als^c$ is finite.} \end{array}
\end{equation*}
(cf.  \cite[Remark 4.1.5]{rao}).
For $\sigma\in\caa$ with $0\le\sigma\le\me$ we have
\begin{eqnarray*}
  0\le \sigma(\N)=\sum_{k\in\N}\sigma(\{k\})\le
  \sum_{k\in\N}\me(\{k\})=0
\end{eqnarray*}
and, hence, $\me$ is pure (cf. also \cite[Example 10.4.1]{rao}).
By $\me(\N)=1$ we also see that $\me$ is not $\sigma$-subadditive.
\end{example}

\begin{example}\label{ex:puremr}   
Let $\dom=(0,1)$ and let $\al$ be the algebra generated by all intervals
of the form $(a,b]\subset\dom$. Then we get a pure measure by
\begin{equation*}
  \me(\als):= \bigg\{
  \begin{array}{ll} 1 & \tx{if } B_\delta(0)\cap\als\ne\emptyset
                        \tx{ for all } \delta>0\,, \\
                    0 & \tx{otherwise} \end{array}
\end{equation*}
(cf. \cite[Example 10.4.4]{rao}). 
\end{example}

Notice that contact interactions in continuum mechanics 
naturally lead to pure measures having some similarity with that in the last 
example (cf. Schuricht \cite{schuricht_new_2007}). 
In geometric measure theory the density $\dens_x\als$ of set $\als$ at
point $x$ is a well-known important quantity. But it seems that it hasn't been
considered yet as a set function (for fixed $x$). It turns out to be a typical
example of pure measures.
\begin{example}\label{ex:dzero}   
Let $\dom := \Rn$, let $\al=\bor{\dom}$, and fix some $x\in\dom$.
The density of $\als\in\bor{\dom}$ at $x$ is given by
\begin{equation*}
  \dens_x\als := \lim_{\delta \downarrow 0}
  \frac{\lem(\als \cap \ball{x}{\delta})}{\lem(\ball{x}{\delta})}
\end{equation*}
if the limit exists and, as set function, it is disjointly additive on
these sets. Though not uniquely, one can extend $\dens_x$ to a positive
measure on $\bor{\dom}$ 
such that for all $\als\in\bor{\dom}$
\begin{equation} \label{ex:dzero-1}
  \liminf_{\delta \downarrow 0}
  \frac{\lem(\als \cap \ball{x}{\delta})}{\lem(\ball{x}{\delta})} \le
  \dens_x\als \le
  \limsup_{\delta \downarrow 0}
  \frac{\lem(\als \cap \ball{x}{\delta})}{\lem(\ball{x}{\delta})}
\end{equation}
(cf. Proposition~\ref{tt-s3} below with $\edom=\alse=\dom$, $\K=\{x\}$, 
$\gamma(\delta)=\lem(\ball{x}{\delta})$, $f\equiv 1$,
and $\phi=\chi_\als$). We call $\dens_x$ {\it density measure} at $x$.
For $\sigma\in\caa$ with $0\le\sigma\le\dens_x$ we get
\begin{equation*}
  0\le \sigma\big(\R^n\setminus\{x\}\big) =
  \lim_{\delta \downarrow 0}\,\sigma\big(\ball{x}{\delta}^c\big) \le 
  \lim_{\delta \downarrow 0}\,\dens_x \big(\ball{x}{\delta}^c\big) = 0 
\end{equation*}
and $0\le\sigma(\{x\})\le\dens_x(\{x\})=0$. Hence $\sigma=0$ and $\dens_x$ is
pure. Moreover $\dens_x\sing\lem$ but not $\dens_x\ssing\lem$.
\end{example}

For measures $\lme$, $\me$ on algebra $\al$, 
we call $\me$ \textit{\tac} 
with respect to $\lme$ if for every $\eps > 0$ there is some
$\delta > 0$ such that
\begin{equation*}
  |\me(\als)| < \eps \qmq{if} \als\in\al \zmz{with}
  |\lme|(\als) < \delta \z    
\end{equation*}
and we write $\me\ac\lme$. 
We call $\me$ \textit{\twac}
with respect to $\lme$ in the case where
\begin{equation*}
  \me(\als) = 0 \qmq{if} \als\in\al \zmz{with} |\lme|(\als) = 0
\end{equation*}
and we write $\me\wac\lme$. Clearly $\me\wac\lme$ if $\me\ac\lme$ 
and all notions coincide for 
positive bounded $\sigma$-measures on a $\sigma$-algebra
(cf. \cite[p.~159 ff.]{rao}). 
For any measure $\lme$ on $(\dom,\al)$ ($\lme$ not necessarily bounded) we
introduce the set 
\begin{equation*}
  \baaw{\lme} := \big\{\me\in\baa \mid \me\wac\lme \big\} \,
\end{equation*}
(which plays an important role in our analysis in the special case 
$\bawl{\dom}$). It turns out that 
the decomposition of $\me\in\baa$ stated in Proposition \ref{pm-s1} 
doesn't leave the set $\baaw{\lme}$
(cf. \cite[Theorem~3.16]{schonherr_pure_2017} for the case 
$\lme\in\baa$).
\begin{proposition} \label{pm-s2}
  Let $\al$ be an algebra on $\dom$, let $\lme$ be a positive measure on
  $(\dom,\al)$, and let
  $\me \in \baaw{\lme}$ with (unique) decomposition $\me=\me_c+\me_p$ 
into $\sigma$-additive and pure part according to 
Proposition~{\rm \ref{pm-s1}}. Then 
\begin{equation*}
  \me_c,\me_p \in \baaw{\lme} \,.
\end{equation*}
\end{proposition}
\noindent
The statement remains true for any (non-positive) $\lme\in\baa$.
The next characterization of pure measures in $\basws$ with $\sigma$ being a
$\sigma$-measure adumbrates their relation to traces
(cf. \cite[p.~244]{rao} and \cite[p.~56]{yosida_finitely_1951}).

\begin{proposition} \label{pm-s3}
  Let $\sal$ be a $\sigma$-algebra on $\dom$ and let $\sme\in\cas$ be
  positive. Then $\me \in \basws$ is pure if and only if there is a decreasing
  sequence $\{A_k\}\subset \sal$ such that 
\begin{equation*}
  \sme(\als_k) \to 0 \qmq{and}
  |\me|(\als_k^c) = 0 \zmz{for all}  k\in\N\,.
\end{equation*}
\end{proposition}
\noindent
Roughly speaking one "feels" a pure measure $\me\in\basws$ in
any small vicinity of a $\sme$-null set $\bigcap_{k\in\N} \als_k$.
We call $\als\in\al$ an \textit{aura} of measure $\me\in\baa$ if  
\begin{equation*}
  |\me|(\als^c) = 0\,.
\end{equation*}
For a pure measure $\me\in\basws$ a
decreasing sequence $\{\als_k\}$ of auras with $\sme(\als_k)\to 0$
is said to be an \textit{aura sequence} for $\me$.
Notice that the intersection of an aura sequence might be empty.
Hence the support as used for $\sigma$-measures on $\dom$ with a
topological structure (that identifies a preferably small set where the
complete mass of the measure concentrates,
cf. \cite[p.~30]{ambrosio}) might be empty for a
nontrivial measure. In fact, for the pure measure $\dens_x$
from Example~\ref{ex:dzero} the support would be $\{x\}$, but
\begin{equation*}
  \dens_x(\{x\})=0 \qmq{and}
  \dens_x(B_\delta(x)\setminus\{x\})=1 \zmz{for all} \delta>0\,. 
\end{equation*}
Notice that we can define $\dens_x$ merely on $\dom=\R^n\setminus\{x\}$
which leads to a support outside of $\dom$.
Thus we have to realize that the usual support cannot localize where a general
measure ``lives''.
For $\dom \subset M$ with a compact topological space~$M$ and an
algebra $\al$ containing all relatively open 
sets in $\dom$, we introduce 
the \textit{core} of $\me \in \baa$ as the set
\begin{equation*}
  \cor{\me} := \big\{ x\in M \;\big|\; |\me|(U\cap\dom)>0 \tx{ for all }
  U\subset M \tx{ open with $x\in U$} \big\} \,.
\end{equation*}
Obviously $\cor{\me}\subset\ol\dom$, it is closed in $M$,
and it need not be contained in $\dom$.
For $\mu\ne 0$ one has $\cor{\me}\ne \emptyset$ and
\begin{equation}\label{pm-e1}
  |\me|(U^c\cap\dom)=0 \qmq{for any open \; $U\subset M$ \; with \;
  $\cor{\me}\subset U$\,}
\end{equation}
(cf. Proposition \ref{acm-s0} below). From $\cor(\dens_x)=\{x\}$ we readily
see that a nonempty core might be a null set. For $\dom=\R^n$ or $\dom=\N$ 
the core is not defined, since $\dom$ is not compact (with the usual topology).
Notice that formally $\cor{\me}=\emptyset$ for the non-zero measure $\me$ from 
Example~\ref{ex:puremn} and we would sloppily say that $\me$  
``concentrates near~$\infty$''. We can describe such situations 
precisely by (tacitly) using the {com\-pacti\-fi\-cations}
\begin{equation*}
  \ol{\N}:=\N\cup\{\infty\} \qmq{and}
  \ol{\R^n}:=\R^n\cup\{\infty\} \,
\end{equation*}
in the definition of core. 
This way we get $\cor{\me}=\{\infty\}$ for $\me$ from Example~\ref{ex:puremn}
and, e.g., for any $\me\ne 0$ on $\R^n$ with $\me(A)=0$ for all bounded
$A\subset\R^n$ (cf. also \cite[Example 10.4.1]{rao}).
For $\dom\in\cB(\R^n)$ and $\me\in\bawl{\dom}$ we have that
\begin{equation*}
  \me \qmq{is pure if} \lem(\cor\me\cap\dom)=0\,.
\end{equation*}
(cf. Proposition~\ref{acm-s1} below). 
It turns out that core and aura are reasonable tools to describe where
the measure is concentrated. 

For a measure $\me$ on $(\dom,\al)$, a sequence of functions
$\f_k:\dom\to\R$ converges \textit{in measure} to $\f : \dom \to \R$
if
\begin{equation*}
  \lim_{k \to \infty} |\me|^* \big\{ x\in\dom \;\big|\;
  |\f_k(x) - \f(x)| > \eps \big\} = 0 \qmq{for all} \eps > 0
\end{equation*}
and we write $\f_k\convim{\me}\f$. The limit is unique if we identify
functions $\f,\g:\dom\to\R$ that agree \textit{in measure}
(i.m.) on $\dom$ in the sense that
\begin{equation*}
  |\me|^* \big(\big\{ x\in\dom \;\big|\; |\f(x)-\g(x)|>\eps \big\}\big) = 0 
  \text{ for every } \eps > 0
\end{equation*}
(cf. \cite[p.~88, 92]{rao}). The stronger condition
that $\f$ and $\g$ agree \textit{almost everywhere} (a.e.) on $\dom$, i.e. 
\begin{equation*}
  |\me|^* \big(\big\{ x\in\dom \;\big|\; |\f(x)-\g(x)|\ne 0 \big\}\big) = 0 \,,
\end{equation*}
is sufficient but not necessary for $\f=\g$ i.m. in general, however it is 
equivalent
in the case of a $\sigma$-measure $\me$ on a $\sigma$-algebra $\al$
(cf. \cite[Proposition~4.2.7]{rao}).
For demonstration we consider the measure $\mu=\dens_{\,0}$ 
from Example~\ref{ex:dzero}
and the functions $\f_c(x)=c|x|$ on $\R^n$ with $c\in\R$. Then 
$\f_c=\f_d$ i.m. on $\R^n$ for all $c,d\in\R$. But 
$\f_c$ and $\f_d$ do not agree a.e. on $\R^n$ for $c\ne d$, since 
\begin{equation*}
  |\me|^* \big(\big\{ x \;\big|\; |\f_c(x)-\f_d(x)|\ne 0 \big\}\big) =
  \me\big(\R^n\setminus\{0\}\big) = 1\,. 
\end{equation*}
For $f,g:\dom\to\R$ we say $f\le g$ \textit{in measure} (i.m.) on $\dom$ if
\begin{equation*}
  |\me|^* \big(\big\{ x\in\dom \;\big|\; \g(x)-\f(x) < -\eps \big\}\big) = 0 
  \qmq{for every} \eps > 0
\end{equation*}
(or, equivalently, $f\le g+h$ on $\dom$ for some $h$ with $h=0$ i.m. on
$\dom$, cf. \cite[p.~88]{rao}). 
The condition $\f\le\g$ \textit{almost everywhere} (a.e.) on $\dom$, 
i.e. 
\begin{equation*}
  |\me|^* \big(\big\{ x\in\dom \;\big|\; \g(x)-\f(x)<0 \big\}\big) = 0 \,,
\end{equation*}
is sufficient but not necessary for $\f\le\g$ i.m. in general.
Clearly $f=g$ i.m. if $f\le g$ and $g\le f$ i.m. on $\dom$
(cf. \cite[p.~88]{rao}).

We call $\simf:\dom\to\R$ \textit{simple function} related to measure $\me$ on
$(\dom,\al)$ if there are finitely many $c_k\in\R$ and $\als_k\in\al$ such that 
\begin{equation*}
  \simf = \sum_{k=1}^m c_k\chi_{\als_k} \qmq{on} \dom\,.
\end{equation*}
The simple function $\simf$ is \textit{integrable} (with respect to $\me$) if  
$|\me|(\als_k)<\infty$ whenever $c_k\ne 0$ and we set
\begin{equation*}
  \I{\dom}{\simf}{\me} := \sum_{k=1}^m c_k\me(\als_k) 
\end{equation*}
(with the convention $0\cdot\pm\infty=0$,
cf. \cite[p.~96 ff.]{rao}). 

We say that $\f:\dom\to\R$ is \textit{measurable} (with respect to $\me$) 
if there is a sequence of simple functions $\simf_k : \dom \to \R$ such that
\begin{equation*}
  \simf_k \convim{\me} \f \,.
\end{equation*}
Then $f$ is measurable if and only if for any $\eps>0$ there is a 
partition $\als_0,\dots,\als_m$ of $\dom$ in $\al$ such that
\begin{equation*}
  \me(\als_0)<\eps \qmq{and}
  |f(x)-f(x')|<\eps \zmz{for all} x,x'\in\als_k\,, \; k=1,\dots,m\,
\end{equation*}
(cf. \cite[p.~101]{rao}).

We call $\f:\dom\to\R$ \textit{integrable} (or $\me$-integrable) on $\dom$ 
if there is a sequence of integrable simple functions
$\simf_k:\dom\to\R$ such that
\begin{equation*}
  \simf_k \convim{\me} \f \qmq{and}
  \lim_{k,l \to \infty} \int_\dom |\simf_k - \simf_l| \, d|\me| = 0 \,.
\end{equation*}
In this case $\f$ is measurable and we define the \textit{integral} of $\f$ on
$\dom$ as 
\begin{equation*}
  \I{\dom}{\f}{\me} := \lim\limits_{k \to \infty} \I{\dom}{\simf_k}{\me} \,
\end{equation*}
while $\{\simf_k\}$ is called a \textit{determining sequence}
for it (cf. \cite[p.~104]{rao}). The integral is linear and
integrability of $\f$ is
equivalent to that of $|\f|$ (cf. \cite[p.~113]{rao}).
An integrable $\f$ is also integrable with respect to $|\me|$ and it is
integrable on any $\als\in\al$. 
Integrable functions
$\f,\g:\dom\to\R$ agree i.m. on $\dom$ if and only if
\begin{equation*}
   \I{\als}{\f}{\me} =  \I{\als}{\g}{\me} \qmq{for all} \als\in\al \,.
\end{equation*}
If $f$ is integrable, then 
\begin{equation} \label{pm-e5}
  \lme(\als)=\I{\als}{f}{\me} \qmq{for} \als\in\al
\end{equation}
gives a measure $\lme\in\baa$ that is absolutely continuous with respect to
$\me$. For more basic properties we refer to \cite[p.~105-107]{rao}.
If $f_k,g:\dom\to\R$ are integrable with
\begin{equation*}
  f_k \convim{\me}: f \qmq{and} |f_k|\le g \;\zmz{i.m. on}\dom\,,
\end{equation*}
we have dominated convergence, i.e.
\begin{equation*}
  \lim_{k \to \infty} \I{\dom}{f_k}{\me} = \I{\dom}{f}{\me} 
\end{equation*}
(cf. \cite[p.~131]{rao}). For a vector measure
$\me=(\me_1,\dots,\me_m)$ a scalar function $f:\dom\to\R$ is said to be
$\me$-integrable if $f$ is $\me_j$-integrable for all $j$. A vector-valued
function $F=(F_1,\dots,F_j):\dom\to\R^m$ is said to be $\me$ integrable if
each $F_j$ is $\me_j$-integrable and we set
\begin{equation*}
  \I{\dom}{F}{\me} := \sum_{j=1}^m \I{\dom}{F_j}{\me_j}\,. 
\end{equation*}

For the integral $\I{\dom}{\f}{\me}$ it is sufficient to 
integrate on an aura $A\subset\dom$ of $\me$,
but it is not enough to integrate on $\cor{\me}$ (if it is defined). 
For $\dens_x$ from Example~\ref{ex:dzero}, e.g., we have 
for $\f$ continuous and $\delta>0$ that 
\begin{equation*}
  \I{B_\delta(x)}{\f\,}{\dens_x} = \f(x) \qmq{but}
  \I{\{x\}}{\f\,}{\dens_x} = 0 \,
\end{equation*}
(where $B_\delta(x)$ is an aura and $\{x\}$ the core of $\me$).
In order to indicate a more precise information about the domain of
integration, we use the notation (if $\cor\me$ is defined)
\begin{equation*}
  \sI{C}{\f}{\me} := \I{U\cap\dom}{\f}{\me}  \qmq{if}
  \cor{\me} \subset C \subset U\,, \z U \tx{ open}\,,
\end{equation*}
which is well-defined by \reff{pm-e1}.

Notice that the usual notion of measurable functions, based on a 
$\sigma$-measure $\sme$ on $(\dom,\sal)$ with $\sigma$-algebra $\sal$, 
relies on convergence a.e. and it is weaker than the one used here in general,
but both agree if $\sme(\dom)<\infty$. 
Nevertheless integrability and the integral as introduced here always agree
with the usual $\sigma$-variant, since usual convergence in $\cL^1(\dom)$
implies convergence in measure due to Chebyshev's inequality. 

For a measure $\me$ on $(\dom,\al)$ and $1\le p <\infty$ we define 
\begin{equation*}
  \Lpam{p} :=
  \big\{ \f:\dom\to\R \;\big|\; f \tx{ measurable}, \;
                              |\f|^p \tx{ integrable} \big\} 
\end{equation*}
with
\begin{equation*}
  \norm{p}{\f} := \bigg(\I{\dom}{|\f|^p}{|\mu|}\bigg)^{\frac{1}{p}}
\end{equation*}
and
\begin{equation*}
  \Lpam{\infty} := \big\{ \f:\dom\to\R \;\big|\; \tx{$\f$ measurable},\;
  \norm{\infty}{\f}<\infty \big\} 
\end{equation*}
where
\begin{equation*}
  \norm{\infty}{\f} := \essup{\dom}{|\f|} :=
  \inf_{\substack{N\subset\dom \\ |\me|^*(N)=0}} \;
                    \sup_{\dom\setminus N} \: |\f| \,.
\end{equation*}
For vector-valued functions $F\in\Lpam{p}^m$ we replace $|f|$ with the
Euclidian norm $|F|$ in the previous definitions.
The corresponding sets of equivalence classes with respect to equality
i.m., denoted by
\begin{equation*}
  \lpam{p} \,, 
\end{equation*}
are normed spaces but they are not complete in general
(cf.  \cite[p.~125 ff.]{rao}). 

With $\f\in\lpam{p}$ we get a measure $\f\me\in\baa$, $\f\me\ac\me$ by
\begin{equation}\label{pm-e-fm}
  (\f\me)(\als) = \I{\als}{\f}{\me} \qmq{for all} \als\in\al \,.
\end{equation}
For positive $\me\in\baa$ the completion of $\lpam{p}$ for $p\in[1,\infty]$ is  
\begin{equation*}
  \llp^p(\dom,\al,\me) := \big\{ \lme\in\baa \;\big|\;
  \lme\ac\me\,,\;\|\lme\|_p<\infty \big\}
\end{equation*}
with
\begin{equation*}
  \|\lme\|_p^p =  \lim_{\partition \in \partitions}  
  \sum_{\substack{\als\in \partition\\\me(\als) \neq 0}}
  \Big|\frac{\lme(\als)}{\me(\als)}\Big|^p \me(\als) =
  \RI{\dom}{\Big| \frac{\lme}{\me}\Big|^p \me}
  \qquad \big(1\le p < \infty\big)\,,
\end{equation*}
where $\RI{}$ is the refinement integral (cf. 
\cite{kolmogoroff_untersuchungen_1930},
\cite[p.~231]{rao}) and 
the limit is taken in the sense of nets over the directed set $\partitions$ of
all finite partitions $\partition\subset\al$ of $\dom$, and  
\begin{equation*}
  \|\lme\|_\infty := \sup\bigg\{ \Big|\frac{\lme(\als)}{\me(\als)}\Big|
  \;\bigg|\; \als\in\al \bigg\}  
\end{equation*}
(use convention $\tfrac{0}{0}=0$, cf. \cite[p.~185]{rao}). 
Hölder's inequality is satisfied in all spaces for $1\le p\le\infty$
(cf. \cite[p.~122]{rao})
and the integrable simple functions are dense in all spaces for $1\le
p<\infty$ (cf. \cite[p.~132]{rao}). We briefly write
\begin{equation*}
  \Lpbm{p}:= L^p(\dom,\bor{\dom},\me)\,, \quad
  L^p(\dom):= L^p(\dom,\bor{\dom},\lem)  \,,
\end{equation*}
for vector-valued functions we write
\begin{equation*}
  L^p(\dom,\bor{\dom},\me)^m\,,\quad  \Lpbm{p}^m\,, \quad
  L^p(\dom,\R^m):=L^p(\dom)^m \,.
\end{equation*}
and we use and analogous notation for the $\cL^p$-spaces.

The dual of $\liss$ can be identified with $\basws$ and 
plays an important role in our analysis
(cf. \cite[p.~139, 140]{rao} or also
\cite[p.~53]{yosida_finitely_1951},
\cite[p.~118]{yosida_book}). Let us formulate a vector-valued version as
needed later. The more refined decomposition, that 
makes precise how $\lpss{1}$ as subspace of the dual has to be supplemented,  
relies on Proposition \ref{pm-s2}.

\begin{proposition}  \label{pm-s5}
  Let $\sigma$ be a positive $\sigma$-measure on $(\dom,\sal)$ with
  $\sigma$-algebra $\sal$. Then
\begin{equation*}
  \big(\liss^m\big)^* = \basws^m
\end{equation*}
if we identify $\de\in\big(\liss^m\big)^*$ and $\me\in\basws^m$ by 
\begin{equation*}
  \df{\de}{\tf} = \I{\dom}{\tf}{\me} \quad\qmq{for all} \tf\in\liss^m
\end{equation*}
where  
\begin{equation} \label{pm-s5-1}
  \norm{}{\de} = \norm{}{\me} = |\me|(\dom) \,
\end{equation}
and
\begin{equation} \label{pm-tva}
  |\me|(\als) = \sup_{\substack{\tf\in\liss^m\\\|\tf\|_\infty\le 1}} 
  \I{\als}{\tf}{\me} \qmz{for all} \als\in\al^\sigma\,.
\end{equation}
Moreover, a measure $\me\in\basws^m$ can be decomposed uniquely with some pure
$\me_p\in\basws^m$ and some $h \in \lpss{1}^m$ such that
\begin{equation*}
  \I{\dom}{\f}{\me} = \I{\dom}{\f\cdot h}{\sme} + \I{\dom}{\f}{\me_p}
  \quad\qmq{for all} \f\in\liss^m
\end{equation*}
(i.e. $\me=\me_\sigma+\me_p$ with $\me_\sigma=h\sme$). 
\end{proposition}
\noindent
Notice that \reff{pm-tva} readily follows from \reff{pm-e3} for $m=1$. 

\begin{proof}    
For all assertions despite \reff{pm-tva} it is sufficient to consider the 
case $m=1$. As already mentioned, the characterization of the dual space in
the scalar case follows from \cite[p.~139, 140]{rao}. The decomposition 
is taken from \cite[Theorem 4.14]{schonherr_pure_2017} and 
relies on Proposition \ref{pm-s2}.

It remains to show \reff{pm-tva} for a vector measure $\me=(\me_1,\dots,\me_m)$
(cf. \cite[p.~21]{ambrosio} for the case of a $\sigma$-measure).
For that we fix $\eps>0$. 
First, by \reff{pm-e3}, there is a pairwise disjoint
decomposition $A=\bigcup_{j=1}^k A_j$ such that $\me(A_j)\ne 0$,
$A_j\in\als^\sigma$, and  
\begin{equation*}
  |\me|(A) - \eps \le \sum_{j=1}^k|\me(A_j)| = 
  \sum_{j=1}^k\I{A_j}{a_j}{\me} = \I{A}{\tilde\tf}{\me} \le
  \sup_{\substack{\tf\in\liss^m\\\|\tf\|_{\cL^\infty}\le 1}}
  \I{A}{\tf}{\me}
\end{equation*}
where $a_j=\frac{\me(A_j)}{|\me(A_j)|}$ and $\tilde\tf=a_j$ on $A_j$.
Second, there is some $\tilde\tf\in\liss^m$ with 
$\|\tilde\tf\|_\infty\le 1$ and there is a step function 
$h=\sum_{j=1}^ka_j\chi_{A_j}$ ($A_j\in\als^\sigma$ pairwise disjoint)
with $|\tilde\tf-h|<\eps$, $|h|\le 1$ on $A$ such that
\begin{eqnarray*}
  \sup_{\substack{\tf\in\liss^m\\\|\tf\|_{\cL^\infty}\le 1}}
  \I{A}{\tf}{\me} - \eps
&\le&
  \I{A}{\tilde\tf}{\me} \le
  \I{A}{h}{\me} + c\eps  \\
&=&
  \sum_{j=1}^k \I{A_j}{a_j}{\me} + c\eps 
\: \le \: 
  \sum_{j=1}^k |\me(A_j)| + c\eps \\[1mm]
&\le& 
  |\me|(A) + c\eps 
\end{eqnarray*}
where $c=\sum_l|\me_l|(A)$.
Since $\eps>0$ is arbitrary, the assertion follows.
\end{proof}

\noindent
Now we provide some examples for integration. 

\begin{example} \label{pm-ex5a}   
With $\me$ on $(\N,\al)$ from Example~\ref{ex:puremn}
and functions on $\N$ taken as sequences $\{a_n\}\subset\R$, we have
integrability if there is some $c\in\R$ with $a_k=c$ for almost all $k\in\N$
and then
\begin{equation*}
  \I{\N}{a_k}{\me}=c
\end{equation*}
(cf. \cite[p.~111]{rao}).
\end{example} 

\begin{example} \label{pm-ex5b}   
The space $\ell^\infty:=\cL^\infty(\N,\cP(\N),\sigma)$ with the 
counting measure $\sigma$ is just the set of bounded sequences
$\{a_k\}\subset\R$. On the subspace $\ell_0^\infty$ of convergent sequences
one has a linear continuous functional by
\begin{equation*}
  \{a_k\} \to \lim_k a_k \,
\end{equation*}
which can be extended (not uniquely) on $\ell^\infty$ by Hahn-Banach. Hence
there is some 
$\me\in \op{ba}(\N,\cP(\N),\sigma)$ such that all sequences of $\ell^\infty$
are integrable with respect to $\me$ and 
\begin{equation*}
   \I{\N}{a_k}{\me} = \lim_k a_k \qmq{on} \ell^\infty_0
\end{equation*}
(cf. also \cite[p.~39 ff.]{rao} and 
Banach limits in \cite[p.~104]{yosida_book}). 
Sequences $\{a_k\}$ that are zero up to some $a_l=1$ are simple functions on
$\N$ and we get
\begin{equation*}
  \me(\{l\})=\I{\N}{a_k}{\me}=0 \qmq{for all} l\in\N\,.
\end{equation*}
Hence $\me$ is pure by arguments as in Example~\ref{ex:puremn}.
\end{example}

\begin{example} \label{pm-ex5c}   
Obviously $\dens_x$ from Example~\ref{ex:dzero} belongs to
$\bawl{\R^n}$. Hence all $\f\in L^\infty(\R^n)$ are integrable with respect
to $\dens_x$ for any $x\in\R^n$. Proposition~\ref{tt-s3} below implies that
\begin{equation*}
  \bar\f(x) := \I{\R^n}{\f}{\dens_x} = \sI{\{x\}}{\f}{\dens_x} 
\end{equation*}
agrees with $\f(x)$ at Lebesgue points $x$ of $\f$,
since there 
\begin{equation*}
  \f(x)=\lim_{\delta \downarrow 0} \mI{B_\delta(x)}{\f}{\lem}\,,
\end{equation*}
while 
$\bar f(x)$ may depend on the special choice of the measure $\dens_x$
at other points. This way we have an integral representation for some
$\bar f\in L^\infty(\R^n)$ at all $x\in\R^n$ that agrees $\lem$-a.e. with the 
precise representative of $f$ (cf. also Remark~\ref{pm-rem9} below). 
\end{example} 

\noindent
Let us still justify the definition of core before we consider  
some facts about weakly absolutely continuous measures. 

\begin{proposition} \label{acm-s0}
Let $M$ be a compact topological space, let $\dom \subset M$, let
$\al$ be an algebra containing all relatively open 
sets in $\dom$, and let $\me \in \baa$ with $\me\ne 0$. Then
$\cor{\me}\ne\emptyset$ and 
\begin{equation*}
  |\me|(U\cap\dom) = |\me|(\dom)\,, \quad  |\me|(U^c\cap\dom)=0 
\end{equation*}
for any open $U\subset M$ with $\cor{\me}\subset U$.
\end{proposition}
 
\begin{proof}    
Assume that $\cor{\me}=\emptyset$. Then, by definition of core and by
compactness of $M$, there is a
finite open covering $\{U_k\}_{k=1}^m$ of $M$ with
$|\me|(U_k\cap\dom)=0$ for all $k$. Hence,
\begin{equation*}
  |\me|(\dom) \le \sum_{k=1}^m |\me|\big(U_k\cap\dom\big) = 0
\end{equation*}
in contradiction to $\me\ne 0$. 

Now let $U\subset M$ be open such that $\cor{\me}\subset U$. 
Then, for any $y\in U^c$, we find some open neighborhood $U_y$ with
$|\me|(U_y\cap\dom)=0$ by definition of core. Since
$U^c$ is compact, the open covering $\{U_y\}_{y\in U^c}$ contains a finite
covering $\{U_j\}_{j=1}^l$. Consequently,
\begin{equation*}
  0 \le |\me|(U^c\cap\dom) \le \sum_{j=1}^l |\me|(U_j\cap\dom) = 0
\end{equation*}
and
\begin{equation*}
  |\me|(\dom) = |\me|(U\cap\dom) + |\me|(U^c\cap\dom) =
  |\me|(U\cap\dom) \,.
\end{equation*}
\end{proof}

\subsection{Some weakly absolutely continuous measures}
\label{acm}

The measures in $\bawl{\dom}$,
that are weakly absolutely continuous with respect to $\lem$ and 
that coincide with the 
linear continuous functionals on $\cL^\infty(\dom,\bor{\dom},\lem)$, 
are of particular interest for our general treatment of traces. 
Therefore let us discuss some special aspects. 
First we provide a very useful sufficient condition for such measures to be
pure. 

\begin{proposition} \label{acm-s1}
Let $\dom\in\cB(\R^n)$. Then $\me\in\bawl{\dom}$ is pure if
\begin{equation*}
  \lem(\cor\me\cap\dom)=0\,.
\end{equation*}
\end{proposition}

\begin{proof}   
We have that $\cor{\me}$ is closed in the compact topological space 
$\ol{\R^n}$ and there is a decreasing sequence $B_k$ of open 
neighborhoods of $\cor{\me}$ such that $\cor{\me} = \bigcap_k B_k$. Then
\begin{equation*}
  |\me|\big( \dom\setminus B_k \big) = 0 \qmz{for all} k\in\N 
\end{equation*}
by Proposition~\ref{acm-s0}. For a positive $\sigma$-measure $\sme$ with
$0 \leq \sme \leq |\me|$ we get
\begin{equation*}
  0 \le \sme(\dom\setminus B_k) \le |\me|(\dom\setminus B_k) = 0 \,.
\end{equation*}
From $\lem(\cor\me\cap\dom)=0$ we derive
\begin{equation*}
  0 \le \sme(\cor{\me}\cap\dom) \le |\me|(\cor{\me}\cap\dom) =0 \,.
\end{equation*}
Therefore
\begin{equation*}
  \sme(\dom) = \sme(\cor\me\cap\dom) + 
               \sme\bigg( \bigcup_{k} \dom\setminus B_k \bigg)
             = \lim_{k\to\infty} \sme\big(\dom\setminus B_k\big) = 0 
\end{equation*}
and, thus, $\sme=0$. Since $\sme$ was arbitrary, $\me$ is pure. 
\end{proof}

\newcommand{\C}{\K}

Now we show how any measure $\me\in \ba(\dom,\cB(\dom))$ can be ``restricted''
to a Radon measure supported on $\cor\me$ and how a $\sigma$-measure on some
$\C\subset\ol\dom$ can be extended to a measure $\me\in \bawl{\dom}$
with $\cor\me\subset\ol \C$.
For the proof of the second statement we use the semi norm
on $\cL^\infty(\dom,\R^m)$ given by 
\begin{equation} \label{semi-norm}
  \|\tf\|_\K := \lim_{\delta\downarrow 0} \|\tf\|_{\cL^\infty(\K_\delta\cap\dom)}
  = \inf_{\delta>0} \|\tf\|_{\cL^\infty(\K_\delta\cap\dom)}
\end{equation}
which is well-defined for Borel sets $\C$ and $\dom$ with $\C\subset\ol\dom$
(the limit exists and equals the infimum, since 
the norm on the right hand side is increasing in $\delta$). 

\begin{proposition} \label{acm-ram}
Let $\dom\subset\R^n$ be a bounded Borel set.
\bgl
\item
  For $\me \in \ba(\dom,\cB(\dom))^m$ there is a
  Radon measure $\sme\in\ca(\ol\dom,\cB(\ol\dom))^m$ supported on $\cor{\me}$
  such that
\begin{equation*}
  \I{\cor{\me}}{\phi}{\sme} = 
  \I{\dom}{\phi}{\me}  \quad
  \tx{\small $\bigg(\!\!= \sI{\cor{\me}}{\phi}{\me}\bigg)$   } 
  \quad\qmq{for all} \phi \in C(\ol\dom,\R^m)  \,.
\end{equation*}

\item  
Let $\sme\in\ca(\C,\cB(\C))^m$ be a bounded vector-valued $\sigma$-measure 
on some Borel set $\C\subset\ol\dom$ such that
\begin{equation}\label{ram-e1}
  \lem(\ball{x}{\delta} \cap \dom) > 0 \quad\qmq{for all}
  x \in \C,\, \delta > 0   \,.
\end{equation}
Then there exists $\me \in \bawl{\dom}^m$ with
\begin{equation*}
  \cor{\me} \subset \cl \C \qmq{and}
  |\me|(\dom) = |\sigma|(\C)
\end{equation*}
such that
\begin{equation*}
  \I{\C}{\phi}{\sigma} = \I{\dom}{\phi}{\me} \quad
  \tx{\small $\bigg(\!\!= \sI{\cl \C}{\phi}{\me} \bigg)$   } 
  \quad\qmq{for all} \phi \in C(\ol\dom,\R^m)\,.
\end{equation*}
If $\lem(\cl \C \cap \dom) = 0$, then $\me$ is pure.
\el  
\end{proposition}

\noindent
Obviously \reff{ram-e1} is satisfied if e.g.
$\C \subset \mbd{\dom} \cup \mint{\dom}$. The construction of measure $\me$
in (2) uses the Hahn-Banach theorem and is not unique.
Before giving the proof we consider some example as illustration.

\begin{example} \label{pm-ex6}   
Let $\dom\subset\R^n$ be open and bounded, let 
$\C=\{x\}$ for some $x\in\dom$, and let $\delta_x$ be 
the Dirac measure concentrated at $x$. By Proposition~\ref{acm-ram} 
there is some pure $\me_x\in\bawl{\dom}$ such that 
\begin{equation} \label{pm-ex6-1}
  \phi(x) = \I{\{x\}}{\phi}{\delta_x} = \sI{\{x\}}{\phi}{\me_x}
  \quad\qmq{for all} \phi \in C(\dom)\,.
\end{equation}
Obviously $\me_x=\dens_x$ would be a possible choice
(cf. Example~\ref{pm-ex5c}). Alternatively we can consider  
some density $\dens_x^E\in\bawl{\dom}$ of $\als\in\bor{\dom}$ at $x$ with
respect to $E\in\bor{\dom}$ by a construction in analogy to
Example~\ref{ex:dzero} such that 
\begin{equation} \label{pm-e11}
  \liminf_{\delta \downarrow 0}
  \frac{\lem(\als \cap B_\delta^E(x))}{\lem(B_\delta^E(x))}
  \le \dens_x^E(\als) \le
  \limsup_{\delta \downarrow 0}
  \frac{\lem(\als \cap B_\delta^E(x)}{\lem(B_\delta^E(x))}
\end{equation}
where $B_\delta^E(x)=B_\delta(x)\cap E$ (Proposition~\ref{tt-s3} below ensures
the existence of such measures). We call $\dens_x^E$ also a
{\it density measure} at $x$ with respect to $E$.
In particular we can choose $\me_x=\dens_x^E$ for some 
open $E$ having an outward cusp at $x\in\bd E$. 
Then \reff{pm-ex6-1} remains true with $\mu_x=\dens_x^E$
(cf. Proposition~\ref{tt-s3}). But, for 
$\tf=\chi_E\in\cL^\infty(\dom)$
we have that $x$ 
is Lebesgue point and  
\begin{eqnarray*}
  0
&=&
  \chi_E(x) \:=\:
  \lim_{\delta\downarrow 0} \mI{B_\delta(x)}{\chi_E}{\lem} \: = \: 
   \dens_x(E) \:=\: 
  \sI{\{x\}}{\chi_E}{\dens_x} \\
&\ne& 
  1 \: = \: 
  \lim_{\delta\downarrow 0} \mI{E \cap B_\delta(x)}{\chi_E}{\lem} \: = \:
  \dens_x^E(E) \:=\:
  \sI{\{x\}}{\chi_E}{\dens_x^E}\,.
\end{eqnarray*}
Hence a point evaluation $\bar\phi$ with $\dens_x^E$ in analogy to
Example~\ref{pm-ex5c} would not agree with $\phi$ at Lebesgue points in
general.  
This illustrates the variety of extensions $\me_x$ of $\delta_x$ provided by
Proposition~\ref{acm-ram}.
\end{example}

\medskip

\begin{proof+}{ of Proposition~\ref{acm-ram}}   
For (1) it is sufficient to consider $m=1$. We first notice that
\begin{equation*}
  \Big|\I{\dom}{\tf}{\me}\Big| \le |\me|(\dom)\,\|\tf\|_{C(\ol\dom)}
  \qmq{for all} \phi \in \Cefun{\dom}\,.
\end{equation*}
Then $\tf\to\I{\dom}{\tf}{\me}$ is a
linear continuous functional on $C(\ol\dom)$ and, by Riesz' Representation
theorem, there is a Radon measure $\sme\in\ca(\ol\dom,\cB(\ol\dom))$ such that
\begin{equation*}
  \I{\ol\dom}{\tf}{\sme} = \I{\dom}{\tf}{\me}
  \qmq{for all} \phi \in \Cefun{\dom}\,.
\end{equation*}
For any $x\in\ol\dom\setminus\cor\me$ there is $\delta>0$ with
$B_\delta(x)\cap\cor\me=\emptyset$, since $\cor\me$ is closed. Thus, 
\begin{equation*}
  \I{\ol\dom}{\tf}{\sme} = \I{\dom}{\tf}{\me} = 0
\end{equation*}
for all $\tf\in C(\ol\dom)$ compactly supported on $B_\delta(x)$. Hence 
$\supp \sme\subset\cor\me$.

For (2) we use the semi norm from \reff{semi-norm} and observe that 
\begin{equation*}
  \|\tf\|_\C =  \|\tf_{|\C}\|_{C(\C)} \le 
  \|\tf\|_{\cL^\infty(\dom)} 
  \qmq{for all} \tf\in C(\ol\dom,\R^m) \,,
\end{equation*}
since $|\tf(x)|\le\|\tf\|_{\cL^\infty(\dom\cap B_\delta(x))}$ 
for all $x\in\C$ and $\delta>0$ by \reff{ram-e1}. Thus 
\begin{equation} \label{ram-e11}
  \Big|\I{\C}{\tf}{\sme}\Big| \le |\sme|(\C) \, \|\tf_{|\C}\|_{C(\C)} =
  |\sme|(\C) \, \|\tf\|_\C  
  \qmq{for all} \tf\in C(\ol\dom,\R^m) \,.
\end{equation}
Hence $g_0^*:C(\ol\dom,\R^m)\to\R$ with $\df{g_0^*}{\tf}=\I{\C}{\tf}{\sme}$
is a linear continuous functional on a subspace of $\cL^\infty(\dom,\R^m)$. 
By the Hahn-Banach theorem there is a linear continuous extension $g^*$ of
$g_0^*$ to all of $\cL^\infty(\dom,\R^m)$ preserving \reff{ram-e11}. Thus
\begin{equation} \label{ram-e12}
  |\langle g^*,\tf \rangle|\le |\sme|(\C)\|\tf\|_\C \le
  |\sme|(\C)\|\tf\|_{\cL^\infty(\dom)}
  \qmq{for all} \tf\in\cL^\infty(\dom,\R^m) \,.
\end{equation}
By Proposition~\ref{pm-s5} there is $\me\in\bawl{\dom}^m$ 
such that 
\begin{equation*}
   \df{g^*}{\tf} = \I{\dom}{\tf}{\me} \qmq{for all}
   \tf\in\cL^\infty(\dom,\R^m) \,.
\end{equation*}
Consequently,
\begin{equation*}
  |\me|(\dom) = \|g^*\| \le |\sme|(\C) \,.
\end{equation*}
Since every $\tf\in C_c(\C,\R^m)$ can be extended to some 
$\ol\tf\in C(\ol\dom,\R^m)$
under preservation of the norm 
(cf. \cite[Proposition 2.1]{zeidler_I}) and since $\sme$ is bounded, 
\begin{eqnarray*}
  |\sme|(\C) 
&=& 
  \sup_{\substack{\tf\in C_c(\C,\R^m) \\ \|\tf\|_{C(\C)} \le 1}}
  \I{\C}{\tf}{\sme} = 
  \sup_{\substack{\tf\in C_c(\C,\R^m) \\ \|\tf\|_{C(\C)} \le 1}}
  \I{\dom}{\ol\tf}{\me} \\
&\le&
  \sup_{\substack{\psi\in C(\ol\dom,\R^m) \\ \|\psi\|_{C(\ol\dom)} \le 1}} 
  \I{\dom}{\psi}{\me} =
  \sup_{\substack{\psi\in C(\ol\dom,\R^m) \\ \|\psi\|_{\cL^\infty(\dom)} \le 1}}
  \I{\dom}{\psi}{\me} \le |\me|(\dom) \,.
\end{eqnarray*}
Therefore $|\sme|(\C)=|\me|(\dom)$. If $\delta>0$ and 
$\tf\in\cL^\infty(\dom,\R^m)$ with $\tf=0$ on $\C_\delta$, then
\begin{equation*}
  \df{g^*}{\tf}=\I{\dom}{\tf}{\me}=0
\end{equation*}
by \reff{ram-e12}. This implies $\cor{\me}\subset\ol\C$.
If $\lem(\ol\C\cap\dom)=0$, then $\me$ is a pure measure 
by Proposition~\ref{acm-s1}.
\end{proof+}

\medskip

For some measure $\me\in \bawl{\dom}$ with $\dom\in\bor{\R^n}$, 
all functions $f\in L^\infty(\dom)$ are
integrable by Proposition~\ref{pm-s5} and, thus,
\begin{equation*}
  L^\infty(\dom,\bor{\dom},\lem)\subset L^1(\dom,\bor{\dom},\me)\,.
\end{equation*}
By a more detailed analysis of the special case $\me=\dens^\dom_x$ 
(cf. \reff{pm-e11})
we not only show that the inclusion can be strict, but we also demonstrate 
how the integration theory can be applied to an important prototype of 
pure measures. 
We say that $f\in L^1_{\rm loc}(\dom)$ has the {\it approximate limit}
$\alpha$ at $x\in\ol\dom$ (with respect to $\dom$), denoted by 
\begin{equation*}
  \alim_{y\to x} f(y) = \alpha \,,
\end{equation*}
if for all $\eps>0$
\begin{equation}   \label{ram-e13}
  \lim_{\delta\to 0}
  \frac{\lem\big(\dom\cap B_\delta(x)\cap\{|f-\alpha|\ge\eps\}\big)}
       {\lem(\dom\cap B_\delta(x))}
  = 0 
\end{equation}
(cf. \cite[p.~46]{evans}).

\begin{proposition}  \label{pm-prop7}
Let $\dom\in\bor{\R^n}$, let $f\in L^1_{\rm loc}(\ol\dom)$,
let $x\in\ol\dom$ be such that 
\begin{equation*}
  \lem(\dom\cap B_\delta(x))>0 \qmq{for all} \delta>0\,,
\end{equation*}
and let $\dens^\dom_x$ be a
density measure at $x$ with respect to $\dom$.

\bgl
\item
If there is some $\tilde\delta>0$ such that 
\begin{equation*}
   \mI{\dom\cap B_\delta(x)}{|f|}{\lem} \qmq{is bounded for all}
   0<\delta<\tilde\delta \,,
\end{equation*}
then $f$ is integrable with respect to $\dens^\dom_x$ and
\begin{equation*}
  \I{\dom}{|f|}{\dens^\dom_x} \le 
   \limsup_{\delta\to 0} \mI{\dom\cap B_\delta(x)}{|f|}{\lem} \,.
\end{equation*}

\item
If we have
\begin{equation*}
  \alim_{y\to x} f(y)=\alpha \qmq{for some} \alpha\in\R\,,
\end{equation*}
then $f=\alpha$ i.m. and $f$ is $\dens^\dom_x$-integrable with
\begin{equation*} \label{pm-prop7-02}
  \sI{\{x\}}{f}{\dens^\dom_x} = \alpha\,.
\end{equation*}

\item
If there is $\alpha\in\R$ with
\begin{equation} \label{pm-prop7-03}
  \lim_{\delta\to 0} \mI{\dom\cap B_\delta(x)}{|f-\alpha|}{\lem} = 0\,,
\end{equation}
then $\alim_{y\to x} f(y)=\alpha$ and $f$ is $\dens^\dom_x$-integrable with 
\begin{equation} \label{pm-prop7-04}
  \alpha = \sI{\{x\}}{f}{\dens^\dom_x} =
  \lim_{\delta\to 0} \mI{\dom\cap B_\delta(x)}{f}{\lem} \,.
\end{equation}

\el
\end{proposition}

\begin{remark} \label{pm-rem9}   
(1) Let us discuss the results for $\dom$ open and $x\in\dom$.
Notice first that, in this case, the results are also true for any density
measure $\dens_x$ according to Example~\ref{ex:dzero}, since it agrees with
some $\dens^\dom_x$ on $\dom$. We also observe that \reff{pm-prop7-03} is
valid with $\alpha=f(x)$ at any Lebesgue point $x\in\dom$ of $f$ by definition. 
This way we see that any $f\in L^1_{\rm loc}(\dom)$ is
$\dens^\dom_x$-integrable at 
least for $\lem$-a.e. $x\in\dom$ (cf. \cite[p.~44]{evans}). 
Consequently, for fixed $x\in\dom$ there is a large class of 
$\dens^\dom_x$-integrable functions beyond $L^\infty(\dom)$.
Notice that, conversely, $\dens^\dom_x$-integrability of $f$ does not imply that
$x$ is a Lebesgue point of $f$ or that $f$ has an approximate limit $x$
(take e.g. $f=\chi_A$ for some
$A\subset\bor{\dom}$ such that liminf and limsup in \reff{pm-e11} do not
coincide). 

(2) For $\dom$ open the results also provide an integral representation for a
slightly modified {\it precise representative} of $f$ by 
\begin{equation*}
  \pr{f}(x) =
  \bigg\{ \mbox{\small $ 
  \begin{array}{ll} \sI{\{x\}}{f}{\dens^\dom_x} & 
                    \text{if the integral exists}\,, \\[3pt]
                    0 & \text{otherwise} \,.
  \end{array}$ } 
\end{equation*}
The usually used precise representative, that equals 
\begin{equation*}
  \lim_{\delta\to0}\:\mI{B_\delta(x)}{f}{\lem}
\end{equation*}
if the limit exists,
obviously agrees with $\pr{f}(x)$ at all $x\in\dom$ where \reff{pm-prop7-03}
is satisfied. Thus it differs from $\pr{f}(x)$ 
at most on an $\lem$-zero set (cf. \cite[p.~46]{evans}).
More precisely, $f\in L^1_{\rm loc}(\dom)$ might not be $\dens^\dom_x$-integrable
if merely $\lim_{\delta\to0}\mI{B_\delta(x)}{f}{\lem}$ exists, since
roughly speaking $\sI{\{x\}}{f^\pm}{\dens^\dom_x}$ can be infinite for the
positive and negative part of~$f$. If $f$ is 
$\dens^\dom_x$-integrable, the limit 
$\lim_{\delta\to0}\mI{B_\delta(x)}{f}{\lem}$ can exist but differ from 
$\pr{f}(x)$ (cf. Example~\ref{pm-ex9} below).
Moreover, $f$ can be $\dens^\dom_x$-integrable if the limit
$\lim_{\delta\to0}\mI{B_\delta(x)}{f}{\lem}$ 
does not exist, since , e.g., the inequalities in \reff{tt-s3-1}
can be strict. In this last case the integral will depend on the special
choice of $\dens^\dom_x$. 

For $f$ in $W^{1,1}(\dom)$ or $BV(\dom)$ we readily obtain that it is 
integrable with respect to $\dens^\dom_x$ for $\cH^{n-1}$-a.e.
$x\in\dom$ and that $\pr{f}$ agrees with the usual precise representative
$\cH^{n-1}$-a.e. on $\dom$ (cf. \cite[p.~160, 213]{evans}). 

(3) If $\dom$ is open with Lipschitz boundary and if 
$f$ is in $W^{1,1}(\dom)$ or $BV(\dom)$, then the integral
$\sI{\{x\}}{f}{\dens^\dom_x}$ exists and agrees with the usual trace of $f$ for 
$\cH^{n-1}$-a.e. $x\in\bd\dom$ (cf. \cite[p.~133, 181]{evans}).  
Therefore, the precise representative $\pr{f}$ also provides an integral
representation for the pointwise trace of $f$ in this case. 

(4) Let $\dom\subset\R^n$ be open and let $f\in\cW^{1,1}(\dom,\R^m)$. Then $f$
is differentiable $\lem$-a.e. and its derivative $Df$ equals its
weak derivative $\lem$-a.e. (cf. \cite[p.~235]{evans}). 
For $x\in\dom$ where $f$ is differentiable we have
\begin{equation*}
  \alim_{y\to x}\frac{\big| f(y)-f(x) - Df(x)(y-x) \big|}{|y-x|} = 0
\end{equation*}
and, thus,
\begin{equation*}
  \sI{\{x\}}{\frac{\big| f(y)-f(x) - Df(x)(y-x) \big|}{|y-x|}}{\dens_x^\dom} 
  = 0 \,.
\end{equation*}
Moreover, if $h\in\R^n$ with $h\ne 0$, then
\begin{equation*}
  \lim_{t\to 0} \frac{f(x+th)-f(x)}{t} = Df(x)\,h =
  \alim_{t\to 0} \frac{f(x+th)-f(x)}{t} \,.
\end{equation*}
With a density measure $\dens_0^\R\in\op{ba}(\R,\bor{\R},\cL^1)$ at $t=0$ in
$\R$ we obtain 
\begin{equation*}
  Df(x)\,h = \sI{\{0\}}{\frac{f(x+th)-f(x)}{t}}{\dens_0^\R}
  \qmz{for all} h\in\R^n\,,\; h\ne 0\,,
\end{equation*}
which is an integral representation of the directional derivatives
for $\lem$-a.e. $x\in\dom$. 
\end{remark}

For $f\in BV(\dom)$ we have an even stronger result saying 
that, on a jump set of $f$, the precise representative $\pr{f}$ gives 
the mean value of the approximate limits from both sides for 
$\cH^{n-1}$-a.e. $x$ (cf. also \cite[p.~213]{evans}, 
\cite[p.~175]{ambrosio}).

\begin{corollary}  \label{pm-cor8}   
Let $\dom\subset\R^n$ be open, let $f\in BV(\dom)$, and let $\dens^\dom_x$ be
any density measure at $x\in\dom$ with respect to $\dom$. Then
\begin{equation*}
  \pr{f}(x) = \sI{\{x\}}{f}{\dens^\dom_x} =
  \frac{f^i(x)+f^s(x)}{2} \quad\qmq{$\cH^{n-1}$-a.e. on $\dom$}
\end{equation*}
where 
\begin{equation*}
  f^i(x) = \op{ap} \liminf_{y\to x} f(y) =
  \sup \bigg\{ \alpha \:\Big|\: \lim_{\delta\to 0} 
  \text{\small $\frac{\lem\big(B_\delta(x)\cap\{f<\alpha\}\big)}
       {\lem( B_\delta(x))} = 0 $ } \bigg\}\,,
\end{equation*}
\begin{equation*}
  f^s(x) = \op{ap} \limsup_{y\to x} f(y) =
  \inf \bigg\{ \alpha \:\Big|\: \lim_{\delta\to 0} 
  \text{\small $\frac{\lem\big( B_\delta(x)\cap\{f>\alpha\}\big)}
       {\lem( B_\delta(x))} = 0 $ } \bigg\}\,
\end{equation*}
are the lower and the upper approximate limit of $f$ at $x$, respectively.
\end{corollary}

\medskip

For $\dom\in\bor{\R^n}$ and $x\in\ol\dom$ such that 
$\lem(\dom\cap B_\delta(x))>0$ for all $\delta>0$, 
we readily get from \reff{pm-e11} 
for all $A\in\bor{\dom}$ 
\begin{equation*}
  \I{\dom}{\chi_A}{\dens_x} = \dens_x(A) \le
  \limsup_{\delta\to 0}\frac{\lem\big(A\cap B^\dom_\delta(x)\big)}
          {\lem(B^\dom_\delta(x))} =
  \limsup_{\delta\to 0} \mI{B^\dom_\delta(x)}{\chi_A}{\lem}
\end{equation*}
and an analogous relation with $\liminf$. This implies for 
simple functions $h$ that
\begin{equation} \label{pm-e2}
  \liminf_{\delta\to 0} \mI{B^\dom_\delta(x)}{h}{\lem} \le
  \I{\dom}{h}{\dens_x} \le 
  \limsup_{\delta\to 0} \mI{B^\dom_\delta(x)}{h}{\lem} \,.
\end{equation}
Now we can ask how far this remains true for $\dens^\dom_x$-integrable
$f\in L^1_{\rm loc}(\dom)$. For any $f\in L^\infty(\dom)$
this follows from \reff{tt-s3-1} below (take $\K=\{x\}$ and $\alse=\dom$).
If $f\in L^1_{\rm loc}(\dom)$ satisfies \reff{pm-prop7-03}, then \reff{pm-e2} 
follows with equality. Let us provide an example that this might fail if 
merely $\alim_{y\to x} f(y)=\alpha$. 

\begin{example} \label{pm-ex9}   
Let $\dom=B_1(0)\subset\R^2$. With polar coordinates  
$(r,\beta)$ we set (radius $r>0$, angle $\beta\in[0,2\pi)$) 
\begin{equation*}
  \dom'= \big\{(r,\beta)\,\big|\, 0<r<1,\: 0<\beta<r^2 \big\} \subset\dom\,.
\end{equation*}
Now we consider $f\in  L^1(\dom)$ given by
\begin{equation*}
  f(x) = \bigg\{ 
  \mbox{\small $ 
  \begin{array}{cl} \frac{1}{|x|} & \text{on } \dom'\,, \\[3pt]
                    0 & \text{otherwise\,.}
  \end{array}$ } 
\end{equation*}
Then, for $0<\delta<1$,
\begin{equation*}
  \mI{B_\delta(0)}{f}{\cL^2} = 
  \frac{1}{\pi\delta^2} \int_0^\delta\int_0^{r^2}\frac{1}{r}\,d\beta\, dr =
  \frac{1}{\pi\delta^2} \int_0^\delta r\,dr = \frac{1}{2\pi}
\end{equation*}
From \reff{pm-e11} we obtain for $x=0$
\begin{equation*}
  0 \le \dens_0(\dom') \le \limsup_{\delta \downarrow 0}
  \frac{\lem(\dom' \cap B_\delta(0))}{\lem(B_\delta(0))} = 0 \,.
\end{equation*}
Since $\dom'\cap B_{\frac{1}{\eps}}(0)=\{|f-0|>\eps\}$ for $\eps>0$,
\begin{equation} \label{pm-ex9-3}
  \alim_{y\to 0} f(y)=0 \,.
\end{equation}
Consequently, by Proposition~\ref{pm-prop7}, $f=0$ i.m. $\dens^\dom_0$
and
\begin{equation} \label{pm-ex9-4}
  0 = \sI{\{0\}}{f}{\dens^\dom_0} < 
  \liminf_{\delta\to 0} \mI{B_\delta(0)}{f}{\lem}  \,.
\end{equation}
For $-f$ we get the opposite inequality with $\limsup$. Thus \reff{pm-e2} is
not valid with $\pm f$ instead of $h$. Since 
$\lim_{\delta\to 0}\mI{B_\delta(0)}{f}{\lem}>0$ exists, $\pr{f}(x)=0$
differs from the usual precise representative. If we define $f$ with 
$\frac{1}{|x|^2}$ instead of $\frac{1}{|x|}$ on $\dom'$, then 
we still have \reff{pm-ex9-3} but
\begin{equation*}
  \mI{B_\delta(0)}{f}{\lem}=
  \frac{1}{2\pi\delta}\overset{\delta\to 0}{\longrightarrow}\infty \,,
\end{equation*}
i.e. the equality in \reff{pm-ex9-4} remains the same while the 
$\liminf$ becomes even infinite. The example also shows that the boundedness
of $\mI{\dom\cap B_\delta(0)}{|f|}{\lem}$ for small $\delta>0$ is not
sufficient for the second equality in \reff{pm-prop7-04}. 
\end{example}

\medskip

\begin{proof+}{ of Proposition~\ref{pm-prop7}}   
We have that $f\in L^1(\dom\cap B_{\bar\delta}(x))$
for some $\bar\delta>0$. 

For (1) there is some $c>0$ such that 
$\mI{B^\dom_\delta(x)}{|f|}{\lem}<c$ for all $\delta\in(0,\tilde\delta)$.
We define for each $k\in\N$
\begin{equation*}
  \dom_k := \big\{y\in\dom \:\big|\: |f(y)|<k \big\}\,, \quad 
  \dom_k^0 := \dom\setminus\dom_k\,,
\end{equation*}
\begin{equation*}
  h_k(x) := \bigg\{ \mbox{\small $ 
  \begin{array}{ll} \tfrac{l}{2^k} & \text{on } 
    |f|^{-1}\big(\big[\tfrac{l}{2^k},\tfrac{l+1}{2^k}\big)\big) 
    \cap \dom_k \text{ for all } l\in\N\,,\\[4pt]
    0 & \text{on } \dom_k^0 \,.
  \end{array}$ } 
\end{equation*}
Obviously all $h_k$ are Borel measurable and, thus, they are 
simple functions related to $\dens^\dom_x$. 
Clearly, $h_k\le|f|$ on $\dom$
for all $k\in\N$. By $|\,h_k-|f|\,|<\tfrac{1}{2^k}$ on $\dom_k$ we get
\begin{equation*}
  \big\{ y\in\dom\:\big|\: |\,h_k-|f|\,|>\eps \big\} \subset \dom_k^0
  \qmq{if} \tfrac{1}{2^k}<\eps \,.
\end{equation*}
Therefore
\begin{equation} \label{pm-prop7-1}
  \dens_x^\dom\big\{ |\,h_k-|f|\,|>\eps \big\} \le \dens_x^\dom(\dom_k^0) \,.
\end{equation}
Let us first assume that $\dens^\dom_x(\dom_k^0)\ge\tfrac{2}{\sqrt{k}}$. 
Then we choose $\delta_k\in(0,\tilde\delta)$ such that
\begin{equation*}
  \limsup_{\delta\to 0}
  \frac{\lem\big(\dom_k^0\cap B^\dom_\delta(x)\big)}{\lem(B^\dom_\delta(x))} \le 
  \frac{\lem\big(\dom_k^0\cap B^\dom_{\delta_k}(x)\big)}
       {\lem(B^\dom_{\delta_k}(x))} +
  \frac{1}{\sqrt{k}} \,.
\end{equation*}
Consequently,
\begin{eqnarray*}
  \mI{B^\dom_{\delta_k}(x)}{|f|}{\lem} 
&=&
  \frac{1}{\lem(B^\dom_{\delta_k}(x))} \bigg(\,
  \I{B^\dom_{\delta_k}(x)\cap\dom_k}{|f|}{\lem} +
  \I{B^\dom_{\delta_k}(x)\cap\dom_k^0}{|f|}{\lem} \bigg) \\
&\ge&
  k \, \frac{\lem\big(\dom_k^0\cap B^\dom_{\delta_k}(x)\big)}
            {\lem(B^\dom_{\delta_k}(x))} 
\; \ge \;
  k\,\bigg(\limsup_{\delta\to 0} 
  \frac{\lem\big(\dom_k^0\cap B^\dom_\delta(x)\big)}{\lem(B^\dom_\delta(x))} -
  \frac{1}{\sqrt{k}} \bigg) \\
&\overset{\reff{pm-e11}}{\ge}&
  k\, \big( \dens^\dom_x(\dom_k^0) - \tfrac{1}{\sqrt{k}}\, \big)
\; \ge \;
  \sqrt{k} \,.
\end{eqnarray*}
But this is impossible for $\sqrt{k}>c$ by the boundedness of the left hand
side and, therefore, $\dens^\dom_x(\dom_k^0)<\tfrac{2}{\sqrt{k}}$ for such $k$. 
Using \reff{pm-prop7-1} we get 
\begin{equation*}
  h_k \convim{\dens^\dom_x} |f| \,.
\end{equation*}
With \reff{pm-e11} we obtain for all $A\in\bor{\dom}$
\begin{equation*}
  \I{\dom}{\chi_A}{\dens_x} = \dens_x(A) \le
  \limsup_{\delta\to 0}\frac{\lem\big(A\cap B^\dom_\delta(x)\big)}
          {\lem(B^\dom_\delta(x))} =
  \limsup_{\delta\to 0} \mI{B^\dom_\delta(x)}{\chi_A}{\lem}
\end{equation*}
and an analogous relation with $\liminf$. This implies for 
the simple functions $h_k$ that
\begin{equation*} \label{pm-prop7-2}
  \liminf_{\delta\to 0} \mI{B^\dom_\delta(x)}{h_k}{\lem} \le
  \I{\dom}{h_k}{\dens_x} \le 
  \limsup_{\delta\to 0} \mI{B^\dom_\delta(x)}{h_k}{\lem} \,.
\end{equation*}
Since $0\le h_k\le|f|$ on $\dom$, we obtain that 
\begin{equation} \label{pm-prop7-4}
  0 \le \I{\dom}{h_k}{\dens_x} \le 
   \limsup_{\delta\to 0} \mI{B^\dom_\delta(x)}{|f|}{\lem} < c \,.
\end{equation}
By construction, the sequence $\{h_k\}$ of simple functions is
increasing. Hence 
\begin{equation*}
  \I{\dom}{|h_k-h_l|}{\dens_x} \to 0 \qmq{as} k,l\to\infty\,. 
\end{equation*}
Consequently, $|f|$ is $\dens^\dom_x$-integrable with determining sequence
$\{h_k\}$ and, hence, also $f$ is $\dens^\dom_x$-integrable. 
Taking the limit $k\to\infty$ in \reff{pm-prop7-4} we get the remaining 
estimate.

For (2) we fix $\eps>0$ and set
$\dom^\eps := \big\{ y\in\dom \:\big|\: |f-\alpha|\ge\eps\big\}$. 
Then, by \reff{pm-e11},
\begin{equation*}
  0 \le \dens^\dom_x (\dom^\eps) \le 
  \limsup_{\delta\to 0} 
  \frac{\lem\big(\dom^\eps\cap B^\dom_\delta(x)\big)}{\lem(B^\dom_\delta(x))} 
  = 0
\end{equation*}
where the last equality follows from 
$\alim_{y\to x} f(y)=\alpha$.
Thus, $\dens_x^\dom(\dom^\eps)=0$ for all $\eps>0$,
which implies $f=\alpha$ i.m. on $\dom$. Hence the constant sequence 
$\{\alpha\}_k$ is a determining
sequence for $f$. Therefore $f$ is $\dens^\dom_x$-integrable with
\begin{equation*}
  \sI{\{x\}}{f}{\dens^\dom_x} = 
  \I{\dom}{\alpha}{\dens^\dom_x} = \alpha \dens^\dom_x(\dom) =
  \alpha \,.  
\end{equation*}

For (3) we use $\dom^\eps$ as defined in the proof of (2) to get
\begin{eqnarray*}
  0 
& \le & 
  \limsup_{\delta\to 0} 
  \frac{\lem\big(\dom^\eps\cap B^\dom_\delta(x)\big)}{\lem(B^\dom_\delta(x))}  \\
&\le&
   \limsup_{\delta\to 0} 
  \frac{1}{\eps\lem(B^\dom_\delta(x))} \:
  \I{\dom^\eps\cap B^\dom_\delta(x)}{|f-\alpha|}{\lem}             \\
&\le&
  \limsup_{\delta\to 0} 
  \frac{1}{\eps} \:\mI{B^\dom_\delta(x)}{|f-\alpha|}{\lem} 
\; = \; 0 \,.
\end{eqnarray*}
Hence $\alim_{y\to x} f(y)=\alpha$ by the definition of $\dom^\eps$. Thus (2)
implies that $f$ is $\dens^\dom_x$-integrable with
$\sI{\{x\}}{f}{\dens^\dom_x}=\alpha$.  
By \reff{pm-prop7-03}, 
\begin{equation*}
  \Big|\: \mI{\dom\cap B_\delta(x)}{\alpha-f\,}{\lem}\,\Big| \le
  \mI{\dom\cap B_\delta(x)}{|\alpha-f|}{\lem} 
  \,\overset{\delta\to 0}{\longrightarrow}\, 0 \,.
\end{equation*}
Hence
\begin{equation*}
  \alpha = \lim_{\delta\to 0} \mI{\dom\cap B_\delta(x)}{\alpha}{\lem} =
  \lim_{\delta\to 0} \mI{\dom\cap B_\delta(x)}{f}{\lem}  \,
\end{equation*}
which verifies the assertion.
\end{proof+}

\medskip

\begin{proof+}{ of Corollary~\ref{pm-cor8}}   
We have
\begin{equation} \label{pm-cor8-1}
  -\infty < f^i(x) \le f^s(x) < \infty \qmq{for $\cH^{n-1}$-a.e. $x\in\dom$\,}
\end{equation}
(cf. \cite[p.~211]{evans}). 
We use $\tilde f(x)=\tfrac{1}{2}(f^i(x)+f^s(x))$ and 
show the statement for $x\in\dom$
where \reff{pm-cor8-1} is satisfied. If the approximate limit of $f$ at $x$
exists, then 
\begin{equation*}
  f^i(x)=f^s(x)=\alim_{y\to x}f(y)
\end{equation*}
and $\pr{f}(x)=\tilde f(x)$ by Proposition~\ref{pm-prop7}. 
Otherwise there are disjoint open half
spaces $H_\pm\subset\R^n$ such that $x\in\bd H_\pm$ and
\begin{equation*}
  f^i(x) = \alim_{\substack{y\to x\\y\in H_-}} f(y) \,, \quad
  f^s(x) = \alim_{\substack{y\to x\\y\in H_+}} f(y) \,
\end{equation*}
(cf. \cite[p.~213]{evans} and notice that merely 
half balls $B_\delta^{H_\pm}(x)$ enter the computation of $\alim$). 
By $\lem(H^+\cap H^-)=0$, the measures
\begin{equation*}
   \dens_x^{\dom\cap H^\pm} := 2\,\reme{\dens_x^\dom}{H^\pm}
\end{equation*}
are density measures at $x$ with respect to $H^\pm$. 
From Proposition~\ref{pm-prop7} (2) with $\dom\cap H^\pm$ instead of $\dom$
we get
\begin{equation*}
  f^i(x)=\sI{\{x\}}{f}{\dens^{\dom\cap H^-}_x}\,, \quad
  f^s(x)=\sI{\{x\}}{f}{\dens^{\dom\cap H^+}_x} \,. 
\end{equation*}
With 
\begin{equation*}
  \dens^\dom_x = \tfrac{1}{2} 
  \big( \dens^{\dom\cap H^-}_x + \dens^{\dom\cap H^-}_x \big) 
\end{equation*}
we get $\pr{f}(x)=\tilde f(x)$ also in this case. 
\end{proof+}

\section{Theory of traces}
\label{tt}

For the treatment of partial differential equations, Sobolev and BV functions
play an essential role. Since they cannot be evaluated directly on the
boundary, it is common to consider a trace operator for sufficiently
regular $\dom$.  As typical example available today we can take
$\dom\subset\R^n$ open and bounded with Lipschitz boundary to have a
linear continuous operator
\begin{equation*}
  T:\cW^{1,1}(\dom)\to\cL^1(\bd\dom,\cH^{n-1})
\end{equation*}
such that for all $f\in\cW^{1,1}(\dom)$ and $\phi\in C^1(\ol\dom,\R^n)$\\
\begin{equation} \label{tt-e0}
  \I{\dom}{f\divv\phi}{\lem} + \I{\dom}{\phi\cdot Df}{\lem} =
  \I{\bd\dom}{\phi\cdot\normal{\dom}\, Tf }{\cH^{n-1}}
\end{equation}
(cf. \cite[p.~133]{evans}, \cite[p.~168]{pfeffer}). Here the surface integral
on the right hand side is related to the vector-valued Radon measure 
$\normal{\dom}\, Tf \cH^{n-1}\lfloor\bd\dom$. This basically restricts
\reff{tt-e0} to sets $\dom$ of finite perimeter, since these are the sets 
having a suitable normal field on their
boundary. We will overcome that limitation by a much more general approach.

\subsection{General traces}
\label{tt-gt}

Notice that the left hand side in \reff{tt-e0} can be considered as linear
continuous functional $f^*\in C^1(\ol\dom,\R^n)^*$ such that 
\begin{equation} \label{tt-e1} 
  \df{f^*}{\tf}=0 \qmq{if}
  \tf_{|\bd\dom}=0 \,. 
\end{equation}
In this light we introduce a more general notion of trace. 
Let $\edom\subset \Rn$ be a Borel set, let $\K\subset\cl\edom$, and let $X$
be a normed space of functions $\tf:U\to\R^m$. A \textit{trace} or
\textit{trace functional} on $\K$ over $X$ is some $f^*\in X^*$ such that
for all $\tf\in X$ 
\begin{equation} \label{tt-e2} 
  \df{f^*}{\tf}=0 \quad\qmq{if}\quad \tf_{|\K_\delta\cap\edom}=0 
  \zmz{for some} \delta>0\,.
\end{equation}
Clearly $f^*$ in \reff{tt-e1} is a trace on $\bd\dom$ related to~$f$.
Since $\K_\delta=(\ol \K)_\delta$ for all $\delta>0$, it is
sufficient to consider traces on closed $\K$.

Let us motivate our approach by some traces over
$\cL^\infty(\edom)$ that in particular show that  
\reff{tt-e1} would be too restrictive for general traces we intend to study.

\begin{proposition}  \label{tt-s3}
Let $\edom \subset \Rn$ be a Borel set, let $\alse\in\bor{\edom}$, 
let $\K\subset\ol\edom$ be closed, 
and let $\gamma:(0,\infty)\to(0,\infty)$ be continuous such that
\begin{equation} \label{tt-s3-0}
  c := \limsup_{\delta \downarrow 0} \frac{1}{\gamma(\delta)}
  \I{\K_\delta\cap\alse}{}{\lem}  \quad \text{is finite.}
\end{equation}
Then there exists a measure    
\begin{equation*}
  \me_\K\in \bawl{\edom} \qmq{with} \cor\me_\K\subset \K\,, \quad
  |\me_\K|(\edom)\le c
\end{equation*}
such that $\f_\K^*\in\cL^\infty(\edom)^*$ related to 
$\f\me_\K$ is a trace on $\K$ for all $\f\in \cL^\infty(\edom)$
and 
\begin{eqnarray} 
  \liminf_{\delta \downarrow 0} \frac{1}{\gamma(\delta)}\,
  \I{\K_\delta\cap\alse}{\tf\f}{\lem} 
&\le& 
  \df{f^*_\K}{\tf} = \sI{\K}{\tf\f}{\me_\K} \nonumber \\
&\le&  \label{tt-s3-1}
  \limsup_{\delta \downarrow 0} \frac{1}{\gamma(\delta)}\,
  \I{\K_\delta\cap\alse}{\tf\f}{\lem} \,
\end{eqnarray}
for all $\tf \in \cL^\infty(\edom)$. If the limsup in \reff{tt-s3-0} is a
limit, then $|\me_\K|(\edom)=c$. The mapping
\begin{equation*} 
  T:\cL^\infty(\edom)\to\cL^\infty(\edom)^* \qmq{with} Tf=f_\K^*
\end{equation*}
is linear and continuous. For fixed $f,\tf\in\cL^\infty(\edom)$ there is
a sequence $\delta_j\downarrow 0$ with
\begin{equation}  \label{tt-s3-2}
  \sI{\K}{\tf\f}{\me_\K} = 
  \lim_{j\to\infty} \frac{1}{\gamma(\delta_j)}\,
  \I{\K_{\delta_j}\cap\alse}{\tf\f}{\lem} \,.
\end{equation}
\end{proposition}
\noi
For a nontrivial measure $\mu_\K$ one obviously needs
\begin{equation*}
  \I{\K_\delta\cap\alse}{}{\lem} = \lem(\K_\delta\cap\alse)
  >0 \qmz{for all} \delta>0\,.
\end{equation*}
In applications we consider the special choices $\gamma(\delta)=\delta$ and
\begin{equation}  \label{tt-s3-3}
  \gamma(\delta)=\lem(\K_\delta\cap\alse) \quad 
  \mbox{\Big( {\small  thus \;\;
     $\displaystyle
      \frac{1}{\gamma(\delta)}\, \I{\K_\delta\cap\alse}{\tf\f}{\lem} =
      \mI{\K_\delta\cap\alse}{\tf\f}{\lem}$  }\Big).}
\end{equation}
Notice that $|\me_\K|(\edom)=1$ in case \reff{tt-s3-3}. 
In some examples the following approximation result turns out to be helpful.

\begin{corollary}  \label{tt-s4}
Let $\edom \subset \R^n$, $\alse\in\bor{\edom}$, $\K\subset\ol\edom$, 
$\gamma\in C(\R_{>0})$, and $\me_\K$ be as in Proposition~{\rm \ref{tt-s3}}.
Moreover let $\tilde\delta>0$ and $\tf_k,\tf\in\cL^\infty(\edom)$ be such
that 
\begin{equation} \label{tt-s4-1}
  \frac{1}{\gamma(\delta)} \I{\K_\delta\cap\alse}{\tf}{\lem} =
  \lim_{k\to\infty} \frac{1}{\gamma(\delta)} \I{\K_\delta\cap\alse}{\tf_k}{\lem} 
  \qmz{uniformly for} \delta\in(0,\tilde\delta)\,,
\end{equation}
\begin{equation}  \label{tt-s4-2}
  \sI{\K}{\tf_k}{\me_\K} =
  \lim_{\delta \downarrow 0} \frac{1}{\gamma(\delta)}
  \I{\K_\delta\cap\alse}{\tf_k}{\lem} 
  \qmz{for all} k\in\N\,.
\end{equation}
Then 
\begin{equation*}
   \sI{\K}{\tf}{\me_\K} = \lim_{k\to\infty}\sI{\K}{\tf_k}{\me_\K}\,.
\end{equation*}
\end{corollary}

Let us discuss the results before we provide the proofs.
For $\alse=\edom$, $\K=\{x\}$ and $\gamma$ as in \reff{tt-s3-3},
the proposition provides a measure $\me_x$, that we call $\dens_x$
in accordance with \reff{ex:dzero-1}, and a trace $f^*_x$
on $\{x\}$ over $\cL^\infty(\edom)$.
We readily notice that $f_x^*$ cannot satisfy \reff{tt-e1} with $\{x\}$
instead of $\bd\dom$, since $\tf_{|\{x\}}$ cannot be defined in a reasonable
way. However, the trace $f_x^*$ provides an
evaluation of $f$ at $x$ due to
\begin{equation*}
  \bar f(x) := \sI{\{x\}}{f}{\dens_x} = \df{f_x^*}{1} 
  \qmq{for all} f\in\cL^\infty(\dom) \,.
\end{equation*}
By \reff{tt-s3-1} this agrees with $f(x)$ if $x$ is Lebesgue point of 
$f\in L^\infty(\edom)$
(let us mention that this need not be the case for any extension $\me_x$ of
$\delta_x$ according to Proposition~\ref{acm-ram} as, e.g., in
Example~\ref{pm-ex6}). The mapping $f\to f_x^*$ can be considered as a
trace operator on $\cL^\infty(\dom)$ at $x$. If we fix $f$ and vary $x$, then
we have a pointwise integral representation of $f$ that agrees a.e. with its 
precise representative $\pr{f}$ (cf. Remark~\ref{pm-rem9}). 

Let us discuss Proposition~\ref{tt-s3} with 
\begin{equation*}
  \edom=\alse=(0,1)^2\subset\R^2\,, \quad
  \K=\bd\edom \,, \quad \gamma(\delta)=\delta\,.
\end{equation*}
Then, for fixed $f\in\cL^\infty(\edom)$, the measure $f_{\bd\edom}^*$ 
is a trace on $\bd\edom$. 
For $g,\tf\in\cL^\infty(\edom)$ we obviously have
\begin{equation*}
  \Big|\, \sI{\bd\edom}{\tf(f-g)}{\me_\K} \Big|\le 
  \|\tf\|_{\cL^\infty(\edom)} \, \|f-g\|_{\cL^\infty((\bd\edom)_\delta)}
  \qmq{for all} \delta>0\,.
\end{equation*}
Therefore the trace $g_{\bd\edom}^*=g\me_{\bd\edom}$ agrees with 
$f_{\bd\edom}^*$ for all
functions $g$ in the affine linear subspace
\begin{equation*}
  X_f = \big\{g\in\cL^\infty(\edom)\:\big|\: \|f-g\|_{\bd\edom}=0\big\} \,
\end{equation*}
(cf. \reff{semi-norm}). This somehow means that $g\in X_f$ behaves as 
$f$ arbitrarily close to $\bd\edom$ and $f_{\bd\edom}^*$ appears to be an
appropriate tool to describe that behavior. If we restrict our attention to 
$\tf\in C(\ol\edom)$, then we can identify $f_{\bd\edom}^*$ with a
$\sigma$-measure $f_{\bd\edom}^\sigma$ supported on $\bd\edom$. 
By the application of \reff{tt-s3-2} to smoothened versions of $\tf=\chi_R$ 
with rectangles $R$ intersecting $\bd\edom$, we obtain that $f_{\bd\edom}^\sigma$
is (weakly) absolutely continuous with respect to $\reme{\cH^1}{\bd\edom}$. 
Hence there is a density function $f^\sigma$ on $\bd\edom$ such that
$f_{\bd\edom}^\sigma=f^\sigma\reme{\cH^1}{\bd\edom}$. However, in the general
case with $\tf\in\cL^\infty(\edom)$ we cannot find a function on $\bd\edom$
representing the measure $f_{\bd\edom}^*$.

We still consider
\begin{equation*}
  \dom=\dom_0\cup\dom_1\subset\R^2 \qmq{with}
  \dom_j = (j,j+1)\times(0,1) 
\end{equation*}
and $f\in\cL^\infty(\dom)$ given by
\begin{equation*}
  f=0 \zmz{on} \dom_0\,, \quad f=1 \zmz{on} \dom_1
\end{equation*}
(which is merely a representative of an equivalence class).
Obviously $f\in \cB\cV(\dom)$. 
For $\K=\bd\dom$ and $\gamma(\delta)=\delta$
the proposition provides a measure $\me_\K$ and a trace 
$f_\K^*$ on $\K$ such that
\begin{equation*}
  \df{f^*_\K}{\tf} = \sI{\K}{\tf\f}{\me_\K} \qmq{for all} 
  \tf\in\cL^\infty(\dom)\,.
\end{equation*}
In the light of usual traces we can try to assign a function $f_\K$ on $\K$ to
the trace $f_\K^*$ by the requirement
\begin{equation}\label{tt-e4}
  \sI{\K}{\tf\f}{\me_\K} = \I{\K}{\tf\f_\K}{\cH^1} 
\end{equation} 
for suitable $\tf$ that, however, have to be extendable up to $\K$. 
For $\tf\in C(\ol\dom)$ we get from \reff{tt-s3-1} 
\begin{equation*}
  f_\K = \Bigg\{ 
  \mbox{\small $ 
  \begin{array}{ll} 0 & \text{on $\bd\dom_0\!\setminus\bd\dom_1$}\,, \\
                    1 & \text{on $\bd\dom_1\!\setminus\bd\dom_0$} \,, \\
          \frac{1}{2} & \text{on $\bd\dom_0\!\cap\bd\dom_1$} \,.   
  \end{array}$ } 
\end{equation*}
But notice that $f_\K$ cannot provide the precise behavior of $f$ near 
$\bd\dom_0\!\cap\bd\dom_1$ while $f_\K^*$ can give the full information
by using $\tf\in C(\dom)$. In addition, $f_\K^*$ is not restricted to $\tf$
that are extendable up to $\K$. This property of general (finitely additive)
measures 
allows the construction of much more general Gauss-Green formulas than before.

Finally let us roughly sketch how we extend the Gauss-Green formula
\reff{tt-e0} to general Borel sets $\dom$ contained in some open set
$\edom\subset\R^n$. For $f\in \cW^{1,1}(\edom)$ the left hand side can
obviously be considered as functional $f^*\in\cW^{1,\infty}(\edom,\R^n)^*$. 
Based on Proposition~\ref{pm-s5} it can be shown that $f^*$ is
related to measures 
\begin{equation*}
   \lf \in \bawl{\edom}^n \qmq{and} \mf \in \bawl{\edom}^{n\times n}
\end{equation*}
such that
\begin{equation*}
  \I{\dom}{f\divv\phi}{\lem} + \I{\dom}{\phi\cdot Df}{\lem} =
  \df{\lf}{\tf} + \df{\mf}{D\tf} \,
\end{equation*}
for all $\tf\in\cW^{1,\infty}(\edom,\R^n)$ where, in full generality, 
the core of $\lf$ and $\mf$ belongs to a small neighborhood of $\bd\dom$.
In 'better' cases their core belongs to $\bd\dom$
and we get
\begin{equation*}
  \I{\dom}{f\divv\phi}{\lem} + \I{\dom}{\phi\cdot Df}{\lem} =
  \sI{\bd\dom}{\tf}{\lf} + \sI{\bd\dom}{D\tf}{\mf}\,.
\end{equation*}
In 'even better' cases $\lf$ can be considered as Radon measure on $\bd\dom$ and
$\mf$ might disappear. If $\dom$ has some inner boundary, the measures 'know'
the function $f$ on both sides of it and $\lf$ cannot be a $\sigma$-measure
as in the previous example surrounding
\reff{tt-e4}. In some cases, the measure $\mf$ disappears 
and we get more structure for the other boundary term where, in particular, 
$f$ enters explicitly. More precisely, in these cases
there is a so-called normal measure $\nu\in\bawl{\edom}^n$ with
$\cor{\nu}\subset\bd\dom$, an extension of 
the pointwise outer normal function, such that
\begin{equation*}
  \I{\dom}{f\divv\phi}{\lem} + \I{\dom}{\phi\cdot Df}{\lem} =
  \sI{\bd\dom}{f\tf}{\nu}  \qmz{for all} 
  \tf\in\cW^{1,\infty}(\dom,\R^n)\,.
\end{equation*}
For some normal measures $\nu$ we even get that
\begin{equation*}
  \nu = \nu^\dom\dens_{\bd\dom}
\end{equation*}
for the normal field $\nu^\dom=D \big(\distf{\dom}-\distf{\dom^c}\big)$
and some density measure $\dens_{\bd\dom}$ as in Proposition~\ref{tt-s3}.

Summarizing it turns out that the results derived below would not be
possible in that generality with a notion of trace relying merely on pointwise
trace functions on the boundary $\bd\dom$.  
We will develop our theory first for vector fields having
divergence measure. Then the results for Sobolev
functions and BV functions are at least partially direct consequences.

\medskip

\begin{proof+}{ of Proposition \ref{tt-s3}} 
Let $X:=\cL^\infty(\edom)$. Then  
\begin{equation*}
  X_0 := \Big\{\tf \in X \;\Big|\; 
  \lim_{\delta\downarrow 0} \frac{1}{\gamma(\delta)}
  \I{\K_\delta \cap \alse}{\tf}{\lem} 
  \text{ exists} \Big\} \,
\end{equation*}
is a linear subspace. $g_0^*:X_0\to\R$ with
\begin{equation*}
  g_0^*(\tf) := \lim_{\delta \downarrow 0} \frac{1}{\gamma(\delta)}
  \I{\K_\delta \cap \alse}{\tf}{\lem}
\end{equation*}
is a continuous linear functional on $X_0$ majorized by the  
positively homogeneous and subadditive functional
$\tilde g:\cL^\infty(\edom) \to \R$ given by
\begin{equation*}
  \tilde g(\tf) := \limsup_{\delta\downarrow 0} \frac{1}{\gamma(\delta)}
  \I{\K_\delta\cap\alse}{\tf}{\lem} 
  \le c\|\tf\|_{\cL^\infty}\,.
\end{equation*}
The Hahn-Banach theorem provides an extension 
$g^*\in\cL^\infty(\edom)^*$ of $g_0^*$ that is also majorized by $\tilde g$ on
$X$.   
Hence $\|g^*\|\le c$ and
\begin{eqnarray}
  \liminf_{\delta \downarrow 0} \frac{1}{\gamma(\delta)}
  \I{\K_\delta \cap \alse}{\tf}{\lem}
&=&
  -\limsup_{\delta \downarrow 0} \frac{1}{\gamma(\delta)}
  \I{\K_\delta \cap \alse}{- \tf}{\lem} =
  - \tilde g(-\tf)  \nonumber \\
&\le& -\df{g^*}{-\tf} = \df{g^*}{\tf} \nonumber \\
&\le&
  \tilde g(\tf) = \limsup_{\delta\downarrow 0} \frac{1}{\gamma(\delta)}
  \I{\K_\delta\cap\alse}{\tf}{\lem}  \label{tt-s3-5}
\end{eqnarray}
for all $\tf\in\cL^\infty(\edom)$. In the case where the limsup in
\reff{tt-s3-0} is a limit, we have for $\tf\equiv 1$ that 
$\|\tf\|_{\cL^\infty}=1$
and $\df{g^*}{\tf}=c$ by \reff{tt-s3-5} and, thus, $\|g^*\|=c$. 
If $\tf_{|\K_{\delta'}}=0$ 
for some $\delta' >0$, then obviously
\begin{equation} \label{tt-s3-6}
  0 = \lim \limits_{\delta\downarrow 0} \frac{1}{\gamma(\delta)} 
  \I{\K_\delta\cap\alse}{\tf}{\lem} = \df{g^*}{\tf}  
\end{equation}
and, hence, $g^*$ is a trace on $\K$. 

Let $\me_\K\in \bawl{\edom}$ be related to $g^*\in\cL^\infty(\edom)^*$ as in
Proposition~\ref{pm-s5}. Then
\begin{equation*}
  \df{g^*}{\tf} = \sI{\K}{\phi}{\me_\K} \qmq{for all}
  \phi\in\cL^\infty(\edom)\,,
\end{equation*}
$\cor{\me_\K} \subset \K$ by \reff{tt-s3-6}, and $|\me_\K|(\edom)=\|g^*\|$.

For $f\in\cL^\infty(\edom)$ we now consider 
$\f_\K^*\in\cL^\infty(\edom)^*$ related to $\f\me_\K$ and we have
\begin{equation*}
  \df{f_\K^*}{\tf} = \df{g^*}{\tf f} = \sI{\K}{\phi f}{\me_\K}
  \qmq{for all} \phi\in\cL^\infty(\edom)
\end{equation*} 
(cf. \reff{pm-e5}).
Obviously $f_\K^*$ is also a trace on $\K$.
From \reff{tt-s3-5} we obtain \reff{tt-s3-1}. 
Clearly, the mapping $T$ is linear and, by
$\|Tf\|\le\|f\|_{\cL^\infty}|\me_\K|(\edom)$,
also continuous.

For the last statement we fix $f,\tf\in\cL^\infty(\edom)$ and set
\begin{equation*}
  \beta^i:=\liminf_{\delta \downarrow 0} \frac{1}{\gamma(\delta)} 
  \I{\K_\delta\cap\alse}{\tf\f}{\lem}\,,
  \quad
  \beta^s:=\limsup_{\delta \downarrow 0} \frac{1}{\gamma(\delta)}
  \I{\K_\delta\cap\alse}{\tf\f}{\lem}\,. 
\end{equation*}
By \reff{tt-s3-1} we have
\begin{equation*}
  \beta^i \le \beta:=\sI{\K}{\phi f}{\me_\K} \le \beta^s\,.
\end{equation*}
If $\beta=\beta^i$ or $\beta=\beta^s$ we use the definition of $\liminf$
or $\limsup$, respectively, to get the assertion. For 
$\beta\in(\beta^i,\beta^s)$ we first observe that 
\begin{equation*}
  \delta\to I(\delta) := \frac{1}{\gamma(\delta)}
  \I{\K_\delta\cap\alse}{\tf\f}{\lem}
\end{equation*}
is continuous for $\delta>0$. Hence the mapping $I$ attains the value 
$\gamma$ on each interval $(0,\tilde\delta)$ with $\tilde\delta>0$. But this
implies the statement.  
\end{proof+}

\medskip

\begin{proof+}{ of Corollary~\ref{tt-s4}} 
Let us fix $\eps>0$. By \reff{tt-s4-1} there is some $k_0$
such that for all $k>k_0$ and all $\delta\in(0,\tilde\delta)$
\begin{equation*}
  \frac{1}{\gamma(\delta)} \I{\K_\delta\cap\alse}{\tf}{\lem} - \eps \le
  \frac{1}{\gamma(\delta)} \I{\K_\delta\cap\alse}{\tf_k}{\lem} \le
  \frac{1}{\gamma(\delta)} \I{\K_\delta\cap\alse}{\tf}{\lem} +\eps \,.
\end{equation*}
Using the limit from \reff{tt-s3-2} with $f\equiv 1$ and using 
\reff{tt-s4-2} we obtain
\begin{equation*}
  \sI{\K}{\tf}{\me_\K}- \eps \le
  \sI{\K}{\tf_k}{\me_\K} \le
  \sI{\K}{\tf}{\me_\K} + \eps \qmz{for all} k>k_0\,.
\end{equation*}
Now the arbitrariness of $\eps>0$ implies the assertion.
\end{proof+}

\subsection{Traces of vector fields with divergence measure}
\label{tt-tvf}

In our further treatment we are interested in traces that describe the
behavior near the boundary $\bd\dom$ for vector fields 
where the distributional divergence is a Radon measure.
As special cases we consider Sobolev functions and BV functions.
In our subsequent treatment we always assume that 
$\edom\subset\R^n$ is an open set.

Let us first recall some notation. For vector fields 
$F=(F_1,\dots,F_m)$ we use 
\begin{equation*}
  \cL^p(\edom,\R^m):=\cL^p(\edom)^m\,, \quad 
  \cW^{1,p}(\edom,\R^m):=\cW^{1,p}(\edom)^m 
\end{equation*}
with the norms
\begin{eqnarray}
  \|F\|_p = \|F\|_{\cL^p}
&:=&
  \Big( \I{\edom}{|F|^p}{\lem}\,\Big)^\frac{1}{p}
  \qmq{for} 1\le p<\infty\,, \nonumber \\
  \|F\|_{\cW^{1,p}}
&:=&
  \big(\|F\|_p^p + \|DF\|_p^p\big)^{\frac{1}{p}}
  \qmz{for} 1\le p<\infty\,, \nonumber \\[2mm]
  \|F\|_\infty = \|F\|_{\cL^\infty} 
&:=& \label{norm-li}
  \essup{\edom}{|F|} \,, \\
  \|F\|_{\cW^{1,\infty}} 
&:=&
  \max \big\{\|F\|_\infty,\|DF\|_\infty\big\}   \nonumber
\end{eqnarray}
(where $|\cdot|$ is the Euclidean norm and $DF$ is interpreted as $mn$-vector).
Moreover
\begin{equation*}
  \|f\|_{\cB\cV} := \|f\|_1 + |Df|(\edom) \qmq{for}
  f\in\cB\cV(\edom) \,,
\end{equation*}
\begin{equation*} 
  \|\me\| := |\me|(\edom) 
  \qmq{for} \me\in\bawl{\edom}^m \,.
\end{equation*}
Proposition~\ref{pm-s5} tells us that $\bawl{\edom}^m$ is
the dual of $\cL^\infty(\edom,\R^m)$.

We say that $F\in\cL^1_{\rm loc}(U,\R^n)$
has {\it divergence measure} if there is a signed Radon measure on $\edom$
denoted by $\divv F$ such that
\begin{equation}  \label{tt-e3}
  \I{\edom}{F\cdot D\tf}{\lem} = -\I{\edom}{\tf}{\divv F}
  \qmq{for all} \tf\in C^\infty_c(\edom)\,
\end{equation}
(i.e. the distributional divergence of $F$ is a signed Radon measure).
By approximation, \reff{tt-e3} is even valid for all
$\tf\in\cW^{1,\infty}(\edom)$ having compact support in $\edom$
(take mollifications $\tf_n\in C^\infty_c(\edom)$ of $\tf$
such that $D\tf_n\to D\tf$ a.e. on $\edom$ and use dominated convergence on
the left hand side). The space of vector fields in $\cL^p$ having
divergence measure is denoted by 
\begin{equation*}
  \cD\cM^p(U):=\big\{ F\in\cL^p(U,\R^n) \:\big|\: 
  \left|\divv F\right|(\edom)<\infty \big\}\,, \quad
  1\le p\le\infty\,, 
\end{equation*}
where the total variation $\left|\divv\F\right|(\edom)$ equals
\begin{equation*}
  \left|\divv F\right|(\edom) = 
  \sup\Big\{ \I{\edom}{F\cdot D\tf}{\lem} \:\Big|\:
  \tf\in C^1_c(\edom),\;|\tf|\le 1 \Big\}\,.
\end{equation*}
We have that $\cD\cM^p(U)$ is a Banach space with the norm
\begin{equation*}
  \|F\|_{\cD\cM^p} := \|F\|_{\cL^p} + \left|\divv F\right|(U) \,. 
\end{equation*}

For a Borel set $\K\subset\ol\edom$ we use the semi norm on 
$\cL^\infty(\edom,\R^m)$ given by
\begin{equation*}
  \|\tf\|_\K := 
  \lim_{\delta\downarrow 0} \|\tf\|_{\cL^\infty(\K_\delta\cap\edom,\R^m)}
\end{equation*}
(cf. \reff{semi-norm}). With the subspace 
\begin{equation} \label{tt-e5}
  Z := \{ \tf\in\cL^\infty(\edom,\R^m) \mid \|\tf\|_\K=0 \}
\end{equation}
we define the factor space
\begin{equation} \label{tt-e6}
  \cL^\infty_\K(\edom,\R^m) := \cL^\infty(\edom,\R^m)/Z \,.
\end{equation} 
The equivalence class containing $\tf\in\cL^\infty(\edom,\R^m)$ is denoted by 
\begin{equation*} \label{eq-class}
  \tf_{\wr\K}:=\tf+Z \,.
\end{equation*}
This way we can describe $\tf$ infinitesimally close to~$\K$, 
but not at $\K$. If $\tf$ is continuous 
on a neighborhood of $\K$ with a continuous extension up to $\K$, we can
identify $\tf_{\wr\K}$ with the restriction $\tf_{|\K}$.

Now we are able to provide a large class of traces that will be the basis 
for upcoming general Gauss-Green formulas.

\begin{theorem} \label{ex:dmo_woi}
Let $\edom \subset \Rn$ be open and bounded and let $\dom \in \bor{\edom}$. 
Then there is a
linear continuous operator $T: \dmo{\edom}\to\woinf{\edom}^*$ such that
\begin{equation} \label{ex:dmo_woi-1}
  \df{TF}{\tf} =  \divv(\tf F)(\dom) =
  \I{\dom}{\tf}{\divv{\funv}} + \I{\dom}{\Deriv{\tf}\cdot \funv}{\lem} 
\end{equation}
for all $\tf\in\woinf{\edom}$
and $TF$ is a trace on $\bd{\dom}$ over $\woinf{\edom}$ for all
$F\in\dmo{\edom}$. Moreover
\begin{equation*}
  \df{TF}{\tf}=0 \qmq{if} \tf_{|(\bd\dom)_\delta\cap\dom}=0\,
\end{equation*}
for some $\delta>0$.
\end{theorem}
\noi
We call $T$ trace operator. 
In Theorem \ref{dt-s1} below 
we will exploit the structure of $\woinf{\edom}^*$ to get a general
representation for these traces.

\begin{remark}\label{ex:dmo_div} 
The functional $T^*\in \dmo{\edom}^*$ given by 
\begin{equation*}
  \df{T^*}{F} = \divv{F}(\dom) \,
\end{equation*}
is a trace on $\bd\dom$ over $\dmo{\edom}$ as one can see similar to the proof
of the theorem. Thus we could take $\df{T^*}{\tf F}$ instead of
$\df{TF}{\tf}$ in \reff{ex:dmo_woi-1}. The advantage would be 
to have merely one functional $T^*$ for all $F$. However, the lack of knowledge
about the structure of $\dmo{\edom}^*$ prevents a direct
representation of traces that way in general. 
\end{remark}

\begin{corollary} \label{ex:dmo_woi_c}
Let $\edom\subset\Rn$ be open and bounded, $\dom \in \bor{\edom}$,
$F\in\dmo{\edom}$, and let $T$ be as in 
\reff{ex:dmo_woi-1}. If $\tf\in\woinf{\edom}$ with
\begin{equation*}
  \tf_{\wr\bd\dom}=0 \quad \qmq{(i.e. $\|\tf\|_{\bd\dom}=0$)}
\end{equation*}
and 
\begin{equation} \label{dmo_woi_c-1}
  (D\tf)_{\wr\bd\dom}=0  \quad\qmq{or}\quad  \lem(\dom\setminus\op{int}\dom)=0\,,
\end{equation}
then $\df{TF}{\tf}=0$.
\end{corollary}
\noindent
Notice that \reff{dmo_woi_c-1} is satisfied if $\dom$ is open or if
$\lem(\bd\dom)=0$. For $\edom=\dom$ open and bounded 
the previous result is similar to Theorem 2.3 in \v Silhav\'y
\cite{silhavy_divergence_2009} where the right hand side in
\reff{ex:dmo_woi-1} is considered as functional over bounded 
$\tf\in \op{Lip}(\R^n)$ and it is shown that this functional agrees with a
linear continuous functional on $\op{Lip}(\bd\dom)$.

\begin{remark}  \label{tt-s15}
Corollary~\ref{ex:dmo_woi_c} readily implies that the trace $TF$ from
Theorem~\ref{ex:dmo_woi} is uniquely determined if it is known for all
$\tf\in\cW^{1,\infty}_c((\bd\dom)_\delta)$ having compact support in
$(\bd\dom)_\delta$ for some $\delta>0$.  
\end{remark}

As simple consequence of Theorem~\ref{ex:dmo_woi} we get some 
analogous statement for Sobolev functions and BV functions.
  
\begin{proposition} \label{tt-s5}
Let $\edom \subset \Rn$ be open and bounded and let $\dom \in \bor{\edom}$. 
Then there is a
linear continuous operator $T: \cB\cV(\edom)\to\woinf{\edom,\R^n}^*$ 
such that
\begin{equation} \label{ex:woo_woi-1}
  \df{Tf}{\tf} = \divv(f\tf)(\dom) =
  \I{\dom}{f\divv\tf}{\lem} + \I{\dom}{\tf}{Df} 
\end{equation}
for all $\tf\in\woinf{\edom,\R^n}$
and $Tf$ is a trace on $\bd{\dom}$ over $\woinf{\edom,\R^n}$
for all functions $f\in\cB\cV(\edom)$. If
\begin{equation*}
  \tf_{|(\bd\dom)_\delta\cap\dom}=0 
  \zmz{for some} \delta>0\,
\end{equation*}
or if
\begin{equation*}
   \tf_{\wr\bd\dom}=0 \qmq{and \reff{dmo_woi_c-1} is satisfied,}
\end{equation*}
then we have $\df{Tf}{\tf} = 0$.
\end{proposition}

\begin{remark} \label{tt-s6}
(1) For $f\in\cB\cV(\edom)$ the distributional partial derivatives $D_{x_k}f$ 
are signed Radon measures and the distributional gradient $Df$ is the
vector-valued Radon measure
\begin{equation} \label{tt-s6-1}
  Df=(D_{x_1}f,\dots,D_{x_n}f)
\end{equation}
(cf. \cite[p.~117]{ambrosio}). Thus the most right integral in 
\reff{ex:woo_woi-1} has to be taken as
\begin{equation*}
  \I{\dom}{\tf}{Df} = \sum_{k=1}^n \I{\dom}{\tf^k}{D_{x_k}f}
  \qmq{where} \tf=(\tf^1,\dots,\tf^n)\,.
\end{equation*}

(2) Proposition~\ref{tt-s5} is obviously valid for all Sobolev functions  
$f\in\cW^{1,1}(\edom)$, since they belong to $\cB\cV(\edom)$
(cf. \cite[p.~170]{evans}). For such $f$ the measure $Df$ equals
$Df(\cdot)\lem$ with the weak gradient $Df(\cdot)$ as density.
Therefore
\begin{equation*}
  |Df|(\dom) = \I{\dom}{|Df|}{\lem}\,, \quad 
  \|f\|_{\cB\cV} = \|f\|_{\cW^{1,1}} \,,
\end{equation*}
and in \reff{ex:woo_woi-1} we can replace
\begin{equation*}
  \I{\dom}{\tf}{Df} = \I{\dom}{\tf\cdot Df}{\lem}\,.
\end{equation*}

(3) Let $\edom\subset\R^n$ be open, bounded and let $\dom\in\bor{\edom}$.
We consider the space
\begin{equation*}
  X:=\{f\in\cW^{1,1}(\edom)\mid Df\in\cD\cM^1(\edom)\}\,, \quad
  \|f\|_X:=\|f\|_{\cL^1} + \|Df\|_{\cD\cM^1} \,. 
\end{equation*}
This means that $\Delta f$ in the sense of distributions is a Radon
measure. Now we define $T:X\to\woinf{\edom,\R^n}^*$ by
\begin{equation*}
  \df{Tf}{\tf} = \I{\dom}{\tf}{\Delta f} + \I{\dom}{Df\cdot D\tf}{\lem}\,.
\end{equation*}
Theorem~\ref{ex:dmo_woi} implies that $Tf$ is a trace on $\bd\dom$ and we
readily conclude that $T$ is continuous. For $\dom$ open with Lipschitz
boundary and $f$ smooth we obviously have
\begin{equation*}
  \df{Tf}{\tf} = \I{\bd\dom}{\tf Df\cdot\nu^\dom}{\hm} \,.
\end{equation*}
\end{remark}

\medskip

\begin{proof+}{ of Proposition \ref{tt-s5}}  
For $f\in\cB\cV(\edom)$ we consider the vector fields
\begin{equation*}
  F_k=(F_k^1,\dots,F_k^n) \qmq{with}
  F_k^k=f,\z F_k^j=0 \zmz{for} k\ne j 
\end{equation*}
where $k=1,\dots,n$.
Obviously $F_k\in\dmo{\edom}$ for all $k$ and, with the notation from
\reff{tt-s6-1}, 
\begin{eqnarray*}
  \|F_k\|_{\cD\cM^1} 
&=&
 \|F_k\|_{\cL^1}+|\divv F_k|(\edom) =
  \|f\|_{\cL^1}+|D_{x_k}f|(\edom) \\
&\le&
  \|f\|_{\cL^1}+|Df|(\edom) = \|f\|_{\cB\cV} \,.
\end{eqnarray*}
Hence, by Theorem~\ref{ex:dmo_woi},
there are linear and continuous mappings  
\begin{equation*}
  T_k:\cB\cV(\edom)\to\woinf{\edom,\R^n}^* \quad (k=1,\dots,n)
\end{equation*}
such that each $T_kf$ is a trace on $\bd\dom$ and such that
\begin{equation*}
   \df{T_kf}{\tf} = \divv(\tf^kF_k)(\dom) = \I{\dom}{\tf^k}{D_{x_k}f} + 
  \I{\dom}{f\tf^k_{x_k}}{\lem} \,
\end{equation*}
for all $\tf=(\tf^1,\dots,\tf^n)\in\woinf{\edom,\R^n}$.
Summing over $k$ we get the first statement of the proposition for
$T=\sum_kT_k$. The assertions related to $\df{Tf}{\tf} = 0$ follow directly
from Theorem~\ref{ex:dmo_woi} and Corollary~\ref{ex:dmo_woi_c}.
\end{proof+}

\bigskip

\begin{proof+}{ of Theorem \ref{ex:dmo_woi}}    
First we note that
\begin{equation} \label{ex:dmo_woi-4}
    \tf F \in \dmo{\edom} \qmq{for all} 
    F\in\dmo{\edom},\; \tf \in \cW^{1,\infty}(\edom)
\end{equation}
and that, as measures,
\begin{equation} \label{ex:dmo_woi-5}
  \divv(\tf F) = \tf\divv F + F\cdot D\tf \,\lem
\end{equation}
(cf. Proposition 2.2 and the subsequent comment in 
\cite{silhavy_divergence_2009}). 
For $F\in\dmo{\edom}$ we set
\begin{equation*}
    \df{TF}{\tf} := \divv (\tf F)(\dom) 
  \qmq{for all} \tf\in\cW^{1,\infty}(\edom) \,.
\end{equation*}
From \reff{ex:dmo_woi-5} we get \reff{ex:dmo_woi-1} and
\begin{equation*}
    |\df{TF}{\tf}| \le 
    \|\tf\|_\infty\left|\divv F\right|(U) + \|D\tf\|_\infty\|F\|_1 \le
    \|\tf\|_{\cW^{1,\infty}}\|\,\|F\|_{\cD\cM^1}\,. 
\end{equation*}
Hence $TF\in\cW^{1,\infty}(\edom)^*$ with $\|TF\|\le\|F\|_{\cD\cM^1}$.
Therefore $T$ is a linear continuous operator as stated.

We now consider $\tf\in\cW^{1,\infty}(\edom)$ with 
\begin{equation}\label{ex:dmo_woi-6}
  \tf_{|(\bd\dom)_\delta\cap\dom} = 0 \qmq{for some} \delta > 0 \,.
\end{equation}
Then
\begin{equation} \label{ex:dmo_woi-8}
  D\tf=0 \qmq{$\cL^n$-a.e. on} (\bd\dom)_\delta\cap\dom 
\end{equation}
(cf. \cite[p.~130]{evans}). 
For $\delta'=\frac{\delta}{3}$ we define   
$\chi\in\cW^{1,\infty}(\R^n)$ by
\begin{equation}\label{ex:dmo_woi-7}
  \chi(x):= \Bigg\{ 
  \mbox{\small $ 
  \begin{array}{ll} 1 & \text{for }x\in\dom_{-2\delta'}\,, \\
                    0 & \text{for }x\not\in\dom_{-\delta'}\,,  \\
      1- \tfrac{1}{\delta'} \dist{(\dom_{-2\delta'})}{x} & 
     \text{otherwise}\,.
  \end{array}$ } 
\end{equation}
Clearly,
\begin{equation*}
  0\le\chi\le 1\,, \quad \chi=1 \zmz{on} \ol{\dom_{-\delta}} \,, \quad
  \supp\chi\subset\ol{\dom_{-\delta'}} 
           \subset\ol{\edom_{-\delta'}} \,,
\end{equation*}
\begin{equation*}
  (1-\chi)\tf=0 \zmz{on}\dom\,, \z \tf D\chi=0 \zmz{on} \edom\,, 
\end{equation*}
\begin{equation*}
  (1-\chi)D\tf=0 \z\zmz{$\cL^n$-a.e. on} \dom \,.
\end{equation*}
Using \reff{ex:dmo_woi-1} we obtain
\begin{eqnarray*}
  \df{TF}{\tf} 
& = &
 \I{\dom}{(1-\chi)\tf}{\divv F} + \I{\dom}{(1-\chi) D\tf\cdot F}{\lem} 
  \nonumber\\
&&
 + \I{\dom}{\chi\tf}{\divv F} + 
   \I{\dom}{\chi D\tf \cdot F}{\lem}   \nonumber\\ 
&=&
  \I{\edom}{\chi\tf}{\divv F} + 
   \I{\edom}{\big(\chi D\tf + \tf D\chi\big) \cdot F}{\lem}  
   \nonumber\\ 
& = & 
  \I{\edom}{\chi\tf}{\divv F} + 
  \I{\edom}{F\cdot D(\chi\tf)}{\lem} \,. 
\end{eqnarray*}
Obviously $\chi\tf\in\cW^{1,\infty}(\edom)$
with compact support in $\edom$. The definition of divergence measure in 
\reff{tt-e3} and the subsequent comment imply
\begin{equation*}
  \I{\edom}{\chi\tf}{\divv F} + 
  \I{\edom}{F\cdot D(\chi\tf)}{\lem} = 0 \,.
\end{equation*}
Consequently $\df{TF}{\tf}=0$ for all $\tf\in\cW^{1,\infty}(\edom)$
satisfying \reff{ex:dmo_woi-6}. This shows the last statement in the theorem
and readily implies that $TF$ is a trace on $\bd{\dom}$ over
$\cW^{1,\infty}(\edom)$. 
\end{proof+}

\begin{proof+}{ of Corollary \ref{ex:dmo_woi_c}}    
For $\delta>0$ we use $\chi=\chi_\delta$ as in \reff{ex:dmo_woi-7} and we have 
\begin{equation*}
  0\le\chi_\delta\le 1\,, \quad \chi_\delta=1 \zmz{on} \ol{\dom_{-\delta}} \,, 
  \quad  \supp\chi_\delta\subset\ol{\dom_{-\frac{\delta}{3}}} 
           \subset\ol{\edom_{-\frac{\delta}{3}}} \,.
\end{equation*}
If $\tf\in\cW^{1,\infty}(\edom)$, then 
$\chi_\delta\tf\in\cW^{1,\infty}(\edom)$, it has compact support on 
$\edom$, and it vanishes outside $\dom$. 
Hence $D\chi_\delta=0$ $\cL^n$-a.e. on 
$\edom\setminus(\dom\setminus\dom_{-\delta})$
(cf. \cite[p.~130]{evans}). Therefore the definition of divergence
measure in \reff{tt-e3} and the subsequent comment give for all $\delta>0$
\begin{eqnarray*}
  0
&=&
  \I{\edom}{\chi_\delta\tf}{\divv F} + 
  \I{\edom}{F\cdot D(\chi_\delta\tf)}{\lem} \\
&=&
  \I{\dom}{\chi_\delta\tf}{\divv F} + 
  \I{\dom}{\chi_\delta F\cdot D\tf}{\lem} +
  \I{\dom}{\tf F\cdot D\chi_\delta}{\lem} \,.
\end{eqnarray*}
Consequently
\begin{eqnarray}
  \df{TF}{\tf} 
& = &
 \I{\dom}{(1-\chi_\delta)\tf}{\divv F} + 
 \I{\dom}{(1-\chi_\delta) D\tf\cdot F}{\lem}   \nonumber\\
&&
 + \I{\dom}{\chi_\delta\tf}{\divv F} + 
   \I{\dom}{\chi_\delta D\tf \cdot F}{\lem}   \nonumber\\ 
& = & 
  \I{\dom}{(1-\chi_\delta)\tf}{\divv F} + 
  \I{\dom}{(1-\chi_\delta) D\tf\cdot F}{\lem} \nonumber\\
&& \label{ex:dmo_woi_c-5}
  - \I{\dom}{\tf F\cdot D\chi_\delta}{\lem} \,.
\end{eqnarray}

Let now $\tf_{\wr\bd\dom}=0$. 
Then $\tf(x)=0$ for $x\in\bd\dom\cap\dom$ by continuity of $\tf$ on $\edom$.
For $x\in\op{int}\dom\setminus\dom_{-\delta}$ there is $x'\in\bd\dom$ such that
\begin{equation*}
  |x-x'|=\dist{\bd\dom}{x} \le\delta 
\end{equation*}
and, consequently,
\begin{equation*}
  B_{|x-x'|}(x)\subset\dom \,.
\end{equation*}
For any $\delta'>0$ we find $x''\in(x',x)$ (open segment connecting $x$, $x'$)
with $|\tf(x'')|<\delta'$ by $\|\tf\|_{\bd\dom}=0$. Therefore
\begin{equation*}
  |\tf(x)| \le |\tf(x)-\tf(x'')| + |\tf(x'')| \le
  \delta\|D\tf\|_\infty + \delta'\,.
\end{equation*}
The arbitrariness of $\delta'$ and $x$ implies
\begin{equation*}
  |\tf(x)|\le \delta\|D\tf\|_\infty  
  \qmq{on} \dom\setminus\dom_{-\delta}\,.
\end{equation*}
Using $|D\chi_\delta|\le\frac{3}{\delta}$ $\cL^n$-a.e. on $\edom$ and 
$D\chi_\delta=0$ $\cL^n$-a.e. on 
$\edom\setminus\big(\dom\setminus\dom_{-\delta} \big)$ we get
\begin{equation*}
  |\tf D\chi_\delta|\le 3\|D\tf\|_\infty 
  \qmq{$\cL^n$-a.e. on} \edom \,.
\end{equation*}
Since $D\chi_\delta(x)\to 0$ for all $x\in\edom$, dominated convergence gives 
\begin{equation*}
  \I{\edom}{\tf F\cdot D\chi_\delta}{\lem} \to 0 \qmq{as} \delta\to 0\,.
\end{equation*}
From $\tf_{\wr\bd\dom}=0$ we also obtain that
\begin{equation*}
  \big|(1-\chi_\delta(x))\tf(x)\big| \le 
  \|\tf_{|\dom\setminus\dom_{-\delta}}\|_\infty
  \overset{\delta\to 0}{\longrightarrow} 0 \qmq{for all} x\in\dom
\end{equation*}
and, therefore,
\begin{equation*}
  \I{\dom}{(1-\chi_\delta)\tf}{\divv F} \to 0 \qmq{as} \delta\to 0\,.
\end{equation*}
Analogously, $(D\tf)_{\wr\bd\dom}=0$ implies
$\big|(1-\chi_\delta(x))D\tf(x)\big|\overset{\delta\to 0}{\longrightarrow} 0$
for all $x\in\dom$. Thus
\begin{equation*}
  \I{\dom}{(1-\chi_\delta)F\cdot D\tf}{\lem} \to 0 \qmq{as} \delta\to 0\,.
\end{equation*}
The last convergence is also obtained in the case where
$\lem(\dom\setminus\op{int}\dom)=0$, since
then $\chi_\delta\to 1$ $\lem$-a.e. on $\dom$.

Now we can take the limit $\delta\to 0$
in \reff{ex:dmo_woi_c-5} to get $\df{TF}{\tf}=0$.
\end{proof+}

\subsection{Representation of traces}
\label{tt-r}

For a powerful theory we now need suitable representations for the traces 
introduced in the previous section. In general we are interested 
in traces over $\woinf{\edom,\R^m}$ which requires
representations of functionals in $\woinf{\edom,\R^m}^*$.
That we do not interrupt the presentation of the subsequent essential results, 
we collect all proofs in the next subsection.

Let us start with some preliminary considerations. 
First we consider the semi norm $\|\tf\|_\K$ and the related
factor space $ \cL^\infty_\K(\edom,\R^m)$ introduced in \reff{tt-e6}.
\begin{lemma}\label{lem-fs}
Let $\edom\subset\R^n$ be open and let $\K\subset\ol\edom$ be a Borel set.
Then the subspace $Z=\{\tf\in\cL^\infty(\edom,\R^m)\mid\|\tf\|_\C=0\}$ 
(cf. \reff{tt-e5})
is closed and, for all $\tf\in\cL^\infty(\edom,\R^m)$, 
\begin{equation}\label{lem-fs-1}
  \|\tf\|_\K = \op{dist}_Z \tf := 
  \inf_{\psi\in Z} \|\tf-\psi\|_{\cL^\infty(\edom)} \,,
\end{equation}
\begin{equation}\label{lem-fs-2}
  \|\tf\|_\K=\|\tf+\psi\|_\K \qmq{if} \psi\in Z \,.
\end{equation}
Moreover 
\begin{equation*}
  \cL^\infty_\K(\edom,\R^m) = \big\{ \tf_{\wr\K} \:\big|\:
  \tf_{\wr\K} = \tf +Z,\: \tf\in\cL^\infty(\edom,\R^m) \big\}  
\end{equation*}
is a Banach space with the norm $\|\tf_{\wr\K}\|:=\|\tf\|_\K$ and 
\begin{equation*}
  \cL^\infty_\K(\edom,\R^m)^* \cong 
  \{f^*\in \cL^\infty(\edom,\R^m)^*\mid \df{f^*}{\tf}=0
  \text{ for all } \tf\in Z \}
\end{equation*}
as isometric isomorphism. For each $f_\K^*\in \cL^\infty_\K(\edom,\R^m)^*$ 
there is some vector-valued measure $\me\in\bawl{\edom}^m$ such that
\begin{equation*}
  \df{f_\K^*}{\tf_{\wr\K}} = \I{\edom}{\tf}{\me} \qmz{for all}
  \tf_{\wr\K}\in\cL^\infty_\K(\edom,\R^m)
\end{equation*}
and 
\begin{equation*}
  \I{\edom}{\tf}{\me}=0 \qmz{for all} \tf\in Z 
  \zmz{(i.e. $\|\tf\|_\K=0$)\,.}
\end{equation*}
If $\edom$ is bounded, then $\cor{\me}\subset\ol \K$.
\end{lemma}

For the characterization of $\woinf{\edom,\R^m}^*$ we identify 
$\woinf{\edom,\R^m}$ with a subspace of a product space of the form
\begin{equation*}
  \big\{ (\tf,D\tf)\in X_0\times\cL^\infty(\edom,\R^{mn}) \:\big| \: 
  \tf\in \woinf{\edom,\R^m} \big\}
\end{equation*}
where $X_0$ is a suitable Banach space (e.g. $\cL^\infty(\edom,\R^m)$ 
in the general case or $C(\K,\R^m)$ if all $\tf$ are continuously extendable
up to $\K\subset\ol\edom$).

\begin{lemma} \label{lem-ps2}
Let $X_0$ be a Banach space with $\|.\|_{X_0}$ and let 
$\edom\subset\R^n$ be open. Then 
\begin{equation*}
  X=X_0\times\cL^\infty(\edom,\R^{mn}) \qmq{with}
  \|(\tf,\tF)\|_X= \max\big\{ \|\tf\|_{X_0},\|\tF\|_{\cL^\infty} \big\}
\end{equation*}
is a Banach space and $X^*$ is isometrically isomorphic to 
\begin{equation*}
  X_0^*\times \bawl{\edom}^{mn} \qmq{with}
  \|(f^*,\me)\| = \|f^*\| + \|\me\| 
\end{equation*}
such that
\begin{equation} \label{lem-ps2-1}
  \big\langle (f^*,\me),(\tf,\tF) \big\rangle = 
  \df{f^*}{\tf} + \I{\edom}{\tF}{\me} 
\end{equation}
for all $(\tf,\tF)\in X$, $(f^*,\me)\in X^*$ 
and $\|\me\|=|\me|(\edom)$. 
\end{lemma}

That we can use the lemma for the description of $\woinf{\edom,\R^m}^*$,
we need an injection from $\woinf{\edom,\R^m}$ onto a
subspace of $X$. 

\begin{proposition} \label{prop:dual1}
Let $X_0$ be a Banach space, let $\edom\subset\R^n$ be open, and assume that
there is a linear mapping 
$\iota_0:\cW^{1,\infty}(\edom,\R^m)\to X_0$ such that 
\begin{equation*}
  \iota:\cW^{1,\infty}(\edom,\R^m)\to X_0\times\cL^\infty(\edom,\R^{mn})
  \qmq{with} \iota(\tf) = (\iota_0(\tf),D\tf)
\end{equation*}
has a continuous inverse 
$\iota^{-1}$ on its image $\iota\big(\cW^{1,\infty}(\edom,\R^m)\big)$
equipped with 
\begin{equation*}
  \|\iota(\tf)\|=\max \big\{ \|\iota_0(\tf)\|_{X_0}, \|D\tf\|_\infty \big\} \,.
\end{equation*}
Then for each $f^*\in\cW^{1,\infty}(\edom,\R^m)^*$ there are 
$f_0^*\in X_0^*$ and $\me\in\bawl{\edom}^{mn}$ such that
\begin{equation*}
  \df{f^*}{\tf} = \df{f_0^*}{\iota_0(\tf)} + \I{\edom}{D\tf}{\me} 
\end{equation*}
for all $\tf\in\cW^{1,\infty}(\edom,\R^m)$. Moreover
\begin{eqnarray}
  \|(f_0^*,\me)\| 
&=&
  \|f_0^*\| + \|\me\| \: = \:
  \|f^*\circ\iota^{-1}\| \nonumber \\
&\le&  \label{prop:dual1-3}
  \|\iota^{-1}\| \|f^*\| \: = \:
  \|\iota^{-1}\|
  \sup_{\substack{\tf\in\cW^{1,\infty}(\edom,\R^m)\\
        \|\tf\|_{\cW^{1,\infty}}\le 1} }
    \df{f^*}{\tf} \,.  
\end{eqnarray}
\end{proposition}

\begin{remark} \label{rem:dual1}
The assumption for $\iota$ is satisfied if 
$\iota\big(\cW^{1,\infty}(\edom,\R^m)\big)$ is closed and $\iota$ 
is continuous and injective, since then $\iota^{-1}$ is continuous by the open
mapping theorem. Alternatively an estimate 
\begin{equation*}
  \|\tf\|_{\cW^{1,\infty}}\le \tilde c\, \|\iota(\tf)\| 
  \qmz{for all} \tf\in \cW^{1,\infty}(\edom,\R^m) 
\end{equation*}
with some $\tilde c\ge 0$ would be sufficient for existence and continuity of
$\iota^{-1}$. 
\end{remark}

For $X_0=\cL^\infty(\edom,\R^m)$ and $\iota_0(\tf)=\tf$, the assumptions of
Proposition~\ref{prop:dual1} are satisfied and we obtain a
representation of $\woinf{\edom,\R^m}^*$ related to measures on $\edom$
without further specification of their core. 
At this point we have to realize that the representation of 
$f^*\in\woinf{\edom,\R^m}^*$ by means of measures is not unique in general
and that even traces on $\K\subset\ol\edom$ over
$\woinf{\edom,\R^m}$ can be represented with measures that are supported on
all of $\edom$. Indeed, if we take the trace $f^*=TF$ from
\reff{ex:dmo_woi-1}, then the right hand side of \reff{ex:dmo_woi-1} itself
gives a representation of it related to the measures
\begin{equation*}
  \divv F\in\cM(\edom) \qmq{and} F\cL^n\in\bawl{\edom}^{n}\,.
\end{equation*}
These measures might be provided by Proposition~\ref{prop:dual1} for $f^*$
if we consider $\divv F$ as extension to some element of $\bawl{\edom}$ 
according to Proposition~\ref{acm-ram}. However, having in mind
usual Gauss-Green formulas, we are rather interested in representations of
traces on $\K$ with measures having core on or at least near $\K$. 

For such a localization we use the 
\textit{\tent\ function} $\port{\K}_\delta:\R^n\to\R$ 
of $\K$ and $\delta>0$ given by
\begin{equation}\label{eq:port_fun}
    \port{\K}_\delta := \ind{\K_{\frac{\delta}{2}}} +
    \ind{\K_\delta\setminus \K_{\frac{\delta}{2}}}
    \big( 2 - \tfrac{2}{\delta}\distf{\K}\big) 
\end{equation}
which is 1 on $\K_{\frac{\delta}{2}}$ and 0 outside $\K_\delta$ (cf. 
Figure \ref{fig:portbd}). Note that $\port{\K}_\delta \in \woinf{\edom}$, 
since it is Lipschitz continuous on $\Rn$.
Let us mention that other choices of $\port{\K}_\delta$ provide the same
results as long as $\port{\K}_\delta$ is Lipschitz continuous with support
on $\ol{\K_\delta}$ and with $\port{\K}_\delta =1$
on some neighborhood of~$\K$.

\begin{figure}[H] 
  \centering
  \begin{tikzpicture}
    \draw (-2.5,0) -- (2.5,0);
    \draw (-1,1.5) -- (1,1.5);
    \draw (1,1.5) -- (2,0);
    \draw (0,0) -- (0,1.5);
    \draw (1,1.5) -- (1,0);
    \draw (-1,1.5) -- (-1,0);
    \draw (-2,0) -- (-1,1.5);
    \node at (0,-0.5) {$\K$};
    \node at (-1,-0.5) {$\frac{\delta}{2}$};
    \node at (1,-0.5) {$\frac{\delta}{2}$};
    \node at (2,-0.5) {$\delta$};
    \node at (-2,-0.5) {$\delta$};
  \end{tikzpicture}
  \vspace{-2.5ex}
  \caption{\Tent\ function $\port{\K}_\delta$ of $\K$ and $\delta$.}
  \label{fig:portbd}
\end{figure}
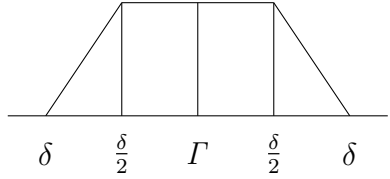

\begin{proposition}\label{prop:dual0}
Let $\edom\subset\R^n$ be open, let $\K\subset\cl\edom$ be closed, let
$\delta>0$, and let $f^*\in\cW^{1,\infty}(\edom,\R^m)^*$
be a trace on $\K$.
Then there is some 
\begin{equation}\label{prop:dual0-1}
  f_\delta^*\in\cW^{1,\infty}(\K_\delta\cap\edom,\R^m)^* \qmq{with}
  \df{\de}{\tf} = \df{f^*}{\port{\K}_\delta \tf} =
  \df{f_\delta^*}{\tf_{|\K_\delta\cap\edom}}
\end{equation}
for all $\tf\in\cW^{1,\infty}(\edom,\R^m)$ and 
\begin{equation*}
  \|f_\delta^*\|\le 2\|\port{\K}_\delta\|_{\cW^{1,\infty}(\edom,\R^m)}\|f^*\| 
  \,.
\end{equation*}
\end{proposition}

Now we are able to represent any trace $f^*$ on $\K$ over  
$\cW^{1,\infty}(\edom,\R^m)$ by a functional   
$f^*_\delta\in\cW^{1,\infty}(\K_\delta\cap\edom,\R^m)^*$. This certainly 
leads to sharper results for the core of the related measures. 
But notice that $f^*_\delta$ really can depend on $\delta>0$ in general 
(cf. Example~\ref{dt-ex2} below). However if the  
$f_\delta^*$ are bounded in some sense, then $f^*$ can be represented 
by measures with core in $\K$ (cf. Proposition~\ref{prop:special-a} below).

For a more precise analysis we apply Proposition \ref{prop:dual1} 
to $f_\delta^*$ of Proposition~\ref{prop:dual0}. Here we use  
three choices of $X_0$ with corresponding 
$\iota_0:\cW^{1,\infty}(\edom,\R^m)\to X_0$
that turn out to be of particular relevance.
We say that we have case (G), (L), or (C) for $\K\subset\cl\edom$ and
$\delta>0$ if the assumption in 
Proposition~\ref{prop:dual1} is satisfied with $\K_\delta\cap\edom$
instead of $\edom$ and with

\smallskip

\bgl
\item[(G)] 
$X_0=\cL^\infty(\K_\delta\cap\edom,\R^m)$ and $\iota_0(\tf)=\tf$,

\item[(L)]
$X_0=\cL^\infty_\K(\K_\delta\cap\edom,\R^m)$ and $\iota_0(\tf)=\tf_{\wr\K}$, 

\item[(C)] 
$X_0=C(\K,\R^m)$ and $\iota_0(\tf)=\tf_{|\K}$ 
where all $\tf\in\cW^{1,\infty}(\K_\delta\cap\edom,\R^m)$ 
are assumed to be continuously extendable up to $\K$.
\el

\smallskip

Obviously we always have the {\it general case} (G), since
here $\iota$ is an isometry for all $\delta>0$. 
The other cases, that we call {\it Lebesgue case} (L) and 
{\it continuity case} (C), turn out to be a condition for the
geometry of $\K\subset\ol\edom$ related to $\delta$.
With Lemma~\ref{lem-fs} we readily get for bounded $\edom$
that $f_0^*\in X_0^*$ from  
Proposition~\ref{prop:dual1} corresponds to a measure $\ome$ where
\begin{equation*}
  \text{(G): } \cor{\ome}\subset\ol{\K_\delta\cap\edom}\,, \quad
  \text{(L),\;(C): } \cor{\ome}\subset\cl\K 
\end{equation*}
and where $\ome$ is even a $\sigma$-measure on $\K$ in the strongest case (C). 
This shows us the relevance of the different cases. 
Using these cases we now provide general representations of traces on $\K$ over
$\cW^{1,\infty}(\edom,\R^m)$. In Section~\ref{sec:ub} we combine these
results with Theorem~\ref{ex:dmo_woi} to derive Gauss-Green formulas
on arbitrary Borel sets $\dom\subset\edom$
for vector fields in $\cD\cM^1(\edom)$.

\begin{theorem}\label{prop:dual}
Let $\edom\subset\R^n$ be open and bounded, let $\K\subset\cl\edom$ be closed,
and let $f^*\in\cW^{1,\infty}(\edom,\R^m)^*$ be a trace on $\K$. 
Then for each $\delta>0$ there are measures
\begin{equation}\label{prop:dual-1}
  \ome \in \bawl{\edom}^m \qmq{and} \me \in \bawl{\edom}^{mn}
\end{equation}
with
\begin{equation*}\label{prop:dual-2}
  \cor{\ome},\: \cor{\me} \subset \ol{\K_\delta\cap\edom}
\end{equation*}
such that
\begin{equation}\label{prop:dual-3}
  \df{f^*}{\tf}= \I{\K_\delta\cap\edom}{\tf}{\ome} + 
                \I{\K_\delta\cap\edom}{D\tf}{\me} \,
\end{equation}
for all $\tf\in\woinf{\edom,\R^m}$ and
\begin{eqnarray}
  \|(\ome,\me)\| 
&=&
  \|\ome\| + \|\me\| \: = \: 
  |\ome|(\K_\delta\cap\edom) + |\me|(\K_\delta\cap\edom)
  \nonumber \\
&\le& \label{prop:dual-4}
  c \;
  \sup_{\substack{\tf\in\cW^{1,\infty}(\edom,\R^m)\\
                  \|\tf_{|\K_\delta\cap\edom}\|_{\cW^{1,\infty}}\le 1} }
    \df{f^*}{\port{\K}_\delta\tf}  \,
\end{eqnarray}
for a constant $c\ge 1$. 
In the particular cases we have in addition
\bgl
\item[{\rm (G):}]
equality in norm estimate \reff{prop:dual-4} with $c=1$,

\item[{\rm (L):\,}]
$\cor{\ome}\subset\K$ such that \reff{prop:dual-3} can be written as
\begin{equation*}
   \df{\de}{\tf}= \sI{\K}{\tf}{\ome} + 
                \I{\K_\delta\cap\edom}{D\tf}{\me} \,,
\end{equation*}
\item[{\rm (C):}]
$\ome$ corresponds to a Radon measure $\sme$ with $\supp{\sme}\subset\K$ 
such that   
\begin{equation*}
  \df{\de}{\tf}= \I{\K}{\tf}{\sme} + 
                \I{\K_\delta\cap\edom}{D\tf}{\me}  \,.
\end{equation*}
\el
\end{theorem}

\noindent
In cases (L) and (C) we get from the proof that
\begin{equation}\label{prop-dual-a}
   c = \|\iota^{-1}\|
\end{equation}
in \reff{prop:dual-4} for the related $\iota$ 
according to Proposition~\ref{prop:dual1}.   
Before we discuss cases (L) and (C) in some more detail we consider 
important special cases of Theorem~\ref{prop:dual}.

\newpage

In general the measures $(\ome,\me)=(\ome_\delta,\me_\delta)$ 
on $\K_\delta\cap\edom$ related to $f^*$ depend on $\delta$. 
But, if they are somehow bounded with respect to $\delta$,
they have a weak$^*$ cluster point giving a representation of $f^*$ 
independent of $\delta$.

\begin{proposition}\label{prop:special-a}
Let $\edom \subset \Rn$ be open and bounded, let $\K \subset \cl{\edom}$ be 
closed, and let $\de\in \woinf{\edom,\R^m}^*$ be a trace on $\K$.
Then we can choose $\ome$, $\me$ in \reff{prop:dual-3} independent of $\delta$
such that 
\begin{equation*}
    \df{\de}{\tf} = \sI{\K}{\tf}{\ome} + \sI{\K}{D\tf}{\me} 
  \qmq{and} \cor{\ome},\cor{\me}\subset \K
\end{equation*}
if and only if 
\begin{equation}\label{eq:monster}
  \liminf_{\delta \downarrow 0} 
  \sup_{\substack{\tf \in \woinf{\edom,\R^m}\\
        \|\tf_{|\K_\delta\cap\edom}\|_{\cW^{1,\infty}} \leq 1}} 
  \df{\de}{\port{\K}_\delta \tf} < \infty \,.
\end{equation}
If we have in addition case {\rm (C)} for some $\delta>0$, then
$\ome$ corresponds to a Radon measure supported on $\K$.
\end{proposition}

\noi
We call the trace $f^*$ {\it finite} if \reff{eq:monster} is satisfied. 
Notice that this condition gives some uniform bound for 
$(\ome,\me)$ with respect to $\delta$ according to \reff{prop:dual-4}. 
In the light of usual Gauss-Green formulas as in
\reff{tt-e0} it is also desirable to characterize the case where the measure 
$\me$ can disappear in \reff{prop:dual-3}.

\begin{proposition}\label{prop:special-b}
Let $\edom \subset \Rn$ be open and bounded, let $\K \subset \cl{\edom}$ be 
closed, let $\de\in \woinf{\edom,\R^m}^*$ be a trace on $\K$.
\bgl
\item
For $\delta>0$ we can choose $\me=0$ in \reff{prop:dual-3} such that 
\begin{equation*}
  \df{\de}{\tf} = \I{\K_\delta\cap\edom}{\tf}{\ome} \qmq{and}
  \cor{\ome}\subset\ol{\K_\delta\cap\edom}
\end{equation*}
if and only if
\begin{equation} \label{eq:special}
  \sup_{\substack{\tf\in\woinf{\edom,\R^m}\\
        \|\tf_{|\K_\delta\cap\edom}\|_{\cL^\infty}\leq 1}} 
  \df{f^*}{\tf} < \infty \,.
\end{equation}
If we have in addition case {\rm (L)} or {\rm (C)} for $\delta$, then
$\cor{\ome}\subset\K$.
\item
We can choose $\me=0$ in \reff{prop:dual-3} such that 
\begin{equation*}
  \df{\de}{\tf} = \sI{\K}{\tf}{\ome}
  \qmq{and} \cor{\ome}\subset \K
\end{equation*}
if and only if 
\begin{equation} \label{eq:special-1}
  \liminf_{\delta\downarrow 0}\sup_{\substack{\tf\in\woinf{\edom,\R^m}\\
        \|\tf_{|\K_\delta\cap\edom}\|_{\cL^\infty}\leq 1}} 
  \df{f^*}{\tf} < \infty \,.
\end{equation}
\el
If we have in addition case {\rm (C)} for some $\delta>0$, then 
measure $\ome$ corresponds to a Radon measure supported on $\K$ in both cases.
\end{proposition}

\noindent
Condition \reff{eq:special} somehow says that the trace $f^*$ can be
considered as 
a linear continuous functional on $\cL^\infty(\K_\delta\cap\edom,\R^m)$. 
But notice that the 
possible choice $\me=0$ does not exclude other representations of $f^*$ 
with $\me\ne 0$ (cf. Example~\ref{dt-ex5} below).
The proof of (2) shows that \reff{eq:special-1} implies
\reff{eq:monster}.

Let us come back to cases (L) and (C). Since they play an important role,
we will provide some conditions that help to identify them. We say that 
$\K_\delta\cap\edom$ is {\it bounded path connected with} $\K$ if
there is a maximal length $\ell>0$ such that for any $x\in\K_\delta\cap\edom$  
and any $\delta'>0$ there are a point $y\in\K_{\delta'}\cap\edom$ and a curve
connecting $x$, $y$ inside $\K_\delta\cap\edom$ with length 
less than $\ell$.

\begin{proposition}\label{prop:cases}
Let $\edom\subset\R^n$ be open, let $\K\subset\cl\edom$ be closed, and let
$\delta>0$.  
\bgl
\item
If we have case {\rm (L)} or {\rm (C)} for $\delta$ 
and if $V$ is a component of 
$\K_\delta\cap\edom$, then 
\begin{equation*}
  \K_{\delta'}\cap V\ne\emptyset \qmq{for all} \delta'>0 \,.
\end{equation*}

\item
If $\K_\delta\cap\edom$ is bounded path connected with $\K$, 
then we have case {\rm (L)} for $\delta$ with $c=1+\ell$ in \reff{prop:dual-4}.

\item
If any $\tf\in\woinf{\edom,\R^m}$ is continuously extendable up to $\K$
and if $\K_\delta\cap\edom$ is
bounded path connected with $\K$, then we have case {\rm (C)} for $\delta$
with $c=1+\ell$ in \reff{prop:dual-4}. 
\el
\end{proposition}

\noi
These general assertions imply some important special cases. 

\begin{corollary} \label{cor:cases}
Let $\edom\subset\R^n$ be open and bounded and let $\K\subset\cl\edom$ be
closed. 
\bgl
\item
If $\K=\bd\edom$, then we have case {\rm (L)} for any
$\delta>0$ and $c\le 1+\delta$ in \reff{prop:dual-4}. 
\item
If $\K=\bd\edom$ and $\edom$ has Lipschitz boundary,
then we have case {\rm (C)} for all $\delta>0$
and $c\le 1+\delta$ in \reff{prop:dual-4}.
\item 
If $\K_\delta\csubset\edom$ for $\delta>0$, 
then 
we have case {\rm (C)} for $\delta$   
and $c\le 1+\delta$ in \reff{prop:dual-4}.
\el
\end{corollary}

\noindent
Let us illuminate the cases (G), (L), (C) by
applying Proposition~\ref{prop:cases} and Corollary~\ref{cor:cases}
to some typical examples in $\R^2$.

\begin{example}\label{tt-r-ex1}   
We consider $\edom,\K\subset\R^2$ with 
\begin{equation*}
  \edom:=\big((0,1)\cup(1,2)\big)\times(0,1)\,,
  \quad \K:=\{1\}\times(0,1)\,.
\end{equation*}
Since $\K_\delta\cap\edom$ has two components, some
$\tf\in\woinf{\K_\delta\cap\edom,\R^m}$ that equals different constants on each
component cannot be extended continuously up to~$\K$. Therefore we do not 
have (C) for any $\delta>0$. But we readily verify the
assumption of Proposition \ref{prop:cases} (2) and, thus, we have 
(L) for all $\delta>0$. 
\end{example}

\begin{example}\label{tt-r-ex2}   
In $\R^2$ we take
\begin{equation*}
    \edom := \bigcup_{k=1}^\infty R_k \qmq{with}
    R_k := \big(\tfrac{1}{2k+1},\tfrac{1}{2k}\big) \times (0,1)\,, \quad
    \K:=\{0\}\times(0,1)\,.
\end{equation*}
Obviously we cannot continuously extend all 
$\tf\in\woinf{\K_\delta\cap\edom,\R^m}$
up to $\K$ and, thus, we do not have (C) for any $\delta>0$. 

For fixed $\delta>0$ we now choose some 
$R_{k'}\subset\K_\delta$. This is obviously a component of $\K_\delta\cap\edom$ 
and clearly $\K_{\delta'}\cap R_{k'}=\emptyset$ for all sufficiently small
$\delta'>0$. Hence we do not have (L) for any $\delta>0$ by 
Proposition~\ref{prop:cases}~(1).
Therefore, for the treatment of a trace on $\K$ we can 
merely use the general case (G) in Theorem~\ref{prop:dual}.

While $\edom$ has infinite perimeter, we get the same results for some 
$\edom$ with finite perimeter if we replace $R_k$ and $\K$ with
\begin{equation*}
    \tilde R_k := \big(\tfrac{1}{2^{2k+1}},\tfrac{1}{2^{2k}}\big) \times 
           \big( 0,\tfrac{1}{2^k} \big)
    \qmq{and} \tilde \K:=\{(0,0)\}\,.
\end{equation*}
\end{example}

\begin{example}\label{tt-r-ex2a}   
With $R_k$ as in the previous example we now choose in $\R^2$
\begin{equation*}
  \edom=B_2(0)\,, \quad \dom=\bigcup_{k=1}^\infty R_k\,, \quad 
  \K\subset\bd\dom \zmz{closed\,.}
\end{equation*}
In contrast to $\K\subset\bd\edom$ in Example~\ref{tt-r-ex2}, we now 
have $\K\csubset\edom$. This changes the situation essentially 
and, by Corollary~\ref{cor:cases} (3), we have case (C) for all small 
$\delta>0$.
\end{example}

\begin{example}\label{tt-r-ex3}   
In $\R^2$ we consider $\K:=\{0\}\times (0,1)$ and
\begin{equation*}
  \edom:=\big((0,1)\times(0,1)\big)\setminus
  \Big( \bigcup_{l=1}^\infty 
  \big\{\tfrac{1}{2l}\big\}\times\big(0,\tfrac{3}{4}\big] \cup
  \big\{\tfrac{1}{2l+1}\big\}\times\big[\tfrac{1}{4},1\big) \Big)
\end{equation*}
(cf. Figure~\ref{fig:u0}). 

\begin{figure}[h!]          
\centering
  \includegraphics[width=3cm]{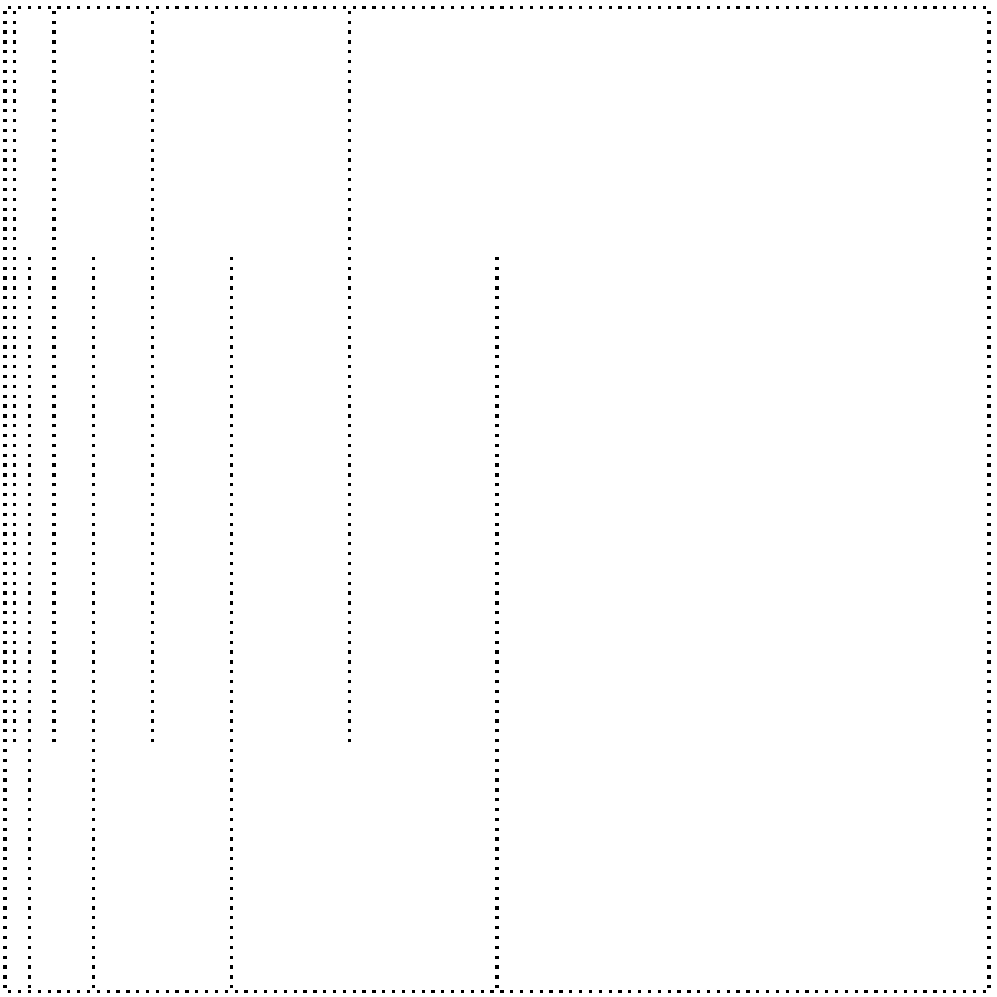} 
  \caption{The open set $U$.} 
  \label{fig:u0}
\end{figure}

\noi
Obviously we cannot extend all $\tf\in\woinf{\K_\delta\cap\edom,\R^m}$
up to $\K$, since $\tf$ can oscillate between the inner boundaries. 
Therefore we do not have (C) for any $\delta>0$. 
For treating (L) we first observe that $\K_\delta\cap\edom$ is not bounded
path connected with $\K$ for any $\delta>0$. Thus we cannot use
Proposition~\ref{prop:cases} and we have to check the assumption for $\iota$ in
Proposition~\ref{prop:dual1} directly. For that we fix $\delta>0$. Then for
each $k\in\N$ there is some $\tf_k\in\woinf{\K_\delta\cap\edom}$ with
\begin{equation*}
  \|\tf_k\|_\infty=k\,, \quad \|D\tf_k\|_\infty=1\,, \quad
  \tf_{k\wr\K}=0 \,
\end{equation*}
(roughly speaking, choose $\tf_k=k$ near $\{\delta\}\times(0,1)$, 
decrease $\tf_k$ to zero towards $\K$ by respecting $\|D\tf_k\|_\infty=1$,
and set $\tf_k=0$ in a small remaining neighborhood of~$\K$). Then
\begin{equation*}
  \|\tf_k\|_{\cW^{1,\infty}} = k+1\,,  \quad 
  \|\iota(\tf_k)\| = 
  \max \big\{ \|\tf_{k\wr\K}\|,\, \|D\tf\|_\infty \big\} = 1 \,.
\end{equation*}
But this prevents continuity of $\iota^{-1}$ and, hence, we do not have 
(L) for any $\delta>0$. 

If we take $\K':=\{1\}\times (0,1)$ instead of $\K$, then 
$\K'_\delta\cap\edom$ is bounded path connected with $\K'$ for
$\delta<1$ while it is not for $\delta\ge 1$. Thus we get (L) for $\delta<1$
from Proposition~\ref{prop:cases} while we do not have (L) for $\delta\ge 1$
by arguments as above. For $\K''=\bd\edom$ we have (L) for all $\delta>0$ by
Corollary~\ref{cor:cases}. 
\end{example}

\subsection{Proofs}

Now, the proofs of the previous results will be given. 

\medskip

\begin{proof+}{ of Lemma \ref{lem-fs}}    
Since $\|\cdot\|_\K$ is a semi norm, we readily get for 
$\tf,\psi\in\cL^\infty(\edom,\R^m)$
\begin{equation}\label{lem-fs-5}
  \big| \|\tf\|_\K - \|\psi\|_\K \big| \le \|\tf-\psi\|_\K \le 
  \|\tf-\psi\|_{\cL^\infty}\,.
\end{equation}
Hence $Z$ is a closed subspace. If $\psi\in Z$, then
\begin{equation*}
  \|\tf\|_\K \osm{\reff{lem-fs-5}}{\le} \|\tf+\psi\|_\K \le
  \|\tf\|_\K + \|\psi\|_\K = \|\tf\|_\K \,,
\end{equation*}
which implies \reff{lem-fs-2}.
For $\tf\in\cL^\infty(\edom,\R^m)$ we define 
\begin{equation*}
  \tf_\delta := \bigg\{
  \begin{array}{ll} \tf & \tx{on } \edom\setminus \K_\delta\,, \\   
                    0   & \tx{on } \edom\cap \K_\delta
  \end{array}
  \quad \zmz{for all} \delta >0\,.
\end{equation*}
Obviously $\tf_\delta\in Z$ for all $\delta$ and
\begin{equation*}
  \op{dist}_Z \tf \le \inf_{\delta>0} \|\tf-\tf_\delta\|_{\cL^\infty(\edom)} =
  \lim_{\delta\downarrow 0} \|\tf\|_{\cL^\infty(\K_\delta\cap\edom)} = 
  \|\tf\|_\K \,.
\end{equation*}
Assume that $\|\tf\|_\K > \op{dist}_Z \tf$, then there is $\psi\in Z$ with
\begin{equation*}
  \|\tf-\psi\|_{\cL^\infty} < \|\tf\|_\K \osm{\reff{lem-fs-2}}{=} 
  \|\tf-\psi\|_\K \le
  \|\tf-\psi\|_{\cL^\infty}\,, 
\end{equation*}
which is a contradiction. Hence $\|\tf\|_\K = \op{dist}_Z \tf$.
Therefore, by standard results, 
$\cL^\infty_\K(\edom,\R^m)$ is a Banach space with
$\|\tf_{\wr\K}\|=\|\tf\|_\K$ and its dual space is isometrically isomorphic to
the stated set  
(cf. \cite[p.~34, 99]{werner_00}, \cite[p.~185]{zeidler_109}).

Since $\cL^\infty(\edom,\R^m)^*$ can be identified with $\bawl{\edom}^m$, for 
$f_\K^*\in\cL^\infty_\K(\edom,\R^m)^*$ there is $\me\in \bawl{\edom}^m$ such that
\begin{equation*}
  \df{f_\K^*}{\tf_{\wr\K}} = \I{\edom}{\tf}{\me} \qmq{for all} 
  \tf_{\wr\K}\in\cL^\infty_\K(\edom) 
\end{equation*}
while
\begin{equation}\label{lem-fs-10}
  \I{\edom}{\tf}{\me}=0 \qmz{for all} \tf\in Z\,.
\end{equation}

Let now $\edom$ be bounded and assume that $x\in\cor{\me_k}\setminus\ol \K$
for some component $\me_k$ of $\me$.   
Then there is some $\delta>0$ and some open 
$V\subset\R^n\setminus \K_{2\delta}$
containing $x$, such that $|\me_k|(V\cap\edom)>0$. Hence we can  
find some $\psi\in\cL^\infty(\edom,\R^m)$ with 
$\psi_{|\K_\delta}=0$, with $\psi_j\equiv 0$ for $j\ne k$, and 
with $\I{\edom}{\psi}{\me}>0$.
But this contradicts \reff{lem-fs-10}, since $\psi\in Z$. Therefore
$\cor{\me}\subset\ol \K$. 
\end{proof+}

\begin{proof+}{ of Lemma \ref{lem-ps2}}  
Obviously $X$ is a Banach space. Moreover $\bawl{\edom}^{mn}$ 
with $\|\me\|=|\me|(\edom)$ is the dual of
$\cL^\infty(\edom,\R^{mn})$ 
by Proposition~\ref{pm-s5}.  
Then $X^*$ is the dual of $X$ with \reff{lem-ps2-1}
by standard arguments. For the norm in $X^*$ we fix $\eps>0$. Then
there is $(\tf^\eps,\tF^\eps)\in X$ with
\begin{equation*}
  \|(\tf^\eps,\tF^\eps)\|\le 1\,, \quad
  \|f^*\|\le \df{f^*}{\tf^\eps}+\eps\,, \quad
  \|\me\|\le \df{\me}{\tF^\eps}+\eps \,.
\end{equation*}
Hence, by \reff{lem-ps2-1},
\begin{equation*}
  \|(f^*,\me)\| \le \|f^*\| + \|\me\| \le 
  \df{(f^*,\me)}{(\tf^\eps,\tF^\eps)} +\eps 
  \le \|(f^*,\me)\| + \eps\,.
\end{equation*}
The arbitrariness of $\eps>0$ implies equality and completes the proof.
\end{proof+}

\begin{proof+}{ of Proposition \ref{prop:dual1}}   
We fix $f^*\in\cW^{1,\infty}(\edom,\R^m)^*$. Since    
$\tilde X:=\iota\big(\cW^{1,\infty}(\edom,\R^m)\big)$
is a linear subspace of $X:=X_0\times\cL^\infty(\edom,\R^{mn})$, we can use  
the Hahn-Banach theorem to extend $f^*\circ\iota^{-1}\in\tilde X^*$ 
to some $g^*\in X^*$ under preservation of norm. 
Then, by Lemma~\ref{lem-ps2}, there are $f_0^*\in X_0^*$ 
and $\me\in\bawl{\edom}^{mn}$ such that 
\begin{equation*}
  \df{f^*}{\tf} = \df{f^*\circ\iota^{-1}}{\iota(\tf)} =
  \df{g^*}{\iota(\tf)} =
  \df{f_0^*}{\iota_0(\tf)} + \I{\edom}{D\tf}{\me} 
\end{equation*} 
for all $\tf\in\cW^{1,\infty}(\edom,\R^m)$. Moreover 
\begin{eqnarray*}
\|f_0^*\| + \|\me\|  
&=&
\|(f_0^*,\me)\| = \|g^*\| = \|f^*\circ\iota^{-1}\|   \\
&\le&
  \|f^*\|\,\|\iota^{-1}\| = \|\iota^{-1}\|
  \sup_{\substack{\tf\in\cW^{1,\infty}(\edom,\R^m)\\
        \|\tf\|_{\cW^{1,\infty}}\le 1}} \df{f^*}{\tf}    
\end{eqnarray*}
which verifies the final estimate. 
\end{proof+}

\medskip

\begin{proof+}{ of Proposition \ref{prop:dual0}}   
Let $\delta>0$ be fixed. For $\tf\in\woi{m}$ we have 
\begin{equation*}
    \port{\K}_\delta \tf \in \woi{m} \qmq{and}
    \refun{(1-\port{\K}_\delta)\tf}{\dnhd{\K}{\frac{\delta}{2}}} = 0 \,.
\end{equation*}
Since $f^*\in\cW^{1,\infty}(\edom,\R^m)^*$ is a trace on $\K$,
\begin{equation}\label{prop:dual0-5}
    \df{\de}{\tf} = \df{\de}{\port{\K}_\delta \tf} + 
                    \df{\de}{(1-\port{\K}_\delta)\tf}
    = \df{\de}{\port{\K}_\delta \tf} \,. 
\end{equation}
With 
\begin{equation*}
  \const_\delta := \|\port{\K}_\delta\|_{\cW^{1,\infty}(\edom,\R^m)} =
  \|\port{\K}_\delta\|_{\cW^{1,\infty}(\K_\delta\cap\edom,\R^m)} 
  \ge 1
\end{equation*}
and with the product rule for $D(\port{\K}_\delta \tf)$ we get
\begin{eqnarray*}
  \norm{\woi{m}}{\port{\K}_\delta \tf} 
&=&
  \|\port{\K}_\delta \tf\|_{\cW^{1,\infty}(\K_\delta\cap\edom,\R^m)} \\
&\le&
  \max \big\{ \|\port{\K}_\delta
      \tf\|_{\cL^\infty(\K_\delta\cap\edom,\R^m)},\, \\
&& \hspace{11mm}
  \|\port{\K}_\delta D\tf\|_{\cL^\infty(\K_\delta\cap\edom,\R^m)} +
  \|\tf D\port{\K}_\delta\|_{\cL\infty(\K_\delta\cap\edom,\R^m)}
   \big\}   \\
&\leq&
  \max\big\{\|\tf\|_{\cL^\infty(\K_\delta\cap\edom,\R^m)},  \\
&& \hspace{11mm}
  \|D\tf\|_{\cL^\infty(\K_\delta\cap\edom,\R^m)} +
  c_\delta\|\tf\|_{\cL\infty(\K_\delta\cap\edom,\R^m)}\big\}   \\
&\le&
  2c_\delta\max \big\{ \|\tf\|_{\cL^\infty(\K_\delta\cap\edom,\R^m)},\,
  \|D\tf\|_{\cL^\infty(\K_\delta\cap\edom,\R^m)} \big\}    \\
&=&
  2\const_\delta \|\tf\|_{\cW^{1,\infty}(\K_\delta\cap\edom,\R^m)}
\le 2\const_\delta \norm{\woi{m}}{\tf} 
\end{eqnarray*}
for all $\tf\in\woi{m}$.

We now consider the subspace (that might be strict)
\begin{equation*}
  X_\delta := \big\{ \psi\in\cW^{1,\infty}(\K_\delta\cap\edom,\R^m) 
  \;\big|\; \psi = \tf_{|\K_\delta\cap\edom} \zmz{for some} 
  \tf\in\cW^{1,\infty}(\edom,\R^m) \big\}
\end{equation*}
and define a linear functional $f^*_\delta$ on $X_\delta$ by
\begin{equation*}
  \df{f_\delta^*}{\tf_{|\K_\delta\cap\edom}} =
  \df{f^*}{\port{\K}_\delta \tf} \qmq{for} \tf\in\cW^{1,\infty}(\edom,\R^m) \,
\end{equation*}
(notice that $f_\delta^*$ is well-defined this way). Since
\begin{equation*}
  \big|\df{f_\delta^*}{\tf_{|\K_\delta\cap\edom}}\big| = 
  \big|\df{f^*}{\port{\K}_\delta \tf}\big|
  \le 2 c_\delta \|f^*\| \|\tf\|_{\cW^{1,\infty}(\K_\delta\cap\edom,\R^m)} \,,
\end{equation*}
we have that $f^*_\delta\in X_\delta^*$ and 
$\|f^*_\delta\|\le 2c_\delta\|f^*\|$.
By a norm preserving extension with the Hahn-Banach theorem,
we can identify $f_\delta^*$ with some 
$f_\delta^*\in\cW^{1,\infty}(\K_\delta\cap\edom,\R^m)^*$.
Using \reff{prop:dual0-5} we obtain
\begin{equation*}
  \df{\de}{\tf} = \df{f_\delta^*}{\tf_{|\K_\delta\cap\edom}}
  \qmq{for} \tf\in\cW^{1,\infty}(\edom,\R^m) \,
\end{equation*}
which verifies the assertion.
\end{proof+}

\medskip

\begin{proof+}{ of Theorem \ref{prop:dual}}   
For $f^*\in \cW^{1,\infty}(\edom,\R^m)^*$ and $\delta>0$ we fix  
\begin{equation*}
  f_\delta^*\in\cW^{1,\infty}(\K_\delta\cap\edom,\R^m)^*
\end{equation*}
according to Proposition \ref{prop:dual0}. Then we apply 
Proposition~\ref{prop:dual1} with $\K_\delta\cap\edom$
instead of $\edom$, with a suitable choice of $X_0$, $\iota_0$, and with
$\|(\tf,\tF)\|_{X_0\times\cL^\infty} = \max\{\|\tf\|_{X_0},\|\tF\|_\infty\}$. 

Let us first consider the general case (G) with 
\begin{equation*}
  X_0=\cL^\infty(\K_\delta\cap\edom,\R^m) \qmq{and} \iota_0(\tf)=\tf.
\end{equation*}
Then  
\begin{equation*}
  \iota:\cW^{1,\infty}(\K_\delta\cap\edom,\R^m) \to
  \cL^\infty(\K_\delta\cap\edom,\R^m) \times
  \cL^\infty(\K_\delta\cap\edom,\R^{mn})
\end{equation*}
with $\iota(\tf)=(\tf,D\tf)$ is a linear and isometric mapping. Moreover, it is
bijective onto $Y:=\iota\big(\cW^{1,\infty}(\K_\delta\cap\edom,\R^m)\big)$.
Hence there is a continuous inverse $\iota^{-1}$ on $Y$. With 
\begin{equation*}
  X_0^*=\bawl{\K_\delta\cap\edom}^m
\end{equation*}
we obtain the existence of
\begin{equation}\label{prop:dual-5}
  \ome \in \bawl{\K_\delta\cap\edom}^m\,, \quad 
  \me \in \bawl{\K_\delta\cap\edom}^{mn}
\end{equation}
such that the representation of $\df{f^*}{\tf}$ as in \reff{prop:dual-3} is
satisfied. The measures $\ome$, $\me$ can be extended on $\edom$ by zero to
get \reff{prop:dual-1} and, clearly, 
\begin{equation*}
  \cor{\ome},\, \cor{\me} \subset \ol{\K_\delta\cap\edom}\,.
\end{equation*}
Using \reff{prop:dual1-3}, \reff{prop:dual0-1},
the isometry of $\iota$, and that
$f_\delta^*$ is a norm preserving extension from the subspace $X_\delta$ to 
$X_0$ (cf. the proof of Proposition~\ref{prop:dual0}), 
we finally have
\begin{eqnarray*}
  \|\ome\| + \|\me\| 
&=&
  |\ome|(\edom) + |\me|(\edom) 
\: = \:
  |\ome|(\K_\delta\cap\edom) + |\me|(\K_\delta\cap\edom)
  \nonumber \\
&=&
  \|(\ome,\me)\| = \|f^*_\delta\circ\iota^{-1}\|  =
  \sup_{\substack{\psi\in\cW^{1,\infty}(\K_\delta\cap\edom,\R^m)
                  \\\|\iota(\psi)\|\le 1} }
  \df{f_\delta^*\circ\iota^{-1}}{\iota(\psi)}   \\
&=& 
  \sup_{\substack{\psi\in\cW^{1,\infty}(\K_\delta\cap\edom,\R^m)
                  \\\|\psi\|\le 1} }
  \df{f_\delta^*}{\psi}  
= \sup_{\substack{\psi\in X_\delta \\\|\psi\|\le 1} }
  \df{f_\delta^*}{\psi}   \\
&=&
   \sup_{\substack{\tf\in\cW^{1,\infty}(\edom,\R^m)\\\|\tf_{|\K_\delta\cap\edom}\|\le 1} }
    \df{f_\delta^*}{\tf_{|\K_\delta\cap\edom}} =  
  \sup_{\substack{\tf\in\cW^{1,\infty}(\edom,\R^m)\\\|\tf_{|\K_\delta\cap\edom}\|\le 1} }
    \df{f^*}{\port{\K}_\delta\tf}  \,.
\end{eqnarray*}
This verifies the first assertion for case (G)
(cf. also Adams \cite[Theorem 3.8]{adams:78} for the duals of Sobolev spaces).

For case (L) we choose $X_0=\cL^\infty_\K(\K_\delta\cap\edom,\R^m)$ and
$\iota_0(\tf)=\tf_{\wr\K}$. Then we combine Proposition~\ref{prop:dual1} with
Lemma~\ref{lem-fs} to get the existence 
of measures $\ome$, $\me$ as in \reff{prop:dual-5} with
$\cor{\ome}\subset \K$, $\cor{\me}\subset\ol{\K_\delta\cap\edom}$,
and such that the representation of $\df{f^*}{\tf}$ as in the assertion is
true. For \reff{prop:dual-1} we extend the measures by zero. 
With $c=\|\iota^{-1}\|\ge 1$ we get similar to the general case that
\begin{eqnarray*}
  \|\ome\| + \|\me\| 
&=&
  \|(\ome,\me)\| = \|f^*_\delta\circ\iota^{-1}\|  \\
&\le& 
  \|\iota^{-1}\|\, \|f^*_\delta\| =
  c  \sup_{\substack{\psi\in\cW^{1,\infty}(\K_\delta\cap\edom,\R^m)
                  \\\|\psi\|\le 1} }
  \df{f_\delta^*}{\psi}  \\
&=&
  c
  \sup_{\substack{\tf\in\cW^{1,\infty}(\edom,\R^m)\\\|\tf_{|\K_\delta\cap\edom}\|\le 1} }
    \df{f^*}{\port{\K}_\delta\tf}  \,.
\end{eqnarray*}
For case (C) we use $X_0=C(\K,\R^m)$ and $\iota_0(\tf)=\tf_{|\K}$ and 
argue as in case~(L).
\end{proof+}

\medskip

\begin{proof+}{ of Proposition \ref{prop:special-a}}   
We first assume that $f^*$ is finite, i.e. \reff{eq:monster} is
satisfied. Then there are $\delta_k>0$ with $\delta_k\to 0$ and
\begin{equation}\label{eq:mon_aux}
  \sup \limits_{k \in \N} \; 
  \sup_{\substack{\tf \in \woinf{\edom,\R^m}\\\|\tf_{|\K_{\delta_k}\cap\edom}\| \leq 1}} 
  \df{\de}{\port{\K}_{\delta_k} \tf} < \infty \,. 
\end{equation}
Theorem \ref{prop:dual} provides measures
\begin{equation*}
  \ome_k\in\bawl{\edom}^m\,, \quad \me_k\in\bawl{\edom}^{mn}
\end{equation*}
related to $\delta_k$
such that $\cor{\ome_k}, \cor{\me_k}\subset\ol{\K_{\delta_k}\cap\edom}$ and,
for all $\tf\in\woi{m}$,
\begin{equation*}
  \df{f^*}{\tf} = \I{\K_{\delta_k}\cap\edom}{\tf}{\ome_k} + 
                \I{\K_{\delta_k}\cap\edom}{D\tf}{\me_k} \,.
\end{equation*}
By \reff{prop:dual-4} for case (G), where $c$ is independent of $\delta$,
and by \reff{eq:mon_aux} there is some $\tilde c>0$ with
\begin{equation*}
  \|(\ome_k,\me_k)\| = \|\ome_k\| + \|\me_k\| 
  \le \tilde c 
  \qmz{for all} k\,.
\end{equation*}
Therefore $\{(\ome_k,\me_k)\}$ is a bounded sequence in 
$\big(\cL^\infty(\edom,\R^m) \times \cL^\infty(\edom,\R^{mn})\big)^*$ and, by
the Banach-Alaoglu theorem, there is a 
weak* cluster point $(\ome,\me)$ with
\begin{equation*}
  \df{(\ome,\me)}{(\tf,\tF)} = \I{\edom}{\tf}{\ome} + \I{\edom}{\tF}{\me} 
\end{equation*}
for all $(\tf,\tF)\in \cL^\infty(\edom,\R^m) \times \cL^\infty(\edom,\R^{mn})$.
Hence there is a subnet of 
$\{(\ome_k,\me_k)\}$ converging to $(\ome,\me)$. Thus, 
for any $(\tf,\tF)\in \cL^\infty(\edom,\R^m) \times \cL^\infty(\edom,\R^{mn})$
with 
\begin{equation*}
  \tf_{|\K_\delta}=0\,, \z \tF_{|\K_\delta}=0  \qmq{for some} \delta>0\,
\end{equation*}
there is a subsequence $\big\{(\ome_{k'},\me_{k'})\big\}$ such that 
\begin{equation*}
  \big\langle (\ome,\me),(\tf,\tF) \big\rangle = 
  \lim_{k'\to\infty} \big\langle (\ome_{k'},\me_{k'}),(\tf,\tF)\big\rangle =
  \lim_{k'\to\infty}\I{\edom}{\tf}{\ome_{k'}} + \I{\edom}{\tF}{\me_{k'}} = 0 \,.
\end{equation*}
(recall that $\cor{\ome_k},\cor{\me_k}\subset\K_{\frac{\delta}{2}}$ for $k$
large). Consequently $\cor{\ome}, \cor{\me} \subset \K$ and, clearly,
$\df{\de}{\tf} = \sI{\K}{\tf}{\ome} + \sI{\K}{D\tf}{\me}$.

For the reverse statement we consider $\ome$, $\me$ as in \reff{prop:dual-1}
with $\cor{\ome}, \cor{\me} \subset \K$ such that for all
$\tf\in\cW^{1,\infty}(\edom,\R^m)$
\begin{equation*}
    \df{\de}{\tf} = \sI{\K}{\tf}{\ome} + \sI{\K}{D\tf}{\me}   \,. 
\end{equation*}
Obviously
\begin{equation*}
  \port{\K}_\delta \tf = \tf\,, \z D(\port{\K}_\delta \tf) = D\tf
  \qmz{on} \K_{\frac{\delta}{2}}\cap\edom  \,.
\end{equation*}
Thus, for any $\delta>0$ 
and any $\tf\in\cW^{1,\infty}(\edom,\R^m)$ with 
$\|\tf_{|\K_\delta\cap\edom}\|_{\cW^{1,\infty}}\le 1$ we use that 
$\ome$, $\me$ have core in $\K$ to get  
\begin{eqnarray*}
  \df{\de}{\port{\K}_\delta \tf}
&=& 
  \sI{\K}{\port{\K}_\delta \tf}{\ome} + \sI{\K}{D(\port{\K}_\delta \tf)}{\me} \\
&=&
  \I{\K_{\delta/2} \cap\edom}{\tf}{\ome} + 
  \I{\K_{\delta/2} \cap\edom}{D\tf}{\me}  \\
&\le&
  \|\tf_{|\K_\delta\cap\edom}\|_\infty \|\ome\| + 
  \|(D\tf)_{|\K_\delta\cap\edom}\|_\infty \|\me\| \\
&\le& 
   \|\ome\| + \|\me\| \,.
\end{eqnarray*}
Since the right hand side does not depend on $\delta$, we obtain
\reff{eq:monster}.

If we have case (C) for $\delta>0$, then all 
$\tf\in\cW^{1,\infty}(\edom,\R^m)$ can be considered as continuous up to $\K$ 
and $\ome$ can be replaced by a Radon measure $\sme$ with the stated properties
according to Proposition~\ref{acm-ram}.
\end{proof+}

\medskip

\begin{proof+}{ of Proposition \ref{prop:special-b}}   
For (1) we fix $\delta>0$. First we assume that $\me=0$ in
\reff{prop:dual-3}. Then  
\begin{equation*}
  |\df{f^*}{\tf}| = \Big|\I{\K_\delta\cap\edom}{\tf}{\ome} \,\Big|
  \le \|\tf_{|\K_\delta\cap\edom}\|_\infty \|\ome\|
\end{equation*}
for all $\tf\in\woinf{\edom,\R^m}$, which verifies the statement. 
For the other direction we assume \reff{eq:special}. 
With $f_\delta^*$ from Proposition \ref{prop:dual0} we have  
\begin{equation*}
  |\df{f^*}{\tf}| = \big|\df{f^*_\delta}{\tf_{|\K_\delta\cap\edom}}\big| \le
  \tilde c \, \|\tf_{|\K_\delta\cap\edom}\|_\infty\,.
\end{equation*}
for all $\tf\in\woinf{\edom,\R^m}$ and some constant $\tilde c>0$. 
Hence $f^*_\delta$ can be extended to some 
$g^*_\delta\in\cL^\infty(\K_\delta\cap\edom)^*$ by the Hahn-Banach theorem. 
Consequently there is some measure
$\ome\in \bawl{\edom}^m$ with $\cor{\ome}\in\cl{\K_\delta\cap\edom}$ such that
\begin{equation*}
  \df{f^*}{\tf} = \I{\K_\delta\cap\edom}{\tf}{\ome}
  \qmq{for all}  \tf\in\woinf{\edom,\R^m}\,,
\end{equation*}
which gives the opposite statement.
The remaining assertion follows directly from Theorem~\ref{prop:dual}.

For (2) we first assume that there is a measure $\ome$ with
$\cor{\ome}\subset\K$ and
\begin{equation*}
  \df{\de}{\tf} = \sI{\K}{\tf}{\ome}
  \qmq{for all} \tf \in \woinf{\edom,\R^m} \,.
\end{equation*}
Then, using \reff{prop:dual0-1},  
we have for any $\delta>0$ and all $\tf$ that
\begin{equation*}
  \df{\de}{\tf} = \df{\de}{\port{\K}_{\delta} \tf} 
  \le \|\tf_{|\K_\delta\cap\edom}\|_\infty \|\ome\|
  \le \|\tf_{|\K_\delta\cap\edom}\|_{\cW^{1,\infty}} \|\ome\| \,.
\end{equation*}
This readily implies \reff{eq:special-1}.
For the reverse statement we choose $\delta_k\downarrow 0$ 
such that the liminf in \reff{eq:special-1} is realized. 
By the first assertion there are $\ome_k$ with 
$\cor{\ome_k}\subset\ol{\K_{\delta_k}\cap\edom}$ and
\begin{equation*}
  \df{\de}{\tf} = \I{\K_{\delta_k}\cap\edom}{\tf}{\ome_k} 
  \qmq{for all} \tf \in \woinf{\edom,\R^m} \,.
\end{equation*}
Using \reff{prop:dual0-1} and the assumption we get for some $\tilde c>0$
\begin{equation*}
  \sup_{\substack{\tf \in \woinf{\edom,\R^m}\\
        \|\tf_{|\K_{\delta_k}\cap\edom}\|_{\cW^{1,\infty}} \leq 1}} 
  \df{\de}{\port{\K}_{\delta_k} \tf} \le
  \sup_{\substack{\tf\in\woinf{\edom,\R^m}\\
        \|\tf_{|\K_{\delta_k}\cap\edom}\|_{\cL^\infty}\leq 1}} 
  \df{f^*}{\tf} < \tilde c
\end{equation*}
for all $k$. Hence, by \reff{prop:dual-4} with $c=1$ 
for case (G), we get $\|\ome_k\|\le\tilde c$. 
Now we can argue as in the proof of Proposition~\ref{prop:special-a}
to get a weak$^*$ cluster point $\ome$ of 
$\{\ome_k\}$ with $\cor{\ome}\subset\K$ and
\begin{equation*}
  \df{\de}{\tf} = \sI{\K}{\tf}{\ome}
  \qmq{for all} \tf \in \woinf{\edom,\R^m} \,.
\end{equation*}
Notice that we cannot just apply Propositions~\ref{prop:special-a} and
assertion (1) simultaneously, since $\ome$ in (1)
might differ from that in the previous proposition 
due to non-uniqueness.

For case (C) we argue as in the proof of Proposition~\ref{prop:special-a}.
\end{proof+}

\medskip

\begin{proof+}{ of Proposition \ref{prop:cases}}   
For (1) we assume that there is a component $V$ of 
$\K_\delta\cap\edom$ and some
$\delta'>0$ such that $\K_{\delta'}\cap V=\emptyset$. We consider 
$\tf_0,\tf_1\in \cW^{1,\infty}(\K_\delta\cap\edom)$ with
\begin{equation*}
 \tf_0=0 \zmz{on} \K_\delta\cap\edom\,, \quad
 \tf_1= \Big\{ 
  \mbox{\small $ 
  \begin{array}{ll} 1 & \text{on $V$}\,, \\
                    0 & \text{otherwise} \,. 
  \end{array}$ } 
\end{equation*}
Obviously $\tf_0\ne \tf_1$. But, for cases (L) and (C) with $\delta$, we have
$\iota_0(\tf_0)=\iota_0(\tf_1)=0$ and $D\tf_0=D\tf_1=0$. 
Hence $\iota(\tf_0)=\iota(\tf_1)$ . 
Therefore $\iota$ is not injective. Thus both (L) and (C) are not
met, which verifies the assertion. 

For (2) we notice that $\tf\in\cW^{1,\infty}(\K_\delta\cap\edom,\R^m)$
is locally Lipschitz continuous with Lipschitz constant $\|D\tf\|_\infty$.
For $\eps>0$ there are $x\in\K_\delta\cap\edom$ and $\delta'\in(0,\delta)$ 
such that
\begin{equation*}
  \|\tf\|_{\cL^\infty(\K_\delta\cap\edom)}-\eps\le |\tf(x)|\,, \quad 
  \|\tf\|_{\cL^\infty(\K_{\delta'}\cap\edom)} \le \|\tf\|_\K + \eps \,.
\end{equation*}
Now we find $y\in\K_{\delta'}\cap\edom$ that can be connected with $x$ within
$\K_\delta\cap\edom$ by a curve of length less than $\ell$. Hence, by local
Lipschitz continuity,  
\begin{equation*}
  |\tf(x)| \le |\tf(y)| + \ell\|D\tf\|_\infty \,.
\end{equation*}
Consequently,
\begin{equation*}
  \|\tf\|_{\cL^\infty(\K_\delta\cap\edom)} \le 
  \|\tf\|_\K + \ell\|D\tf\|_\infty + 2\eps \,.
\end{equation*}
Since $\eps>0$ is arbitrary and 
$a\cdot b\le|a|_1|b|_\infty$ for $a,b\in\R^2$, we get
\begin{equation*}
  \|\tf\|_\infty \le \|\tf\|_\K + \ell\|D\tf\|_\infty \le
  (1+\ell) \max \big\{ \|\tf\|_\K,\|D\tf\|_\infty\big\} \,.
\end{equation*}
Using that $\iota(\tf)=(\tf_{\wr\K},D\tf)$ for case (L) and that 
the right hand side is larger than $\|D\tf\|_\infty$,
we obtain
\begin{equation*}
  \|\tf\|_{\cW^{1,\infty}} = \max\big\{\|\tf\|_\infty,\, \|D\tf\|_\infty \big\}
  \le (1+\ell)\,\|\iota(\tf)\| \,.  
\end{equation*}
Therefore $\iota$ is injective and $\iota^{-1}$ is continuous on its image
with $\|\iota^{-1}\|\le 1+\ell$. Observing \reff{prop-dual-a} 
we get the assertion (cf. also Remark~\ref{rem:dual1}).

For (3) we argue basically as for (2). However we use $\|\tf\|_{C(\K)}$ 
instead of $\|\tf\|_\K$, $\iota(\tf)=(\tf_{|\K},D\tf)$, and  
we choose $\delta'>0$ by continuity of $\tf$ such that 
\begin{equation*}
  \|\tf\|_{C( \K \,\cup\,(\K_{\delta'}\cap\edom) )} 
  \le \|\tf\|_{C(\K)}+\eps\,.
\end{equation*}
Then we can proceed as above.
\end{proof+}

\medskip

\begin{proof+}{ of Corollary \ref{cor:cases}}  
For (1) we fix $\delta>0$, $x\in(\bd\edom)_\delta\cap\edom$, and $\delta'>0$. 
Then there is $x'\in\bd\edom$ with  
\begin{equation*}
  |x-x'|=\op{dist}_{\bd\edom}x < \delta \,.
\end{equation*}
Clearly, the open segment $(x',x)$ belongs to $(\bd\edom)_\delta\cap\edom$
and there is $y\in(x',x)$ with $|y-x'|<\min\{\delta,\delta'\}$. 
Then the closed segment $[y,x]$ connects $x$, $y$ inside
$(\bd\edom)_\delta\cap\edom$ and has length less than $\delta$. 
Hence $(\bd\edom)_\delta\cap\edom$ is bounded path connected with $\bd\edom$
and maximal length $\ell=\delta$. Thus we have case (L) for $\delta$ by
Proposition~\ref{prop:cases} (2) and $c\le 1+\delta$ in \reff{prop:dual-4}.

For (2) we argue as in (1) and use that any
$\tf\in\cW^{1,\infty}\big((\bd\edom)_\delta\cap\edom,\R^m\big)$
can be extended continuously up to $\bd\edom$. 

For (3) we fix $\delta>0$ and observe that all
$\tf\in\cW^{1,\infty}(\K_\delta\cap\edom,\R^m)$ 
are continuous on $\K_\delta\subset\edom$. 
For any $x\in\K_\delta$ and any $\delta'>0$ we clearly find $y\in\K_{\delta'}$
such that the line segment $[y,x]$ has length less than $\delta$ 
and belongs to $\K_\delta$. Hence $\K_\delta\cap\edom$ is bounded path
connected with $\K$ and we have case (C) for $\delta$. For the estimate 
$c\le 1+\delta$ we can argue as in the proof of assertion (1). 
\end{proof+}

\section{ Divergence theorems}
\label{dt}

We derive general divergence theorems for vector fields in $\cD\cM^1(\edom)$
by representing corresponding traces with the results of
the previous section. As long as nothing else is mentioned the cases (L) and
(C) are taken for $\K=\bd\dom$.

\subsection{Divergence measure fields}\label{sec:ub}

For $F\in\cD\cM^1(\edom)$ and $\dom \in \bor{\edom}$ we have that
\begin{equation*}
  \df{TF}{\tf} = \divv(\tf\F)(\dom) = 
  \I{\dom}{\tf}{\divv\F} + \I{\dom}{\F D\tf}{\lem}\,
\end{equation*}
is a trace on $\bd\dom$ over $\woinf{\edom}$ according to
Theorem~\ref{ex:dmo_woi}. Then, with Theorem~\ref{prop:dual},
we obtain a general Gauss-Green formula for $\cD\cM^1$-vector fields. 
Notice that we have to choose $m=1$ for the particular cases (L) and (C) 
in this section. 

\begin{theorem} \label{dt-s1}
Let $\edom \subset \Rn$ be open and bounded, let $\dom \in \bor{\edom}$, let
$\delta>0$, and assume that $\F\in \dmo{\edom}$. 
Then there exist measures 
\begin{equation*}
  \lF \in \bawl{\edom} \qmq{and} \mF \in \bawl{\edom}^n
\end{equation*}
with $\cor{\lF},\: \cor{\mF} \subset \ol{(\bd\dom)_\delta\cap\edom}$
such that
\begin{equation}\label{dt-s1-1}
  \df{TF}{\tf} =
  \divv(\tf\F)(\dom) = \I{(\bd\dom)_\delta\cap\edom}{\tf}{\lF} + 
            \I{(\bd\dom)_\delta\cap\edom}{D\tf}{\mF}   
\end{equation}
for all $\tf\in\woinf{\edom}$ with $T:\dmo{\edom}\to\woinf{\edom}^*$ from
Theorem~\ref{ex:dmo_woi}.
In the particular cases with $\K=\bd\dom$ we have in addition

\medskip

{\rm (L):} $\cor{\lF}\subset\bd\dom$ and \reff{dt-s1-1} becomes
\begin{equation*}
   \divv(\tf\F)(\dom) = \sI{\bd\dom}{\tf}{\lF} + 
            \I{(\bd\dom)_\delta\cap\edom}{D\tf}{\mF} \,.   
\end{equation*}

{\rm (C):} $\lF$ corresponds to a Radon measure $\sigma_F$ with 
$\supp{\sigma_F}\subset\bd\dom$ such that 
\begin{equation*}
   \divv(\tf\F)(\dom) = \I{\bd\dom}{\tf}{\siF} + 
            \I{(\bd\dom)_\delta\cap\edom}{D\tf}{\mF} \,.   
\end{equation*}
\end{theorem}
\noindent
We call $(\lF,\mF)$, that represents an element of $\woinf{\edom}^*$,
{\it normal trace} of $F$ on $\bd\dom$. 
In contrast to usual 
Gauss-Green formulas, \reff{dt-s1-1} contains a second boundary term depending 
on $D\tf$ and both boundary terms  depend on a whole neighborhood
of the boundary. It turns out that both extensions cannot be omitted in
general. Example~\ref{dt-ex1} shows the necessity of the additional boundary
term and from Example~\ref{dt-ex2} we see that the dependence on $\delta$ is
needed. 

\begin{proof}       
According to Theorem~\ref{ex:dmo_woi} we have that $TF$ with
\begin{equation*}
  \df{TF}{\tf} = \divv(\tf F)(\dom)
\end{equation*}
is a trace on $\bd{\dom}$ over $\woinf{\edom}$. Then, for each
$\delta>0$, there are measures $\lF$ and $\mF$ as in Theorem~\ref{prop:dual}
with $\K=\bd\dom$. 
\end{proof}

For any $\dom\subset\R^n$ 
we define the (outward) {\it unit normal field} $\nu^\dom$ of $\dom$ to be
the gradient of the signed distance function
\begin{equation} \label{nm-normal}
  \nu^\dom := D \big(\distf{\dom}-\distf{\dom^c}\big) 
  \qmz{$\lem$-a.e. on} \R^n
\end{equation}
(we avoid confusion with our previous notation $\nu^\dom$ by saying that
$\nu^\dom$ is the usual measure theoretic unit normal if we take it on some 
$\rbd\dom$ and otherwise it is the field given above). 
Notice that 
\begin{equation} \label{nm-normal1}
  |\nu^\dom|=1 \qmz{$\lem$-a.e. on} \Int\dom\cup\Ext\dom \qmq{and} 
\end{equation}
\begin{equation*}
  \nu^\dom=0 \qmz{$\lem$-a.e. on} \bd\dom\,
\end{equation*}
(use that $\distf{\bd\dom}$ is Lipschitz continuous with Lipschitz 
constant $1$, that it is differentiable $\lem$-a.e. on $\R^n$ with 
$|D\dist{\bd\dom}{x}|=1$ at points of differentiability outside
$\bd\dom$, since obviously the directional derivative 
$D\op{dist}_{\bd\dom}\big(x;\frac{y-x}{|y-x|}\big)=-1$ if $y$ is a projection
of $x$ onto $\bd\dom$, and recall \cite[p.~235, 130]{evans};
cf. also \cite[p.~114]{chen-comi}).
The coarea formula implies that
$\hm(\bd\dom_\delta)<\infty$ for $\cL^1$-a.e. $\dom_\delta$. Thus, by 
$\mbd\dom_\delta\subset\bd\dom_\delta$ (cf. \cite[p.~50]{pfeffer}),
such sets have finite perimeter and $\nu^\dom$ from \reff{nm-normal}
agrees with the measure theoretic outward unit normal $\hm$-a.e. on 
$\bd\dom_\delta$ 
(cf. \cite[p.~115]{chen-comi} for the last statement).
These properties certainly justify to speak about a normal field of $\dom$.

For $\dom\subset\edom$ open or closed we easily get some characterization 
of the trace $TF$ by several limits using the normal field $\nu^\dom$.

\begin{proposition} \label{dt-s1a}
Let $\edom \subset \Rn$ be open and bounded and 
assume that $\F\in \dmo{\edom}$. If $\dom\subset\edom$ is open, then
\begin{eqnarray} \label{dt-s1a-1a}
  \divv(\tf F)(\dom) 
&=&
  \lim_{\delta\downarrow 0}
  \frac{1}{\delta}\I{(\bd\dom)_\delta\cap\dom}{\tf F\cdot\nu^\dom}{\lem} 
  \nonumber\\
&=& \label{dt-s1a-1b}
  \lim_{\delta\downarrow 0} \frac{1}{\delta} \int_0^\delta
  \I{\bd\dom_{-\tau}}{\tf F\cdot\nu^\dom}{\lem}\, d\tau  \nonumber\\
&=& \label{dt-s1a-1c}
  \esslim_{\delta\downarrow 0} 
  \I{\rbd\dom_{-\delta}}{\tf F\cdot\nu^\dom}{\lem} \nonumber
\end{eqnarray}
and if $\dom\subset\edom$ is closed, then
\begin{eqnarray} \label{dt-s1a-2a}
  \divv(\tf F)(\dom) 
&=& \lim_{\delta\downarrow 0}
  \frac{1}{\delta}\I{(\bd\dom)_\delta\cap\dom^c}{\tf F\cdot\nu^\dom}{\lem}
  \nonumber\\
&=&  \label{dt-s1a-2b}
  \lim_{\delta\downarrow 0} \frac{1}{\delta}
  \int_0^\delta \I{\bd\dom_\tau}{\tf F\cdot\nu^\dom}{\lem}\, d\tau \nonumber\\
&=&  \label{dt-s1a-2c}
  \esslim_{\delta\downarrow 0} 
  \I{\rbd\dom_\delta}{\tf F\cdot\nu^\dom}{\lem} \nonumber
\end{eqnarray}
for all $\tf\in\woinf{\edom}$ where $\esslim$ denotes the limit up to an
$\cL^1$-negligible set.
\end{proposition}

The results are a simple evaluation of $\divv(\tf_\delta\F)(\dom)$ for
suitable $\tf_\delta$. The first two equations for closed $\dom$
can be found in Schuricht \cite[p.~534]{schuricht_new_2007} and 
\cite[p.~189]{schuricht_quaderni} (cf. also 
\v Silhav\'y \cite[p.~449]{silhavy_divergence_2009} for a more general
version). The corresponding equations for open $\dom$ can be shown 
exactly the same way. The third equations can be found in
Chen-Comi-Torres \cite[p.~117-123]{chen-comi}. Here the
approximating sets $\dom_{-\delta}$ or $\dom_\delta$ can be replaced 
by approximating sets from inside or outside  
with smooth boundary by using a standard smoothing argument
(cf. \cite[p.~129, 150]{maggi}). 
The examples below show that the limits on the right hand side need not to be
related to a measure on $\bd\dom$ in general. 
For $\dom$ where $\ham^{n-1}(\bd\dom_\delta)$ is uniformly bounded near 
$\bd\dom$ and for suitable $F$, Proposition~\ref{nm-s9} below 
provides a ``more classical'' version without limit on the right hand side and 
with some density measure near $\bd\dom$ of the type 
given in Proposition~\ref{tt-s3}. Let us provide a short proof for
the convenience of the reader.  

\begin{proof}    
Fix $\tf\in\woinf{\edom}$.
First let $\dom$ be open and consider $\chi_\delta\in\woinf{\edom}$ with
\begin{equation*}
  \chi_\delta := 
  \chi_{\dom^c} + 
  \tfrac{1}{\delta}\chi_{\dom\setminus\dom_{-\delta}}\distf{\dom_{-\delta}}
  \qmq{for} \delta>0 \,.
\end{equation*}
Obviously $\chi_\delta=1$ on $\dom^c$, $\chi_\delta=0$ on $\dom_{-\delta}$,
$ \|1-\chi_\delta\|_{\bd\dom}=0$, and 
$\lem(\dom\setminus\Int\dom)=0$.
Then, by Corollary~\ref{ex:dmo_woi_c}, by  
$D\chi_\delta=\frac{1}{\delta}\nu^\dom$ $\lem$-a.e. on $\supp{D\chi_\delta}$,
and by dominated convergence,
\begin{eqnarray*}
  \divv(\tf F)(\dom)
&=& 
  \divv(\chi_\delta\tf F)(\dom)\\
&=&
  \I{\dom}{\chi_\delta\tf}{\divv F} + 
  \I{\dom}{\chi_\delta\cdot D\tf}{\lem} +
  \I{\dom}{\tf F\cdot D\chi_\delta}{\lem}  \\
&=& 
  \lim_{\delta\downarrow 0} \frac{1}{\delta}
  \I{\dom_\delta}{\tf F\cdot \nu^\dom}{\lem}   \,.
\end{eqnarray*}
The second equation follows from the coarea formula. 
For the third equation we first observe that
$\dom_{-\delta}\subset \mint\,(\dom_{-\delta}) \subset \ol{\dom_{-\delta}}$
for $\delta>0$. Thus the sets $\mint\,(\dom_{-\delta})$ are increasing 
as $\delta\downarrow 0$
and $\bigcup_{\delta>0}\mint\,(\dom_{-\delta})=\dom$. 
Since $\divv(\tf F)$ is a Radon measure on~$\edom$,
\begin{equation*}
  \divv(\tf F)(\dom) = 
  \lim_{\delta\downarrow 0}\,\divv(\tf F)\big(\mint\,(\dom_{-\delta})\big) \,.
\end{equation*}
Moreover we have that 
\begin{equation*}
  \divv(\tf F)\big(\mint\,(\dom_{-\delta})\big) = 
  \I{\rbd\dom_{-\delta}}{\tf F\cdot\nu^\dom}{\lem}
  \qmq{for $\cL^1$-a.e. $\delta>0$} 
\end{equation*}
with the measure theoretic normal $\nu^\dom$ on $\rbd\dom_{-\tau}$
(cf. \cite[p.~212]{degiovanni_cauchy_1999}, \cite[p.~534]{schuricht_new_2007}).
But, for $\cL^1$-a.e. $\tau>0$, we can replace it with $\nu^\dom$ 
from \reff{nm-normal} by the arguments following \reff{nm-normal1}.
This readily gives the third equation. 

If $\dom$ is closed we have $\dom\csubset\edom$ and it is sufficient to show 
the assertion for $\tf_c$ having compact support in $\edom$. 
Since $\dom^c$ is open and since
\begin{equation*}
  0 = \divv(\tf_c F)(\edom) = \divv(\tf_c F)(\dom) + \divv(\tf_c F)(\dom^c) 
\end{equation*}
by Corollary~\ref{ex:dmo_woi_c}, we can apply the first assertion
to $\dom^c$ to get the results for $\dom$.
\end{proof}

We still provide some situation where the right hand side in
Proposition~\ref{dt-s1a} can be represented by a Radon measure supported 
on $\bd\dom$ and we give some relation to 
continuum mechanics. For that we first recall a result from 
Schuricht \cite[p.~537]{schuricht_new_2007}. 

\begin{proposition}  \label{dt-s1b}
Let $\edom \subset \Rn$ be open and bounded and let $\F\in \dmo{\edom}$.
Then there is some $h\in\cL^1_{\rm loc}(\edom)$ with
$|F|\le h$ $\lem$-a.e. on $\edom$ such that for any 
$\dom\csubset\edom$ with finite perimeter and
$\I{\mbd\dom}{h}{\hm}<\infty$ one has $\chi_\dom
F\in\cD\cM^1(\edom)$. Moreover there is some 
$g_\dom\in\cL^\infty(\edom,|\divv F|)$ with values in $[0,1]$
such that, for any $B\in\bor{\edom}$,
\begin{equation} \label{dt-s1b-1}
  \divv(\tf\chi_\dom F)(B) =
  \I{B}{g_\dom \tf}{\divv F} + \I{B}{g_\dom F\cdot D\tf}{\lem}
  - \, \I{\mbd\dom\cap B}{\tf F\cdot\nu^\dom}{\hm} 
\end{equation}
for all $\tf\in\woinf{\edom}$ and $g_\dom(x)=\dens^\edom_x(\dom)$ whenever
$\dens^\edom_x(\dom)$ exists (cf. \reff{pm-e11}).  
\end{proposition}

\begin{remark} \label{dt-s1c}
(1) For $F$ and $\dom$ as in Proposition~\ref{dt-s1b}, 
Theorem~\ref{ex:dmo_woi} gives
\begin{equation*}
  \divv(\tf\chi_{\dom} F)(\edom)=0 \qmq{for all}
  \tf\in\woinf{\edom}\,,
\end{equation*}
since we can change $\tf$ outside $\dom$ to have compact support in $\edom$. 
Then, using the disjoint decomposition 
$\R^n=\mint\dom\cup\mbd\dom\cup\mext\dom$, we directly get from 
\reff{dt-s1b-1} that
\begin{equation*}
  \I{\mint\dom}{\tf}{\divv F} + 
  \I{\mint\dom}{F\cdot D\tf}{\lem} =
  \I{\mbd\dom}{\tf F\cdot\nu^\dom}{\hm} -
  \I{\mbd\dom}{g_\dom\tf}{\divv F} \,.
\end{equation*}
For $\dom=\mint\dom$ we thus have that $\divv(\tf F)(\dom)$ is related to a
Radon measure supported on $\mbd\dom$. 
Notice that the result covers cases where $\divv F$ doesn't vanish
on~$\bd\dom$. If $F\in\cD\cM^\infty(\edom)$
this is true 
with $g_\dom=\frac12$,
since $\divv F\wac\hm$ and $g_\dom=\frac12$ for $\hm$-a.e. point on 
$\mbd\dom$, 
and the related Radon measure has an $\hm$-integrable density on $\mbd\dom$.

(2) The Gauss-Green formula plays also an 
important role for contact interactions in continuum mechanics. 
Let $\edom$ be related to a continuous body, 
let $F$ and $\dom=\mint\dom$ be as in
Proposition~\ref{dt-s1b}, and let $B\subset(\mint\dom)^c$ be a Borel set.
Then we directly get
\begin{equation*}
  \divv(\chi_\dom F)(B) = \I{\mbd\dom}{g_\dom}{\divv F} -
  \I{\mbd\dom}{F\cdot\nu^\dom}{\hm}. 
\end{equation*}
This gives the action exerted from the subbody related to $\dom$ to the 
subbody related to~$B$. Analogously as above, this covers 
cases where $\divv F$ doesn't vanish on~$\bd\dom$ and we can specialize it for 
essentially bounded $F$
(cf. Schuricht \cite[p.~536]{schuricht_new_2007},  
Chen-Torres-Ziemer \cite[p.~298-291]{chen_gauss-green_2009},
Chen-Comi-Torres \cite[p.~157]{chen-comi}).
\end{remark}

Let us now characterize the special cases where $\mF=0$ is possible and 
where the measures $\lF$, $\mF$ can be chosen independent of $\delta$
in \reff{dt-s1-1}. 

\begin{proposition} \label{dt-s2}
Let $\edom \subset \Rn$ be open and bounded, let $\dom \in \bor{\edom}$, and 
assume that $\F\in \dmo{\edom}$. 
\bgl
\item
In Theorem~{\rm \ref{dt-s1}} we can choose $\lF$, $\mF$ with
$\cor{\lF}$, $\cor{\mF} \subset \bd\dom$, i.e. independent of $\delta$,  
if and only if 
\begin{equation} \label{dm-monster}
  \liminf\limits_{\delta\downarrow 0} 
  \sup_{\substack{\tf \in \woinf{\edom} \\ 
        \|\tf_{|(\bd\dom)_\delta\cap\edom}\|_{\cW^{1,\infty}}\le 1}} 
  \divv{(\port{\bd\dom}_\delta\tf\F)}(\dom) < \infty 
\end{equation}
with $\port{\bd\dom}_\delta$ as in \reff{eq:port_fun}. In this case 
\reff{dt-s1-1} becomes 
\begin{equation}\label{dt-s2-1}
   \divv(\tf\F)(\dom) = \sI{\bd\dom}{\tf}{\lF} + 
            \sI{\bd\dom}{D\tf}{\mF} \,.  
\end{equation}

\item
In Theorem~{\rm \ref{dt-s1}} we can choose $\mF=0$
for $\delta>0$ 
if and only if 
\begin{equation}\label{dm-special}
\sup_{\substack{\tf\in\woinf{\edom}\\
      \|\tf_{|(\bd\dom)_\delta\cap\edom}\|_{\cL^\infty}\le 1}} 
  \divv{(\tf\F)}(\dom) < \infty \,. 
\end{equation}

\item
In Theorem~{\rm \ref{dt-s1}} we can take $\mF=0$ and $\lF$ with
$\cor{\lF}\subset \bd\dom$ if and only if
\begin{equation}\label{dm-special1}
  \liminf_{\delta\downarrow 0}
  \sup_{\substack{\tf\in\woinf{\edom}\\
        \|\tf_{|(\bd\dom)_\delta\cap\edom}\|_{\cL^\infty}\le 1}} 
  \divv{(\tf\F)}(\dom) < \infty \,. 
\end{equation}
\el
\end{proposition}

\noindent
Notice that \reff{dm-monster} just means that the trace functional
$\tf\to\divv(\tf F)(\dom)$ is finite.

\begin{proof}   
Let $TF$ be as in Theorem~\ref{dt-s1}. 
Then (1) is a direct consequence of Proposition~\ref{prop:special-a} and
(2), (3) directly follow from Proposition~\ref{prop:special-b}.
\end{proof}

\noindent
We still provide some equivalent conditions.

\begin{lemma}\label{lem:monster2}
Let $\edom \subset \Rn$ be open and bounded, let $\dom\in\bor{\edom}$, and 
assume that $\F\in \dmo{\edom}$. Then \reff{dm-monster} is equivalent to
each of the following two conditions:
\bgl
\item
\begin{equation}\label{eq:monster_ub}
  \liminf \limits_{\delta\downarrow 0}  
  \sup_{\substack{\tf \in \woinf{\edom} \\ 
        \|\tf_{|(\bd\dom)_\delta\cap\edom}\|_{\cW^{1,\infty}}\le 1}}
  \I{(\bd{\dom})_\delta \cap \dom}{\tf\F D\port{\bd\dom}_\delta}{\lem} 
  < \infty \,,
\end{equation}
\item
\begin{equation*}
  \liminf\limits_{\delta\downarrow 0}  
  \sup_{\substack{\tf \in \woinf{\edom} \\ 
        \|\tf_{|(\bd\dom)_\delta\cap\edom}\|_{\cW^{1,\infty}}\le 1}}
     \; \frac{1}{\delta}\I{(\dnhd{(\bd{\dom})}{\delta}\setminus
        \dnhd{(\bd{\dom})}{\frac{\delta}{2}})\cap \dom}{\tf F
        \Deriv{\distf{\bd{\dom}}}}{\lem} < \infty \,. 
\end{equation*}
\el
Moreover, $\op{int}\dom=\emptyset$ implies \reff{dm-monster}.
\end{lemma}

\begin{proof}    
By \reff{ex:dmo_woi-4} and \reff{ex:dmo_woi-5} we have 
for $\tf\in\woinf{\edom}$ that
\begin{equation*}
  \divv(\port{\bd\dom}_\delta \tf F) = 
  \port{\bd\dom}_\delta\tf\divv F + 
  \port{\bd\dom}_\delta F D\tf \lem + 
  \tf F D\port{\bd\dom}_\delta \lem
\end{equation*}
as measures on $\edom$. Obviously
\begin{equation*}
  \big|\port{\bd\dom}_\delta\tf\divv F\big|(\dom) \qmq{and}
  \big|\port{\bd\dom}_\delta F D\tf \lem \big|(\dom) 
\end{equation*}
are uniformly bounded for $\delta>0$ and 
$\|\tf_{|(\bd\dom)_\delta\cap\edom}\|_{\cW^{1,\infty}}\le 1$. Since 
\begin{equation*}
  (\tf F D\port{\bd\dom}_\delta \lem)(\dom)= 
 \I{(\bd{\dom})_\delta \cap \dom}{\tf\F D\port{\bd\dom}_\delta}{\lem}\,,
\end{equation*}
\reff{dm-monster} is equivalent to \reff{eq:monster_ub}.
For the second condition we use that
\begin{equation*}
  D \port{\bd\dom}_\delta = 
  -\frac{2}{\delta} 
 \ind{\dnhd{(\bd{\dom})}{\delta}\setminus\dnhd{(\bd{\dom})}{\frac{\delta}{2}}}  
 D \distf{\bd{\dom}}\,.
\end{equation*}
If  $\op{int}\dom=\emptyset$, then $D \port{\bd\dom}_\delta = 0$ for all $x$
and this readily implies \reff{eq:monster_ub}.
\end{proof}

Let us still give some sufficient conditions that are useful for
applications. 

\begin{proposition}\label{dt-s4}
Let $\edom \subset \Rn$ be an open and bounded set, 
let $\dom\in\bor{\edom}$ with $\lem(\dom\setminus\Int\dom)=0$,
and assume that $\F\in \dmo{\edom}$. If
\begin{equation} \label{dt-s4-1}
  \liminf\limits_{\delta\downarrow 0} 
  \I{\dom}{\big|\F D\port{\bd\dom}_\delta\big|}{\lem} < \infty \,,
\end{equation}
then we can choose $\mF=0$ and $\lF$ with $\cor{\lF}\subset \bd\dom$
in Theorem~{\rm \ref{dt-s1}}. We have \reff{dt-s4-1} if
$F$ is bounded and if there is some $\tilde\delta>0$ such that
\begin{equation}\label{dt-s4-2}
  \sup_{\delta\in(0,\tilde\delta)}  
  \cH^{n-1}(\bd\dom_{-\delta})  < \infty \,.
\end{equation}
If $\dom$ has finite perimeter and if there are $c>0$,
$r>0$ such that
\begin{equation}\label{dt-s4-3}
  \frac{\lem(B_\delta(x)\cap(\ol\dom)^c)}{\lem(B_\delta(x))} \ge c \qmq{for all}
  x\in\bd\dom\,, \; \delta\in(0,r) \,,
\end{equation}
then \reff{dt-s4-2} is satisfied.
\end{proposition}

\noindent
Notice that $\bd\dom_{-\delta}$ has finite perimeter for all $\delta>0$
(cf. \cite[p.~2788]{kraft_measure-theoretic_2016}.
However, even if $\dom$ has finite perimeter, the perimeters
$\cH^{n-1}(\bd\dom_{-\delta})$ might not be bounded uniformly in $\delta$
(cf. \cite[p.~2781]{kraft_measure-theoretic_2016}).
The sufficient condition \reff{dt-s4-3} is a uniform lower
bound for the density of the exterior of $\dom$ at $x\in\bd\dom$.
It in particular excludes points $x\in\bd\dom$ where this density vanishes 
(as, e.g., for $x$ at an inward cusp of $\dom$ or for inner boundary points 
$x\in\op{int}\ol\dom$). But \reff{dt-s4-3} is obviously met if
$\dom$ has Lipschitz boundary. In this case we can also continuously extend 
all $\tf\in\woinf{\edom}$ up to $\bd\dom$ such that $\lF$ can be considered as
Radon measure supported on $\bd\dom$. Let us still refer to 
\cite[p.~449]{silhavy_divergence_2009} where it is shown that a 
condition similar to \reff{dt-s4-1} allows the representation of
$\divv(\tf F)(\dom)$ by a $\sigma$-measure on $\bd\dom$ for open and bounded
$\dom$ and $\tf$ that are Lipschitz continuous on $\R^n$.

\begin{proof}     
By \reff{dt-s4-1} there is some $c>0$ and a sequence $\delta_k\downarrow 0$
such that  
\begin{equation*}
  \I{\dom}{\big|\F D\port{\bd\dom}_{\delta_k}\big|}{\lem}\le c
  \qmq{for all} k\,.
\end{equation*}
For the trace $TF$ from Theorem~\ref{dt-s1}
and for $\tf\in\woinf{\edom}$ we can use dominated convergence to get
\begin{eqnarray*}
  \df{TF}{\tf}
&=&
  \divv{(\tf\F)}(\dom)
\; = \;
  \lim_{k\to\infty} \: \divv{(\port{\bd\dom}_{\delta_k}\tf\F)}(\dom)  \\
&=&
  \lim_{k\to\infty} \I{\dom}{\port{\bd\dom}_{\delta_k}\tf}{\divv\F} + 
  \I{\dom}{\port{\bd\dom}_{\delta_k}\F D\tf}{\lem} +
  \I{\dom}{\tf\F D\port{\bd\dom}_{\delta_k}}{\lem}   \\
&\le&
  \lim_{k\to\infty} \big(|\divv F|(\bd\dom)+c\big)\, 
  \|\tf_{|(\bd\dom)_{\delta_k}\cap\edom}\|_\infty \,.
\end{eqnarray*}
Hence the first statement follows from Proposition~\ref{dt-s2} (3). 
For bounded $F$ the coarea formula implies
\begin{equation*}
  \I{\dom}{\big|\F D\port{\bd\dom}_\delta\big|}{\lem} \le
  \|F\|_\infty \I{\dom}{\big| D\port{\bd\dom}_\delta\big|}{\lem} =
  \tfrac{2}{\delta}\|F\|_\infty
  \int_{-\delta}^{-\frac{\delta}{2}} \cH^{n-1}(\bd\dom_{\tau})\, d\tau
\end{equation*}
which gives the second assertion. For the last statement, \reff{dt-s4-3}
just means that the open set $(\ol\dom)^c$ 
has $(r,c)$-uniform lower density on $\bd\dom$ in the sense of Definition~4 in 
\cite{kraft_measure-theoretic_2016}. Hence, by Theorems~3 and 4 in
\cite{kraft_measure-theoretic_2016}, there are constants
$c_1,c_2>0$ such that for all $\delta\in(0,r)$ and with $\op{Per}$ denoting
the perimeter
\begin{equation} \label{dt-s4-4}
  \cH^{n-1}(\bd\dom_\delta) \le
  c_1 \frac{\lem(\dom\setminus\dom_\delta)}{\delta} 
 \le c_2 \op{Per}\big((\ol\dom)^c\big)\,.
\end{equation}
Since $\op{Per}\big((\ol\dom)^c\big)=\op{Per}(\dom)$, we readily get
\reff{dt-s4-2}.  
\end{proof}

\medskip

For Lipschitz continuous functions $\tf\in\op{Lip}(\Gamma)$ with
$\Gamma\subset\R^n$ we use the norm
\begin{equation*}
  \|\tf\|_{\op{Lip}(\Gamma)}= \|\tf\|_{C(\Gamma)} + \op{Lip}(\tf) 
\end{equation*}
where $\op{Lip}(\tf)$ is the Lipschitz constant of $\tf$ on $\Gamma$. 

\begin{proposition} \label{dt-s5}
Let $\edom \subset \Rn$ be open and bounded
and let $F\in\cD\cM^1(\edom)$. 
\bgl
\item
If $\dom \in \bor{\edom}$ is such that any $\tf\in\woinf{\edom}$ has a
continuous extension onto~$\cl\dom$, if $\lem(\dom\setminus\op{int}\dom)=0$,
and if there are $\const >0$ and $\tilde\delta>0$ such that 
\begin{equation} \label{dt-s5-0}
  \|\tf_{|\bd\dom}\|_{\op{Lip}(\bd\dom)} \le
  c \|\tf\|_{\woinf{(\bd\dom)_\delta\cap\edom}} \qmq{for all}
  \tf\in\woinf{\edom},\; \delta\in(0,\tilde\delta)\,,
\end{equation}
then \eqref{dm-monster} is satisfied.

\item If $\dom\subset\edom$ is open with Lipschitz boundary, then 
\eqref{dm-monster} is satisfied.
\el
\end{proposition}

\noi
Notice that \reff{dt-s5-0} is a condition for $\dom$ that does not depend 
on $F$. It somehow says that $\cW^{1,\infty}$-functions should be 
(globally) Lipschitz continuous near the boundary. 
We postpone the quite technical proof to the end of this section and now
provide several examples illuminating the previous results. 
First we show by an example on some open $\dom\subset\R^2$ with Lipschitz
boundary and some unbounded $F$ that $\me_F=0$ in \reff{dt-s1-1} is not
possible in general. This way we see that the new term is really needed
for a general Gauss-Green formula.  

\begin{example}\label{dt-ex1}   
Let $\edom=B_2(0)\subset\R^2$, let $\dom=(0,1)^2$, and let 
$F\in\cL^1(\edom,\R^2)$ be given by 
\begin{equation*}
    F(x,y) := \frac{1}{x^2+y^2}
    \begin{pmatrix}  -y \\ x   \end{pmatrix}\,.
\end{equation*}
(cf. also \cite[Example 2.5]{silhavy_divergence_2009}).
We clearly have $\divv{F} = 0$ on $\edom\setminus\{0\}$
in the classical sense and $F\cdot\nu^{B_\delta(0)} = 0$ on $\bd{B_\delta(0)}$ for
$\delta>0$. Then, for all $\tf\in C^1_c(\edom)$, 
\begin{eqnarray*}
  \I{\encl{\dom}}{F D\tf}{\cL^2} 
&=&
  \lim_{\delta \downarrow 0} \I{\edom\setminus B_\delta(0)}{F D\tf}{\cL^2} \\
&=&
  - \lim_{\delta \downarrow 0} \Big(
  \I{\edom\setminus B_\delta(0)}{\tf \divv{F}}{\cL^2} +
  \I{\bd{B_\delta(0)}}{\tf F \cdot \nu^{B_\delta(0)}}{\ham^1} \Big) \\
&=&
  0  \,.
\end{eqnarray*}
Therefore $\divv{F}$ is the zero measure on $\edom$
and, thus, $F\in\dmo{\edom}$. 
For small $\delta>0$ we have case (C) by Corollary~\ref{cor:cases} (3)
and \reff{dm-monster} is satisfied by Proposition~\ref{dt-s5}~(2).
Hence, by Theorem~\ref{dt-s1}, there are 
a Radon measure $\siF\in\cM(\edom)$ supported on $\bd\dom$ and a 
measure $\mF\in \op{ba}(\edom,\cB(\edom),\cL^2)^2$ with
core in $\bd\dom$ such that
\begin{equation}\label{dt-ex1-1}
  \divv(\tf F)(\dom) =  
  \I{\bd\dom}{\tf}{\siF} + \I{\bd\dom}{D\tf}{\mF}
\end{equation}
for all $\tf\in\woinf{\edom}$. For
\begin{equation*}
  \tf_k := \ind{\left (\frac{1}{k},\infty\right) \times \R} +
           \ind{\left (0,\frac{1}{k}\right)\times \R} k \distf{\{0\}\times\R} 
\end{equation*}
we have 
\begin{equation*}
  \tf_k\in \woinf{\edom}\,, \quad
  D\tf_k = \ind{\left (0,\frac{1}{k}\right) \times \R} \, k 
  {\scriptsize \begin{pmatrix}  1 \\ 0   \end{pmatrix} } \,,  \quad
  0\le\tf_k\le 1 
\end{equation*}
and, thus,
\begin{eqnarray*}
  \big| \divv(\tf_kF)(\dom) \big|
& = &  
  \Big| \I{\dom}{F D\tf_k}{\cL^2} + \I{\dom}{\tf_k}{\divv{F}} \Big| 
\; = \; 
  \Big| \I{\dom}{F D\tf_k}{\cL^2} \Big|  \\ 
& = &
  \mI{\left	(0,\frac{1}{k}\right)}{\I{(0,1)}{\frac{\eRR}{\eR^2+
	\eRR^2}}{\eRR}}{\eR} \\
& = &
   \mI{\left(0,\frac{1}{k}\right)}{\left [\frac{1}{2} \ln (\eR^2+
  \eRR^2)\right]_{y=0}^1}{\eR} \\
& = &
  \mI{\left(0,\frac{1}{k}\right)}{\frac{1}{2}\ln\left(\frac{1}{\eR^2}
  +1\right)}{\eR} \\
& \geq &
  \frac{1}{2} \ln (k^2 + 1) \xrightarrow{k \to \infty} \infty \,.
\end{eqnarray*}
Moreover
\begin{equation*}
  \Big|\I{\bd\dom}{\tf_k}{\siF}\Big| \le |\siF|(\edom)  
\end{equation*}
for all $k$. Hence $\mF=0$ is impossible in \reff{dt-ex1-1}.

In Example~\ref{ex:rotsing} below we consider the same vector field on a 
slightly modified set $\dom$. There we construct for \reff{dt-ex1-1}
possible Radon measures $\sigma_F$ and pure measures $\mF$ depending on some
scalar parameter. This way we provide  
an uncountable family of possibilities for $\sigma_F$ and $\mF$. 
\end{example}

Next we provide an example with a constant vector field $F$ and an
open $\dom\subset\R^2$ having infinite perimeter where
\reff{dm-monster} fails. This means by Proposition~\ref{dt-s2} 
that the dependence of the measures $\lF$, $\mF$ on $\delta>0$ cannot be
removed. 

\begin{example} \label{dt-ex2}  
Let $\edom:=B_2(0)\subset\R^2$, let 
\begin{equation*}
    \dom := \bigcup \limits_{k=1}^\infty R_k \qmq{with}
    R_k := \big(\tfrac{1}{2k+1},\tfrac{1}{2k}\big) \times (0,1)\,,
\end{equation*}
and take the constant vector field 
\begin{equation*}
  F\in\cL^1(\edom) \qmq{with} 
  F=\big(\begin{smallmatrix} 1 \\ 0 \end{smallmatrix}\big) \,.
\end{equation*}
(cf. also Example~\ref{tt-r-ex2}).
Obviously $\divv{\funv} = 0$ on $\edom$ and thus $F\in\dmo{\edom}$. 

Let us show that \reff{eq:monster_ub} doesn't hold. 
For $\delta>0$ we choose $k_\delta\in\N$ to be the largest number
such that
\begin{equation*}
  \delta < \tfrac{1}{4}\big( \tfrac{1}{2k}-\tfrac{1}{2k+1}\big)
  \qmq{for all} k\le k_\delta+1 \,.
\end{equation*}
Then $k_\delta\to\infty$ for $\delta\to 0$. Moreover, for $\delta>0$ fixed,
we set 
\begin{equation*}
  R_k^-:=\big(\tfrac{1}{2k+1},\tfrac{1}{2k+1}+\delta\big) \times (0,1)\,, \quad
  R_k^+:=\big(\tfrac{1}{2k}-\delta,\tfrac{1}{2k}\big) \times (0,1)\,.
\end{equation*}
Then we obviously find some 
\begin{equation*}
  \tf^\delta\in\woinf{\edom} \qmq{with} 
  \|\tf^\delta_{\,|(\bd\dom)_\delta\cap\edom}\|_{\cW^{1,\infty}} \le 1
\end{equation*}
such that
\begin{equation*}
  \tf^\delta(x,y) = \bigg\{ 
  \begin{array}{lll}
    \pm \tfrac{1}{2}\big(\tfrac{1}{2} - |y-\tfrac{1}{2}| \big) &
    \zmz{on} R_k^\pm & \text{for }\; k\le k_\delta \,, \\
    0 & \zmz{on} R_k & \text{for }\; k>k_\delta\,.
  \end{array}
\end{equation*}
Notice that
\begin{equation*}
  \tf^\delta(x,0)=\tf^\delta(x,1)=0\,, 
  \z \tf^\delta(x,\tfrac{1}{2})=\pm\tfrac{1}{4}\,, \z 
  |D\tf^\delta(x,y)|=\tfrac{1}{2} \qmq{on $R_k^\pm$ for $k\le k_\delta$\,,}
\end{equation*}
\begin{equation*}
  F D\port{\bd\dom}_\delta=0 \qmz{on} 
  \big((\bd\dom)_\delta\cap\dom\big)\setminus\big(R_k^-\cup R_k^+)\big) \,.
\end{equation*}
Hence
\begin{eqnarray*}
  \I{(\bd\dom)_\delta\cap\dom}{\tf^\delta F D\port{\bd\dom}_\delta}{\cL^2}
&=&
  \sum_{k=1}^{k_\delta}
  \I{R_k^-\cup R_k^+}{\tf^\delta F D\port{\bd\dom}_\delta }{\cL^2}  \\
&\ge&
  \sum_{k=1}^{k_\delta} 2
  \int_\delta^{1-\delta}\tfrac{1}{2}\big(\tfrac{1}{2}-|y-\tfrac{1}{2}|\big)\,dy\\
&=&
  2k_\delta \int_\delta^{\frac{1}{2}} y\, dy 
\; = \;
  k_\delta \big( \tfrac{1}{4}-\delta^2\big) 
  \overset{\delta\to 0}{\longrightarrow} \infty \,.
\end{eqnarray*}
But this means that \reff{eq:monster_ub} is not satisfied. 

For $\delta\in\big(0,\tfrac{1}{2}\big)$ and $\K=\bd\dom$ 
we have case (C) by 
Corollary~\ref{cor:cases}~(3) and \reff{dt-s1-1} becomes 
\begin{equation*}
   \divv(\tf\F)(\dom) = \I{\bd\dom}{\tf}{\siF} + 
            \I{(\bd\dom)_\delta\cap\edom}{D\tf}{\mF} \,.   
\end{equation*}
(notice that this doesn't contradict Example~\ref{tt-r-ex2} where 
$\edom$ and $\K$ are different). 

Let us provide possible choices of $\siF$ and $\mF$.
For $\delta>0$ we first fix some $m_\delta\in\N$ such that
\begin{equation*}
  \frac{1}{2m_\delta} - \frac{1}{2m_\delta+1} < \delta \,.
\end{equation*}
Then we get for every $\tf\in\woinf{\edom}$
\begin{eqnarray*}
  \divv(\tf F)(\dom)
&=&
  \sum_{k=1}^\infty \divv(\tf F)(R_k)
\; = \;
   \sum_{k=1}^\infty \I{R_k}{\tf\divv F + FD\tf}{\cL^2}  \\
&=&
  \sum_{k=1}^{m_\delta} \I{\bd R_k}{\tf F\nu^{R_k}}{\cH^1} +
  \sum_{k=m_\delta+1}^\infty \I{R_k}{FD\tf}{\cL^2} \,
\end{eqnarray*}
where we have used the classical Gauss-Green formula for $k\le m_\delta$.
Hence we can choose 
\begin{equation*}
  \siF^\delta = \sum_{k=1}^{m_\delta} F\nu^{R_k}\ham^1 \lfloor \bd R_k  \,,
  \quad
  \mF^\delta = \sum \limits_{k=m_\delta+1}^\infty F\cL^2\lfloor R_k \,.
\end{equation*}
Notice that $R_k\subset (\bd\dom)_\delta\cap\edom$ for $k\ge m_\delta+1$
and that the measures $\mF^\delta$ are even $\sigma$-measures.
Let us also mention that we cannot sum over all $k\in\N$ for $\sigma_F$, since
this would not give a bounded measure. 
The dependence of the measures on $\delta$ comes through $m_\delta$.
But, for fixed $\delta>0$, we also have some freedom to choose $m_\delta\in\N$. 
Therefore the choice of $\sigma_F$ and $\me_F$ is not unique even for given
$\delta$. Since the measures $\siF^\delta$, $\mF^\delta$ given above are
restricted to $\dom$, the situation would not change if we take $\edom=\dom$
instead of $\edom=B_2(0)$. 
\end{example}

We now give an example where we can choose $\mF=0$ but $\lF$ cannot be 
taken as $\sigma$-measure. 

\begin{example} \label{dt-ex3}
We set
\begin{equation*}
  \dom_1:=(0,1)\times(0,1)\,, \quad \dom_2:=(1,2)\times(0,1)\,
\end{equation*}
and consider
\begin{equation*}
  \edom=\dom=\dom_1\cup\dom_2
\end{equation*}
with the discontinuous vector field
\begin{equation*}
  F= \big(\begin{smallmatrix} 1 \\ 0 \end{smallmatrix}\big) \zmz{on} \dom_1\,, 
  \quad
  F= \big(\begin{smallmatrix}-1 \\ 0 \end{smallmatrix}\big) \zmz{on} \dom_2\,.
\end{equation*}
We readily verify \reff{dt-s4-2} and, hence, we can choose $\mF=0$ and
$\lF$ with $\cor{\lF}\subset \bd\dom$ in Theorem~{\rm \ref{dt-s1}}.
Since $\tf\in\woinf{\edom}$ cannot be extended continuously up to $\bd\dom$ in
general, we do not have case (C) and we cannot expect that $\lF$ is in fact a
Radon measure on $\bd\dom$. But notice that we have case (L) by
Corollary~\ref{cor:cases}~(1). 

Though the classical Gauss-Green formula is not applicable on $\dom$, we can use
it on $\dom_j$, since $\tf_{|\dom_j}$ is continuously extendable to a
Lipschitz function $\ol\tf_j$ on $\ol\dom_j$ for all $\tf\in\woinf{\edom}$.
Therefore, with
$\sigma_F^j=F\nu^{\dom_j}\cH^1\lfloor\bd\dom_j$,
\begin{equation*}
  \divv(\tf F)(\dom_j) = \I{\bd\dom_j}{\ol\tf_j F\nu^{\dom_j}}{\cH^1}
  = \I{\bd\dom_j}{\ol\tf_j}{\sigma_F^j} \qmq{for} \tf\in\woinf{\edom} \,.
\end{equation*}
But, due to $\tf$
without continuous extension up to $\bd\dom$, we cannot just sum
up the $\sigma_F^j$ for $\dom$. However we can apply  
Proposition~\ref{acm-ram} (2) to $\dom_j$ and obtain pure measures 
\begin{equation*}
  \lF^j\in\op{ba}(\dom_j,\cB(\dom_j),\cL^2) \qmq{with} \cor{\lF^j}=\bd\dom_j\,
\end{equation*}
such that
\begin{equation*}
  \divv(\tf F)(\dom_j) = 
  \sI{\bd\dom_j}{\tf_{|\dom_j}}{\lF^j} \qmq{for} \tf\in\woinf{\edom} \,.
\end{equation*}
We can consider the $\lF^j$ as measures on $\dom$ by extending them
with zero. Then $\dom_j$ is an aura of $\lF^j$ and we readily obtain 
for $\lF=\lF^1+\lF^2$ that
\begin{equation*}
  \divv(\tf F)(\dom) = 
  \sI{\bd\dom}{\tf}{\lF} \qmq{for} \tf\in\woinf{\edom} \,.
\end{equation*}
We can interpret the situation along the common boundary of the $\dom_j$ so,
that one part of $\lF$ takes care for $\tf_{|\dom_1}$ and the other part takes
care for $\tf_{|\dom_2}$. This nicely shows the relevance of the aura. 
In the case of a crack along the inner boundary  we would be able 
to describe the situation on each side of the crack separately by choosing
suitable functions $\tf$.

Notice that the usual Gauss-Green formula using $\mint{\dom}$ and 
$\rbd\dom$ (cf. \cite[p.~209]{evans}) is substantially different,
since here the interior part of $\bd\dom$ belongs to $\mint{\dom}$ and 
$\tf$ has to be continuous on $\mint{\dom}$. 
\end{example}

In the next example we consider some vector field $F$ where $\divv F$ has some
point concentration on $\bd\dom$. We demonstrate the difference between an open 
$\dom$ and its closure and we discuss some lower dimensional $\dom$.

\begin{example}\label{dt-ex4}   
We consider the vector field, that is a classical example in the literature,
\begin{equation*}
  F(x):=\frac{x}{2\pi|x|^2} \qmq{on} \edom=B_2(0)\subset\R^2\,
\end{equation*}
and we discuss the Gauss-Green formula for several $\dom$.
Clearly $F\in\cL^1(\edom)$ and we easily compute that $\divv F=0$ on 
$\edom\setminus\{0\}$ in the classical sense. Moreover, with the usual
Gauss-Green formula we get for $\tf\in C^1_c(\edom)$
\begin{eqnarray*}
  \I{\encl{\dom}}{F D\tf}{\cL^2} 
&=&
  \lim_{\delta \downarrow 0} \I{\edom\setminus B_\delta(0)}{F D\tf}{\cL^2} \\
&=&
  - \lim_{\delta \downarrow 0} \Big(
  \I{\edom\setminus B_\delta(0)}{\tf \divv{F}}{\cL^2} +
  \I{\bd{B_\delta(0)}}{\tf F \cdot \nu^{B_\delta(0)}}{\ham^1} \Big) \\
&=&
  -\lim_{\delta \downarrow 0} \:
  \frac{1}{2\pi\delta}\:\I{\bd{B_\delta(0)}}{\tf}{\ham^1} \; = \; -\tf(0) \,.
\end{eqnarray*}
Hence, by \reff{tt-e3}, 
$\divv F=\delta_0$ as measure
(Dirac measure concentrated at the origin) and $F\in \dmo{\edom}$.

First we consider 
\begin{equation*}
  \dom := B_1(0)\cap C_\alpha\subset\R^2
\end{equation*}
where $C_\alpha$ is an open cone with vertex at the origin and 
opening angle $\alpha\in(0,2\pi)$. 
Let us check condition \reff{dt-s4-1}. There we have to integrate over the
support of $\port{\bd\dom}_\delta$. For the integral over the curved part near 
$\bd B_1(0)$ we use the coarea formula to get a bound for small $\delta>0$ by
\begin{equation*}
  \I{(\bd B_1(0))_\delta\cap B_1(0)}{\big|\F D\port{\bd\dom}_\delta\big|}{\cL^2} 
  = \frac{2}{2\pi\delta}\int_{1-\delta}^{1-\frac{\delta}{2}}\frac{2\pi r}{r}\,dr 
  = 1 \,.
\end{equation*}
For $\bd C_\alpha$ we assume that it contains  
the positive $x_1$-axis. Then, for small $\delta>0$, 
\begin{equation*}
  \dom\cap Q_\delta\ne\emptyset \qmq{for} 
  Q_\delta=(\tfrac{\delta}{2},1)\times(\tfrac{\delta}{2},\delta) \,.
\end{equation*}
Using the symmetry, we can estimate 
\begin{eqnarray*}
  \I{\dom}{\big|\F D\port{\bd\dom}_\delta\big|}{\cL^2}
&\le&
  1 + 2\I{Q_\delta}{\frac{(x_1,x_2)}{2\pi(x_1^2+x_2^2)}\cdot
                   \big(0,\tfrac{2}{\delta}\big)}{\cL^2} \\
&=&
  1+\frac{2}{\delta\pi}\int_{\tfrac{\delta}{2}}^\delta  
  \int_{\tfrac{\delta}{2}}^1 \frac{x_2}{x_1^2+x_2^2}\, dx_1dx_2 \\
&\le&
  1+\frac{2}{\delta\pi} \frac{\delta}{2}
  \int_{\tfrac{\delta}{2}}^1 \frac{\delta}{x_1^2+\tfrac{\delta^2}{4}}\,dx_1\\
&=&
  1 + \frac{2}{\pi}
  \Big[\arctan\frac{2x_1}{\delta}\Big]_{\delta/2}^1 \\
&=&
  1 + 
  \frac{2}{\pi} \Big(\arctan\frac{2}{\delta} - \arctan 1\Big) \,.
\end{eqnarray*}
Hence \reff{dt-s4-1} is satisfied. Therefore
we can choose $\mF=0$ and $\lF$ with $\cor{\lF}\subset \bd\dom$
in Theorem~{\rm \ref{dt-s1}}. Moreover we have case (C) by
Corollary~\ref{cor:cases}. Thus \reff{dt-s1-1} becomes
\begin{equation} \label{dt-ex4-1}
  \divv(\tf\F)(\dom) = \I{\bd\dom}{\tf}{\sigma_F}
\end{equation}
for a Radon measure $\sigma_F$ supported on $\bd\dom$. 

For the open set $\dom$ we now 
consider $\tf\in\woinf{\edom}$ supported outside
$B_r(0)$ for small $r>0$. Then the usual Gauss-Green formula gives
\begin{equation*}
  \divv(\tf\F)(\dom) = \I{\bd\dom}{\tf F\nu^\dom}{\cH^1} = 
  \I{\bd\dom}{\tf}{\sigma_F} \,.
\end{equation*}
Since $\woinf{\edom}$ is dense in $C(\edom)$, we obtain
\begin{equation*}
  \sigma_F=F\nu^\dom\cH^1\lfloor\bd\dom  \qmq{on} \bd\dom\setminus\{0\} \,.
\end{equation*}
From \reff{dt-ex4-1} we get with $\tf=1$ on $\edom$ 
\begin{equation*}
  0 = \divv F(\dom) = 
  \sigma_F(\{0\}) + \I{\bd B_1(0)\cap\bd\dom}{F\nu^\dom}{\cH^1} =
  \sigma_F(\{0\}) + \frac{\alpha}{2\pi} \,.
\end{equation*}
Hence 
\begin{equation*}
  \sigma_F(\{0\}) = - \frac{\alpha}{2\pi} 
\end{equation*}
and, therefore,
\begin{equation} \label{dt-ex4-2}
  \sigma_F=F\nu^\dom\cH^1\lfloor\bd\dom - \tfrac{\alpha}{2\pi}\delta_0 \,.
\end{equation}

For $\ol\dom$ instead of $\dom$ we can argue the same way. Then the
concentration at the origin is contained. This leads to 
\reff{dt-ex4-1} with $\ol\dom$ and some measure $\ol\sigma_F$ that is given
by
\begin{equation*}
  \ol\sigma_F = F\nu^\dom\cH^1\lfloor\bd\dom + 
  \tfrac{2\pi-\alpha}{2\pi}\delta_0 \,.
\end{equation*}
Notice that
\begin{equation*}
  \sigma_F(\{0\}) = \lim_{r\to 0}
  \I{\bd B_r(0)\cap\dom}{F\nu^{B_r(0)^c}}{\ham^1}
\end{equation*}
but
\begin{equation*}
  \ol\sigma_F(\{0\}) = \lim_{r\to 0}
  \I{\bd B_r(0)\cap\dom^c}{F\nu^{B_r(0)}}{\ham^1} \,.
\end{equation*}
This allows the interpretation that $\sigma_F$ is a trace from inside, i.e.
it can be computed from $F$ restricted to the interior of $\dom$ 
and this way disregards the concentration of $\divv F$ on the boundary, while
$\ol\sigma_F$ is a trace from outside, i.e. it can be computed from $F$
restricted to the exterior of $\dom$ and includes the concentration of 
$\divv F$ at the origin. However we have that $\sigma_F(\{0\})$ and 
$\ol\sigma_F(\{0\})$ depend on the opening angle $\alpha$ and, in 
\reff{dt-ex4-1}, they merely contribute 
a part of the concentration of $\divv F$.

We can also treat any Borel set $\tilde\dom$ with
\begin{equation*}
  \dom\subset\tilde\dom \subset\ol\dom \,.
\end{equation*}
Then we readily get for the associated measure $\tilde\sigma_F$ that
\begin{equation*}
  \tilde\sigma_F = \bigg\{
  \begin{array}{ll}
    \sigma_f & \zmz{if} 0\not\in\tilde\dom \,, \\
    \ol\sigma_F & \zmz{if} 0\in\tilde\dom \,.
  \end{array}
\end{equation*}

This exact treatment of concentrations on the boundary  
allows in applications e.g. a very precise description how a point load
at a body is balanced by its parts (cf.~\cite{podio_2004},
\cite[Example 2]{schuricht_new_2007}). 
Moreover it shows that a dependence 
of such a concentration on a normal at $x=0$ doesn't make sense. 
Notice that point concentrations can also occur at a cuspidal
corner and need not be proportional to the opening angle $\alpha$ in general
(cf. \cite{podio}).

As limit case $\alpha=0$ we still consider
\begin{equation*}
  \dom = (0,1)\times\{0\} \,.
\end{equation*}
Then we have for all $\tf\in\woinf{\edom}$
\begin{equation*}
  \divv(\tf\F)(\dom) = \I{\dom}{\tf}{\divv\F} = 0\,, \quad
  \divv(\tf\F)(\ol\dom) = \tf(0) \,.
\end{equation*}
Hence we readily get with the zero measure $\sigma_0$
\begin{equation*}
  \divv(\tf\F)(\dom) = \I{\bd\dom}{\tf}{\sigma_0} \qmq{and}
  \divv(\tf\F)(\ol\dom) = \I{\bd\dom}{\tf}{\delta_0} \,.
\end{equation*}
This agrees with $\sigma_F$ and $\ol\sigma_F$ from above if we take 
as outer ``unit'' normal of $\dom$ the sum of the upward and the downward
normal, i.e.
\begin{equation*}
  \nu^\dom = {\scriptsize \begin{pmatrix} 0 \\ 1 \end{pmatrix} } +
  {\scriptsize \begin{pmatrix} 0 \\ -1 \end{pmatrix} } = 
  {\scriptsize \begin{pmatrix} 0 \\ 0 \end{pmatrix} } \,.
\end{equation*}
Obviously we can also treat $\dom=\{0\}$ that way.

\end{example}

In the light of $\tilde\dom$ in the previous example
let us provide a further simple example demonstrating 
how the Gauss-Green formula exactly takes care
for the points belonging to $\dom$. 

\begin{example} \label{dt-ex4a}   
For $\edom=B_2(0)\subset\R^2$ we consider $F\in\cD\cM^1(\edom)$ given by
\begin{equation*}
  F = \bigg\{
  \begin{array}{ll}
    (2,0) & \zmz{on} (0,1)^2 \,, \\
    (1,0) & \zmz{otherwise\,.} 
  \end{array}
\end{equation*}
The first component of $F$ is a BV function and we readily get
\begin{equation*}
  \divv F = \reme{\cH^1}{\{0\}\times(0,1)} - \reme{\cH^1}{\{1\}\times(0,1)}
\end{equation*}
(cf. \cite[p.~169]{evans}). Let $\dom$ be a Borel set satisfying
\begin{equation*}
  (0,1)^2\subset\dom\subset[0,1]^2 \,.
\end{equation*}
By Corollary~\ref{cor:cases} and Proposition~\ref{dt-s4} 
we have case (C) for small $\delta>0$ and we can choose \reff{dt-s1-1} in
the form 
\begin{equation*} 
  \divv(\tf\F)(\dom) = \I{\bd\dom}{\tf}{\sigma_F}
  \qmq{for all} \tf\in\woinf{\edom}
\end{equation*}
with a Radon measure $\sigma_F$ supported on $\bd\dom$. For the determination 
of $\sigma_F$ we consider rectangles $R\subset\edom$
intersecting $\bd\dom$. Then we approximate $\chi_R$ by
\begin{equation*}
  \chi_R^\delta := 
  \chi_R + \chi_{R_\delta\setminus R}(1-\tfrac{1}{\delta}\distf{R}) 
  \in\woinf{\edom} \zmz{for small} \delta>0\,
\end{equation*}
and use $\divv(\chi_R^\delta)=\chi_R^\delta\divv F+F\cdot
D\chi_R^\delta$. This way we get with 
\begin{equation*}
  \K_\dom:=\bd\dom\cap\dom\,, \quad
  F_{\rm int}:=(2,0)\,, \quad F_{\rm ext}:=(1,0)\,,
\end{equation*}
that
\begin{equation*}
  \divv(\tf\F)(\dom) = 
  \I{\bd\dom\setminus\K_\dom}{\tf F_{\rm int}\cdot\nu^\dom}{\cH^1} +
  \I{\K_\dom}{\tf (F_{\rm int}-F_{\rm ext})\cdot\nu^\dom}{\cH^1}
\end{equation*}
for all $\tf\in\woinf{\edom}$.
\end{example}

\medskip

In Example~\ref{dt-ex2} we have already seen that the choice of $\lF$ and
$\mF$ in \reff{dt-s1-1} is not unique in general. 
For a simple case we now demonstrate that the usual boundary integral
$\I{\bd{\dom}}{\tf F\cdot\nu}{\ham^{n-1}}$
in the Gauss-Green formula can be completely replaced with
$\I{\bd\dom}{D\tf}{\me_F}$ for some suitable $\me_F$. 

\begin{example} \label{dt-ex5}   
For $\edom=\dom=(-1,1)^2\subset\R^2$ and the vector field 
\begin{equation*}
  F(x,y) = 
  \Big(\,\mbox{\small $\begin{matrix}  |x| \\ 0  \end{matrix}$}\,\Big)
  \qmq{on}\dom  
\end{equation*}
we have classically
\begin{equation} \label{dt-ex5-1}
  \I{\dom}{F D\tf}{\lem} + \I{\dom}{\tf\divv{F}}{\lem} =
  \I{\bd{\dom}}{\tf F\cdot\nu^\dom}{\ham^1} 
\end{equation}
for all $\tf\in\woinf{\dom}$. For a transformation of the right hand side 
we first consider 
$\tf\in C^1(\ol\dom)$ and use integration by parts piecewise on 
$\bd\dom$ to get
\begin{eqnarray}
  \I{\bd{\dom}}{\tf F\cdot\nu^\dom}{\ham^1}
&=& 
  \int_{-1}^1 \tf(1,y) - \tf(-1,y) \, dy \nonumber\\
&=&
  - \int_{-1}^1 y\tf_y(1,y) - y\tf_y(-1,y)\, dy +
  \big[y\tf(1,y)-y\tf(-1,y)\big]_{-1}^1  \nonumber\\
&=& 
  - \int_{-1}^1 y\tf_y(1,y) - y\tf_y(-1,y)\, dy  \nonumber\\
&&  
  +  \int_{-1}^1 \tf_x(x,1) + \tf_x(x,-1) \, dx \nonumber \\
&=& \label{dt-ex5-2} 
  \I{\bd{\dom}}{G D\tf}{\ham^1}
\end{eqnarray}
where 
\begin{equation*}
  G(x,y) = 
    \bigg\{ \mbox{\small $ 
    \begin{array}{cl} (0,-xy) & 
                    \text{for } |x|>|y|\,, \\[3pt]
                    (|y|,0) & \text{for } |x|<|y| \,.
    \end{array}$ } 
\end{equation*}
For some general $\tf\in\woinf{\dom}$ we choose smooth 
$\tf_k\in C^\infty(\ol\dom)$
approximating $\tf$ within $\cW^{1,1}(\dom)$ 
and $\me_{\bd\dom}$ be the measure from Proposition~\ref{tt-s3} related to 
$\alse=\dom$ and $\gamma(\delta)=\delta$. 
Piecewise continuity of the integrand up to the boundary, the coarea
formula, and \reff{tt-s3-2} imply
\begin{equation*}
   \I{\bd{\dom}}{G D\tf_k}{\ham^1} =
  \lim_{\delta\downarrow 0} \frac{1}{\delta}
  \I{(\bd\dom)_\delta\cap\dom}{G D\tf_k}{\cL^2}
  = \sI{\bd\dom}{GD\tf_k}{\me_{\bd\dom}} \qmz{for all $k$\,.}
\end{equation*}
One readily checks for small $\delta>0$ that \reff{dt-ex5-2} is also valid
with $\dom_{-\delta}$ instead of $\dom$. Moreover we extend $\nu^\dom$ onto
$\dom$ by setting $\nu^\dom=\nu^{\dom_{-\delta}}$ on $\bd\dom_{-\delta}$. 
Then the coarea formula yields
\begin{equation*} 
  \frac{1}{\delta} \I{(\bd\dom)_\delta\cap\dom}{\tf_k F\cdot\nu^\dom}{\cL^2} =
  \frac{1}{\delta} \I{(\bd\dom)_\delta\cap\dom}{GD\tf_k}{\cL^2} 
  \qmz{for all} k\,. 
\end{equation*}
By the uniform convergence $\tf_k\to\tf$, the limit $k\to\infty$ on the left
hand side is uniform in $\delta$. With $D\tf_k\to D\tf$ in $\cL^1(\dom)$ we
therefore get that the limit
\begin{equation*}
  \lim_{k\to\infty}\frac{1}{\delta} \I{(\bd\dom)_\delta\cap\dom}{GD\tf_k}{\cL^2} 
  = \frac{1}{\delta} \I{(\bd\dom)_\delta\cap\dom}{GD\tf}{\cL^2} 
\end{equation*}
is also uniform in $\delta$. Using Corollary~\ref{tt-s4} we get
\begin{eqnarray*}
  \I{\bd{\dom}}{\tf F\cdot\nu^\dom}{\ham^1} 
&=& 
  \lim_{k\to\infty}\I{\bd{\dom}}{\tf_k F\cdot\nu^\dom}{\ham^1} =
  \lim_{k\to\infty}\I{\bd{\dom}}{G D\tf_k}{\ham^1} \\
&=&
  \lim_{k\to\infty}\sI{\bd{\dom}}{G D\tf_k}{\me_{\bd\dom}} =
  \sI{\bd{\dom}}{G D\tf}{\me_{\bd\dom}}\,.
\end{eqnarray*}
Thus we end up with 
\begin{equation} \label{dt-ex5-4}
  \I{\dom}{F D\tf}{\lem} + \I{\dom}{\tf\divv{F}}{\lem} =
  \sI{\bd{\dom}}{G D\tf}{\me_{\bd\dom}}
\end{equation}
for all $\tf\in\woinf{\dom}$. From \reff{dt-ex5-2} we see that we can replace
$\me_{\bd\dom}$ with $\reme{\cH^1}{\bd\dom}$ for $\tf\in C^1(\ol\dom)$.
But notice that $\me_{\bd\dom}$ is pure by Proposition~\ref{acm-s1} and that
we cannot take a
$\sigma$-measure in \reff{dt-ex5-4} in general. 
Finally it is a simple observation by taking $\tf=1$ on $\dom$ that a
Gauss-Green formula like \reff{dt-ex5-4} is only possible if 
$(\divv F)(\dom)=0$ (the example shows that $\divv F=0$ is not
necessary).
\end{example}

Let us still give some more general example in $\R^3$ for a 
Gauss-Green formula of the form \reff{dt-ex5-4}.
We refrain from looking for the most general version, since we just want to
provide the essential idea for a class of vector fields satisfying
$\divv F=0$.

\begin{example} \label{dt-ex6}   
Let $\edom\subset\R^3$ be open and bounded, let $\dom\csubset\edom$ be open
and connected with smooth boundary, and let $F\in C^1(\edom,\R^3)$ be a
bounded vector field such that 
\begin{equation*}
  F=\curl G \qmq{where} G\in C^2(\edom,\R^3)\,.
\end{equation*}
Obviously $\divv F=0$ and, thus, $F\in\cD\cM^1(\edom)$. By the smooth boundary
we have \reff{dt-s4-2}. Hence, by Proposition~\ref{dt-s4},
we can choose $\mF=0$ and $\lF$ with
$\cor{\lF}\subset\bd\dom$ in Theorem~\ref{dt-s1}  
while the usual Gauss-Green formula gives
\begin{equation} \label{dt-ex6-1}
  \divv{(\tf F)}(\dom) = \I{\dom}{F\cdot D\tf}{\cL^3} =
 \I{\bd\dom}{\tf F\cdot\nu^\dom}{\cH^2}  
\end{equation}
for all $\tf\in C^1(\edom)$. Now we find two disjoint smooth manifolds
$\K_j\subset\bd\dom$ with smooth boundary such that
$\bd\dom=\ol\K_1\cup\ol\K_2$ (since $\bd\dom$ is locally the graph of a smooth
function, we can intersect $\bd\dom$ transversally with 
a small circular cylinder to get the desired decomposition with a 'big' and a
'small' manifold).
By the classical Stoke's theorem and a coherent orientation of the boundaries
$\bd\K_j$ by some tangent field $t_j$ we have
\begin{equation*}
  \I{\K_j}{\curl{(\tf G)}\cdot\nu^\dom}{\cH^2} = 
  \I{\bd\K_j}{\tf G\cdot t_j}{\cH^1}
\end{equation*}
for all $\tf\in C^1(\edom)$. From 
\begin{equation*}
  \curl (\tf G) = \tf\curl G + D\tf\times G 
\end{equation*}
we obtain some kind of integration by parts on the manifolds $\K_j$ 
\begin{equation*}
  \I{\K_j}{\big(\tf\curl G\big)\cdot\nu^\dom}{\cH^2} + 
  \I{\K_j}{\big(D\tf\times G\big)\cdot\nu^\dom}{\cH^2} =
  \I{\bd\K_j}{\tf G\cdot t_j}{\cH^1} \,.
\end{equation*}
If we sum up the integrals on $\K_1$ and $\K_2$, the boundary terms cancel out
by $t_1=-t_2$. Moreover, 
\begin{equation*}
  \big(D\tf\times G\big)\cdot\nu^\dom = \big(G\times\nu^\dom\big)\cdot D\tf \,.
\end{equation*}
Thus, \reff{dt-ex6-1} becomes
\begin{equation*}
  \divv{(\tf F)}(\dom) = \I{\dom}{F\cdot D\tf}{\lem} = 
  \I{\bd\dom}{\big(G\times\nu^\dom\big)\cdot D\tf}{\cH^2} 
\end{equation*}
for all $\tf\in C^1(\edom)$. For $\tf\in\woinf{\edom}$ 
we consider $\me_{\bd\dom}$ from Proposition~\ref{tt-s3} with $\alse=\dom$ 
and $\gamma(\delta)=\delta$. By arguments as in  
Example~\ref{dt-ex5} we can apply Corollary~\ref{tt-s4} to get
\begin{equation*}
  \divv{(\tf F)}(\dom) = \I{\dom}{F\cdot D\tf}{\lem} = 
  \sI{\bd\dom}{\big(G\times\nu^\dom\big)\cdot D\tf}{\mu_{\bd\dom}} 
\end{equation*}
for all $\tf\in\woinf{\edom}$.
\end{example}

In contrast to Example~\ref{dt-ex2} we now consider a constant vector field
$F$ and a similar open set 
$\dom\subset\R^2$ having infinite perimeter but such that 
\reff{dm-monster} is satisfied. This tells us that the choice of the measures 
$\lF$, $\mF$ independent of $\delta>0$ is not restricted to $\dom$ with
finite perimeter. 

\begin{example}  \label{dt-ex7}
In $\R^2$ we consider the open set
\begin{equation*}
  \edom=\dom:=\bigcup_{k=1}^\infty R_k \qmq{with}
  R_k:=\big(\tfrac1k,\tfrac1k+\tfrac{1}{k^2}\big) \times 
  \big(0,\tfrac{1}{k}\big) 
\end{equation*}
and the constant vector field
\begin{equation*}
  F\in\cL^1(\edom) \qmq{with} 
  F=\big(\begin{smallmatrix} 1 \\ 0 \end{smallmatrix}\big) \,.
\end{equation*}
Since $\frac1k+\frac{1}{k^2}<\frac{1}{k-1}$, the $R_k$ are pairwise disjoint.
Clearly $F\in\cD\cM^1(\edom)$ and $\dom$ has infinite perimeter. 
In order to verify \reff{eq:monster_ub} we first study the supremum in
\reff{eq:monster_ub} over
\begin{equation*}
  M_\delta := \big\{\tf\in\woinf{\edom} \:\big|\: 
  \|\tf_{|(\bd\dom)_\delta\cap\dom}\|_{\cW^{1,\infty}}\le 1 \}
\end{equation*}
merely on some fixed $R_k$ for some given $\delta>0$. The coarea
formula gives that
\begin{equation} \label{dt-ex7-1}
  \I{(\bd R_k)_\delta\cap R_k}{\tf\F D\port{\bd\dom}_\delta}{\cL^2} =
  \int_{\frac{\delta}{2}}^\delta 
  \I{\bd((R_k)_{-\tau})}{\tf\F D\port{\bd\dom}_\delta}{\cH^1}\, d\tau \,
\end{equation}
where the inner integral on the right hand side vanishes if $(R_k)_{-\tau}$ is
empty. For ``relevant'' $\tau$ that inner integral  
has to be taken on the boundary of a (translated) rectangle of the form
\begin{equation} \label{dt-ex7-2}
  R:= (-a,a)\times(0,b) \qmq{with} 0<a<\tfrac{1}{2k^2}\,,\;
  0<b<\tfrac1k\,.
\end{equation}
The integral vanishes on the short sides of $R$ by
$F D\port{\bd\dom}_\delta=0$ and one has
\begin{equation*}
  \F D\port{\bd\dom}_\delta=\pm\tfrac2\delta \qmq{on} 
  \K_\pm := \{\pm a\}\times(0,b) \,. 
\end{equation*}
For an estimate of the supremum in \reff{eq:monster_ub} we consider 
$\cW^{1,\infty}$-functions $\tf$ on a small neighborhood of $\bd R$
with $\|\tf\|_\infty\le 1$ and $\|D\tf\|_\infty\le 1$. 
Using $\tf_\pm(y):=\tf(\pm a,y)$ we want to maximize 
\begin{equation} \label{dt-ex7-3}
  \int_0^b\tf_+-\tf_-\,dy 
\end{equation}
for such $\tf$ (cf. Figure~\ref{fig:monster}).
\begin{figure}[h!]          
\centering
  \includegraphics[width=3.5cm]{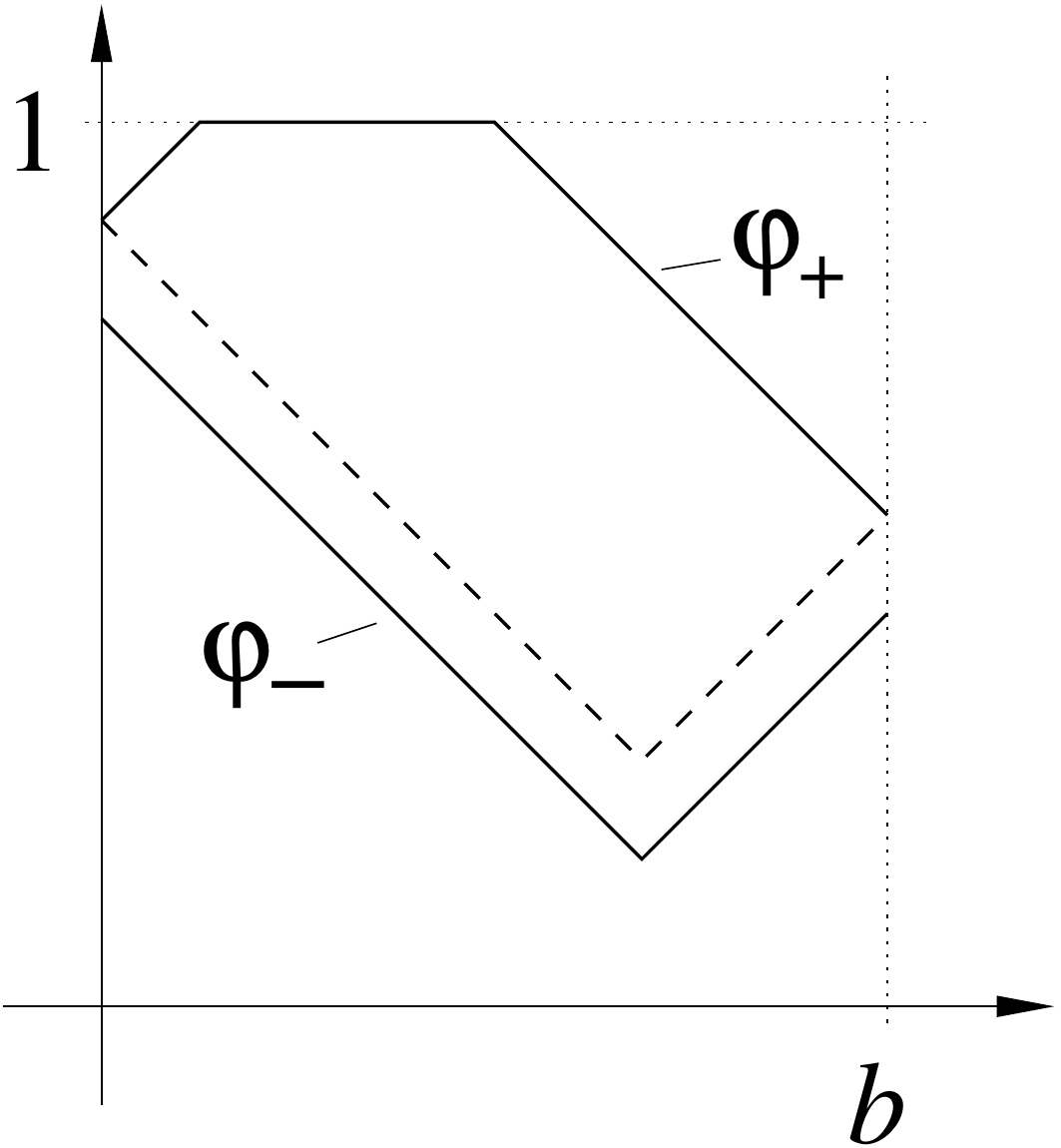} 
  \caption{The figure shows the graph of $\tf_+$ and $\tf_-$ where the dashed 
   graph is a translation of $\tf_-$.}  
  \label{fig:monster}
\end{figure}
Thus we have to look for $\tf$ with $\tf_+\ge\tf_-$ such that the 
area between the graphs of $\tf_\pm$ becomes maximal. 
Since an additive constant for $\tf$ doesn't change the integral, we can
assume that the maximum of $\tf_+$ on $[0,b]$ equals~$1$.  
Hence, by $\|D\tf\|_\infty\le 1$ and the size of $R$, 
we always have $\tf\ge-1$, i.e. we do not
have to take care explicitely for that constraint.
Let us now briefly assume that $\tf_+(0)$, $\tf_+(b)$ are fixed. 
For such $\tf_+$ we denote the smallest $y\in[0,b]$ with $\tf_+(y)=1$ 
by $y_1$ and the largest by $y_2$. In order to maximize the integral 
for such $\tf_+$, we use 
\begin{equation*}
  \tf_+(y)=\tf_+(0) + \int_0^y\tf_+'(\tau)\, d\tau
\end{equation*}
to see that $\tf_+'$ has to equal $1$ on $[0,y_1]$. Analogously  
$\tf_+'$ has to equal $-1$ on $[y_2,b]$ and, clearly, $\tf_+=1$ on
$[y_1,y_2]$. The same way we get that the optimal 
$\tf_-$ with fixed values on the boundary first has to decay with slope
$-1$ and then it grows with slope $1$, where we have used that 
always $\tf_-\ge-1$ (cf. Figure~\ref{fig:monster}). 
If $\tf_+$ equals $1$ on a nontrivial interval,
then we can enlarge the integral in \reff{dt-ex7-3}
by a translation of $\tf$ such 
that $\tf_+=1$ merely at a single point $\tilde y\in[0,b]$.
Moreover, for a maximal value in \reff{dt-ex7-3} 
the $\tf$ should have maximal slope on the short sides of $R$ such that 
\begin{equation*}
  \tf_-(0)=\tf_+(0)-2a\,, \quad \tf_-(b)=\tf_+(b)-2a\,.
\end{equation*}
Now it is sufficient to look for an optimal $\tf$ in this subclass of 
admissible $\tf$. Clearly the ``shape'' of the optimal $\tf_\pm$ is not
influenced if we add $2a$ to $\tf_-$ for a moment (cf. dashed graph in
Figure~\ref{fig:monster}) such that
the graphs of $\tf_\pm$ form a rectangle with sides having 
length $\sqrt{2}\tilde y$ and $\sqrt{2}(b-\tilde y)$.
Hence the sum is independent of $\tilde y$ and always equals $\sqrt{2}b$.
Therefore the area of the rectangle between the graphs becomes maximal if it
is a square. Consequently there is some $\tf$ that maximizes  
\reff{dt-ex7-3} under the considered constraint and it satisfies
(with the actual not shifted graph of $\tf_-$) 
\begin{equation*}
  \tf_+\big(\tfrac{b}{2}\big)=1\,, \z 
  \tf_+(0)=\tf_+(b)=1-\tfrac{b}{2}\,, \z
\end{equation*}
\begin{equation*}
    \tf_-(0)=\tf_-(b)=1-\tfrac{b}{2}-2a\,, \z
  \tf_-\big(\tfrac{b}{2}\big)=1-b-2a\,.
\end{equation*}
Therefore
\begin{equation*}
  \int_0^b\tf_+-\tf_-\,dy = \Big(\frac{\sqrt{2}b}{2}\Big)^2 + 2ab = 
  \frac{b^2}{2}+2ab\,
\end{equation*}
for the maximal $\tf$. Since $(R_k)_{-\tau}\subset R_k$
is a nonempty rectangle for $\tau<\frac{1}{2k^2}$, we can use the bounds from
\reff{dt-ex7-2} to get 
\begin{equation*}
  \sup_{\tf\in M_\delta} \:
  \I{\bd((R_k)_{-\tau})}{\tf\F D\port{\bd\dom}_\delta}{\cH^1} \le
  \sup_{\tf\in M_\delta} \I{\bd R_k}{\tf\F D\port{\bd\dom}_\delta}{\cH^1} =
  \frac{1}{\delta k^2} + \frac{2}{\delta k^3} \,
\end{equation*}
for $\tau\in\big[\frac\delta2,\delta\big]\cap\big(0,\frac{1}{2k^2}\big)$. 
Since the left hand side vanishes if $(R_k)_{-\tau}=\emptyset$, the estimate
is even true for a.e. $\tau\in\big[\frac\delta2,\delta\big]$. 
Though a 
``good'' $\tf\in M_\delta$ for the supremum in \reff{eq:monster_ub}
might not be optimal for each $\bd((R_k)_{-\tau})$, from \reff{dt-ex7-1}
we get at least the estimate  
\begin{equation*}
  \sup_{\tf\in M_\delta}
  \I{(\bd R_k)_\delta\cap R_k}{\tf\F D\port{\bd\dom}_\delta}{\lem} \le
  \frac{1}{2k^2} + \frac{1}{k^3} \qmq{for all} k\in\N\,,\: \delta>0\,.
\end{equation*}
Hence
\begin{equation*}
  \sup_{\tf\in M_\delta}
  \I{(\bd\dom)_\delta\cap\dom}{\tf\F D\port{\bd\dom}_\delta}{\lem} \le
  \sum_{k=1}^\infty\Big(\frac{1}{2k^2} + \frac{1}{k^3}\Big) < \infty
  \qmq{for all} \delta>0\,.
\end{equation*}
But this implies \reff{eq:monster_ub}. Thus,
according to Proposition~\ref{dt-s2} and Lemma~\ref{lem:monster2}, 
we can choose $\lF$, $\mF$ in \reff{dt-s1-1} independent of 
$\delta$ with core in $\bd\dom$.

Assume that $\mF=0$. Then we can consider functions $\tf\in\woinf{\dom}$ 
that vanish outside some fixed $R_k$. These $\tf$ are continuously extendable
onto $\ol{R_k}$. From the usual 
Gauss-Green formula on $R_k$ we obtain that the restriction of $\lF$ on $R_k$
is related to the Radon measure 
\begin{equation*}
  \sigma_k=\reme{F\nu^{R_k}}{\bd R_k} \qmq{where}
  |\lF|(R_k)=|\sigma_k|(\bd R_k)=\tfrac{2}{k} 
\end{equation*}
(cf. Proposition~\ref{acm-ram}). But this is a contradiction, since 
$|\lF|(\dom)$ wouldn't be finite. Therefore $\mF\ne 0$. Let us  
briefly sketch how the measures $\lF$ and $\mF$
can be chosen. We fix some $\tilde k\in\N$ and set
\begin{equation*}
  \dom_1:=\bigcup_{k=1}^{\tilde k} R_k\,, \quad
  \dom_2:=\bigcup_{k=\tilde k+1}^\infty R_k\,.
\end{equation*}
Obviously $\dom_1$ is an open set with Lipschitz boundary and we can choose 
$\lF$ corresponding to the Radon measure $\sigma_F$ on $\bd\dom_1$ 
from the boundary integral in the usual Gauss-Green formula on $\dom_1$
(cf. also Example~\ref{dt-ex2}).  
For $\mF$ we first argue similar to Example~\ref{dt-ex5} on some fixed $R_k$
to get for $\tf\in C^1(\ol{R_k})$
\begin{eqnarray*}
  \I{\bd R_k}{\tf F\cdot\nu^{R_k}}{\cH^1}
&=&
 \int_0^{\frac1k} \tf(\tfrac1k+\tfrac{1}{k^2},y)-\tf(\tfrac1k,y)\, dy \\
&=&
  \int_0^{\frac1k} 
  y\big(\tf_y(\tfrac1k,y)-\tf_y(\tfrac1k+\tfrac{1}{k^2},y)\big) \, dy \\
&&
  + \, \tfrac1k 
  \big(\tf(\tfrac1k+\tfrac{1}{k^2},\tfrac1k)-\tf(\tfrac1k,\tfrac1k)  \big) \\
&=&
  \int_0^{\frac1k} 
  y\big(\tf_y(\tfrac1k,y)-\tf_y(\tfrac1k+\tfrac{1}{k^2},y)\big) \, dy \\
&&
 + \, \tfrac1k \int_{\frac1k}^{\frac1k+\frac{1}{k^2}}
 \tf_x(x,\tfrac1k) \, dx \\
&=&
  \I{\bd R_k}{G D\tf}{\cH^1} 
\end{eqnarray*}
with a suitable vector field $G$ on $\ol{R_k}$. The measure
\begin{equation*}
  \reme{G\cH^1}{\bd R_k} \qmq{has total variation}
  \tfrac{1}{k^2}+\tfrac{1}{k^3}
\end{equation*}
and we can construct $\mF$ on $R_k$ similar to Example~\ref{dt-ex5}. Summing
up over $k>\tilde k$ we finally get $\mF$ on $\dom_2$.
\end{example}

Let us come back to Example~\ref{dt-ex1}. We construct an uncountable
family of measures $(\lF^\rho,\mF^\rho)$ satisfying
\reff{dt-s1-1} where even $\lF=0$ is possible.

\begin{example}\label{ex:rotsing}   
Let $\edom=B_2(0)$ as in Example~\ref{dt-ex1}, but 
for technical simplicity we now consider the cone
\begin{equation*}
  \dom := (0,\infty)^2 \cap B_1(0)\subset \R^2 \,.
\end{equation*}
For the vector field 
\begin{equation*}
  F(x,y) := \frac{1}{2\pi(x^2+y^2)}
    \begin{pmatrix} -y \\ x \end{pmatrix}\,
\end{equation*}
we already now that $F\in\dmo{\edom}$ with $\divv F=0$ and that $\mF\ne 0$ is
impossible in \reff{dt-s1-1}. Moreover we have \reff{dm-monster} and case (C)
for small $\delta>0$. For any $\rho\in(0,1]$ we now construct 
Radon measures $\sigma^\rho_F$ supported on $\bd\dom$ and $\mF^\rho$ 
in $\op{ba}(\edom,\cB(\edom),\cL^2)^2$ with
core in $\bd\dom$ such that
\begin{equation*}\label{ex:rotsing-1}
  \divv(\tf F)(\dom) = 
  \I{\bd\dom}{\tf}{\sigma_F^\rho} + \sI{\bd\dom}{D\tf}{\mF^\rho} 
  \qmq{for all} \tf\in\woinf{\edom}\,.
\end{equation*}
Notice that the measures $\mF^\rho$ are pure by Proposition~\ref{acm-s1}.

For fixed $\rho\in (0,1]$ we now set
\begin{equation*}
  \dom^\delta := \big(\dom\setminus B_\delta(0)\big)\cap B_\rho(0), \quad
  \tilde\dom := \dom\setminus \ol{B_\rho(0)} \qmq{for} \delta\in(0,\rho)\,.
\end{equation*}
Obviously $F$ is tangential to circles around the origin. Then, with the
classical Gauss-Green theorem, we get for all $\tf\in\woinf{\edom}$ 
\begin{eqnarray*}
  \divv(\tf F)(\dom^\delta\cup\tilde\dom) 
&=&
   \I{\bd\dom^\delta}{\tf F\nu^{\dom^\delta}}{\ham^1} +
  \I{\bd{\tilde\dom}}{\tf F\nu^{\tilde\dom}}{\ham^1} \\
&=&
  I_\delta + \I{\bd{\tilde\dom}}{\tf}{\tilde\sigma_F}
  \;\overset{\delta\to 0}{\longrightarrow}\; \divv(\tf F)(\dom)
\end{eqnarray*}
with the Radon measure
\begin{eqnarray}
  \tilde\sigma_F
&=&
  -\frac{1}{2\pi x}\cH^1\lfloor\big([\rho,1]\times\{0\}\big) 
  + \frac{1}{2\pi y}\cH^1\lfloor\big(\{0\}\times[\rho,1]\big) 
  \quad\qmq{and}  \nonumber\\
  I_\delta 
&=&\label{ex:rotsing-2}
  \frac{1}{2\pi} \bigg(
  \int_\delta^\rho \frac{\tf(0,0)-\tf(x,0)}{x} \: dx +
  \int_\delta^\rho \frac{\tf(0,y)-\tf(0,0)}{y} \: dy \bigg)  \,
\end{eqnarray}
where $\tf(0,0)$ is included additionally. We use the Lipschitz continuous
function 
\begin{equation*}
  \psi(x,y) := \tf(0,y)-\tf(x,y)  \qmq{satisfying}
  |\psi(x,y)|\le x\, \|D\tf\|_\infty  \,
\end{equation*}
to treat the first integral in $I_\delta$. Though the limit of the integral
for $\delta\to 0$ exists by dominated convergence, the limit of the measures
$\tfrac{1}{x}\cL^1\lfloor\,[\delta,\rho]$ for $\delta\to 0$ does not give a
finite measure. Hence the limit of $I_\delta$ cannot contribute to $\sigma_F$.
We therefore use integration by parts to get a 
contribution to $\mF$. The difficulty that $D\tf$ might not exist on
$\bd\dom^\delta$ is circumvented by fattening up the boundary.   
Clearly, $\psi(\cdot,y)$ is absolutely continuous for all $y$.
Since $D\tf$ exists $\cL^2$-a.e., the fundamental theorem of calculus 
implies 
\begin{equation}\label{ex:rotsing-5}
  \psi(x,y)\ln x\Big|_{x=\delta}^\rho =
  \int_\delta^\rho \Big( \psi_x(x,y)\ln x + \psi(x,y)\frac{1}{x} \Big) \,dx
  \quad\qmz{for a.e.} y
\end{equation}
(cf. \cite[p.~1019]{zeidler_IIB}, \cite[p.~164, 235]{evans}).
Notice that
\begin{equation*}
  |\psi(x,y)\ln x| \le x|\ln x|\,\|D\tf\|_\infty
  \overset{x\to 0}{\longrightarrow} 0\,, \quad
  \Big|\frac{\psi(x,y)}{x}\Big| \le \|D\tf\|_\infty\,.
\end{equation*}
Using dominated convergence and $\psi_x=-\tf_x$, 
we can take the limit $\delta\to 0$ in \reff{ex:rotsing-5}. Then 
we integrate over $[0,\tau]$ with respect to $y$ and divide everything 
by $\tau$ to get
\begin{equation}\label{ex:rotsing-6}
  \frac{1}{\tau}\int_0^\tau\psi(\rho,y)\ln\rho \, dy =
  \I{[0,\rho]\times[0,\tau]}{D\tf}{\me'_{F,\tau}} +
  \frac{1}{\tau}\int_0^\tau\int_0^\rho 
  \psi(x,y)\frac{1}{x} \,dx dy \,
\end{equation}
with the vector measures
\begin{equation*}
  \me_{F,\tau}' = - \tfrac{1}{\tau} 
  {\scriptsize \begin{pmatrix}\ln x \\ 0 \end{pmatrix} }\cL^2
  \lfloor \big([0,\rho]\times[0,\tau]\big) \in
  \op{ba}(\edom,\cB(\edom),\cL^2)^2\,.
\end{equation*}
 Obviously 
\begin{equation*}
  |\me'_{F,\tau}|(\edom) \le \int_0^1|\ln x| \: dx = 1  
  \qmz{for all} \tau>0\,, \quad
  \cor{\me'_{F,\tau}}\subset [0,\rho]\times[0,\tau]\,.
\end{equation*}
Hence, by the Alaoglu theorem, the measures
$\big\{\me'_{F,\frac{1}{k}}\big\}_k$ have a weak$^*$ cluster point 
\begin{equation*}
  \mF' \in \op{ba}(\edom,\cB(\edom),\cL^2) \qmq{with}
  \cor{\mF'}\subset\bd\dom\cap B_\rho(0) \,.
\end{equation*}
In \reff{ex:rotsing-6} the limit for $\tau\to 0$ exists for the most left and
the most right integral by continuity of $\psi$ (notice that 
$\frac{\psi(x,y)}{x}$ is bounded and continuous on 
$(0,\rho]\times[0,\tilde\tau]$ for some $\tilde\tau>0$ and that 
$\frac{1}{\tau}\int_0^\tau\frac{\psi(x,y)}{x}\,dy\to\frac{\psi(x,0)}{x}$ as
$\tau\to 0$ for all $x\in(0,\rho)$).
Hence we can take the limit $\tau\to 0$ in \reff{ex:rotsing-6}
and, with the definition of $\psi$, we obtain  
\begin{equation}\label{ex:rotsing-7}
  \big(\tf(0,0)-\tf(\rho,0)\big) \ln\rho \:=\:
   \sI{[0,\rho]\times\{0\}}{\!D\tf}{\mF'} +
   \int_0^\rho  \frac{\tf(0,0)-\tf(x,0)}{x} \,dx \,
\end{equation}
for all $\tf\in\woinf{\edom}$.
Analogous arguments for the second integral in \reff{ex:rotsing-2}
give a weak$^*$ cluster point $\mF''$ of the vector measures 
\begin{equation*}
  \me''_{F,\tau} = \tfrac{1}{\tau} 
  {\scriptsize \begin{pmatrix} 0 \\ \ln y \end{pmatrix} }\cL^2
  \lfloor \big([0,\tau]\times [0,\rho]\big) \in
  \op{ba}(\edom,\cB(\edom),\cL^2)^2\,
\end{equation*}
with $\cor{\mF''}\subset\bd\dom\cap B_\rho(0)$. Now we use
\reff{ex:rotsing-7} and the analogous equation with $\mF''$
to replace the two integrals in \reff{ex:rotsing-2}.
Then, with
\begin{equation*}
  \mF^\rho:=-\tfrac{1}{2\pi}(\mF'+\mF'')\,, \quad
  \sigma_F^\rho := \tilde\sigma_F + 
  \tfrac{\ln\rho}{2\pi}\big(\delta_{(0,\rho)} - \delta_{(\rho,0)}\big) \,,
\end{equation*}
we finally obtain
\begin{eqnarray}
  \divv(\tf F)(\dom) 
&=&
  \lim_{\delta\to 0}I_\delta + \I{\bd{\tilde\dom}}{\tf}{\tilde\sigma_F} 
  \nonumber\\
&=&  \label{ex:rotsing-8}
  \I{\bd\dom}{\tf}{\sigma_F^\rho} + 
  \sI{\bd\dom\cap B_\rho(0)}{D\tf}{\mF^\rho} \,
\end{eqnarray}
for all $\tf\in\woinf{\edom}$ and all $\rho\in(0,1]$ where
\begin{equation*}
  \I{\bd\dom}{\tf}{\sigma_F^\rho}
  = \I{\bd\dom\setminus B_\rho(0)}{\tf F\nu^\dom}{\cH^1}
  + \frac{\ln\rho}{2\pi}\big( \tf(0,\rho)-\tf(\rho,0) \big) \,.
\end{equation*}
Notice that this covers the special case $\sigma_F^1=0$ for $\rho=1$.
For $\tf\in C^1(\edom)$ we can argue by continuity to get 
\begin{equation*}
  \sI{\bd\dom\cap B_\rho(0)}{D\tf}{\mF^\rho} =
  \frac{1}{2\pi} 
  \Big( \int_0^\rho \tf_x(x,0)\ln x\, dx - 
        \int_0^\rho \tf_y(0,y)\ln y\, dy \Big) \,.
\end{equation*}
Moreover, from \reff{ex:rotsing-8} with $\tf=1$ on $\edom$, we get the
Gaussian formula 
\begin{eqnarray*}
  0 = \I{\dom}{\divv F}{\lem} = 
  \I{\bd\dom\setminus B_\rho(0)}{\tf F\nu^\dom}{\cH^1} 
\end{eqnarray*}
for any $\rho\in(0,1]$. For $\rho=1$ this includes the exotic special case
\begin{equation*}
  0 = \I{\dom}{\divv F}{\lem} = 
  \I{\emptyset}{\tf F\nu^\dom}{\cH^1} \,.
\end{equation*}
Let us finally provide an explicit computation for \reff{ex:rotsing-8} 
with the simple function $\tf(x,y)=x$ as demonstration. 
Using polar coordinates for the first volume integral we readily get
\begin{eqnarray*}
  -\frac{1}{2\pi} 
&=& \I{\dom}{FD\tf}{\cL^2} \: = \: \divv(\tf F)(\dom) \\
&=&
  - \int_\rho^1 \frac{\tf(x,0)}{2\pi x}\, dx - \frac{\ln\rho}{2\pi}\tf(\rho,0)
  + \int_0^\rho \frac{\tf_x(x,0)\ln x}{2\pi}\, dx  \\
&=&
  - \frac{1-\rho}{2\pi} - \frac{\rho\ln\rho}{2\pi} + 
    \frac{\rho\ln\rho-\rho}{2\pi}
\: = \:
  -\frac{1}{2\pi} \,
\end{eqnarray*}
for all $\rho\in(0,1)$. 
\end{example}

\medskip

\begin{proof+}{ of Proposition~\ref{dt-s5}}   
For (1) let $\tf\in\woinf{\edom}$ and $\delta\in(0,\tilde\delta)$
such that \reff{dt-s5-0} is satisfied. Hence $\tf_{|\bd\dom}$ is Lipschitz
continuous on $\bd\dom$ and there is a Lipschitz continuous extension 
 $\til{\phi}$ onto $\R^n$ with 
\begin{equation*}
  \|\tilde\tf\|_{\op{Lip}(\edom)} \le \|\tf_{|\bd\dom}\|_{\op{Lip}(\bd\dom)} 
\end{equation*}
(cf. \cite[p.~452]{silhavy_divergence_2009}).
By $\tilde\tf=\tf$ on $\bd\dom$, Corollary~\ref{ex:dmo_woi_c} implies
\begin{equation*}
  \divv (\phi F) (\dom) = \divv (\tilde\phi F) (\dom) \,.
\end{equation*}
Therefore, by \reff{ex:dmo_woi-1} and \reff{prop:dual0-1}
we get
\begin{eqnarray*}
  \divv{(\port{\bd\dom}_\delta \phi\funv)}(\dom) 
& = &
  \divv (\phi F) (\dom) = \divv (\tilde\phi F) (\dom)    \\
& \le &
  \|F\|_{\cD\cM^1(\edom)} \|\tilde\tf\|_{W^{1,\infty}(\edom)} \\
& \le &
  \|F\|_{\cD\cM^1(\edom)} \|\tilde\tf\|_{\op{Lip}(\edom)} \le 
  \|F\|_{\cD\cM^1(\edom)} \|\tf_{|\bd\dom}\|_{\op{Lip}(\bd\dom)}   \\
&\le&
  c \|F\|_{\cD\cM^1(\edom)} \|\tf\|_{\woinf{(\bd\dom)_\delta\cap\edom}} \,.
\end{eqnarray*}
But this readily implies \reff{dm-monster}.

For (2) we verify the assumptions of (1). Obviously 
$\lem(\dom\setminus\op{int}\dom)=0$ and all 
$\tf\in\woinf{\edom}$ have a continuous extension onto~$\cl\dom$. 
By definition of Lipschitz boundary we
can cover $\bd\dom$ by finitely many open cylinders $C_j$ with $j=1,\dots,m$
such that for Lipschitz continuous functions $\gamma_j:\R^{n-1}\to\R$
and suitable $r_j,h_j>0$ up to translation and rotation
\begin{equation*}
  C_j=\{(x',t)\mid |x'|<r_j,\: |t|<h_j \}\,, \quad
  |\gamma_j(x')|<\tfrac{h_j}{3}\,,
\end{equation*}
\begin{equation*}
  \dom\cap C_j = \{(x',t)\mid |x'|<r_j,\: \gamma_j(x')<t \} \quad
\end{equation*}
where $(x',t)\in\R^{n-1}\times\R$ 
(cf. \cite[p.~127, 177]{evans}). By compactness of $\bd\dom$ 
we can find some 
$\rho>0$ such that $\rho<\tfrac{h_j}{3}$ for all $j$ and such
that for all $x\in\bd\dom$ there is $j_x$ with   
\begin{equation*}
  B_\rho(x)\subset C_{j_x}\,, \quad j_x\in\{1,\dots,m\} \,.
\end{equation*}
Moreover we can assume that $|x-y|>\rho$ if $x,y$ belong to different
components of~$\ol\dom$. 
Since all $C_j$ are bounded, there is some $c_0>0$ such that
\begin{equation} \label{dt-s5-1}
  \op{diam}(C_j)< c_0\rho \qmz{for all} j
\end{equation}
where $\op{diam}(C_j)$ denotes the diameter. 

Let us now fix $\delta\in(0,\rho)$ and $\tf\in\woinf{\edom}$. Then 
\begin{equation*}
  L=\|\tf\|_{\woinf{(\bd\dom)_\delta\cap\edom}}
\end{equation*}
is a bound for $\tf$ on $(\bd\dom)_\delta\cap\edom$ and it is 
a Lipschitz constant of $\tf$ on segments 
$[x,y]\subset(\bd\dom)_\delta\cap\edom$.
First we assume that $x,y\in\bd\dom$ such that $y\in B_\rho(x)$. Then 
\begin{equation*}
  [x,y] \qmq{belongs to}  C_{j_x} \,.
\end{equation*}
We can move the points of the segment $[x,y]$ parallel to the axis
of the cylinder $C_{j_x}$ to get a polygonal curve 
$P\subset (\bd\dom)_\delta\cap\ol\dom$ connecting finitely many points 
\begin{equation*}
  x=x^0,x^1,\dots,x^{k}=y\in (\bd\dom)_\delta\cap\ol\dom 
\end{equation*}
(cf. Figure~\ref{fig:lipbd}).
\begin{figure}[h!]          
\centering
  \includegraphics[width=3.5cm]{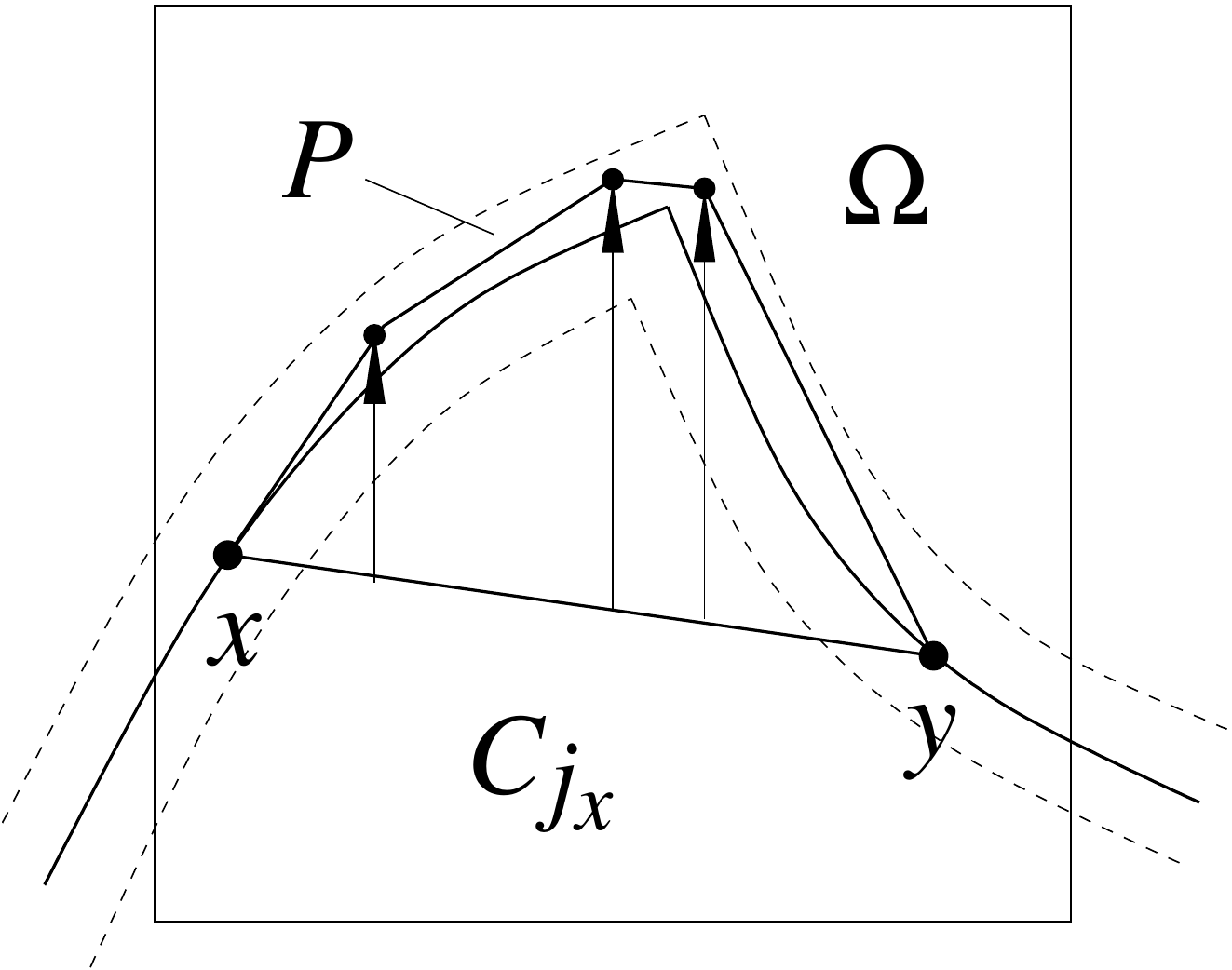} 
  \caption{The figure shows the boundary $\bd\dom$ (solid curve), its
    neighborhood $(\bd\dom)_\delta$ (dashed curves), the segment $[x,y]$ in 
    $C_{j_x}$, and the polygonal curve $P$ in $\ol\dom$.}  
  \label{fig:lipbd}
\end{figure}
There is a constant $c_1>0$ depending merely on the largest Lipschitz constant
of the $\gamma_j$ such that
\begin{equation*}
  \op{length}(P) \le c_1|x-y| \,.
\end{equation*}
Hence
\begin{equation*}
  |\tf(x)-\tf(y)| \le
  \sum_{j=1}^{k} |\tf(x^{j-1})-\tf(x^j)| \le 
  L \sum_{j=1}^{k} |x^{j-1}-x^j|  \le  c_1L|x-y| \,.
\end{equation*}
Let now $x,y\in\bd\dom$ be such that $y\not\in B_\rho(x)$ but that $x,y$ 
belong to the same component of $\ol\dom$. By the convexity of the
$C_j$ there are at most $m$ points 
$x_0=x,x_1,\dots,x_l=y\in\bd\dom$ such that the closed segments 
\begin{equation*}
   [x_{i-1},x_i] \qmq{belong to some}  C_j \,.
\end{equation*}
As above we can construct polygonal curves 
$P_i\subset (\bd\dom)_\delta\cap\ol\dom$ connecting $x_{i-1}$ and $x_i$ within 
$(\bd\dom)_\delta\cap\ol\dom$ to get
\begin{equation*}
  |\tf(x_{i-1})-\tf(x_i)| \le
  c_1L|x_{i-1}-x_i| \,.
\end{equation*}
Using that $y\not\in B_\rho(x)$ we get
\begin{eqnarray*}
  |\tf(x)-\tf(y)| 
& \le &
  \sum_{i=1}^l |\tf(x_{i-1})-\tf(x_i)| 
\; \overset{l\le m}{\le} \;
  mc_1L|x_{i-1}-x_i| \\
&\overset{\reff{dt-s5-1}}{\le}&
  mc_1Lc_0\rho \; \le \; c_0c_1mL|x-y| \,.
\end{eqnarray*}
Finally let $x,y\in\bd\dom$ belong to different components of $\ol\dom$. Then
$|x-y|>\rho$ and
\begin{equation*}
  |\tf(x)-\tf(y)| \le |\tf(x)| + |\tf(y)| \le 2L \le 
  \tfrac{2L}{\rho}|x-y| \,.
\end{equation*}
Summarizing all cases we use $c=1+\max\{1,\frac{2}{\rho},c_0c_1m\}$ to get
\begin{equation*}
  \|\tf_{|\bd\dom}\|_{\op{Lip}(\bd\dom)} =  
  \|\tf_{|\bd\dom}\|_{C(\bd\dom)} + \op{Lip}(\tf_{|\bd\dom})   \le cL =
  c \|\tf\|_{\woinf{(\bd\dom)_\delta\cap\edom}} \,.
\end{equation*}
Since $c$ is independent of $\delta\in(0,\rho)$ and $\tf\in\woinf{\edom}$,
the assumption of (1) is verified. 
\end{proof+}

\subsection{Normal measures}
\label{nm}

The linearity of the trace operator $T$ from Proposition
\ref{ex:dmo_woi} hints at a linear dependence of the measures $\lF$
and $\mF$ on $\F$. Moreover, the usual Gauss-Green formula
for smooth $F$ and regular $\dom$ given by
\begin{equation*}
  \divv(\tf F)(\dom) = \I{\bd\dom}{\tf F\cdot\nu^\dom}{\cH^{n-1}}
\end{equation*}
indicates some dependence of the boundary term on the outer unit normal field
$\nu^\dom$. Therefore we are interested in more structural information 
about $\lF$ and $\mF$. Here it turns out that the usage of a pointwise normal
field $\nu^\dom$ on the boundary $\bd\dom$ is too restrictive even if it
exists (in particular if $\divv F$ has concentrations there) and
that a pointwise trace function $F$ on the boundary might not exist. 
Therefore we are looking for
measures $\nu$ that extend the notion of pointwise normal fields. 
As a first idea we could consider some extension $\nu\in \bawl{\dom}^n$ of the
Radon measure $\nu^\dom\cH^{n-1}\lfloor\mbd{\dom}$ according to 
Proposition~\ref{acm-ram} for suitable sets of finite perimeter.
However, by the variety of extensions, 
we wouldn't get enough information about $\nu$ for a general Gauss-Green
formula. If $\divv F$ has concentrations on $\bd\dom$, we have e.g. to take
care for parts of the boundary belonging to $\dom$ by controlling
the aura of a corresponding $\nu$. Thus 
we need a more careful construction of such measures than
just any extension. Therefore we provide a general approach that even allows
some weight on 
$\bd\dom$. This way we finally obtain a more precise representation
of the boundary term in the general Gauss-Green formula \reff{dt-s1-1}
for a large class of cases.

For $\edom\subset\R^n$ open, bounded and $\dom\in\bor{\edom}$,
a measurable function $\chi: \edom \to [0,1]$ is said to be a 
\textit{good approximation} for $\chi_\dom$ if there are  
$\chi_k\in W^{1,\infty}(\edom)$ such that
\bgl
\item $\chi_k:\edom\to[0,1]$ is compactly supported on $\edom$ for all 
$k\in\N$,

\item \; $\lim\limits_{k \to \infty} \chi_k = \chi$ \z $\ham^{n-1}$-a.e. on
$\edom$, 

\item  $\chi_k = 1$ on $\dom_{-\frac{1}{k}}$, \z
$\chi_k=0$ on $\edom\setminus\dom_{\frac{1}{k}}$ for all $k$,

\item 
$\liminf\limits_{k \to \infty}  \|D\chi_k\|_{\cL^1(\edom)} < \infty$\,.
\el

\noi
We call $\{\chi_k\}$ an \textit{approximating sequence} for $\chi$.
Obviously 
\begin{equation*}
  \chi = 1 \zmz{on} \Int{\dom} \qmq{and}
  \chi = 0 \zmz{on} \Ext{\dom} \,.
\end{equation*}
Notice that the $\chi_k$ cannot be taken as
equivalence classes in $\woinf{\edom}$ and that 
$\dom$ need not to be compactly contained in $\edom$ for a good approximation. 
It might be an option for the treatment of bounded vector fields $F$ to 
consider also $\chi_k$ that merely belong to $\cB\cV(\edom)$
(recall that $\chi_k F\in\cD\cM^\infty(\edom)$ in this case;
cf. \cite[p.~97]{chen-comi}). However there was no need for that extension in
the present treatment. Now we see that,
despite degenerate cases, $\dom$ allowing a good approximation have to have
finite perimeter in~$\edom$. 

\begin{proposition} \label{nm-s2}
Let $\edom\subset\R^n$ be open and bounded and let $\dom\in\bor{\edom}$.
If there is a good approximation $\chi$ for $\chi_\dom$  with
$\|\chi - \chi_\dom\|_{\cL^1(\edom)}=0$, then $\dom$ has finite perimeter 
in~$\edom$. 
\end{proposition}

\noi
Note that $\lem(\bd{\dom\cap\edom})=0$ implies 
$\|\chi-\chi_\dom\|_{\cL^1(\edom)}=0$.

\begin{proof} 
Let $\chi$ be a good approximation for $\chi_\dom$ with corresponding 
approximating sequence $\{\chi_k\}\subset W^{1,\infty}(\edom)$. Since $\dom$ is
bounded and every $\ham^{n-1}$-null set is $\lem$-null set, 
\begin{equation*}
  \chi_k \xrightarrow{L^1} \chi_\dom \,.
\end{equation*}
By the lower semicontinuity of the total variation of BV functions,
\begin{equation*}
  |D\chi_\dom|(\edom) \le \liminf_{k\to\infty} \|D\chi_k\|_{\cL^1} < \infty \,.
\end{equation*}
Hence $\chi_\dom\in BV(\edom)$ and the assertion follows. 
\end{proof}

Let us now demonstrate that good approximations provide measures.

\begin{theorem}\label{bd-s5}
Let $\edom\subset\R^n$ be open and bounded, let $\dom\in\bor{\edom}$, and let 
$\chi$ be a good approximation for $\chi_\dom$ with approximating sequence
$\{\chi_k\}$. Then there is an associated measure  
$\nu\in\bawl{U}^n$  
such that 
\begin{equation*}
  \cor{\nu}\subset\bd\dom\,, \quad 
  |\nu|(\edom)=\|\nu\| \le \liminf_{k \to \infty} \|D\chi_k\|_{\cL^1(\edom)} \,,
\end{equation*}
\begin{equation} \label{bd-s5-0a}
  |\nu|(B) \le \limsup_{k\to\infty}\|D\chi_k\|_{\cL^1(B)}
  \qmq{for all} B\in\bor{\edom}\,,
\end{equation}
\begin{equation} \label{bd-s5-0}
  A = \bigcup_{k\ge k_0} \supp(D\chi_k) 
\end{equation}
is an aura of $\nu$ for each $k_0\in\N$, 
and for any $\tf\in\cL^\infty(\edom,\R^n)$ there is a subsequence 
$\{\chi_{k'}\}$ with 
\begin{equation}\label{bd-s5-1}
  \lim_{k'\to\infty}\I{U}{\tf\cdot D\chi_{k'}}{\lem} = 
  -\,\sI{\bd\dom}{\tf}{\nu} \,.
\end{equation}
Moreover, $\nu$ is a trace on $\bd\dom$ over $\cL^\infty(\edom,\R^n)$.
If $\|\chi - \chi_\dom\|_{\cL^1(\edom)}=0$, then
\begin{equation} \label{bd-s5-3}
  \sI{\bd\dom}{\tf}{\nu} =
  \I{\mbd\dom\cap\edom}{\tf\cdot\nu^\dom}{\ham^{n-1}}
  \qmq{for all} \tf\in C_c(\edom,\R^n) \,
\end{equation}
and 
\begin{equation} \label{bd-s5-2}
  |\nu|(B) \ge \big(\reme{\ham^{n-1}}{(\rbd{\dom}\cap\edom)}\big)(B) 
  \qmz{for all open} B\subset\edom \,.
\end{equation}
If $\edom=\dom$ with Lipschitz boundary, then 
\begin{equation} \label{bd-s5-3a}
  \sI{\bd\dom}{\tf}{\nu} =
  \I{\bd\dom}{\tf\cdot\nu^\dom}{\ham^{n-1}}
  \qmq{for all} \tf\in C(\ol\dom,\R^n) \,
\end{equation}
and we have \reff{bd-s5-2} with $\bd\dom$ instead of $\rbd\dom\cap\edom$ 
for all relatively open $B\subset\ol\dom$. 

\end{theorem}

\noi
We call $\nu$ (outward) {\it normal measure} of $\dom$.
The proof shows that $\nu$ is a weak$^*$ cluster point of the measures 
$-D\chi_k\lem$ and, thus, it might be not unique in general. 
However it is uniquely determined by the subsequences entering \reff{bd-s5-1}. 
Below we provide examples showing that, for given $\dom$, there are different 
normal measures due to different auras. In particular we consider normal
measures with aura being completely inside or completely outside $\dom$ where,
in both cases, the vectors $\nu(B)$ are directed outward for small balls $B$
intersecting the boundary $\bd\dom$.
But notice that we always have the
same right hand side in \reff{bd-s5-3} for continuous $\tf$. If $\dom$ has
Lipschitz boundary, then $\mbd\dom=\bd\dom$ (cf. \cite[p.~50]{pfeffer}).
The inequality in \reff{bd-s5-2} can be strict as e.g. in the simple case
of an open ball $\dom=B\csubset\edom$ and $\nu=\nu_{\rm int}$ from 
Example~\ref{nm-int} below where 
$|\nu|(B)=\cH^{n-1}(\bd\dom)>(\reme{\cH^{n-1}}{\bd\dom})(B)=0$.

\begin{proof}   
Since $\chi$ is a good approximation, the measures $\nu_k:=-D\chi_k\lem$ 
with norm $\|\nu_k\|=\|D\chi_k\|_{\cL^1(\edom)}$
have a subsequence $\nu_{k'}$ that is bounded in $\cL^\infty(\edom,\R^n)^*$
and satisfies 
\begin{equation*}
  \lim_{k'\to\infty}\|D\chi_{k'}\|_{\cL^1(\edom)} = 
  \liminf_{k\to\infty}\|D\chi_k\|_{\cL^1(\edom)} \,.
\end{equation*}
This subsequence has a weak$^*$ cluster point in $\cL^\infty(\edom,\R^n)^*$
that we can identify with some 
\begin{equation*}
  \nu\in\bawl{U}^n \qmq{where} 
  |\nu|(\edom)=\|\nu\| \le \liminf_{k\to\infty} \|D\chi_k\|_{\cL^1(\edom)}\,.
\end{equation*}
Consequently we have \reff{bd-s5-1} where the 
subsequence $\{\chi_{k'}\}$ might depend on $\tf$. 
Hence $\I{\edom}{\tf}{\nu}=0$ for all $\tf$ vanishing on
$(\bd\dom)_\delta$ with some $\delta>0$. 
Thus $\cor\nu\subset\bd\dom$ and $\nu$ is a trace on
$\bd\dom$ over  $\cL^\infty(\edom,\R^n)$. Taking $\tf$ vanishing on
$\edom\setminus A$, we readily get from \reff{bd-s5-1} that $A$ is an aura of
$\nu$. For $B\in\bor{\edom}$ and $\eps>0$ there is some 
$\tf\in\cL^\infty(\edom,\R^n)$ with $\|\tf\|_\infty\le 1$ 
and a subsequence $\chi_{k'}$ such that
\begin{eqnarray*}
  |\nu|(B) - \eps 
&\le&
  \sI{\bd\dom}{\chi_B\tf}{\nu} 
= \lim_{k'\to\infty}\I{\edom}{\chi_B\tf\cdot D\chi_{k'}}{\lem} \\
&\le&
  \limsup_{k'\to\infty} \I{B}{|D\chi_{k'}|}{\lem} \le
  \limsup_{k\to\infty} \|D\chi_k\|_{\cL^1(B)}
\end{eqnarray*}
(cf. \reff{pm-tva}). 

Let now $\|\chi - \chi_\dom\|_{\cL^1(\edom)}=0$.
Then $\chi_k\to\chi_\dom$ in $\cL^1(\edom)$ and $\dom$ has finite perimeter in
$\edom$ by Proposition~\ref{nm-s2}. 
For $\tf\in C^\infty_c(\edom,\R^n)$ the definition of weak derivatives gives
\begin{equation} \label{bd-s5-5}
   \I{\edom}{\tf\cdot D\chi_k}{\lem} = - \I{\edom}{\chi_k\divv\tf}{\lem}\,.
\end{equation}
By dominated convergence and the divergence theorem
\begin{equation} \label{bd-s5-6}
  \lim_{k\to\infty} \I{\edom}{\chi_k\divv\tf}{\lem} =
  \I{\dom}{\divv\tf}{\lem} = 
  \I{\mbd\dom\cap\edom}{\tf\cdot\normal{\dom}}{\cH^{n-1}}\,
\end{equation}
(cf. \cite[p.141]{pfeffer} and notice that $\phi$ is Lipschitz continuous due
to its compact support).
Using \reff{bd-s5-1} for the left hand side in \reff{bd-s5-5}, we obtain
\begin{equation*}
   \sI{\bd\dom}{\tf}{\nu} =
  \I{\mbd\dom\cap\edom}{\tf\cdot\normal{\dom}}{\cH^{n-1}} 
\end{equation*}
for all $\tf\in C^\infty_c(\edom,\R^n)$. By uniform approximation we 
get this identity even for all $\tf\in C_c(\edom,\R^n)$.

Now let $B\subset\edom$ be open and recall that
\begin{equation*}
  D\chi_\dom=\reme{\nu^\dom\cH^{n-1}}{(\rbd\dom\cap\edom)}\,, \quad
  |D\chi_\dom|=\reme{\cH^{n-1}}{(\rbd\dom\cap\edom)}
\end{equation*}
(cf. \cite[p.~169, 205]{evans}).
Then, with \reff{pm-tva},
\begin{eqnarray*}
  |\nu|(B)
&=&
  \sup_{\substack{\tf\in \cL^\infty(\edom,\R^n)\\\|\tf\|_\infty\le 1}}
  \I{B}{\tf}{\nu} \\
&\ge&
  \sup_{\substack{\tf\in C^1_c(B,\R^n)\\\|\tf\|_\infty\le 1}}
  \I{B}{\tf}{\nu} 
\: = \:
  \sup_{\substack{\tf\in C^1_c(B,\R^n)\\\|\tf\|_\infty\le 1}}
  \I{\edom}{\tf}{\nu} \\
&\overset{\reff{bd-s5-3}}{=}&
  \sup_{\substack{\tf\in C^1_c(B,\R^n)\\\|\tf\|_\infty\le 1}}
  \I{\mbd\dom\cap\edom}{\tf \cdot \nu^\dom}{\ham^{n-1}} \\
&=&
  \sup_{\substack{\tf\in C^1_c(B,\R^n)\\\|\tf\|_\infty\le 1}}
  \I{\edom}{\tf}{D\chi_\dom} 
\: = \:
  \sup_{\substack{\tf\in C^1_c(B,\R^n)\\\|\tf\|_\infty\le 1}}
  \I{B}{\chi_\dom\divv\tf}{\lem}  \\
&=&
  |D\chi_\dom|(B) 
\: = \:
  \big(\reme{\ham^{n-1}}{(\rbd\dom\cap\edom)}\big)(B) \,. 
\end{eqnarray*}
(cf. also \cite[p.~169]{evans} for the second last line). 

If $\edom=\dom$ with Lipschitz boundary, then 
$\|\chi - \chi_\dom\|_{\cL^1(\edom)}=0$ and we argue similar as above.
For \reff{bd-s5-3a} we consider 
\begin{equation*}
  \tf\in C^\infty(\dom,\R^n)\cap C^1(\ol\dom,\R^n) \,.
\end{equation*}
Then we get \reff{bd-s5-5}, 
since $\chi_k$ has compact support in $\dom$, and in \reff{bd-s5-6}
we use the divergence theorem from \cite[p.~168]{pfeffer} to get
\begin{equation*}
  \I{\dom}{\divv\tf}{\lem} = 
  \I{\bd\dom}{\tf\cdot\normal{\dom}}{\cH^{n-1}}\,
\end{equation*}
(notice that $\tf$ is locally Lipschitz continuous on $\ol\dom$). 
For the adaption of \reff{bd-s5-2} we extend $\nu$
by zero on $\R^n$, we argue as above for arbitrary open 
$B\subset\R^n$ with some enlarged $\edom$, and we use that $\rbd\dom=\bd\dom$ 
(cf. \cite[p.~50]{pfeffer}).  
\end{proof}

\begin{proposition} \label{nm-s4}    
Let $\edom\subset\R^n$ be open and bounded, let $\dom\in\bor{\edom}$
be a set of finite perimeter, and let 
$\chi$ be a good approximation for $\chi_\dom$ with approximating sequence
$\chi_k$ and associated normal
measure $\nu\in \bawln{\edom}$. Then for every $B\in\bor{\edom}$
with finite perimeter 
there exists an $\cL^1$-null set $N \subset \R$ and some
$\tilde\delta>0$ such that for all $\delta\in(0,\tilde\delta)\setminus N$  

\begin{eqnarray} 
  \nu(B)
&=&
  - \lim_{k\to\infty}\I{B}{D\chi_k}{\lem}   \nonumber\\
&=&
  - \I{\edom\cap\ol\dom\cap\mbd (B\cap\dom^\delta)}
    {\chi\nu^{B\cap\dom^\delta}}{\ham^{n-1}} \quad\qmq{with}
  \dom^\delta=\dom_\delta\setminus\dom_{-\delta}  \nonumber\\
&=& \label{nm-s4-1}
  \I{\mbd B \cap \bd\dom\cap\edom}{-\chi \nu^B}{\ham^{n-1}} + 
  \lim_{\substack{\delta\downarrow 0\\\delta \notin N}}
  \I{(\mint B) \cap \mbd \dom_{-\delta}\cap\edom}
    {\nu^{\dom_{-\delta}}}{\ham^{n-1}} \,.
\end{eqnarray}
\end{proposition}

\noi
The representation of $\nu(B)$ is illustrated in Figure~\ref{nm-fig1}
for two simple cases. 
For $B_1$ the first integral vanishes and, thus, $\nu(B)$ is directed as
$\nu^\dom$ on $\bd\dom\cap B_1$ with length $\ham^{n-1}(\bd\dom\cap B_1)$. 
For $B_2$ we distinguish two cases. If $\chi=0$ on $\bd\dom$, then
we have a similar situation as for $B_1$ and $-\nu^B$ in the figure does not
apply. If $\chi=1$ on $\bd\dom$, then $\nu(B)=0$, since the two contributions
in \reff{nm-s4-1} cancel out each other.
Below we discuss this situation in more detail for 
several examples. 

\begin{figure}[H]
\centering
\begin{tikzpicture}[scale=0.9]    
	\draw[gray] (-2,-2) rectangle (2,2.2);
	\draw (-0.5,0) node {$\dom$};
	\draw[fill=lightgray] (1.5,-1.5) rectangle (2,1.5);
	\draw[dashed] (0.0,-1.5) rectangle (2,1.5);
	\draw (0.8,-0.7) node {$B_2$};
	\draw (1.75,1.6) node[rotate=-90] {$\{$};
	\draw (1.75,1.9) node {$\delta$};
	\draw[->,thick] (1.5,-0.6) to (3.1,-0.6);
	\draw (2.6,-0.3) node {$\nu^{\dom_{-\delta}}$};
	\draw[->,thick] (2,0.1) to (0.4,0.1);
	\draw (1.0,0.4) node {$-\nu^B$};
	\draw[fill=lightgray] (-1.5,-1.5) rectangle (-2,1.5);
	\draw[dashed] (-0.9,-1.5) rectangle (-3.5,1.5);
	\draw (-2.5,0.5) node {$B_1$};
	\draw (-1.75,1.6) node[rotate=-90] {$\{$};
	\draw (-1.75,1.9) node {$\delta$};
	\draw[->,thick] (-1.5,-0.6) to (-3.1,-0.6);
	\draw (-2.5,-0.3) node {$\nu^{\dom_{-\delta}}$};
\end{tikzpicture}
\caption{Contributions to $\nu(B)$ for two versions of $B$}
         \label{nm-fig1} 
\end{figure}
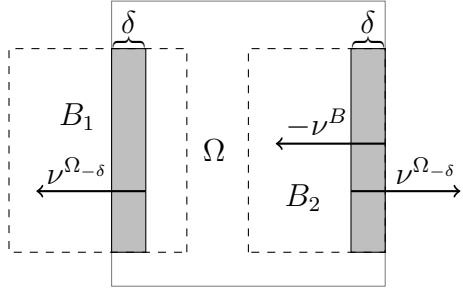

\begin{proof}       
Since $\nu\wac\lem$ 
and since $B$ differs from $\mint B$ only by an $\lem$-null set
(cf. \cite[p.~222]{evans}), we have 
$\nu(B)=\nu(\mint B)$. Thus we can essentially work with $\tilde B:=\mint B$,
but in integrals with $\lem$-measure we can replace it with the original~$B$.
We proceed analogously with $\tilde\dom^\delta:=\mint\dom^\delta$. Clearly
$\mbd B=\mbd{\tilde B}$. By 
$\mint\,(B\cap\dom^\delta)=\tilde B\cap\tilde\dom^\delta$ 
(cf. \cite[p. 50]{pfeffer}), we also have
$\mbd(B\cap\dom^\delta) = \mbd(\tilde B\cap\tilde\dom^\delta)$.
The coarea formula 
implies that $\dom^\delta$ and thus also $\tilde\dom^\delta$
has finite perimeter for a.e. $\delta>0$. 
Therefore $\tilde B\cap\tilde\dom^\delta$ has finite perimeter too
(cf. \cite[p. 130]{maggi}). Notice that $\tilde\dom^\delta$ might not be a
subset of $\edom$ and, though we have assumed $B\subset\edom$, also 
$\tilde B$ might not be a subset of $\edom$. But $\tilde
B\cap\tilde\dom^\delta$ has finite perimeter in $\edom$. 

We now consider the approximating sequence $\chi_k\in W^{1,\infty}(\edom)$ 
for $\chi$ and we choose some $\tf\in C^1(\edom,\R^n)$. Then $\chi_k\tf$ is 
Lipschitz continuous with compact support on $\edom$, since it is locally
Lipschitz continuous. Hence, using
\begin{equation*}
  \edom\cap\mbd(\tilde B\cap\tilde\dom^\delta\cap\edom) = 
  \edom\cap\mbd(\tilde B\cap\tilde\dom^\delta) \,
\end{equation*}
and a general version of the Gauss-Green formula, we obtain
\begin{eqnarray*}
  \I{B\cap\dom^\delta}{\tf\cdot D\chi_k}{\lem} 
&=&
  - \I{B\cap\dom^\delta}{\chi_k\divv\tf}{\lem}  \\
&&
  + \I{\edom\cap\mbd (\tilde B\cap\tilde\dom^\delta)}
            {\chi_k\tf\cdot\nu^{\tilde B\cap\tilde\dom^\delta}}{\ham^{n-1}}
\end{eqnarray*}
(cf. \cite[p.~51, 141]{pfeffer} and \cite[4.5.6]{federer69}).
Since $\chi_k\to\chi$ $\hm$-a.e. on $\edom$, by dominated convergence 
\begin{eqnarray*}
  \lim_{k\to\infty} \I{B\cap\dom^\delta}{\tf\cdot D\chi_k}{\lem} 
&=&
  - \I{B\cap\dom^\delta}{\chi\divv\tf}{\lem}  \\
&&
  + \I{\edom\cap\mbd (\tilde B\cap\tilde\dom^\delta)}
      {\chi\tf\cdot\nu^{\tilde B\cap\tilde\dom^\delta}}{\ham^{n-1}}
\end{eqnarray*}
for all $\tf\in C(\edom,\R^n)$. 
Let us choose $\tf$ to equal a constant vector $a\in\R^n$ on $\edom$. 
By $\cor{\nu}\subset\bd\dom$ and by \reff{bd-s5-1}, where we do not need
a subsequence due to the previous equation,
\begin{eqnarray*}
  \lim_{k\to\infty} \I{B\cap\dom^\delta}{a\cdot D\chi_k}{\lem} 
&=&
  \lim_{k\to\infty} \I{\edom}{\chi_{B\cap\dom^\delta}\, a\cdot D\chi_k}{\lem}
  \\
&=&
  -a \cdot\sI{\bd\dom}{\chi_{B\cap\dom^\delta}}{\nu} =
  -a\cdot\nu(B\cap\dom^\delta) \\
&=&
  -a\cdot\nu(B)  \qmq{for all} a\in\R^n\,.
\end{eqnarray*}
The arbitrariness of $a\in\R^n$ implies
\begin{equation} \label{nm-s4-4}
  \nu(B) =
  - \lim_{k\to\infty} \I{B\cap\dom^\delta}{D\chi_k}{\lem} =
  - \I{\edom\cap\mbd (\tilde B\cap\tilde\dom^\delta)}
    {\chi\nu^{\tilde B\cap\tilde\dom^\delta}}{\ham^{n-1}} 
\end{equation}
for a.e. $\delta>0$. By $\supp\chi_k\subset(\bd\dom)_{\frac{1}{k}}$ we can
omit $\dom^\delta$ in the first integral and by 
$\supp \chi\subset\ol\dom$ we can restrict the second 
integral to $\ol\dom$. Moreover the previous arguments show that 
\reff{nm-s4-4} is also valid with $B$, $\dom^\delta$ instead of 
$\tilde B$, $\tilde\dom^\delta$.

Since $\tilde B$ and $\tilde\dom^\delta$ agree with their measure theoretic
interior, 
\begin{eqnarray*}
  (\tilde B\cap\mbd{\tilde\dom^\delta})\cup(\mbd{\tilde B}\cap \tilde\dom^\delta) 
&\subset&  \mbd(\tilde B\cap \tilde\dom^\delta)  \\
&\subset&
  (\tilde B\cap\mbd{\tilde\dom^\delta})\cup(\mbd{\tilde B}
  \cap\tilde\dom^\delta)\cup (\mbd{\tilde B}\cap\mbd{\tilde\dom^\delta})
\end{eqnarray*}
(cf. \cite[p.~52]{pfeffer}). By $\lem(\mbd{\tilde  B})=0$, the coarea formula
implies
\begin{equation*}
  \ham^{n-1}( \mbd{\tilde  B}\cap\mbd{\tilde\dom^\delta}) = 0
  \qmz{for a.e.} \delta>0\,.
\end{equation*}
Therefore  $\mbd(\tilde B\cap\tilde\dom^\delta)$
differs from 
$(\tilde B\cap\mbd{\tilde\dom^\delta})\cup(\mbd{\tilde B}\cap\tilde\dom^\delta)$
only by an $\ham^{n-1}$-null set. Furthermore
\begin{equation*}
  \nu^{\tilde B\cap\tilde\dom^\delta}=\nu^{\tilde B} \qmq{on} 
  \rbd(\tilde B\cap \tilde\dom^\delta)\cap(\rbd{\tilde B}\cap\tilde\dom^\delta) 
  \qmz{for all} \delta>0\,,
\end{equation*}
since there (with $\dens_x$ from Example~\ref{ex:dzero})
\begin{equation*}
  \dens_x(\tilde B\cap\tilde\dom^\delta) = \dens_x\tilde B = \tfrac{1}{2}
\end{equation*}
and, by $\tilde B\cap\tilde\dom^\delta\subset \tilde B$, 
both sets have to generate the same half-space determining their normal
(cf. \cite[p.~157/158]{ambrosio}).
Analogously we get
\begin{equation*}
  \nu^{\tilde B\cap\tilde\dom^\delta}=\nu^{\tilde\dom^\delta} \qmq{on} 
  \rbd(\tilde B\cap \tilde\dom^\delta)\cap(\tilde B\cap\rbd{\tilde\dom^\delta}) 
  \qmz{for all} \delta>0\,.
\end{equation*}
Using that the reduced boundary agrees $\ham^{n-1}$-a.e. with the measure
theoretic one, that $\mbd B=\mbd{\tilde B}$, and that
$\mbd\dom^\delta=\mbd{\tilde\dom^\delta}$,  
we obtain from \reff{nm-s4-4} that
\begin{equation} \label{nm-s4-5}
  \nu(B)=-\I{\mbd B \cap \tilde\dom^\delta\cap\edom}{\chi \nu^B}{\ham^{n-1}}
  - \I{\tilde B\cap\mbd\dom^\delta\cap\edom}{\chi \nu^{\dom^\delta}}{\ham^{n-1}}
\end{equation}
for all $\delta>0$ despite an $\cL^1$-null set $N$. 
Since 
\begin{equation*}
  A\to\I{\mbd B \cap A\cap\edom}{\chi \nu^B}{\ham^{n-1}}
\end{equation*}
is a $\sigma$-measure
on $\mbd B\cap\edom$ and $\bigcap_{\delta>0}\tilde\dom^\delta=\bd\dom$,
\begin{equation*}
  \lim_{\substack{\delta\downarrow 0 \\ \delta\notin N}}
  \I{\mbd B \cap \tilde\dom^\delta\cap\edom}{\chi \nu^B}{\ham^{n-1}} =
  \I{\mbd B \cap \bd\dom\cap\edom}{\chi \nu^B}{\ham^{n-1}} \,. 
\end{equation*}
For the second integral in \reff{nm-s4-5} we use 
\begin{equation*}
  \supp\chi \cap \mbd\dom^\delta = \mbd\dom_{-\delta}\,, \quad
  \chi=1 \zmz{on} \op{int}\dom\,, \quad
  \nu^{\dom^\delta} = -\nu^{\dom_{-\delta}} \zmz{on} \mbd\dom_{-\delta}\,.
\end{equation*}
Finally we use that 
$\mbd\dom^\delta=\mbd(\dom_\delta\setminus\dom_{-\delta})$
to get the assertion.
\end{proof}

Let us now provide some important normal measures $\nu$ of $\dom$.
Some of the measures use the distance function $\distf{\bd\dom}$ for the
approximating sequence $\chi_k$ 
and require some boundedness
of $\ham^{n-1}(\bd\dom_\delta)$ for $\delta$ near zero. Notice that 
for any open bounded~$\dom$
\begin{equation*}
  \op{Per}(\dom_\delta)\le \ham^{n-1}(\bd\dom_\delta)<\infty
  \qmq{for all} \delta\ne 0\,
\end{equation*}
(cf. \cite[Theorem 3]{kraft_measure-theoretic_2016}), but
there are open bounded sets $\dom$ of finite perimeter where
$\ham^{n-1}(\bd\dom_\delta)$ scales as $\delta^{-s}$ for $s\in(0,1)$ 
(cf. \cite[Theorem 1]{kraft_measure-theoretic_2016}). The other examples 
are applicable to sets of finite perimeter $\dom\csubset\edom$ and 
use mollifications of $\chi_\dom$ for the approximating sequence $\chi_k$. 

As preparation let $\dom\csubset\edom$ have finite perimeter and we consider 
\begin{equation} \label{nm-mol}
  \psi_k := \chi_\dom*\eta_{\frac{1}{k}} 
\end{equation}
with the standard mollifier $\eta_\eps$ supported on $B_\eps(0)$.
Then, for large $k$,
\begin{equation*}
  \psi_k \in  C_c^\infty(\edom,[0,1])\,, \quad
  \psi_k = 1 \zmz{on} \dom_{-\frac{1}{k}}\,, \z
  \psi_k=0 \zmz{on} \edom\setminus\dom_{\frac{1}{k}}\,,
\end{equation*}
\begin{equation} \label{nm-mol1}
  \psi_k \to 
  \tfrac{1}{2}\chi_{\mbd\dom} + \chi_{\mint{\dom}}
  \qmz{$\ham^{n-1}$-a.e. on $\edom$} 
\end{equation}
(cf. \cite[p.~164, 173, 175]{ambrosio}). From
\cite[p.~41, 118]{ambrosio} we obtain for large $k$ 
\begin{equation*}
  D\psi_k(x) = \big(D\chi_\dom*\eta_{\frac{1}{k}}\big)(x) := 
  \I{\edom}{\eta_{\frac{1}{k}}(x-y)}{D\chi_\dom(y)} 
  \qmz{for} x\in\dom\,.
\end{equation*}
Then one has the weak$^*$ limits 
\begin{equation} \label{nm-mol2}
  D\psi_k\lem\overset{*}{\wto} D\chi_\dom \qmq{and}
  |D\psi_k|\lem\overset{*}{\wto} |D\chi_\dom| 
\end{equation}
in the sense of Radon measures. 
For any Borel set $B\subset\edom$ and $k$ large,
we can use $\supp{D\psi_k}\subset\dom_{\frac{1}{k}}\csubset\edom$ and 
$\mbd\dom\subset\bd\dom$ to get
\begin{equation} \label{nm-mol1a}
  \|D\psi_k\|_{\cL^1(B)} = \I{\dom_{1/k}\cap B}{|D\psi_k|}{\lem} \le
  |D\chi_\dom|(\dom_{\frac{2}{k}}\cap B_{\frac{1}{k}}) = 
  \ham^{n-1}(\mbd\dom\cap B_{\frac{1}{k}})  
\end{equation}
(cf. \cite[p.~42]{ambrosio}, \cite[p.~50, 138]{pfeffer} or also 
\cite[p.~49, 129]{maggi}). If $|D\chi_\dom|(\bd B)=0$, then
\begin{equation} \label{nm-mol2a}
  \lim_{k\to\infty}|D\psi_k|(B) = |D\chi_\dom|(B)
\end{equation}
(cf. \cite[p.~54]{evans}). Obviously 
$\psi_k\in W^{1,\infty}(\edom)$ and also the truncations
\begin{equation*}
  \tilde\chi_k^{\rm int}:=(2\psi_k-1)^+\,, \quad
  \tilde\chi_k^{\rm ext}:=-(2\psi_k-1)^-+1 \,.
\end{equation*}
belong to $W^{1,\infty}(\edom)$ with
\begin{equation*}
   D\tilde\chi_k^{\rm int} = \bigg\{ \mbox{\small $ 
  \begin{array}{cl} 2D\psi_k & 
      \text{$\lem$-a.e. on } \{\psi_k>\tfrac{1}{2}\} \,, \\[4pt]
    0 & \text{$\lem$-a.e. on } \{\psi_k\le \tfrac{1}{2}\}\,
  \end{array}$ } 
\end{equation*}
and $D\tilde\chi_k^{\rm ext}$ analogously (cf. \cite[p.~130]{evans}). 
Since $\dom$ and $\mint\dom$ differ merely by an $\lem$-negligible set
(cf. \cite[p.~222]{evans}),  
we can use \reff{tt-s3} with $\tf\in C^1_c(\edom)$ to get
\begin{eqnarray*}
  \lim_{k\to\infty}\I{\edom}{\tf D\tilde\chi_k^{\rm int}}{\lem} 
&=&
  - \lim_{k\to\infty} \I{\edom}{\tilde\chi_k^{\rm int}\divv\tf}{\lem} 
\: = \:
  - \I{\edom}{\chi_{\mint\dom}\divv\tf}{\lem} \\
&=&
  - \I{\edom}{\chi_\dom\divv\tf}{\lem} 
\: = \: 
  \I{\edom}{\tf}{D\chi_\dom} \,
\end{eqnarray*}
and analogously for $\chi_k^{\rm int}$. Since $C^1_c(\edom)$ is dense in
$C_c(\edom)$, we have  
\begin{equation*}
  D\tilde\chi_k^{\rm int}\lem\overset{*}{\wto} D\chi_\dom \qmq{and}
  D\tilde\chi_k^{\rm ext}\lem\overset{*}{\wto} D\chi_\dom \,.
\end{equation*}
From \reff{nm-mol1a} we get for any Borel set $B\subset\edom$
\begin{equation} \label{nm-mol3}
  \|D\tilde\chi_k^{\rm int}\|_{\cL^1(B)} = 
  \I{B\cap\{\psi_k>\frac{1}{2}\}}{2|D\psi_k|}{\lem} 
  \le  2\I{B}{|D\psi_k|}{\lem} 
\le 2\ham^{n-1}(\mbd\dom\cap B_{\frac{1}{k}})  
\end{equation}
and analogously 
\begin{equation} \label{nm-mol4}
  \|D\tilde\chi_k^{\rm ext}\|_{\cL^1(B)}\le 
  2\ham^{n-1}(\mbd\dom\cap B_{\frac{1}{k}}) \,.
\end{equation}

\medskip

Let us now introduce some special examples of normal measures. 
These can be basically applied to the same set $\dom$. However 
they differ in the point how
the boundary is taken into account or in the underlying construction.

\begin{example}(Interior Normal Measures) \label{nm-int}
Let $\edom \subset \Rn$ be open and bounded, and let
$\dom\in\bor{\edom}$. We call a normal measure  
{\it interior normal measure} of $\dom$ if the related good approximation
$\chi$ has the form
\begin{equation*}
  \chi=\chi_{\Int\dom} \qmq{or} \chi=\chi_{\mint\dom}\,.
\end{equation*}
For a first case we assume that $\dom\csubset\edom$, that there is a sequence 
$\delta_k\downarrow 0$ with 
\begin{equation}\label{nm-int-1}
    \lim_{k \to \infty}
    \mI{(0,\delta_k)}{\ham^{n-1}(\bd{\idnhd{\dom}{\delta}})}{\delta} <
    \infty \,,
\end{equation}
and we consider
\begin{equation} \label{nm-int-2}
  \chi_k^{\rm int} := \chi_{\ol{\dom_{-\delta_k}}} +
  \chi_{(\Int\dom\setminus\ol{\dom_{-\delta_k}})} 
  \frac{\distf{\bd\dom}}{\delta_k} \in W^{1,\infty}(\edom) \,.
\end{equation}
We can suppose that $\delta_k\le\frac{1}{k}$ and we obviously have that
\begin{equation*}
  \chi_k^{\rm int} \to \chi_{\Int{\dom}} 
  \qmz{pointwise on} \edom\,.
\end{equation*}
By the coarea formula (cf. \cite[p.117]{evans}),
\begin{eqnarray*}
  \|D\chi_k^{\rm int}\|_{\cL^1(\edom)}
&=&
  \I{(0,\delta_k)}{\I{\bd\dom_{-\delta}}{|D\chi_k^{\rm
    int}|}{\ham^{n-1}}}{\delta} \\ 
&=&
  \I{(0,\delta_k)}
    {\frac{\ham^{n-1}(\bd\dom_{-\delta})}{\delta_k}}{\delta} \\
&=&
  \mI{(0,\delta_k)}{\ham^{n-1}(\bd\dom_{-\delta})}{\delta} \,
\end{eqnarray*}
(notice that $|D\dist{\bd\dom}{x}|=1$ $\lem$-a.e. outside
$\bd\dom$, cf. \reff{nm-normal1})
Hence we obtain a good approximation for $\chi_\dom$ and a related interior
normal measure by
\begin{equation*}
  \chi^{\rm int}:=\chi_{\Int\dom} \qmq{and}  \nu^{\rm int}\in\bawln{\edom} \,.
\end{equation*}
Notice that any $A=(\bd\dom)_{\delta_k}\cap\Int\dom$ is an aura of 
$\nu^{\rm int}$ by \reff{bd-s5-0}. 

As alternative approximating sequence for $\chi_{\Int\dom}$ 
we set $\delta_k':=\frac{\delta_k}{2}$ and take
\begin{equation*} 
  \chi_k^{\rm intc} := \chi_{\ol{\dom_{-2\delta'_k}}} +
  \chi_{(\dom_{-\delta'_k}\setminus\ol{\dom_{-2\delta'_k}})}
  \frac{\distf{\bd(\dom_{-\delta'_k})}}{\delta'_k} \in W^{1,\infty}(\edom) \,
\end{equation*}
with $0<\delta_k\le\frac{1}{k}$ as far as 
\begin{equation*}
    \lim \limits_{k \to \infty}
    \mI{(\delta'_k,2\delta'_k)}{\ham^{n-1}(\bd{\idnhd{\dom}{\delta}})}{\delta} <
    \infty \,
\end{equation*}
(``intc'' indicates compact support of $D\chi^{\rm intc}_k$ on $\Int\dom$). 
In this case we can drop the requirement $\dom\csubset\edom$
and obtain a second interior normal measure related to 
the good approximation $\chi_{\Int\dom}$ of $\chi_\dom$ by
\begin{equation*}
  \chi^{\rm intc}:=\chi_{\Int\dom} \qmq{and}  \nu^{\rm intc}\in\bawln{\edom} \,.
\end{equation*}
Though the measure $\nu^{\rm intc}$
slightly differs from $\nu^{\rm int}$, this is not relevant for the values
of the corresponding integrals 
entering the considered Gauss-Green formulas. 

For a second case we assume that $\dom\csubset\edom$ has finite perimeter
and we consider
\begin{equation*}
  \tilde\chi_k^{\rm int} := (2\psi_k-1)^+
\end{equation*}
with the mollification $\psi_k$ from \reff{nm-mol}. Then, by \reff{nm-mol1},
\begin{equation*}
  \tilde\chi_k^{\rm int} \to \chi_{\mint\dom} \qmz{$\hm$-a.e. on $\edom$.}
\end{equation*}
Obviously $\tilde\chi_k^{\rm int}\in W^{1,\infty}(\edom)$ and, by
\reff{nm-mol3}, 
\begin{equation} \label{nm-int-3}
  \|D\tilde\chi_k^{\rm int}\|_{\cL^1(\edom)} = 
  \I{\{\psi_k>\frac{1}{2}\}}{2|D\psi_k|}{\lem} \le 2\ham^{n-1}(\mbd\dom) \,.  
\end{equation}
This way we get a good approximation for $\chi_\dom$ and a related 
interior normal measure by
\begin{equation*}
  \tilde\chi^{\rm int}:=\chi_{\mint\dom} \qmq{and}
  \tilde\nu^{\rm int}\in\bawln{\edom} \,.
\end{equation*}
\end{example}

\begin{example}(Exterior Normal Measures) \label{nm-ext}
Let $\edom \subset \Rn$ be open and bounded, and let
$\dom\in\bor{\edom}$. We call a normal measure  
{\it exterior normal measure} of $\dom$ if the related good approximation
$\chi$ has the form
\begin{equation*}
  \chi=\chi_{\ol\dom} \qmq{or} \chi=\chi_{\mbd\dom\cup\mint\dom} \,.
\end{equation*}
First we assume that $\dom\csubset\edom$, that there is a sequence 
$\delta_k\downarrow 0$ with 
\begin{equation} \label{nm-ext-1}
    \lim_{k \to \infty}
    \mI{(0,\delta_k)}{\ham^{n-1}(\bd\dom_\delta)}{\delta} <
    \infty \,,
\end{equation}
and we consider
\begin{equation*}
  \chi_k^{\rm ext} = \chi_{\ol\dom} +
  \chi_{(\dom_{\delta_k}\setminus\ol\dom)} 
  \frac{\distf{\bd(\dom_{\delta_k})}}{\delta_k} \in W^{1,\infty}(\edom) \,.
\end{equation*}
We again suppose that $\delta_k\le\frac{1}{k}$ and we have that
\begin{equation*}
  \chi_k^{\rm ext} \to \chi_{\ol\dom} 
  \qmz{pointwise on} \edom\,.
\end{equation*}
As above, 
\begin{equation*}
  \|D\chi_k^{\rm ext}\|_{\cL^1(\edom)} =
  \mI{(0,\delta_k)}{\ham^{n-1}(\bd\dom_{\delta})}{\delta} \,.
\end{equation*}
Therefore we obtain a good approximation for $\chi_\dom$ and a related 
exterior normal measure by
\begin{equation*}
  \chi^{\rm ext}:=\chi_{\ol\dom} \qmq{and} \nu^{\rm ext}\in\bawln{\edom} \,.
\end{equation*}
Now any $A=(\bd\dom)_{\delta_k}\cap\Ext\dom$
is an aura of $\nu^{\rm ext}$ by \reff{bd-s5-0}. 
Similar as in the previous example we can construct some alternative exterior
normal measure $\nu^{\rm extc}$ related to an approximating sequence 
$\{\chi_k^{\rm extc}\}$ for $\chi^{\rm ext}=\chi_{\ol\dom}$ where 
$\chi_k^{\rm extc}=1$ on $\dom_{\delta_k}$. 

For a second example we assume that $\dom\csubset\edom$ has finite perimeter
and we consider
\begin{equation*}
  \tilde\chi_k^{\rm ext} := -(2\psi_k-1)^-+1
\end{equation*}
with $\psi_k$ from \reff{nm-mol}. Then, by \reff{nm-mol1},
\begin{equation*}
  \tilde\chi_k^{\rm ext} \to \chi_{\mbd\dom\cup\mint\dom} 
  \qmz{$\hm$-a.e. on $\edom$.}
\end{equation*}
Clearly
$\tilde\chi_k^{\rm ext}\in W^{1,\infty}(\edom)$ and, by \reff{nm-mol4}, 
\begin{equation*}
  \|D\tilde\chi_k^{\rm ext}\|_{\cL^1(\edom)} \le 2\ham^{n-1}(\mbd\dom)\,.
\end{equation*}
Hence we get a good approximation for $\chi_\dom$ and a related 
exterior normal measure by
\begin{equation*}
  \tilde\chi^{\rm ext}:=\chi_{\mbd\dom\cup\mint\dom} \qmq{and}
  \tilde\nu^{\rm ext}\in\bawln{\edom} \,.
\end{equation*}
\end{example}

\newcommand{\canonical}{symmetric}
\newcommand{\Canonical}{Symmetric}
\newcommand{\canon}{{\op{sym}}}

\begin{example}(\Canonical\ Normal Measures) \label{nm-canon}
Let $\edom\subset\R^n$ be open and bounded, and let 
$\dom\in\bor{\edom}$. We call a normal measure  
{\it \canonical\ normal measure} of $\dom$ if the related good approximation
$\chi$ has the form
\begin{equation*}
  \chi=\frac{1}{2}\chi_{\bd\dom}+\chi_{\Int\dom} \qmq{or}
  \chi=\frac{1}{2}\chi_{\mbd\dom}+\chi_{\mint\dom} \,.
\end{equation*}
Let us first assume that $\dom\csubset\edom$, that there is a sequence 
$\delta_k\downarrow 0$ with 
\begin{equation} \label{nm-canon-1}
    \lim_{k \to \infty}
    \mI{(-\delta_k,\delta_k)}{\ham^{n-1}(\bd\dom_\delta)}{\delta} <
    \infty \,,
\end{equation}
and let us take
\begin{equation*}
  \chi_k^\canon = \chi_{\ol{\dom_{-\delta_k}}} +
  \chi_{(\dom_{\delta_k}\setminus\ol{\dom_{-\delta_k})}} 
  \frac{\distf{\bd(\dom_{\delta_k})}}{2\delta_k} \in W^{1,\infty}(\edom) \,.
\end{equation*}
We can suppose that $\delta_k\le \frac{1}{k}$, we readily get
\begin{equation*}
  \chi_k^\canon \to \tfrac{1}{2}\chi_{\bd\dom} + \chi_{\Int{\dom}}
  \qmz{pointwise on} \edom\,,
\end{equation*}
and, as before,
\begin{equation*}
  \|D\chi_k^\canon\|_{\cL^1(\edom)} =
  \mI{(-\delta_k,\delta_k)}{\ham^{n-1}(\bd\dom_{\delta})}{\delta} \,.
\end{equation*}
Thus we receive a good approximation for $\chi_\dom$ and a related 
\canonical\ normal measure by
\begin{equation*}
  \chi^\canon:=\tfrac{1}{2}\chi_{\bd\dom} + \chi_{\Int{\dom}} \qmq{and}
  \nu^\canon\in\bawln{\edom} \,.
\end{equation*}
Obviously each $A=(\bd\dom)_{\delta_k}$ is an aura of $\nu^\canon$ by
\reff{bd-s5-0}. We can interpret the factor~$\frac{1}{2}$ in this way 
that one half of the source
in a boundary point is considered to flow outward while the other half 
flows inward. 

Now we assume that $\dom\Subset\edom$ is a set of finite perimeter
and consider 
\begin{equation*}
  \tilde\chi_k^\canon:=\psi_k
\end{equation*}
with $\psi_k$ as in \reff{nm-mol}. Then, by \reff{nm-mol1},
\begin{equation*}
  \tilde\chi_k^\canon \to \tfrac{1}{2}\chi_{\mbd\dom} + \chi_{\mint{\dom}}
  \qmz{$\ham^{n-1}$-a.e. on $\edom$.} 
\end{equation*}
With \reff{nm-mol1a} for $B=\edom$ we see that we obtain 
a good approximation for $\chi_\dom$ and a related 
\canonical\ normal measure by
\begin{equation*}
  \tilde\chi^\canon:=\tfrac{1}{2}\chi_{\mbd\dom} + \chi_{\mint{\dom}}  \qmq{and}
  \tilde\nu^\canon\in\bawln{\edom}\,.
\end{equation*}
Theorem~\ref{bd-s5} and \reff{nm-mol1a} imply that 
\begin{equation*}
   |\tilde\nu_\canon|(\edom)\le \ham^{n-1}(\mbd{\dom}) \,
\end{equation*}
and, by \reff{bd-s5-2}, one has equality if 
$\|\tilde\chi^\canon - \chi_\dom\|_{\cL^1(\edom)}=0$.
\end{example}

\medskip

\bigskip

\begin{remark} \label{nm-s6}
(1) While the normal measures based on a distance function can be easily 
represented by means of a normal field and a scalar density measure 
(cf. Proposition~\ref{nm-s9} below), the normal measures based on
mollifications have the 
advantage that they are available for all sets of finite perimeter.
Instead of the distance function one can also construct normal measures based
on other Lipschitz continuous functions vanishing on the boundary $\bd\dom$
(cf. \cite[p.~449]{silhavy_divergence_2009}).

(2) We have that \reff{nm-int-1}, \reff{nm-ext-1} and \reff{nm-canon-1} 
are satisfied if $\dom$ has finite perimeter 
and if it satisfies 
\reff{dt-s4-3} (cf. \reff{dt-s4-4} and \cite{kraft_measure-theoretic_2016}). 
This is in particular the case if $\dom$ has Lipschitz boundary
(cf. also the arguments following Proposition~\ref{dt-s4}). 
\end{remark}

\smallskip

Clearly, $F\in\cL^\infty(\dom,\R^n)$ is $\nu$-integrable for any normal
measure $\nu$. Let us analyze how far also unbounded vector fields are
integrable. 

\newcommand{\au}{A}

\begin{proposition} \label{nm-s5}
Let $\edom\subset\R^n$ be open, bounded, let $\dom\in\bor{\edom}$, 
let $F\in\cL^1(\dom,\R^n)$, 
let $\nu$ be a normal measure related to a good approximation $\chi$
for $\chi_\dom$ and with approximating sequence 
$\{\chi_k\}$ satisfying
$|D\chi_k|\le\gamma k$ $\lem$-a.e. on $\dom$ for some $\gamma>0$, and let
$\au\subset\edom$ be an aura of $\nu$ as in \reff{bd-s5-0}.
If there is some $\tilde\delta>0$ such that 
\begin{equation}\label{nm-s5-1}
  \frac{1}{\delta}\I{(\bd\dom)_\delta\cap\au}{|F|}{\lem} 
  \qmq{is uniformly bounded for} 0<\delta<\tilde\delta \,,
\end{equation}
then $\tf F$ is $\nu$-integrable on $\edom$ 
for all $\tf\in\cL^\infty(\edom)$. 
If, in addition, 
\begin{equation}\label{nm-s5-2}
  \lim_{k\to\infty}
  \frac{1}{\delta}\I{(\bd\dom)_\delta\cap\au\cap\{|F|\ge k\}}{|F|}{\lem} 
  = 0 \qmq{uniformly for $\delta\in(0,\tilde\delta)$\,,}
\end{equation}
then for each $\tf\in\cL^\infty(\edom)$ 
there is a subsequence $\{\chi_{k'}\}$
such that
\begin{equation} \label{nm-s5-3}
   \lim_{k'\to\infty} \I{\edom}{\tf F\cdot D\chi_{k'}}{\lem} = 
  - \,\sI{\bd\dom}{\tf F}{\nu} \,.
\end{equation}
\end{proposition}
\noi
Example~\ref{nm-ex1} below shows that  
\reff{nm-s5-1} is not sufficient for \reff{nm-s5-2} and
that \reff{nm-s5-2} excludes certain concentrations on the boundary $\bd\dom$.

\begin{proof}     
We argue similar to the proof of Proposition~\ref{pm-prop7} and use
\begin{equation*}
  F=(F^1,\dots,F^n)\,, \quad \nu=(\nu^1,\dots,\nu^n)\,.
\end{equation*}
By assumption there is some $c>0$ such that 
$\frac{1}{\delta}\I{(\bd\au)_\delta\cap A}{|F|}{\lem}\le c$ for all 
$\delta\in(0,\tilde\delta)$. For each $k\in\N$ and $j=1,\dots,n$  we set
\begin{equation*}
  \au^j_k := \big\{y\in\au \:\big|\: |F^j(y)|<k \big\}\,, \quad 
  \au^{j,0}_k := \au\setminus\au^j_k\,,
\end{equation*}
\begin{equation*}
  h^j_k(x) := \bigg\{ \mbox{\small $ 
  \begin{array}{ll} \tfrac{l}{2^k} & \text{on } 
    |F^j|^{-1}\big(\big[\tfrac{l}{2^k},\tfrac{l+1}{2^k}\big)\big) 
    \!\cap \au^j_k\text{ for all } l\in\N\,,\\[4pt]
    0 & \text{on } \au^{j,0}_k \,.
  \end{array}$ } 
\end{equation*}
Then all $h^j_k$ are simple functions related to $\nu^j$ and
$h^j_k\le|F^j|$ on $\au$ for all $j$ and all $k\in\N$. 
Using $|\,h^j_k-|F^j|\,|<\tfrac{1}{2^k}$ on $\au^j_k$ we get
\begin{equation*}
  \big\{ y\in\au\:\big|\: |\,h^j_k-|F^j|\,|>\eps \big\} \subset \au^{j,0}_k
  \qmq{if} \tfrac{1}{2^k}<\eps \,.
\end{equation*}
Therefore, for any $\eps>0$ and all $j$,
\begin{equation} \label{nm-s5-5}
  \limsup_{k\to\infty} |\nu^j|\big\{ |\,h^j_k-|F^j|\,|>\eps \big\} \le
  \limsup_{k\to\infty} |\nu^j|(\au^{j,0}_k) \,.
\end{equation}
For fixed $j,k$ there is $\tilde\tf\in\cL^\infty(\edom,\R^n)$ 
with $\|\tilde\tf\|_\infty\le 1$ and $\tilde\tf=0$ on $\au^j_k$
such that
\begin{equation*}
  \tfrac{1}{2}|\nu^j|(\au^{j,0}_k) 
  \le \I{\au^{j,0}_k}{\tilde\tf}{\nu}\,. 
\end{equation*}
(cf. \reff{pm-tva}). 
By \reff{bd-s5-1} there is a subsequence $\{\chi_{l'}\}$ with
\begin{equation*}
  \I{\au^{j,0}_k}{\tilde\tf}{\nu} =
  - \lim_{l'\to\infty}\I{\au^{j,0}_k}{\tilde\tf\cdot D\chi_{l'}}{\lem} 
  \le \liminf_{l'\to\infty}\I{\au^{j,0}_k}{|D\chi_{l'}|}{\lem} \,.
\end{equation*}
Setting $\delta_{l'}:=\frac{2}{l'}$ we have  
$|D\chi_{l'}|\le\gamma l'=\frac{2\gamma}{\delta_{l'}}$. Thus, for $l'$ large, 
\begin{eqnarray*}
  c
&\ge&
  \frac{1}{\delta_{l'}}\I{(\bd\dom)_{\delta_{l'}}\cap\au}{|F^j|}{\lem} \\ 
&=&
  \frac{1}{\delta_{l'}} \bigg(\,
  \I{(\bd\dom)_{\delta_{l'}}\cap\au^j_k}{|F^j|}{\lem} +
  \I{(\bd\dom)_{\delta_{l'}}\cap\au^{j,0}_k}{|F^j|}{\lem} \bigg) \\
&\ge&
  \frac{k}{2\gamma} \, 
  \I{(\bd\dom)_{\delta_{l'}}\cap\au^{j,0}_k}{|D\chi_{l'}|}{\lem} 
\: = \:
  \frac{k}{2\gamma} \, 
  \I{\au^{j,0}_k}{|D\chi_{l'}|}{\lem} 
\end{eqnarray*}
(recall that $D\chi_{l'}=0$ $\lem$-a.e. outside $(\bd\dom)_{\delta_{l'}}$, cf. 
\cite[p.~130]{evans}).
Taking the limit $l'\to\infty$ and using the previous estimates, we obtain
\begin{equation*}
  \frac{4c\gamma}{k} \ge |\nu^j|(\au^{j,0}_k) 
  \qmq{for all $j$ and all $k\in\N$\,.}
\end{equation*}
Thus $|\nu^j|(\au^{j,0}_k)\to 0$ as $k\to\infty$ and, using
\reff{nm-s5-5}, we get  
\begin{equation*}
  h^j_k \convim{\nu^j} |F^j| \,.
\end{equation*}
For the measure $h_k^j|\nu^j|$ there is some $\tilde\tf\in\cL^\infty(\edom)$
with $\|\tilde\tf\|_\infty\le 1$ such that
\begin{equation*}
  \I{\au}{h^j_k}{|\nu^j|} \le  \I{\au}{\tilde\tf h^j_k}{\nu^j} + 1 
\end{equation*}
(cf. \reff{pm-tva}) 
and we use $\tf_k\in\cL^\infty(\edom,\R^n)$ with
\begin{equation*}
  \tf_k = \bigg\{ \mbox{\small $ 
  \begin{array}{cl} 
     \big(0,\dots,0,\tilde\tf h^j_k,0\dots,0\big) & \text{on } \au 
     \,,\\[4pt]
     0 & \text{on } \edom\setminus\au \,.
  \end{array}$ } 
\end{equation*}
By \reff{bd-s5-1} there is some $m\in\N$ (related to $\tf_k$) 
such that, with $\delta_m=\frac{2}{m}$, $|D\chi_m|\le\gamma m$, 
and $|\tf_k|\le h^j_k\le |F|$, 
\begin{eqnarray*}
  \I{\au}{h^j_k}{|\nu^j|}
&\le&
  \I{\au}{\tilde\tf h^j_k}{\nu^j} + 1 \\
&=& 
  \I{\edom}{\tf_k}{\nu} + 1 
\: \le \:
  \Big| \I{\edom}{\tf_k\cdot D\chi_m}{\lem} \Big| + 2  \\
&\le&
  \I{\edom}{|\tf_k|\,|D\chi_m|}{\lem} + 2 
\: \le \:
  \gamma m \I{(\bd\dom)_{\delta_m}\cap A}{|\tf_k|}{\lem} + 2   \\
&\le&
  \frac{2\gamma}{\delta_m}\I{(\bd\dom)_{\delta_m}\cap\au}{|F|}{\lem} + 2
\: \le \: 2\gamma c + 2 \,
\end{eqnarray*}
(notice that $\supp{D\chi_m}\subset(\bd\dom)_{\delta_m}$).
Therefore the integrals on the left hand side are uniformly bounded. 
Since, by construction, the sequence $\{h^j_k\}_k$ of simple functions is
increasing,
\begin{equation*}
  \I{\au}{|h^j_k-h^j_l|}{|\nu^j|} \to 0 \qmq{as} k,l\to\infty\,. 
\end{equation*}
Consequently, $|F^j|$ is $\nu^j$-integrable with determining sequence
$\{h^j_k\}_k$ and, hence, also $F^j$ is $\nu^j$-integrable. 
But this means that $F$ is $\nu$-integrable. 

For $\tf\in\cL^\infty(\edom)$ we have that it is $\nu$-integrable, which
includes that it is also $\nu$-measurable. Thus also $\tf F$ is 
$\nu$-measurable
(cf. \cite[p 102]{rao}). From $|\tf F|\le|F|$ $\lem$-a.e., we get that also  
$\nu$-a.e. (since $\nu$ is weakly absolutely continuous with respect 
to~$\lem$) and, consequently, we have the estimate also i.m. $\nu$. 
But this implies that 
$\tf F$ is $\nu$-integrable too (cf. \cite[p.~113]{rao}.

For the second assertion we consider on $\edom$
\begin{equation}\label{nm-s5-6}
  F^j_k(x) := 
   \bigg\{ \mbox{\small $ 
    \begin{array}{cl} F^j(x) & 
                    \text{if } 
                    |F^j(y)| < k \,, \\[3pt]
                    0 & \text{otherwise}  \,.
    \end{array}$ } 
\end{equation}
Then we get as in the first part
that $F^j_k\overset{\nu^j}{\to} F^j$ for each $j$. 
For $\tf\in\cL^\infty(\edom)$ we have that
$\tf F$ is $\nu$-integrable and, as above, 
$|\tf F^j_k|\le|\tf F^j|$ i.m. $\nu$.
Thus, dominated convergence gives with $F_k:=(F^1_k,\dots,F^n_k)$
\begin{equation*}
  \lim_{k\to\infty}\sI{\bd\dom}{\tf F_k}{\nu} = \sI{\bd\dom}{\tf F}{\nu} \,.
\end{equation*}
Hence, for given $\eps>0$, there is some $k_0\in\N$ such that
\begin{equation} \label{nm-p1}
   \Big|\sI{\bd\dom}{\tf F_k}{\nu} - \sI{\bd\dom}{\tf F}{\nu} \Big| \le \eps
  \qmz{for all} k>k_0\,.
\end{equation}
Since $|F|\le\sqrt{n} |F|_\infty$, we have for subsets of $\edom$ that 
\begin{equation*}
  \tilde\edom_k:=\big\{|F|_\infty\ge k\big\} \subset 
  \big\{|F|\ge k\sqrt{n}\big\} \subset
  \edom_k:=\big\{|F|\ge k\big\}.
\end{equation*}
Then, by construction, $F_k=F$ on $\edom\setminus\tilde\edom_k$.
We now choose some $l_0\in\N$ with 
\begin{equation*}
  \supp(D\chi_l)\subset\au 
  \qmq{and} \tfrac{1}{l}<\tilde\delta
  \qmq{for all} l\ge l_0\,
\end{equation*}
(cf. \reff{bd-s5-0}).
By $\supp(D\chi_l)\subset(\bd\dom)_{\frac{1}{l}}$ and $|D\chi_l|\le\gamma l$
we get with some possibly larger $k_0\in\N$, 
\begin{eqnarray*}
  \Big|\I{\edom}{\tf (F_k-F)\cdot D\chi_l}{\lem}\Big| 
&\le&
  \|\tf\|_\infty\gamma l \I{(\bd\dom)_{1/l}\cap\au}{|F_k-F|}{\lem} \\
&=&
 \|\tf\|_\infty\gamma l 
  \I{(\bd\dom)_{1/l}\cap\au\cap\tilde\edom_k}{|F_k-F|}{\lem} \\
&\overset{|F_k|\le|F|}{\le}& 
  2\|\tf\|_\infty\gamma l \I{(\bd\dom)_{1/l}\cap\au\cap\edom_k}{|F|}{\lem} \\
&\overset{\reff{nm-s5-2}}{<}&
  \eps
\end{eqnarray*}
for all $k>k_0$ and $l>l_0$. Hence
\begin{equation*}
  \I{\edom}{\tf F_k\cdot D\chi_l}{\lem} -\eps \le
  \I{\edom}{\tf F\cdot D\chi_l}{\lem} \le
  \I{\edom}{\tf F_k\cdot D\chi_l}{\lem} +\eps \,.
\end{equation*}
Let us now fix $k>k_0$. Then, by \reff{bd-s5-1} and 
$\tf F_k\in\cL^\infty(\edom,\R^n)$, there is a subsequence $\{\chi_{l'}\}$ with
\begin{equation*}
  \lim_{l'\to\infty} \I{U}{\tf F_k\cdot D\chi_{l'}}{\lem}  
  =  
  -\:\sI{\bd\dom}{\tf F_k}{\nu} \,.
\end{equation*}
Therefore,
\begin{equation*}
  -\:\sI{\bd\dom}{\tf F_k}{\nu} - \eps \le 
  \liminf_{l'\to\infty}\I{\edom}{\tf F\cdot D\chi_{l'}}{\lem} \le
  -\:\sI{\bd\dom}{\tf F_k}{\nu} + \eps  \,.
\end{equation*}
By \reff{nm-p1}
\begin{equation*}
   -\:\sI{\bd\dom}{\tf F}{\nu} - 2\eps \le 
  \liminf_{l'\to\infty}\I{\edom}{\tf F\cdot D\chi_{l'}}{\lem} \le
  -\:\sI{\bd\dom}{\tf F}{\nu} + 2\eps \,.
\end{equation*}
We obviously get the same estimate with limsup. 
Then the arbitrariness of $\eps>0$ gives the assertion.
\end{proof}

Though the measure of the set where an integrable 
function is large has to be small, the next example shows that 
\reff{nm-s5-1} is not sufficient for \reff{nm-s5-2} and \reff{nm-s5-3}. 

\begin{example} \label{nm-ex1}   
Let $\edom=\dom=(0,2)^2\subset\R^2$ and let $F=(0,f)$ with
\begin{equation*}
  f(x,y) :=  \bigg\{ \mbox{\small $ 
  \begin{array}{ll} 
    \tfrac{1}{y} & \text{for }  1<x<1+y \,,\; y<\frac{1}{2} \\[4pt]
    0 & \text{otherwise}\,.
  \end{array}$ }  
\end{equation*}
Since
\begin{equation*}
  \int_0^2 f(x,y)\,dx = \int_1^{1+y} \tfrac{1}{y} \,dx = 1 
  \qmz{for all} y\in(0,1)\,,
\end{equation*}
Fubini's theorem implies $F\in\cL^1(\dom)$ and \reff{nm-s5-1} with $A=\dom$.
Moreover, for $k\in\N$ and $0<\delta<\tfrac{1}{k}$,
\begin{equation*}
  \frac{1}{\delta}\I{(\bd\dom)_\delta\cap\dom\cap\{|F|\ge k\}}{|F|}{\cL^2} =
  \frac{1}{\delta}\int_0^\delta\int_1^{1+y} f(x,y)\,dxdy = 1 \, 
\end{equation*}
and, thus, \reff{nm-s5-2} is not satisfied. To check \reff{nm-s5-3}
we consider an approximating sequence $F_k=(F_{1k},F_{2k})$ of $F$ as in
\reff{nm-s5-6}. Obviously $F_k=0$ on a small neighborhood of $\bd\dom$. 
Then, for any normal measure $\nu$ and any
$\tf\in\cL^\infty(\dom)$,  
we can use $\cor{\nu}=\bd\dom$ and dominated convergence to get 
\begin{equation*}
  0 = \sI{\bd\dom}{\tf F_k}{\nu} \to 
  \sI{\bd\dom}{\tf F}{\nu} = 0 \,.
\end{equation*}
But taking e.g. $\chi_k$ from \reff{nm-int-2} with $\delta_k=\frac{1}{k}$, 
that is related to the interior normal measure $\nu^{\rm int}$,
we readily obtain  
\begin{equation*}
  \lim_{k\to\infty} \I{\dom}{F\cdot D\chi_k}{\cL^2} = 
  \lim_{k\to\infty} k\int_0^\frac{1}{k}\int_1^{1+y} f(x,y)\,dxdy = 1 \,.
\end{equation*}
Hence \reff{nm-s5-3} is not satisfied for $\tf\equiv 1$.  
\end{example}

Now we show how normal measures can be used for Gauss-Green formulas
where we even allow some weight on the boundary $\bd\dom$. 

\begin{theorem} \label{nm-s7}
Let $\edom\subset\R^n$ be open and bounded, let $\dom\in\bor{\edom}$, let
$\nu$ be a normal measure of $\dom$ related to a good approximation $\chi$ for
$\chi_\dom$ with approximating sequence $\{\chi_k\}$, 
let $F\in\cD\cM^1(\edom)$ be $\nu$-integrable such that 
\reff{nm-s5-3} is satisfied, and let
$\chi_k\to\chi$ $\divv F$-a.e. on $\bd\dom$.
Then we have for all $\tf\in\woinf{\edom}$ that
\begin{equation} \label{nm-s7-1}
  \I{\bd\dom}{\chi}{\divv(\tf F)} + \divv(\tf F)(\Int\dom) =
  \sI{\bd\dom}{\tf F}{\nu} \,.
\end{equation}
\end{theorem}

\begin{proof}  
Let $\tf\in\woinf{\edom}$. Then $\tf F\in\cD\cM^1(\edom)$ with
\begin{equation} \label{nm-s7-5}
  \divv(\tf F) = \tf\divv F + F\cdot D\tf\lem
\end{equation}
as measures on $\edom$ (cf. \cite[p.~448]{silhavy_divergence_2009}).
By the assumption and by $\chi_k\to\chi$ $\hm$-a.e. on $\edom$, we get
$\chi_k\to\chi$  $\divv (\tf F)$-a.e. on $\bd\dom$.
By $\chi_k\to\chi=1$ everywhere on $\Int\dom$ and 
by $\chi=0$ on $\Ext\dom$, dominated convergence gives
\begin{equation*}
  \I{\bd\dom}{\chi}{\divv(\tf F)} + \divv(\tf F)(\Int\dom) =
  \I{\edom}{\chi}{\divv(\tf F)} =
  \lim_{k\to\infty} \I{\edom}{\chi_k}{\divv(\tf F)} \,.
\end{equation*}
Since $\chi_k\in W^{1,\infty}(\edom)$ is compactly supported on $\edom$, the
definition of divergence measure (cf. \reff{tt-e3})
and \reff{nm-s5-3} with $\chi_{k'}$ related to $\tf$ imply 
\begin{equation*}
  \lim_{k'\to\infty} \I{\edom}{\chi_{k'}}{\divv(\tf F)} =
  - \lim_{k'\to\infty} \I{\edom}{\tf F\cdot D\chi_{k'}}{\lem} =
  \sI{\bd\dom}{\tf F}{\nu}
\end{equation*}
which implies \reff{nm-s7-1}.
\end{proof}

\begin{remark} \label{nm-s8}
(1) For $\chi=\chi_\dom$ the left hand side in \reff{nm-s7-1}
becomes $\divv(\tf F)(\dom)$ and, in this case, we can choose
\begin{equation*} \label{nm-s8-1}
  \lF=F\nu\,, \quad \mF=0
\end{equation*}
in Theorem~\ref{dt-s1} (cf. also Proposition~\ref{dt-s2}). Notice that we can 
ensure $\nu$-integrability of $F$ by 
\reff{nm-s5-1} and, due to 
\begin{equation*}
  \I{\dom}{\big|\F D\port{\bd\dom}_\delta\big|}{\lem} \le
  \frac{2}{\delta}\I{\dom_\delta}{|F|}{\lem} \,, 
\end{equation*}
\reff{nm-s5-1} also implies \reff{dt-s4-1}. Thus, 
for $\chi=\chi_\dom$,  we get $\mF=0$ and
$\cor{\lF}\subset\bd\dom$ from \reff{nm-s5-1} already 
through Propositions~\ref{dt-s4} without using \reff{nm-s5-3}.
However we do not obtain this way that 
$\lF=F\nu$ with some normal measure $\nu$.

(2) For any good approximations $\chi_1$, $\chi_2$ for $\chi_\dom$ and
associated normal measures $\nu_1$, $\nu_2$ where
$\chi_{j,k}\to\chi_j$ $\divv F$-a.e. on $\bd\dom$, Theorem~\ref{nm-s7} implies
\begin{equation*}
  \I{\bd\dom}{(\chi_1 - \chi_2)}{\divv (\tf F)} =
  \sI{\bd\dom}{\tf F}{(\nu_1-\nu_2)} \,
\end{equation*} 
for all $\tf\in\woinf{\edom}$. By \reff{nm-s7-5}
we can interchange $\nu_1$ and $\nu_2$ in \reff{nm-s7-1} as long as
$\chi_1=\chi_2$ $\lem$-a.e. and $\divv F$-a.e. on $\bd\dom$.
This in particular implies that
$\sI{\bd\dom}{\tf F}{\nu}$ is independent of the choice of the
good approximation $\chi$ and the corresponding normal measure $\nu$ if
$|\divv F|(\bd\dom)=\lem(\bd\dom)=0$.

(3) Since $\chi=1$ on $\Int\dom$ and
\begin{equation*}
  \big(\chi\divv\tf F\big)(\dom) = \I{\bd\dom}{\chi}{\divv(\tf F)} +
  \divv(\tf F)(\Int\dom) \,,
\end{equation*}
we readily see from \reff{nm-s7-1} that $\tf\to\big(\chi\divv\tf F\big)(\dom)$
is a trace on $\bd\dom$ over $\woinf{\edom}$ under the assumptions of
Theorem~\ref{nm-s7}. 
\end{remark}

For $F\in\cD\cM^\infty(\edom)$ we trivially have that $F$ is 
$\nu$-integrable for any normal measure $\nu$
and, due to \reff{bd-s5-1} in Theorem~\ref{bd-s5},
we always have \reff{nm-s5-3} without the assumptions of
Proposition~\ref{nm-s5}. Since $\divv F\wac\cH^{n-1}$ in this case
(cf. \cite[p. 21]{silhavy_2005}), we also 
have $\chi_k\to\chi$ $\divv F$-a.e. on $\bd\dom$ by the definition of a good
approximation. Thus
Theorem~\ref{nm-s7} and Theorem~\ref{bd-s5}
directly imply the next special case.

\begin{corollary} \label{nm-s10}
Let $\edom\subset\R^n$ be open and bounded, let $\dom\in\bor{\edom}$, let
$\nu$ be a normal measure of $\dom$ related to a good approximation $\chi$ for
$\chi_\dom$, 
and let $F\in\cD\cM^\infty(\edom)$.  
Then $F$ is $\nu$-integrable, it satisfies \reff{nm-s5-3}, and 
we have \reff{nm-s7-1} for all $\tf\in\woinf{\edom}$. 
If $F$ is even continuous and $\|\chi-\chi_\dom\|_{\cL^1(\edom)}=0$, then we
have 
\begin{equation*}
  \I{\bd\dom}{\chi}{\divv(\tf F)} + \divv(\tf F)(\Int\dom) =
  \I{\mbd\dom\cap\edom}{\tf F\cdot\nu^\dom}{\hm}
\end{equation*}
for all $\tf\in C^1_c(\edom)$. 
\end{corollary}
\noi

Let us mention that it would be possible in this case to consider 
$\nu$ as trace on $\bd\dom$ over $\cD\cM^\infty(\edom)$. This way one wouldn't
need a trace over $\woinf{\edom}$ for each single~$F$. However both strategies
give essentially the same result. For the representation of such a trace 
one can use that $\cD\cM^\infty(\edom)\subset\cL^\infty(\edom)^n$ 
and that $\tf F\in\cD\cM^\infty(\edom)$ for
$\tf\in\cB\cV(\edom)\cap\cL^\infty(\edom)$ and $F\in\cD\cM^\infty(\edom)$ 
(cf. \cite[p.1014]{chen_structure_2011} and \cite[p.~65]{schonherr_diss}).
Now we derive a more explicit structure for the special normal measures 
\begin{equation*}
  \nu_{\rm int}\,, \z \nu_{\rm intc}\,, \z \nu_{\rm ext}\,,
  \zmz{and} \nu_\canon\,
\end{equation*}
where we use the normal field introduced in \reff{nm-normal}.

\begin{proposition} \label{nm-s9}
Let $\edom\subset\R^n$ be open and bounded, let $\dom\in\bor{\edom}$, let
$\nu^*$ be a normal measure where 
$*$ stands for {\rm int}, {\rm intc}, {\rm ext}, or $\canon$
and let $\chi^*$, $\chi^*_k$, and $\delta_k$ be related to $\nu^*$ 
as in the corresponding examples above. Then there is some 
measure $\dens^*_{\bd\dom}\in\bawl{\edom}$ such that 
\begin{equation} \label{nm-s9-1}
  \nu^* = \nu^\dom\dens^*_{\bd\dom} 
\end{equation}
and for any $\tf\in\cL^\infty(\edom)$ there is a subsequence $\chi^*_{k'}$ with
\begin{equation} \label{nm-s9-2}
  \I{\edom}{\tf}{\dens^*_{\bd\dom}} =
  \lim_{k'\to\infty} \frac{1}{\delta_{k'}}
  \I{(\bd\dom)_{\delta_{k'}}}{\tf\psi^*}{\lem}  \,
\end{equation}
where
\begin{equation*}
  \psi^{\rm int} := \chi_{\Int\dom}\,, \z
  \psi^{\rm intc}:=\psi^{\rm intc}_k= 
  2\chi_{\Int\dom\setminus(\dom_{-\delta_k/2})^c}\,, \z
  \psi^{\rm ext} := \chi_{\Ext\dom}\,, \z
  \psi^{\canon} := \tfrac{1}{2}\,. 
\end{equation*}
If, in addition, $F\in\cD\cM^1(\edom)$ is $\nu^*$-integrable such that  
\reff{nm-s5-3} is satisfied 
and that $\chi^*_k\to\chi^*$ $\divv F$-a.e. on $\bd\dom$,
then we have for all $\tf\in\woinf{\edom}$
\begin{eqnarray*} 
  \divv(\tf F)(\Int\dom) 
&=& 
  \sI{\bd\dom}{\tf F\cdot\nu^\dom}{\dens^{\rm int}_{\bd\dom}}  \,, \\
   \divv(\tf F)(\Int\dom) 
&=&
  \sI{\bd\dom}{\tf F\cdot\nu^\dom}{\dens^{\rm intc}_{\bd\dom}}  \,,  \\
\divv(\tf F)(\ol\dom) 
&=& 
  \sI{\bd\dom}{\tf F\cdot\nu^\dom}{\dens^{\rm ext}_{\bd\dom}} \,, \\
  \tfrac{1}{2}\divv(\tf F)(\bd\dom) + \divv(\tf F)(\Int\dom) 
&=& 
  \sI{\bd\dom}{\tf F\cdot\nu^\dom}{\dens^\canon_{\bd\dom}} .
\end{eqnarray*}
\end{proposition}
\noi
Notice that the measures $\dens^*_{\bd\dom}$ are measures of the type as
constructed in Proposition~\ref{tt-s3}. In the case $\nu^{\rm int}$ we have to
take $\gamma(\delta)=\delta$ and $\alse=\Int\dom$ there. However,
here we cannot just apply Proposition~\ref{tt-s3}, since we have to select 
that weak$^*$ cluster point of the measures
$\frac{1}{\delta_k}\chi_{(\bd\dom)_{\delta_k}\cap\,\Int\dom}\lem$
that is related to $\nu^{\rm int}$. 
In the proof we see that the subsequence $\chi^*_{k'}$ in \reff{nm-s9-2} is the
same as that for $\tf\nu^\dom$ in \reff{bd-s5-1}. 
We refer to the arguments following \reff{nm-normal} for the properties of the
normal field $\nu^\dom$. Notice that the previous proposition
is applicable to all $F\in\cD\cM^\infty(\edom)$, since \reff{nm-s5-3} is
always satisfied in this case 
(cf. \reff{bd-s5-1} in Theorem~\ref{bd-s5} or Corollary~\ref{nm-s10}).

The explicit occurrence of the normal field $\nu^\dom$ in the Gauss-Green
formulas above is due to the fact that, for these normal measures,
the normalized gradient of the associated $D\chi_k$ equals the gradient of
the distance function on the support of $D\chi_k$ and, so far, it is
independent of $k$ near $\bd\dom$. However, this is not met for the normal
measures based on mollification. Therfore we cannot go beyond \reff{nm-s7-1}
in those cases in general. Let us also refer to the fact that the
Radon-Nikodym theorem is only available ``up to a small error $\eps>0$''
for normal measures $\nu$ that are typically pure
(cf. \cite[p.~191]{rao}).

\begin{proof}   
Let us first consider $\nu^{\rm int}$ with $\chi_k=\chi_k^{\rm int}$, 
$\delta_k$ as in
\reff{nm-int-1} and \reff{nm-int-2}. Then 
\begin{equation}\label{nm-s9-5}
  D\chi_k = \bigg\{ 
  \mbox{\small $ 
  \begin{array}{ll} -\frac{1}{\delta_k}\nu^\dom & 
                     \lem\tx{-a.e. on } 
                     \Int\dom\setminus\ol{\dom_{-\delta_k}}\,,\\
                    0 & \text{otherwise}\,.  
  \end{array}$ } 
\end{equation}
Obviously 
\begin{equation*}
  \tf \to \I{\edom}{\tf\nu^\dom}{\nu^{\rm int}} \qmq{for}
  \tf\in\cL^\infty(\edom)
\end{equation*}
belongs to $\cL^\infty(\edom)^*$ and can be identified with some measure
\begin{equation*}
  \dens^{\rm int}_{\bd\dom}\in\bawl{\edom} \,
\end{equation*}
(cf. \cite[p.~106]{rao}). Then
\begin{equation*}
  \I{\edom}{\tf}{\dens^{\rm int}_{\bd\dom}} =  
  \I{\edom}{\tf\nu^\dom}{\nu^{\rm int}} \qmq{for all} \tf\in\cL^\infty(\edom)\,.
\end{equation*}
For $\Phi\in\cL^\infty(\edom)^n$ there is a pointwise unique orthogonal
decomposition $\lem$-a.e. such that 
\begin{equation*}
  \Phi=(\Phi\cdot\nu^\dom)\,\nu^\dom + \Phi^\perp \qmq{where}
  \Phi^\perp\cdot\nu^\dom=0 \,.
\end{equation*}
Then $\sI{\bd\dom}{\Phi^\perp}{\nu^{\rm int}}=0$ by \reff{bd-s5-1} and
\reff{nm-s9-5}. Thus,
for all $\Phi\in\cL^\infty(\edom)^n$,
\begin{equation*}
  \sI{\bd\dom}{\Phi}{\nu^{\rm int}} =
  \sI{\bd\dom}{(\Phi\cdot\nu^\dom)\,\nu^\dom}{\nu^{\rm int}} =
  \sI{\bd\dom}{\Phi\cdot\nu^\dom}{\dens^{\rm int}_{\bd\dom}} \,.
\end{equation*}
Consequently,
\begin{equation*}
  \nu^{\rm int} = \nu^\dom\dens^{\rm int}_{\bd\dom} \,. 
\end{equation*}
For $\tf\in\cL^\infty(\edom)$ we use \reff{nm-s9-5} and 
\reff{bd-s5-1} with $\tf\nu^\dom$ and the corresponding subsequence $\chi_{k'}$ 
to get
\begin{eqnarray*}
  \I{\edom}{\tf}{\dens^{\rm int}_{\bd\dom}}
&=&
  \sI{\bd\dom}{\tf\nu^\dom}{\nu^{\rm int}} 
\: = \: 
  -\lim_{k'\to\infty} \I{\edom}{\tf\nu^\dom\cdot D\chi_{k'}}{\lem} \\
&=&
  \lim_{k'\to\infty} \frac{1}{\delta_{k'}}
  \I{(\bd\dom)_{\delta_{k'}}}{\tf\chi_{\Int\dom}}{\lem} \,. 
\end{eqnarray*}
For the other cases we argue analogously.

Let now $F\in\cD\cM^1(\edom)$ be $\nu^*$-integrable such that  
\reff{nm-s5-3} is satisfied. Then the stated Gauss-Green formulas follow
directly from Proposition~\ref{nm-s7} and \reff{nm-s9-1} with the related
$\chi^*$. 
\end{proof}

\medskip

\noi
The next example shows how the different cases in the previous proposition 
work. 

\begin{example}  \label{nm-ex4}  
Let $\edom=B_3(0)\subset\R^2$, let $\dom=\dom_1\cup\dom_2$ with
\begin{equation*}
  \dom_1=(0,1)^2\,, \quad \dom_2=(1,2)\times(0,1)\,,
\end{equation*}
and consider 
\begin{equation*}
  F(x,y)= 
  \mbox{\small $\left\{ \begin{array}{ll} (1,0) \z & \text{on }\dom_1\,, \\
                          (2,0) & \text{for } x>1\,,   \\
                          (0,0) & \text{otherwise.} \end{array} \right.   $ }
\end{equation*}
For $\chi_k$ related to $\nu_{\rm ext}$ 
and $\delta>0$ small we obviously have that
\begin{equation*}
  D\chi_k=0 \qmq{on} (1-\delta,1+\delta)\times(0,1)\,.
\end{equation*}
Then, by Proposition~\ref{nm-s9} with $\tf\equiv 1$, by 
\reff{nm-s5-3}, and for $\delta>0$ small, we readily get 
\begin{eqnarray*}
  \divv F(\dom) 
&=&
  \sI{\bd\dom}{F\cdot\nu^\dom}{\dens^{\rm int}_{\bd\dom}} 
  \: = \: \sI{\bd\dom}{F}{\nu_{\rm int}}  \\
&=& 
  (F\nu_{\rm int})\big(\{x\in(0,\delta)\}\big) + 
  (F\nu_{\rm int})\big(\{x\in(1-\delta,1+\delta)\}\big) + \\
&&
  (F\nu_{\rm int})\big(\{x\in(2-\delta,2)\}\big) \\
&=&
  -1 + (1-2) + 2 \: = \: 0 \,.
\end{eqnarray*}
\begin{eqnarray*}
  \divv F(\ol\dom) 
&=&
  \sI{\bd\dom}{F\cdot\nu^\dom}{\dens^{\rm ext}_{\bd\dom}} 
  \: = \:  \sI{\bd\dom}{F}{\nu_{\rm ext}}  \\
&=& 
  (F\nu_{\rm ext})\big(\{x\in(-\delta,0)\}\big) + 
  (F\nu_{\rm ext})\big(\{x\in(1-\delta,1+\delta)\}\big) + \\
&&
  (F\nu_{\rm ext})\big(\{x\in(2,2+\delta)\}\big) \\
&=&
  0 + 0 + 2 \: = \: 2 \,.
\end{eqnarray*}
\begin{eqnarray*}
  &&
  \hspace{-20mm}  \tfrac{1}{2}\divv F(\bd\dom) + \divv F(\Int\dom)  \\
&=&
  \sI{\bd\dom}{F\cdot\nu^\dom}{\dens^\canon_{\bd\dom}} 
  \: = \:
  \sI{\bd\dom}{F}{\nu_\canon}  \\
&=& 
  (F\nu_\canon)\big(\{x\in(-\delta,\delta)\}\big) + 
  (F\nu_\canon)\big(\{x\in(1-\delta,1+\delta)\}\big) + \\
&&
  (F\nu_\canon)\big(\{x\in(2-\delta,2+\delta)\}\big) \\
&=&
  -\tfrac{1}{2} + (\tfrac{1}{2} - 1 ) + 2 \: = \: 1 \,.
\end{eqnarray*}
\end{example}

\noi
Let us now consider the application to unbounded vector fields. 

\begin{example} \label{nm-ex5}   
Let $\edom=B_2(0)\subset\R^2$, let $\dom=(0,1)^2$, and take $F=(F^1,F^2)$
given by
\begin{equation*}
    F(x):=\frac{x}{2\pi|x|^2} \qmq{with} \divv F=\delta_0\,
\end{equation*}
(cf. Example~\ref{dt-ex4}). We consider the normal measure 
$\nu=\nu^{\rm int}=(\nu^1,\nu^2)$ and, according to Proposition~\ref{nm-s9},
we have
\begin{equation} \label{nm-ex5-1}
  \nu=\nu^\dom\dens_{\bd\dom}^{\rm int} \,.
\end{equation}
A simple computation shows that
\begin{equation*}
  \frac{1}{\delta}\I{(\bd\dom)_\delta\cap\dom}{|F|}{\lem} 
  \; \overset{\delta\to 0}{\longrightarrow} \; \infty \,.
\end{equation*}
Hence we cannot use Proposition~\ref{nm-s5} to get $\nu$-integrability. 
Therefore let us use a more direct argument. We divide $\dom$ by its diagonals 
into four triangles 
\begin{equation*}
  \dom_{*0}\,,\;  \dom_{*1}\,,\;  \dom_{0*}\,,\;  \dom_{1*}
\end{equation*}
where e.g. $\dom_{*0}$ is the triangle with a side on the line $\{x_2=0\}$.
Clearly $F$ is bounded and, thus, $\nu$-integrable on
$\dom_{*1}\cup\dom_{1*}$. Using \reff{bd-s5-1} we readily get
\begin{equation*}
  \I{\dom_{*1}\cup\dom_{1*}}{F}{\nu} = \frac{1}{4} \,.
\end{equation*}
On $\dom_{*0}$ we obviously have $\nu^\dom=(0,-1)$. Hence, by 
\reff{nm-ex5-1}, $\nu^1$ is the zero measure on $\dom_{*0}$.
Therefore $F^1$ is $\nu^1$-integrable on $\dom_{*0}$ and the integral vanishes. 
Now we consider $F^2$ with respect to $\nu^2$ and set for $\eps>0$
\begin{equation*}
  M_\eps := |\nu^2|\big\{ x\in\dom_{*0}\:\big|\: |F^2|>\eps \big\} \,.
\end{equation*}
Using that $F^2$ is continuous and vanishes for $x_2=0$, 
we have for any $\kappa>0$ 
\begin{equation*}
   M_\eps \cap \{|x|\ge\kappa\} \cap \{x_2=0\}_\delta = \emptyset
\end{equation*}
for some small $\delta>0$. Since $\{x_2=0\}_\delta$ is an aura of 
$\dens_{\bd\dom}^{\rm int}$, 
\begin{equation*}
  |\nu^2|\big(M_\eps \cap \{|x|\ge\kappa\}\big) = 0 \qmz{for all}
  \kappa>0 \,. 
\end{equation*}
Consequently, using \reff{nm-s9-2}, 
\begin{eqnarray*}
   |\nu^2|\big\{ x\in\dom_{*0}\:\big|\: |F^2|>\eps \big\} 
&=&
  |\nu^2|\big\{ x\in\dom_{*0}\:\big|\: |F^2|>\eps\,,\; |x|<\kappa \big\} \\
&\le&
  |\nu^2|\big\{ x\in\dom_{*0}\:\big|\: |x|<\kappa \big\} \\
&\overset{\reff{nm-ex5-1}}{\le}&
  \kappa \qmz{for all} \kappa>0\,.
\end{eqnarray*}
Therefore
\begin{equation*}
  |\nu^2|\big\{ x\in\dom_{*0}\:\big|\: |F^2|>\eps \big\} = 0 \,.
\end{equation*}
This means that $F^2$ agrees i.m. $\nu^2$ with the zero function on
$\dom_{*0}$. Therefore $F^2$ is $\nu^2$-integrable on $\dom_{*0}$ with
vanishing integral. Summarizing we get
\begin{equation*}
  \I{\dom_{*0}}{\!\!F}{\nu}=0 \qmq{and, analogously,}
  \I{\dom_{0*}}{\!\!F}{\nu}=0 \,.
\end{equation*}
Hence, though \reff{nm-s5-1} is not satisfied, $F$ is $\nu$-integrable on
$\dom$ with 
\begin{equation*}
  \sI{\bd\dom}{F}{\nu} = \sI{\bd\dom\setminus B_1(0)}{F}{\nu} = \frac{1}{4} \,.
\end{equation*}
Let us now check \reff{nm-s5-3} directly for $\tf\equiv 1$. 
With $\chi^{\rm int}_k$ from \reff{nm-int-2} we have
\begin{eqnarray*}
   \I{\dom_{*0}}{F\cdot D\chi^{\rm int}_k}{\lem} 
&=&
  \I{(\bd\dom)_{1/k}\cap\dom_{*0}}{kF^2}{\lem} \\
&=&
  \frac{k}{2\pi} \int_0^{\frac{1}{k}} 
                 \int_{x_2}^{1-x_2} \frac{x_2}{x_1^2+x_2^2}\, dx_1dx_2 \\
&=&
  \frac{k}{2\pi} \int_0^{\frac{1}{k}}  
  x_2 \big[\tfrac{1}{x_2}\arctan\tfrac{x_1}{x_2} \big]_{x_1=x_2}^{1-x_2} \, dx_2\\
&=&
  \frac{1}{2\pi}\: \mI{(0,\frac{1}{k})}
  {\big(\arctan\tfrac{1-x_2}{x_2} - \tfrac{\pi}{4}\big)}{x_2} \\
& \overset{k\to\infty}{\longrightarrow}&
  \frac{1}{2\pi}\Big(\frac{\pi}{2}-\frac{\pi}{4}\Big) = \frac{1}{8} \,.
\end{eqnarray*}
By analogous arguments on the other triangles we end up with
\begin{equation*}
  \I{\dom}{F\cdot D\chi^{\rm int}_k}{\lem} 
  \overset{k\to\infty}{\longrightarrow} 0
  \ne - \frac{1}{4} = - \sI{\bd\dom}{F}{\nu} \,.
\end{equation*}
Thus \reff{nm-s5-3} is violated and obviously 
\begin{equation*}
  (\divv F)(\dom) = \delta_0(\dom) = 0 \ne \sI{\bd\dom}{F}{\nu} \,.
\end{equation*}
Therefore \reff{nm-s7-1} doesn't hold and we see that this condition is
essential in Theorem~\ref{nm-s7}. From Example~\ref{dt-ex4} we know that
\begin{equation*}
  \divv(\tf F)(\dom) = \I{\bd\dom}{F\nu^\dom}{\cH^1} - \frac{1}{4} \tf(0)\,
\end{equation*}
for all $\tf\in\woinf{\edom}$. Since the measure 
$F\nu$ is absolutely continuous with respect to $\nu$, we
in fact cannot expect \reff{nm-s7-1} with a normal measure as in
\reff{nm-ex5-1} in the case of a concentration at the origin (cf. 
\cite[p.~106]{rao}).   
\end{example}

Let us still provide an unbounded vector field where Theorem~\ref{nm-s7} is
applicable.  

\begin{example}  \label{nm-ex6}   
For $\edom=B_2(0)\subset\R^3$ and $\dom=(0,1)^3$ we consider as 
in the previous example
\begin{equation*}
    F(x):=\frac{x}{2\pi|x|^2} \,.
\end{equation*}
Then
\begin{equation*}
  s \to \I{(0,1)^2}{|F(x_1,x_2,s)|}{\ham^2(x_1,x_2)} 
\end{equation*}
is continuous and bounded. Thus \reff{nm-s5-1} is satisfied and $F$ is
integrable with respect to any normal measure $\nu$ satisfying the assumption
of Proposition~\ref{nm-s5} (which is the case for the normal measures 
constructed above by means of a distance function). For such a normal measure
$\nu$ with aura in $\dom$ we get for large $l\in\N$
\begin{eqnarray*}
  \I{(\bd\dom)_\delta\cap\dom\cap\{|F|\ge l\}}{|F|}{\lem} 
&\le&
  \I{B_{1/(2\pi l)}(0)\cap(\bd\dom)_\delta\cap\dom}{|F|}{\lem} \\
&\le&
  \frac{3}{4} \int_0^\delta 
  \I{B_{1/(2\pi l)}(0)\cap\{x_3=0\}}
        {\frac{1}{2\pi\sqrt{x_1^2+x_2^2}}}{\ham^2(x_1,x_2)}\,dx_3 \\
&\le&
  \frac{3}{4} \int_0^\delta 
  \int_0^{\frac{1}{2\pi l}} \frac{2\pi r}{2\pi r}\, drdx_3 
\; = \;
   \frac{3\delta}{8\pi l}  \,.  \\
\end{eqnarray*}
Therefore,
\begin{equation*}
  \frac{1}{\delta}
  \I{(\bd\dom)_\delta\cap\dom\cap\{|F|\ge l\}}{|F|}{\lem} \le
  \frac{3}{8\pi l} \qmz{for all} \delta>0
\end{equation*}
which implies \reff{nm-s5-2}. Thus Proposition~\ref{nm-s9} implies e.g. for 
$\nu=\nu^{\rm int}$ that
\begin{equation*}
  \divv(\tf F)(\dom) = 
  \sI{\bd\dom}{\tf F\cdot\nu^\dom}{\dens_{\bd\dom}^{\rm int}}
\end{equation*}
for all $\tf\in\woinf{\edom}$.
\end{example}

Now let us consider the important case of bounded vector fields 
$F\in\cD\cM^\infty(\edom)$ on sets $\dom\subset\edom$ with finite perimeter
in some more detail. Here we use normal measures that are based on 
mollifications, in particular that from 
Examples~\ref{nm-int} and \ref{nm-ext}.
Though some of the assertions in the next proposition are already known from
the literature (cf. \cite{chen_2020}, \cite{chen_divergence-measure_2005},
\cite{chen_gauss-green_2009}, \cite{comi-payne}, \cite{comi-torres},
\cite{silhavy_2005} and the remarks
below), we include them not only for completeness but also to show their
relation to 
the new results. In the proof we essentially use arguments that are 
based on the theory developed here. Recall that, in addition to
the subsequent results, Proposition~\ref{nm-s9} is applicable to $\dom$ if 
$\hm(\dom_\delta)$ is bounded for $|\delta|$ small.

\begin{proposition} \label{nm-s10a}    
Let $\edom\subset\R^n$ be open and bounded, let 
$\dom\subset\edom$ have finite perimeter, 
let $F\in\cD\cM^\infty(\edom)$, and let
$\tilde\nu^{\rm int}$, $\tilde\nu^{\rm ext}$ be the normal measures from 
Examples~\ref{nm-int} and \ref{nm-ext}.

\bgl
\item
If $\dom\csubset\edom$, then there are normal trace functions 
$f^{\rm int},f^{\rm ext}\in\cL^\infty(\mbd\dom,\hm)$ with
$\|f^{\rm int}\|_\infty,\,\|f^{\rm ext}\|_\infty\le\|F\|_{\bd\dom}$ such that
\begin{equation} \label{nm-s10a-1}
  \divv(\tf F)(\mint\dom) = \sI{\bd\dom}{\tf F}{\tilde\nu^{\rm int}}
  = \I{\mbd\dom}{\tf f^{\rm int}}{\hm} \,,
\end{equation}
\begin{equation} \label{nm-s10a-2}
  \divv(\tf F)(\mbd\dom\cup\mint\dom) = \sI{\bd\dom}{\tf F}{\tilde\nu^{\rm ext}} 
  =  \I{\mbd\dom}{\tf f^{\rm ext}}{\hm} 
\end{equation}
for all $\tf\in\woinf{\edom}$ (cf. \reff{semi-norm} for 
$\|\cdot\|_{\bd\dom}$).
If $\dom$ is open we also have
\begin{equation} \label{nm-s10a-3}
  \divv(\tf F)(\dom) = \I{\mbd\dom}{\tf f^{\rm int}}{\hm} -
 \I{\mint\dom\cap\bd\dom}{\tf}{\divv F}
\end{equation}
and if $\dom$ is closed
\begin{equation} \label{nm-s10a-4}
  \divv(\tf F)(\dom) = \I{\mbd\dom}{\tf f^{\rm ext}}{\hm} +
 \I{\mext\dom\cap\bd\dom}{\tf}{\divv F}
\end{equation}
for all $\tf\in\woinf{\edom}$. 

\item
Let $\dom\subset\edom$ be open with $\hm(\bd\dom\cap\mint\dom)<\infty$. 
Then $\hat F\in\cD\cM^\infty(\R^n)$ for the extension $\hat F$ of $F$ 
with zero outside of $\dom$. 
There is also some normal measure $\nu\in\bawl{U}^n$ such that 
$(\bd\dom)_\delta\cap\dom$ is an aura for all $\delta>0$ and
\begin{equation} \label{nm-s10a-6}
  \divv(\tf F)(\dom) = \sI{\bd\dom}{\tf F}{\nu}
\end{equation}
for all $\tf\in\woinf{\edom}$. Moreover there is 
$f\in\cL^\infty(\bd\dom\!\setminus\!\mext\dom,\hm)$ 
satisfying $\|f\|_\infty\le c\|F\|_{\bd\dom}$ for some $c>0$ depending merely
on $\dom$ such that
\begin{equation} \label{nm-s10a-6a}
  \divv(\tf F)(\dom) = 
  \I{\bd\dom\setminus\mext\dom}{\tf f}{\hm}
\end{equation}
for all $\tf\in C^1(\R^n)$ and
\begin{equation} \label{nm-s10a-6b}
  \reme{f\hm}{(\bd\dom\cap\mint\dom)} = 
  \reme{-\divv F}{(\bd\dom\cap\mint\dom)} \,.
\end{equation}
If $\hat F$ is continuous on a neighborhood of $\dom$, then 
$(\divv F)(\bd\dom\cap\mint\dom)=0$ and
\begin{equation} \label{nm-s10a-7}
  \divv(\tf F)(\mint\dom) =
  \divv(\tf F)(\dom) = \I{\mbd\dom}{\tf F\cdot \nu^\dom}{\hm}
\end{equation}
for all $\tf\in C^1(\R^n)$.
\el
\end{proposition}

\begin{remark} \label{nm-s10b}     
(1) The measures $\reme{f^{\rm int}\hm}{\mbd\dom}$ and 
$\reme{f^{\rm ext}\hm}{\mbd\dom}$
are the Radon measures related to $F\tilde\nu^{\rm int}$ and 
$F\tilde\nu^{\rm ext}$, respectively, according to
Proposition~\ref{acm-ram}. Notice that we can only replace 
$\tilde\nu^{\rm int}$ and $\tilde\nu^{\rm ext}$ with a Radon measure in
\reff{nm-s10a-1} and \reff{nm-s10a-2}, respectively, 
if $F$ is continuous (cf. Proposition~\ref{acm-ram}).
With the approximating sequences
$\tilde\chi_k^{\rm int}$, $\tilde\chi_k^{\rm ext}$ corresponding to
$\tilde\nu^{\rm int}$, $\tilde\nu^{\rm ext}$, respectively, we have that
\begin{equation}  \label{nm-s10b-1}
  -F\cdot D\tilde\chi_k^{\rm int}\lem \overset{*}{\wto}
  \reme{f^{\rm int}\hm}{\mbd\dom}
  \qmq{and} 
  -F\cdot D\tilde\chi_k^{\rm ext}\lem \overset{*}{\wto}
  \reme{f^{\rm ext}\hm}{\mbd\dom}
\end{equation}
as weak$^*$ limits for Radon measures. This way we basically
recover several results from the literature about inner and outer traces
on sets of finite perimeter (cf. e.g. 
Chen-Torres-Ziemer \cite[p.~275, 281]{chen_gauss-green_2009},
Chen-Comi-Torres \cite[p.~106]{chen-comi}, 
Comi-Payne \cite[p.~194, 200]{comi-payne}).
Though the sequences in \reff{nm-s10b-1} slightly differ from 
the sequences
\begin{equation*}
  -2\chi_\dom F\cdot D\psi_k\lem  \qmq{and} 
  -2\chi_{\edom\setminus\dom} F\cdot D\psi_k\lem\,,
\end{equation*}
that are usually used in the literature (up to sign),
the weak$^*$ limit is always the same, since in each case
\reff{nm-s10a-1} and \reff{nm-s10a-2} have to be valid for a set of $\tf$ that 
is dense in $C_c(\edom)$. This in particular means that
\begin{equation} \label{nm-s10b-2}
  \I{\mbd\dom}{\tf f^{\rm int}}{\hm} = 
  -\lim_{k\to\infty} \I{\edom}{\tf F\cdot D\tilde\chi_k^{\rm int}}{\lem} =
  -\lim_{k\to\infty} \I{\edom}{2\tf\chi_\dom F\cdot D\psi_k}{\lem}
\end{equation}
for all $\tf\in C_c(\edom)$ and analogously for the other case. 

(2) The proof of Proposition~\ref{nm-s10a} readily shows that the first
assertion is also true for $\tilde\nu^\canon$ with 
\begin{equation*}
  \tfrac{1}{2}\divv(\tf F)(\mbd\dom) + \divv(\tf F)(\mint\dom)
  = \sI{\bd\dom}{\tf F}{\tilde\nu^\canon}
  = \I{\mbd\dom}{\tf f^\canon}{\hm}
\end{equation*}
for some $f^\canon\in\cL^\infty(\mbd\dom,\hm)$ with
$\|f^\canon\|_\infty\le\|F\|_{\bd\dom}$. Moreover, a check of the relevant proofs 
also shows that we can replace the requirement 
$\dom\csubset\edom$ in Proposition~\ref{nm-s10a} (1)
with the restriction to functions $\tf\in\cW^{1,\infty}_c(\edom)$ that have 
compact support in $\edom$. In this case a further restriction to
$\tf$ with
\begin{equation*}
   \tf=0 \qmq{on} \edom\setminus (\bd\dom)_\delta \qmz{for some} \delta >0 
\end{equation*}
is sufficient for the treatment of a Gauss-Green formula (cf.
Remark~\ref{tt-s15}). This way we recover the results from Comi-Payne
\cite[p.~203]{comi-payne}. 

(3) The proof shows that \reff{nm-s10a-6a} and \reff{nm-s10a-6b}
with some $f\in\cL^1(\bd\dom\!\setminus\!\mext\dom,\hm)$ 
are already true without the $\hm$-bound as long as 
the extension of $F$ with zero has divergence measure. 
\end{remark}

\begin{remark} \label{nm-s10c}   
Let $\edom\subset\R^n$ be an open bounded set, let $\dom\subset\edom$ be open 
with finite perimeter such that also $\mint\dom\subset\edom$,
let $F\in\cD\cM^1(\edom)$, and let $\sigma_F^{\rm int}$ be a Radon measure on 
$\mbd\dom$ with $\sigma_F^{\rm int}\wac\hm$ such that 
\begin{equation} \label{nm-s10c-1}
  \divv(\tf F)(\mint\dom) = \I{\mbd\dom}{\tf}{\sigma_F^{\rm int}}
\end{equation}
for all $\tf\in\woinf{\edom}\cap C(\ol\edom)$ (cf. \reff{nm-s10a-1}). 
With the disjoint decomposition
\begin{equation}  \label{nm-s10c-2}
  \mint\dom = \dom \cup (\mint\dom\cap\bd\dom)   \,  
\end{equation}
and since $\dom$ and $\mint\dom$ merely differ by an $\lem$-null set
(cf. \cite[p.~222]{evans}), we easily obtain that 
\begin{eqnarray*}
  \divv(\tf F)(\dom) 
&=&
  \I{\dom}{F\cdot D\tf}{\lem}
  +\I{\dom}{\tf}{\divv F} \\
&=&
  \I{\mint\dom}{F\cdot D\tf}{\lem}
  +\I{\mint\dom}{\tf}{\divv F} 
  -\I{\mint\dom\cap\bd\dom}{\tf}{\divv F}        \\
&=& 
  \I{\mbd\dom}{\tf}{\sigma_F^{\rm int}}
  -\I{\mint\dom\cap\bd\dom}{\tf}{\divv F} \,.
\end{eqnarray*}
Since $\mbd\dom\subset\bd\dom$, we use the Radon measure 
\begin{equation} \label{nm-s10c-3}
  \sigma_F := 
  \sigma_F^{\rm int} - \reme{\divv F}{(\mint\dom\cap\bd\dom)}
\end{equation}
to get
\begin{equation} \label{nm-s10c-4}
  \divv(\tf F)(\dom) = \I{\bd\dom\setminus\mext\dom}{\tf}{\sigma_F}
\end{equation}
for all $\tf\in\woinf{\edom}\cap C(\ol\edom)$. For $F\in\cD\cM^\infty(\dom)$
an $\hm$-integrable density $f$ of $\sigma_F$ is available by the
Radon-Nikodym theorem if $\hm$ is $\sigma$-finite on 
$\bd\dom\setminus\mext\dom$ (cf. \cite[p.~14]{ambrosio}).
This is obviously the case if $\hm(\bd\dom\setminus\mext\dom)<\infty$
and leads to \reff{nm-s10a-6a} with an essentially bounded density.

This shows that we can easily transform
\reff{nm-s10c-1} to a Gauss-Green formula where the full topological boundary
is incorporated into the boundary term, which is favorable in the case of
cracks along some inner boundary. Notice, however, that the form with a normal
measure as e.g. in \reff{nm-s10a-1} or in \reff{nm-s7-1} contains $F$
explicitly in the boundary term and doesn't require a 
normal trace function $f$ on $\bd\dom$. Moreover the larger class of functions
$\woinf{\edom}$ for $\tf$, that do not have to be continuous up
to $\bd\dom$, allows more flexibility for the investigation near $\bd\dom$
(cf. the discussion surrounding \reff{tt-e4}). In Example~\ref{dt-ex4a}
this allows e.g., in contrast to \reff{nm-s10c-4}, 
to describe the behavior on both sides of the inner boundary
$\bd\dom\cap\mint\dom$ by the normal components of $F$, which is
of course helpful for the description of cracks.

In the special situation where $\dom$ is open and bounded and 
where the vector field $F\in\cD\cM^\infty(\dom)$ is such that 
$\hat F\in\cD\cM(\R^n)$ 
for its extension $\hat F$ with zero, we can choose some open bounded
$\hat\edom\subset\R^n$ with $\dom\csubset\hat\edom$. Applying 
Proposition~\ref{nm-s10a}~(1) to $\hat F$ 
we obtain \reff{nm-s10a-1} and, by the previous
arguments, we get \reff{nm-s10c-4}. Using also the second assertion of 
Proposition~\ref{nm-s10a}, 
we basically recover the Gauss-Green formulas stated by  
Chen-Li-Torres \cite[p.~248]{chen_2020}. Let us still mention that 
a vector field $F\in\cD\cM^\infty(\dom)$ can be extended by zero 
as required above if $\dom$ is a bounded open set of finite perimeter
satisfying 
$\hm(\bd\dom\cap\mint\dom)<\infty$ (cf. \cite[p.~242]{chen_2020}).
\end{remark}

The previous remark and the definition of divergence measure in \reff{tt-e3}
lead to a simple extension criterion for open 
$\dom$ with finite perimeter and bounded vector fields $F$
(cf. also \cite[p.~104]{chen-comi},
\cite{comi-payne}, \cite{chen_2020}).

\begin{proposition} \label{nm-s10e}   
Let $\dom\subset\R^n$ be an open bounded set with finite perimeter and let
$F\in\cD\cM^\infty(\dom)$. Then the extension $\hat F$ of $F$ by zero 
belongs to $\cD\cM^\infty(\R^n)$ if and only if there is a Radon measure
$\sigma_F$ supported on $\bd\dom\setminus\mext\dom$ such that
\begin{equation} \label{nm-s10e-1}
  \I{\dom}{F\cdot D\tf}{\lem} + \I{\dom}{\tf}{\divv F} =
  \I{\bd\dom\setminus\mext\dom}{\tf}{\sigma_F}
\end{equation}
for all $\tf\in C^1(\R^n)$.
\end{proposition}
\noi
This basically means that the extension of $F$ by zero has also divergence
measure if and only if there is a Gauss-Green formula where the boundary term
is related to a Radon measure.

\begin{proof}    
First, let $F$ has an extension $\hat F$ as assumed. 
Then we can argue as in Remark~\ref{nm-s10c}
to get \reff{nm-s10e-1}. If otherwise \reff{nm-s10e-1} is satisfied, then
we have for the extension $\hat F$ and for 
\begin{equation*}
  \hat\sigma := \reme{(\divv F)}{\dom} - \sigma_F
\end{equation*}
that
\begin{equation*}
  \I{\R^n}{\hat F\cdot D\tf}{\lem} = -\I{\R^n}{\tf}{\hat\sigma}
\end{equation*}
for all $C^1_c(\R^n)$. Thus $\hat F\in\cD\cM^\infty(\R^n)$. 
\end{proof}

\begin{remark} \label{nm-s10d}     
Let us discuss the general case where $\edom\subset\R^n$ is open and
bounded, $\dom\in\bor{\edom}$, and $F\in\cD\cM^\infty(\edom)$. Then we
always have
\begin{equation} \label{nm-s10d-0}
     \divv(\tf\F)(\dom) = \I{(\bd\dom)_\delta\cap\edom}{\tf}{\lF} + 
            \I{(\bd\dom)_\delta\cap\edom}{D\tf}{\mF}  
\end{equation}
for any $\delta>0$ and all $\tf\in\woinf{\edom}$ by \reff{dt-s1-1}. 
If $\dom=\edom$ then we have case (L) with 
$\cor{\lF}\subset\bd\dom$ and we can remove $\delta$ in the first integral and
if $\dom\csubset\edom$ then we have case (C) and can identify $\lF$ with a
$\sigma$-measure on $\bd\dom$ (cf. Corollary~\ref{cor:cases}). 
If merely $\dom\in\bor{\edom}$ with $\lem(\dom\setminus\Int\dom)=0$ but in
addition 
\begin{equation} \label{nm-s10d-1}
  \sup_{0<\delta<\tilde\delta}\hm(\bd\dom_{-\delta}) < \infty  
  \qmq{for some} \tilde\delta>0\,,
\end{equation}
then we have
\begin{equation} \label{nm-s10d-2}
  \divv(\tf\F)(\dom) = \sI{\bd\dom}{\tf}{\lF} 
\end{equation}
(cf. Proposition~\ref{dt-s4}). Notice that \reff{nm-s10d-1} is valid
for a 
large class of sets of finite perimeter and, if $\dom\csubset\edom$,
it allows the application of the normal measures 
$\nu^{\rm int}$, $\nu^{\rm ext}$, and $\nu^\canon$
that are constructed by means of a distance function.
This leads to Gauss-Green formulas containing $F$ and the normal field
$\nu^\dom$ explicitly in the boundary term (cf. Proposition~\ref{nm-s9}
and the subsequent discussion).
If \reff{nm-s10d-1} is not available, then we can apply the normal measures 
$\tilde\nu^{\rm int}$, $\tilde\nu^\canon$, and $\tilde\nu^{\rm ext}$, that are
based on mollified functions, to any set of finite perimeter
$\dom\csubset\edom$. Then the right hand side in \reff{nm-s10d-2} has the form
$\sI{\bd\dom}{\tf F}{\nu}$ with $\nu$ being one of those normal
measures. But here we cannot explicitly incorporate a normal field
(cf. Theorem~\ref{nm-s7} and Corollary~\ref{nm-s10}). 
We have to realize that \reff{nm-s10d-1},
though it excludes some ``exotic'' sets of finite perimeter, leads
to more structural information about $\lF$ in the boundary term. 

Notice that, even if merely $\dom\subset\edom$, it is possible 
in the general variants \reff{nm-s10d-0} and \reff{nm-s10d-2} to account
precisely for the boundary points belonging to $\dom$, which
is relevant if parts of the boundary belong to the support of $\divv F$. 
If we want to have more structural information in a Gauss-Green formula 
``up to the boundary'', i.e. if we e.g. have $\dom=\edom$, 
then merely the normal measure 
$\nu^{\rm intc}$ is available. It gives 
\begin{equation*}
  \divv(\tf F)(\Int\dom) = \sI{\bd\dom}{\tf F}{\nu^{\rm intc}} =
  \sI{\bd\dom}{\tf F\cdot\nu^\dom}{\dens^{\rm intc}_{\bd\dom}} 
\end{equation*}
for some density measure $\dens^{\rm intc}_{\bd\dom}$ of the type as in
Proposition~\ref{tt-s3}. The missing condition $\dom\csubset\edom$ is
compensated by a compact 
support of the $\chi_k^{\rm intc}$ of the approximating sequence. 
However, a uniform bound on $\hm(\bd\dom_{-\delta})$ for small $\delta>0$
similar to \reff{nm-s10d-1},
that excludes certain sets of finite perimeter, is needed. 
But notice that the bound implies that the extension $\hat F$ of 
any $F\in\cD\cM^\infty(\dom)$ by zero has divergence measure on $\R^n$
(cf. \cite[p.~7, 11]{chen_2020}). Consequently the problem is equivalent to 
that for $\hat F$ with $\dom\csubset\hat\edom$ for some open $\hat\edom$. 
\end{remark}

\begin{proof+}{ of Proposition~\ref{nm-s10a}}     
For (1) we start with the case of the interior normal measure 
$\tilde\nu^{\rm int}$ and use 
\begin{equation*}
  \chi:=\tilde\chi^{\rm int} = \chi_{\mint\dom} \qmq{and}
  \chi_k:=\tilde\chi_k^{\rm int}
\end{equation*}
(cf. Example~\ref{nm-int}). 
The first equation in \reff{nm-s10a-1} follows from 
Corollary~\ref{nm-s10} with
\begin{equation*}
  \I{\bd\dom}{\chi}{\divv(\tf F)} + \divv(\tf F)(\Int\dom) =
  \divv(\tf F)(\mint\dom) \,.
\end{equation*}
Obviously $F\tilde\nu^{\rm int}\in\bawl{U}^n$ and 
$\tf\in\woinf{\edom}$ is continuous on $\ol\dom$.
Then, by Proposition~\ref{acm-ram}, there is some related Radon measure
$\tilde\sigma^{\rm int}_F$. 
Since $\chi_k\in\woinf{\edom}$ has compact support, 
the definition of divergence measure for $\chi_k F$ gives
\begin{equation*}
  \I{\edom}{\tf F\cdot D\chi_k}{\lem} =
  - \I{\edom}{\chi_k\tf}{\divv F} -\I{\edom}{\chi_k F\cdot D\tf}{\lem} 
\end{equation*}
for all $\tf\in C^1(\edom)$.
By $\chi_k\to\chi_{\mint\dom}$ $\hm$-a.e. and $\divv F\wac\hm$
(cf. \cite[p.21]{silhavy_2005}), 
\begin{equation*}
   \lim_{k\to\infty}\I{\edom}{\tf F\cdot D\chi_k}{\lem} =
  - \I{\mint\dom}{\tf}{\divv F} -\I{\mint\dom}{F\cdot D\tf}{\lem} \,. 
\end{equation*}
Hence, we do not need a subsequence in \reff{bd-s5-1} and get 
\begin{equation*}
  \lim_{k\to\infty}\I{\edom}{\tf F\cdot D\chi_k}{\lem} =
  -\, \sI{\bd\dom}{\tf F}{\tilde\nu^{\rm int}} =
  - \I{\bd\dom}{\tf}{\tilde\sigma^{\rm int}_F} \,.
\end{equation*}
Since $C^1_c(\edom)$ is dense in $C_c(\edom)$ we have
$- F\cdot D\chi_k\lem \overset{*}{\wto} \tilde\sigma_F^{\rm int}$ 
as weak$^*$ limit of Radon measures. With the first equation in 
\reff{nm-s10a-1} and the results in \cite[p.~194, 200]{comi-payne}
we get the second equation in \reff{nm-s10b-2}. For $\psi_k$ from
\reff{nm-mol} we also have  
$2\chi_{\dom^c}D\psi_k\lem\overset{*}{\wto}D\chi_\dom$ by 
\cite[p.~189]{comi-payne} and
$|D\psi_k|\lem\overset{*}{\wto}|D\chi_\dom|$ by \reff{nm-mol2}.
Recall also \reff{nm-mol1}. 
Then, for given $\eps>0$ and any open $B\csubset\edom$ 
with $|D\chi_\dom|(\bd B)=0$, there is some
$\tf_\eps\in C_c(B)$ with $\|\tf_\eps\|_\infty\le 1$ such that,
with arguments similar as in Comi-Payne \cite[p.~196]{comi-payne}, 
\begin{eqnarray}
  |\tilde\sigma_F^{\rm int}|(B) -\eps
&\le&
  \I{\bd\dom\cap B}{\tf_\eps}{\tilde\sigma_F^{\rm int}}  
\: = \:
  -\lim_{k\to\infty} \I{\edom\cap B}{\tf_\eps F\cdot D\chi_k}{\lem}  
  \nonumber\\
&=&
  -\lim_{k\to\infty} 2\I{\edom\cap B\cap\dom}{\tf_\eps F\cdot D\psi_k}{\lem} 
  \nonumber\\
&\le&  \label{nm-s10a-20}
  2\|F\|_{\cL^\infty(B)}\lim_{k\to\infty} \I{\edom\cap B\cap\dom}{|D\psi_k|}{\lem} 
  \nonumber \\
&=& 
  2\|F\|_{\cL^\infty(B)}\lim_{k\to\infty} 
  \Big(\I{\edom\cap B}{|D\psi_k|}{\lem} - 
  \I{\edom\cap B\cap\dom^c}{|D\psi_k|}{\lem} \Big)  \nonumber \\
&\le&
  2\|F\|_{\cL^\infty(B)}\lim_{k\to\infty} 
  \Big(\I{\edom\cap B}{|D\psi_k|}{\lem} - 
  \Big| \I{\edom\cap B}{\chi_{\dom^c}D\psi_k}{\lem} \Big| \Big)  
  \nonumber\\
&=& 
  2\|F\|_{\cL^\infty(B)}  \big((|D\chi_\dom|(B) - 
  \tfrac{1}{2}|D\chi_\dom(B)|\big)    \nonumber\\
&=& \label{nm-s10a-21}
  2\|F\|_{\cL^\infty(B)} (\reme{\hm}{\mbd\dom})(B)
  \big( 1 - \tfrac{1}{2} \tfrac{|D\chi_\dom(B)|}{|D\chi_\dom|(B)}\big)  
\end{eqnarray}
(for $|D\chi_\dom|(B)\ne 0$ in the last line). 
By the arbitrariness of $\eps>0$ we can remove it. 
Then
\begin{equation*}
  |\tilde\sigma_F^{\rm int}|(B) \le 
  2\|F\|_{\cL^\infty(B)}(\reme{\hm}{\mbd\dom})(B) \,.
\end{equation*}
Since $\reme{\hm}{\mbd\dom}$ is a Radon measure, we conclude that
\begin{equation*}
  \supp{\tilde\sigma_F^{\rm int}} \subset \mbd\dom  \qmq{and} 
  \tilde\sigma_F^{\rm int} \wac \reme{\hm}{\mbd\dom}\,. 
\end{equation*}
Thus, the Radon-Nikodym theorem implies that $\tilde\sigma_F^{\rm int}$ has an 
integrable density $f^{\rm int}$. 
For fixed $x\in\rbd\dom$ and balls $B_r(x)$ we have that 
$|D\chi_\dom|(\bd B_r(x))=0$
for $\cL^1$-a.e. $r>0$ and, thus, \reff{nm-s10a-21} is valid for such 
$B=B_r(x)$. Then, by the definition of the reduced boundary,  
the fraction in \reff{nm-s10a-21} tends to 1 as 
$r\downarrow 0$ (cf. \cite[p.~154]{ambrosio}). Hence, 
by Lebesgue's differentiation theorem, 
\begin{equation} \label{nm-s10a-9}
  f^{\rm int}\in\cL^\infty(\mbd\dom,\reme{\hm}{\mbd\dom}) \qmq{with}
  \|f^{\rm int}\|_\infty\le \|F\|_{\bd\dom}
\end{equation}
and, consequently,
\begin{equation*}
  \I{\bd\dom}{\tf}{\tilde\sigma^{\rm int}_F} =
  \I{\mbd\dom}{\tf}{\tilde\sigma^{\rm int}_F} =
  \I{\mbd\dom}{\tf f^{\rm int}}{\hm} \,
\end{equation*}
for all $\tf\in\woinf{\edom}$. 
This verifies the first case.
We can argue analogously in the second case by using 
that $\chi=\chi_{\mbd\dom\cap\mint\dom}$ implies 
\begin{equation*}
  \I{\bd\dom}{\chi}{\divv(\tf F)} + \divv(\tf F)(\Int\dom) =
  \divv(\tf F)(\mbd\dom\cap\mint\dom) 
\end{equation*}
and that $\chi_k:=\tilde\chi_k^{\rm ext}\to\chi_{\mbd\dom\cap\mint\dom}$
$\hm$-a.e. on $\edom$. If $\dom$ is open we argue as in Remark~\ref{nm-s10c}
to get \reff{nm-s10a-3}. If $\dom$ is closed we start with \reff{nm-s10a-2},
we use
\begin{equation*}
  \dom = \mint\dom \cup \mbd\dom \cup (\mext\dom\cap\bd\dom) \,,
\end{equation*}
and we argue analogously to get \reff{nm-s10a-4}. 

For (2) we first observe that one always has the disjoint decomposition
\begin{equation*}
  \R^n = \mint\dom \cup \mbd\dom \cup \mext\dom \,.
\end{equation*}
With $\mbd\dom\subset\bd\dom$ we get 
\begin{equation} \label{nm-s10a-15}
  \mbd\dom\cup(\mint\dom\cap\bd\dom) = \bd\dom\!\setminus\!\mext\dom
\end{equation}
(cf. \cite[p.~49, 50]{pfeffer}). Since $\lem(\dom)>0$ and 
$\hm(\bd\dom\cap\mint\dom)<\infty$ we can apply 
Chen-Li-Torres \cite[Theorem~3.1]{chen_2020}
to get sets $\dom_k\csubset\dom$ with finite
perimeter such that 
\begin{equation*}
  \sup_{k\in\N}\hm(\mbd\dom_k)<\infty \qmq{and}
  \bd\dom_k\subset(\bd\dom)_{\frac{1}{2k}} \zmz{for all} k\in\N
\end{equation*}
where the set inclusion is not explicitly stated there but it follows  
from the proof 
(cf. also \cite[p.~208]{comi-payne}).  Then we define
\begin{equation*} 
  \chi_k := \chi_{\dom_k}*\eta_{\delta_k} 
  \qmq{with some} \delta_k\in \big(0,\tfrac{1}{2k}\big)
\end{equation*}
where $\eta_\eps$ is the standard mollifier supported on 
$B_\eps(0)$ (cf. also \reff{nm-mol}). Obviously we have 
$\supp{\chi_k}\subset(\bd\dom)_{\frac{1}{k}}$ and, 
by $\dom_k\csubset\dom$, we can choose $\delta_k>0$ so small
that even $\supp{\chi_k}\csubset\dom$. Moreover
$\chi_k\in W^{1,\infty}(\edom)$ with 
\begin{equation*}
  \chi_k\to 1 \zmz{on} \dom\,, \quad
  \chi_k\to 0 \zmz{otherwise.}
\end{equation*}
Since $|D\chi_{\dom_k}|(\bd\edom)=0$, we can apply \reff{nm-mol2a} to $\dom_k$.
Therefore we can assume $\delta_k>0$ to be so small that 
\begin{equation*}
  |D\chi_k|(\edom) = \I{\edom}{|D\chi_k|}{\lem} \le 
  |D\chi_{\dom_k}|(\edom) +1 = \hm(\mbd\dom_k) + 1 \,.
\end{equation*}
Consequently, we have that $\{\chi_k\}$ is an approximating sequence for the 
good approximation $\chi=\chi_{\Int\dom}$ of $\chi_\dom$ and each
$(\bd\dom)_{\frac{1}{k}}\cap\dom$ is an aura of the associated 
normal measure $\nu$ (cf. Theorem~\ref{bd-s5}).
Then \reff{nm-s10a-6} follows from Theorem~\ref{nm-s7} and
Corollary~\ref{nm-s10}. 

By $\hm(\bd\dom\cap\mint\dom)<\infty$ 
we have $\hat F\in\cD\cM^\infty(\R^n)$ for the extension $\hat F$ of
$F$ with zero (cf. \cite[p.~11]{chen_2020}). Let us choose some open bounded 
$\hat\edom\subset\R^n$ such that $\dom\csubset\hat\edom$. Then we can apply 
assertion (1) to get \reff{nm-s10a-1} for $\hat F$ and all $\tf\in C^1(\R^n)$.
According to Remark~\ref{nm-s10c}
this can be transformed to \reff{nm-s10c-4} and, by
$\divv F = \divv(\hat F)$ on $\dom$, we get 
\begin{equation*}
  \divv(\tf F)(\dom) = \I{\bd\dom\setminus\mext\dom}{\tf}{\sigma_F}
\end{equation*}
for some Radon measure $\sigma_F$ supported on $\bd\dom\setminus\mext\dom$.
The assumptions and \reff{nm-s10a-15} imply that
$\hm(\bd\dom\setminus\mext\dom)<\infty$. Therefore $\sigma_F$ has an 
$\hm$-integrable density $f$ by the Radon-Nikodym theorem and we obtain 
\reff{nm-s10a-6a} (cf. the arguments following \reff{nm-s10c-4}).
Then \reff{nm-s10a-6b} follows
from \reff{nm-s10c-3}. In order to show that $f$ is essentially bounded 
we argue as in Chen-Li-Torres \cite[p.~18]{chen_2020} but without 
smooth approximation of $F$. We fix some $B_r(x)$ with 
$x\in\bd\dom\setminus\mext\dom$ and $r>0$. Moreover we observe that the sets 
$\dom_k\csubset\dom$ from above can be chosen such that
they have smooth boundary, and can therefore assumed to be open, and that 
\begin{equation*}
  |D\chi_{\dom_k}|(B_r(x)) = (\reme{\hm}{\mbd\dom_k})(B_r(x)) \le 
  c \hm\big((\bd\dom\setminus\mext\dom)\cap B_r(x)\big)
\end{equation*}
for some $c>0$ depending merely on dimension $n$
(cf. \cite[p.~237, (5.8)]{chen_2020}) and also 
\cite[p.~208]{comi-payne}).  
Now, for some  
$\tf\in C^1(\R^n)$ with $\|\tf\|_\infty\le 1$ and supported on $B_r(x)$,
we use dominated convergence, $\dom_k=\mint\dom_k$, and 
assertion (1) for $\dom_k$ with density functions $f_k$ 
to get
\begin{eqnarray*}
  \I{\bd\dom\setminus\mext\dom}{\tf f}{\hm}
&=& 
  \divv(\tf F)(\dom) \\
&=&
  \I{\dom}{\tf}{\divv F} + \I{\dom}{F\cdot D\tf}{\lem}  \\
&=&
  \lim_{k\to\infty} \Big(
  \I{\dom_k}{\tf}{\divv F} + \I{\dom_k}{F\cdot D\tf}{\lem} \Big) \\
&=&
  \lim_{k\to\infty} \I{\mbd\dom_k}{\tf f_k}{\hm} \\
&\le&
  \lim_{k\to\infty} \|F\|_{\cL^\infty((\bd\dom)_{1/k})} 
  \I{\mbd\dom_k}{\tf}{\hm}  \\
&\le&
  \|F\|_{\bd\dom}\lim_{k\to\infty}
  (\reme{\hm}{\mbd\dom_k})(B_r(x))   \\
&\le& 
  c\|F\|_{\bd\dom} \hm\big((\bd\dom\setminus\mext\dom)\cap B_r(x)\big) \,.
\end{eqnarray*}
Hence, by Lebesgue's differentiation theorem, 
$\|f\|_{\bd\dom}\le c\|F\|_{\bd\dom}$.  

If $F$ is continuous, we use  
$\chi_{\Int\dom}=\chi_\dom$ to get \reff{bd-s5-3} with the normal measure
$\nu$ derived above (that we assume to be extended with zero on $\R^n$).  
Then \reff{nm-s10a-6} and \reff{bd-s5-3} with the vector function $\tf F$ 
give the second equality in \reff{nm-s10a-7}. 
Since we also have \reff{nm-s10a-6a}, the measure
$\reme{f\hm}{(\bd\dom\setminus\mext\dom)}$ 
has to vanish outside $\mbd\dom$. By \reff{nm-s10a-15} this means that
$(\divv F)(\bd\dom\cap\mint\dom)=0$ (cf. also \cite[p.~198]{comi-payne}).
Using \reff{nm-s10c-2} we get the first equality in \reff{nm-s10a-7}.
\end{proof+}

Next we consider an example with high oscillations that is occasionally
discussed in the literature (cf. 
Chen-Torres-Ziemer \cite[p.~258]{chen_gauss-green_2009},
Comi-Payne \cite[p.~216]{comi-payne}).

\begin{example} \label{nm-ex7}  
Let $\dom=\{y<x\}\cap\{|x|+|y|\le 1\}\subset\R^2$ 
and $F:\R^2\to\R^2$ given by 
\begin{equation*}
  F(x,y) = 
  \big( \sin\big(\tfrac{1}{x-y} \big),\sin\big(\tfrac{1}{x-y}\big)\,\big) 
  \qmz{for} x\ne y \,.
\end{equation*}
Obviously $\divv F=0$ on $\dom$ and, thus, $F\in\cD\cM^\infty(\dom)$.
From Proposition~\ref{nm-s9} we get 
\begin{equation} \label{nm-ex7-1}
  \divv(\tf F)(\dom) = 
  \sI{\bd\dom}{\tf F\cdot\nu^\dom}{\dens_{\bd\dom}^{\rm int}}
\end{equation}
for all $\tf\in\woinf{\dom}$. We can divide $\dom$ by its diagonals 
into two large triangles $\dom^l_j$ and two small triangles $\dom^s_j$
($j=1,2$). Then we have
\begin{equation*}
  F\cdot\nu^\dom=0 \zmz{on the} \dom^l_j \qmq{and}
  F\cdot\nu^\dom=\pm 2\sin\big(\tfrac{1}{x-y}\big) \zmz{on the} \dom^s_j .
\end{equation*}
Therefore we can disregard the large sides in \reff{nm-ex7-1} and the integral
vanishes for constant $\tf$. Moreover,
using \reff{nm-s9-2}, the $\sigma$-measure $\sigma_F$ associated to
$F\cdot\nu^\dom\dens_{\bd\dom}^{\rm int}$ according to
Proposition~\ref{acm-ram} is suported on $\bd\dom$ and we have
\begin{equation*}
  \sigma_F=0 \zmz{on} \bd\dom\cap\bd\dom^l_j \qmq{and}
  \sigma_F=\pm 2\sin\big(\tfrac{1}{x-y}\big)\cH^1 
  \zmz{on} \bd\dom\cap\bd\dom^s_j  \quad (j=1,2)\,.
\end{equation*}
Hence
\begin{equation} \label{nm-ex7-2}
  \divv(\tf F)(\dom) = 
  2\I{\bd\dom\cap\bd\dom^s_1}{\tf\sin\big(\tfrac{1}{x-y}\big)}{\cH^1} -
  2\I{\bd\dom\cap\bd\dom^s_2}{\tf\sin\big(\tfrac{1}{x-y}\big)}{\cH^1}
\end{equation}
for all $\tf\in C^1(\ol\dom)$.
Analogously we can argue for $\ol\dom$ with $\dens_{\bd\dom}^{\rm ext}$.
This in particular implies that $\reme{\divv F}{\{x=y\}}=0$.
Notice that the problem can be treated analogously for other $\dom$ touching
the set $\{x=y\}$.
\end{example}

\subsection{Sobolev functions and BV functions}
\label{sf}

In this section we show that the previous results are 
applicable to Sobolev functions and BV functions.
Recall that, for $\f\in\cB\cV(\edom)$ and $\dom\in\bor{\edom}$, 
\begin{equation*}
  \tf \to \divv(\tf\f)(\dom) = 
  \I{\dom}{\f\divv\tf}{\lem} + \I{\dom}{\tf}{D\f}\,
\end{equation*}
is a trace on $\bd\dom$ over $\woinf{\edom,\R^n}$ according to
Proposition~\ref{tt-s5}. As direct consequence of Theorem~\ref{dt-s1} 
we provide a general Gauss-Green formula for BV functions 
by arguing as in the proof of Proposition~\ref{tt-s5}.
Notice that we have to take $m=n$ for the particular cases (L) and (C) 
in this section. 

\begin{theorem} \label{sf-s1}
Let $\edom\subset\R^n$ be open and bounded, let $\dom\in\bor{\edom}$, let
$\delta>0$, and let $\f\in\cB\cV(\edom)$. Then there exist measures
\begin{equation*}
  \lf \in \bawl{\edom}^n \qmq{and} \mf \in \bawl{\edom}^{n\times n}
\end{equation*}
with $\cor{\lf},\: \cor{\mf} \subset \ol{(\bd\dom)_\delta\cap\edom}$
such that
\begin{equation}\label{sf-s1-1}
  \df{Tf}{\tf} =
  \divv{(\tf\f)}(\dom) = \I{(\bd\dom)_\delta\cap\edom}{\tf}{\lf} + 
  \I{(\bd{\dom})_\delta\cap\edom}{D\tf}{\mf} \,
\end{equation}
for all $\tf\in\woinf{\edom,\R^n}$ with
$T:\cB\cV(\edom)\to\woinf{\edom,\R^n}^*$ from
Proposition~\ref{tt-s5}.
In the particular cases with $\K=\bd\dom$ we have in addition

{\rm (L):} $\cor{\lf}\subset\bd\dom$ and \reff{sf-s1-1} becomes
\begin{equation*}
   \divv(\tf\f)(\dom) = \sI{\bd\dom}{\tf}{\lf} + 
            \I{(\bd\dom)_\delta\cap\edom}{D\tf}{\mf} \,.   
\end{equation*}

{\rm (C):} $\lf$ corresponds to a Radon measure $\sigma_f$ with 
$\supp{\sigma_f}\subset\bd\dom$ such that 
\begin{equation*}
   \divv(\tf\f)(\dom) = \I{\bd\dom}{\tf}{\sif} + 
            \I{(\bd\dom)_\delta\cap\edom}{D\tf}{\mf} \,.   
\end{equation*}
\end{theorem}
\noi
We call $(\lf,\mf)$, representing an element of $\woinf{\edom,\R^n}^*$,
{\it normal trace} of $f$ on $\bd\dom$. 
Notice that $D\tf\,d\me_\f$ in \reff{sf-s1-1} has to be taken as scalar
product of matrices. Since
$\cW^{1,1}(\edom)\subset\cB\cV(\edom)$, the result covers these Sobolev
functions (cf. also Remark~\ref{tt-s6}). 

\begin{proof} 
For vector fields $F_k\in\dmo{\edom}$ related to $f\in\cB\cV(\edom)$ 
as in the proof of Proposition~\ref{tt-s5}, we can apply Theorem~\ref{dt-s1}
to get measures 
$\lambda_{F_k}\in\bawl{\edom}$ and $\me_{\F_k}\in\bawl{\edom}^n$ with 
$\cor{\lambda_{F_k}},\;\cor{\me_{\F_k}}\subset \ol{(\bd\dom)_\delta\cap\edom}$
such that for all functions $\tf=(\tf^1,\dots,\tf^n)\in\woinf{\edom,\R^n}$ 
\begin{equation*}
  \big(D_{x_k}(\tf^k\f)\big)(\dom) = 
  \I{(\bd\dom)_\delta\cap\edom}{\tf^k}{\lambda_{F_k}} + 
  \I{(\bd{\dom})_\delta\cap\edom}{D\tf^k}{\me_{\F_k}} \,
\end{equation*}
with the scalar measure
\begin{equation*}
  D_{x_k}(\tf^k\f) = fD_{x_k}\tf^k\lem + \tf^kD_{x_k}f
\end{equation*}
(cf. Remark~\ref{tt-s6} and \cite[p.118]{ambrosio}).
The sum over $k$ gives the first statement
and the particular cases follow directly from these in Theorem~\ref{dt-s1}.  
\end{proof}

As in Proposition~\ref{dt-s2} we can characterize the cases where 
$\mu_f=0$ is possible and where the measures $\lf$, $\mf$ can be chosen
independent of $\delta$.

\begin{proposition} \label{sf-s2}
Let $\edom\subset\R^n$ be open and bounded, let $\dom\in\bor{\edom}$, 
and assume that $\f\in\cB\cV(\edom)$.

\bgl
\item 
In Theorem~{\rm \ref{sf-s1}} we can choose $\lf$, $\mf$ with
$\cor{\lf}$, $\cor{\mf} \subset \bd\dom$, i.e. independent of $\delta$,  
if and only if 
\begin{equation} \label{bv-monster}
  \liminf\limits_{\delta\downarrow 0} 
  \sup_{\substack{\tf \in \woinf{\edom,\R^n}\\
       \|\tf_{|(\bd\dom)_\delta\cap\edom}\|_{\cW^{1,\infty}}\le 1}} 
  \divv{(\port{\bd\dom}_\delta\tf\f)}(\dom) < \infty 
\end{equation}
with $\port{\bd\dom}_\delta$ as in \reff{eq:port_fun}. In this case 
\reff{sf-s1-1} becomes 
\begin{equation}\label{sf-s2-1}
   \divv(\tf\f)(\dom) = \sI{\bd\dom}{\tf}{\lf} + 
            \sI{\bd\dom}{D\tf}{\mf} \,.  
\end{equation}

\item
In Theorem~{\rm \ref{sf-s1}} we can choose $\mf=0$
for $\delta>0$ if and only if 
\begin{equation}\label{bv-special}
\sup_{\substack{\tf\in\woinf{\edom,\R^n}\\
                \|\tf_{|(\bd\dom)_\delta\cap\edom}\|_{\cL^\infty}\le 1}} 
  \divv{(\tf\f)}(\dom) < \infty \,. 
\end{equation}

\item
In Theorem~{\rm \ref{sf-s1}} we can choose $\mf=0$ and $\lf$ with
$\cor{\lf}\subset \bd\dom$ if and only if
\begin{equation*}
  \liminf_{\delta\downarrow 0}
  \sup_{\substack{\tf\in\woinf{\edom,\R^n}\\
            \|\tf_{|(\bd\dom)_\delta\cap\edom}\|_{\cL^\infty}\le 1}} 
  \divv{(\tf\f)}(\dom) < \infty \,. 
\end{equation*}
\el
\end{proposition}

\begin{proof} 
We use the notation of the proof of Theorem~\ref{sf-s1}. We have for all $k$ 
that
\begin{equation*}
  \port{\bd\dom}_\delta\tf^k F_k = \port{\bd\dom}_\delta\tf\f
  \qmq{if} \tf^j=0 \tx{ for all } j\ne k \,.
\end{equation*}
Hence \reff{bv-monster} is equivalent to \reff{dm-monster} for all $F_k$ and,
by Proposition~\ref{dt-s2}, this is equivalent to 
$\cor{\lambda_{F_k}},\;\cor{\me_{\F_k}}\subset\bd\dom$ for all $k$. But this
gives (1). Analogously we derive (2) and (3) from Proposition~\ref{dt-s2}. 
\end{proof}

Arguing as in the previous proof we can transfer Lemma~\ref{lem:monster2},
Proposition~\ref{dt-s4}, and Proposition~\ref{dt-s5} to BV functions. Let us
briefly rephrase these results for completeness and for the convenience of the
reader.  

\begin{lemma} \label{sf-s3}      
Condition \reff{bv-monster} in Proposition~{\rm \ref{sf-s2}} is equivalent to
each of the following two conditions
\begin{equation*}
  \liminf \limits_{\delta\downarrow 0}  
  \sup_{\substack{\tf \in\woinf{\edom,\R^n}\\
                  \|\tf_{|(\bd\dom)_\delta\cap\edom}\|_{\cW^{1,\infty}}\le 1}}
  \I{(\bd{\dom})_\delta \cap \dom}{\tf\f D\port{\bd\dom}_\delta}{\lem} 
  < \infty \,,
\end{equation*}
\begin{equation*}
  \liminf\limits_{\delta\downarrow 0}  
  \sup_{\substack{\tf\in\woinf{\edom,\R^n}\\
                  \|\tf_{|(\bd\dom)_\delta\cap\edom}\|_{\cW^{1,\infty}}\le 1}}
     \; \frac{1}{\delta}\I{(\dnhd{(\bd{\dom})}{\delta}\setminus
        \dnhd{(\bd{\dom})}{\frac{\delta}{2}})\cap \dom}
        {\tf\f D\distf{\bd{\dom}}}{\lem} < \infty \,. 
\end{equation*}
Moreover, $\Int{\dom}=\emptyset$ implies \reff{bv-monster}.
\end{lemma}

\begin{proposition}\label{sf-s4}      
Let $\edom \subset \Rn$ be an open and bounded set, let $\dom\in\bor{\edom}$
satisfy $\lem(\dom\setminus\Int\dom)=0$, and assume that
$\f\in\cB\cV(\edom)$. If 
\begin{equation} \label{sf-s4-1}
  \liminf\limits_{\delta\downarrow 0} 
  \I{\dom}{\big|\f D\port{\bd\dom}_\delta\big|}{\lem} < \infty \,,
\end{equation}
then we can take $\mf=0$ and $\lf$ with $\cor{\lf}\subset \bd\dom$
in Theorem~{\rm \ref{sf-s1}}. We have \reff{sf-s4-1} if
$f$ is bounded and if there is some $\tilde\delta>0$ such that
\begin{equation}\label{sf-s4-2}
  \sup_{\delta\in(0,\tilde\delta)}  
  \cH^{n-1}(\bd\dom_{-\delta})  < \infty \,.
\end{equation}
If $\dom$ has finite perimeter, then \reff{dt-s4-3} for some $c,r>0$
ensures \reff{sf-s4-2}.
\end{proposition}

\begin{proposition}\label{sf-s5}    
Let $\edom\subset\R^n$ be open and bounded, and let $f\in\cB\cV(\edom)$.
\bgl
\item
If $\dom \in \bor{\edom}$ is such that any $\tf\in\woinf{\edom}$ has a
continuous extension onto~$\cl\dom$, if $\lem(\dom\setminus\op{int}\dom)=0$,
and if there are $\const >0$ and $\tilde\delta>0$ such that 
\begin{equation} \label{sf-s5-0}
  \|\tf_{|\bd\dom}\|_{\op{Lip}(\bd\dom)} \le
  c \|\tf\|_{\woinf{(\bd\dom)_\delta\cap\edom}} \qmq{for all}
  \tf\in\woinf{\edom},\; \delta\in(0,\tilde\delta)\,,
\end{equation}
then \eqref{bv-monster} is satisfied.

\item If $\dom$ is open with Lipschitz boundary, then 
\eqref{bv-monster} is satisfied.
\el
\end{proposition}
 
\noi
Notice that we do not have to consider $\tf\in\woinf{\edom,\R^n}$ in the
previous proposition.

Simple examples of Gauss-Green formulas for Sobolev
functions going beyond the classical ones can be obtained from 
Example~\ref{dt-ex2} or \ref{dt-ex3} if we take the first component of the
vector field $F$ as function $f\in\cW^{1,1}(\edom)$.
Let us now provide a Sobolev function on a set $\dom\subset\R^2$ 
of finite perimeter
where the precise representative is not $\cH^1$-integrable on
$\bd\dom$. This certainly prevents a usual Gauss-Green formula.
By the derivation of measures $\lambda_f$ and $\mu_f$ for \reff{sf-s1-1} we
demonstrate how more general Gauss-Green formulas can be obtained. 

\begin{example}\label{sf-ex2}  
Using the unit square
\begin{equation*}
  \tilde\dom := \{ (x,y)\in\R^2 \mid x,y\in(0,1)\}  \,,  
\end{equation*}
\begin{equation*}
  x_k:=\tfrac1k\,, \; y_k:=\big(\tfrac{1}{k}\big)^{\!\frac54}\,,
  \qmq{segments} \Gamma_k := \{x_k\}\times [0,y_k] \,,  
\end{equation*}
and the closed convex sets
\begin{equation*}
  \tilde\dom_k := \op{conv}\{\Gamma_{2k},\Gamma_{2k+1}\} \,,
\end{equation*}
we define the open sets
\begin{equation}\label{sf-ex2-1}
  \edom := \dom := \tilde\dom \setminus \bigcup_{k=1}^\infty \tilde\dom_k \,.
\end{equation}
Since $\sum_ky_k$ is finite, $\dom$ has finite perimeter. 
Let us consider $f\in\cW^{1,1}(\dom)$ with 
\begin{equation*}
  f(x,y):=\frac{1}{|(x,y)|_\infty^{\:\frac14}} 
\end{equation*}
where $|\cdot|_\infty$ is the $\infty$-norm 
(one even has $f\in\cW^{1,p}(\dom)$ for $1\le p<\frac{8}{5}$). Clearly
\begin{equation*}
  \I{\bd\dom}{|f|}{\ham^1} \ge 
  \I{\bigcup_k\Gamma_k}{|f|}{\ham^1} = 
  \sum_k \big(\tfrac{1}{k}\big)^{\!\frac54} k^{\frac14} =
  \sum_k \tfrac1k
\end{equation*}
and, thus, $f\not\in \cL^1(\bd\dom,\cH^1)$. 
Therefore we do not have a Gauss-Green
formula in the classical sense.

Nevertheless we can apply Theorem~\ref{sf-s1} and \reff{sf-s1-1} is valid. 
Let us first realize that for any two points in $\bar\dom$ there is a
connecting curve in $\dom$ with length less than three times its distance.
Therefore any $\tf\in\woinf{\dom,\R^2}$ is Lipschitz
continuous on $\dom$ and, thus, continuously extendable onto $\bar\dom$.
Moreover we readily verify \reff{dt-s5-0} for any $\tilde\delta>0$. 
This implies \reff{bv-monster} by Proposition~\ref{sf-s5}. Since 
$(\bd\dom)_\delta\cap\dom$ is bounded path connected with $\bd\dom$
for any $\delta>0$, we have case (C) by Proposition~\ref{prop:cases}. 
Consequently there is a vector-valued Radon measure $\sif$ supported on
$\bd\dom$ and a measure $\mf\in\op{ba}(\dom,\cB(\dom),\cL^2)^{2\times 2}$ with
$\cor{\mf}\subset\bd\dom$ such that
\begin{equation*}
  \divv(\tf f)(\dom) = 
  \I{\bd\dom}{\tf}{\sif} + \sI{\bd\dom}{D\tf}{\mf}
\end{equation*}
for all $\tf\in\woinf{\dom,\R^2}$ (identified with their extension onto
$\bar\dom$). 

Let us analyze how $\sif$ and $\mf$, that are not unique,
could look like. First we restrict our attention to smooth 
$\tf=(\tf_1,\tf_2)\in C^1(\R^2,\R^2)$ such that we can also consider
$\mf$ as Radon measures (cf. Proposition~\ref{acm-ram}). 
Notice that $\tilde\dom_k$ has Lipschitz boundary and we decompose
\begin{equation*}
  \bd{\tilde\dom_k} = 
  \Gamma_{2k}\cup\Gamma_{2k+1}\cup\Gamma_k^0\cup\Gamma_k^1
\end{equation*}
where $\Gamma_k^0$ is the part on the $x$-axis and $\Gamma_k^1$ is the
opposite part. Then, taking $\divv(\tf f)$ as a measure, we get 
\begin{equation} \label{sf-ex2-2}
  \divv{(\tf f)}(\dom) =
  \divv{(\tf f)}(\tilde\dom) -
  \sum_{k=1}^\infty \divv{(\tf f)}(\tilde\dom_k) 
\end{equation}
for all $\tf\in C^1(\R^2,\R^2)$. Obviously 
$f\in \cL^1(\bd{\tilde\dom},\cH^1)$ and, by the usual Gauss-Green formula,
\begin{equation}\label{sf-ex2-3}
  \divv{(\tf f)}(\tilde\dom) = 
  \I{\bd{\tilde\dom}}{\f\tf\cdot\nu^{\tilde\dom}}{\cH^1} \,.
\end{equation}
Moreover
\begin{eqnarray}
\hspace*{-5mm}  
\divv{(\tf f)}(\tilde\dom_k) 
&=& 
  \I{\bd{\tilde\dom_k}}{\f\tf\cdot\nu^{\tilde\dom_k}}{\cH^1} \nonumber \\
&=& \label{sf-ex2-4}
  \I{\Gamma_{2k+1}}{\f\tf_1}{\cH^1} -
  \I{\Gamma_{2k}}{\f\tf_1}{\cH^1} +
  \I{\Gamma_k^0\cup\Gamma_k^1}{\f\tf\cdot\nu^{\tilde\dom_k}}{\cH^1} \,.
\end{eqnarray}
If we plug \reff{sf-ex2-3} and \reff{sf-ex2-4} into \reff{sf-ex2-2},
the integrals on $\Gamma_k^0$ will be canceled out.
Thus the integral in \reff{sf-ex2-3} has to be evaluated merely on 
$\tilde\Gamma:=\bd{\tilde\dom}\setminus\big(\bigcup_{k}\Gamma_k^0\big)$.
Hence the Radon measure 
\begin{equation*}
  \sif^0:=\f\nu^{\tilde\dom}\reme{\cH^1}{\tilde\Gamma}
\end{equation*}
can be taken as part of $\sif$. The integral on $\Gamma_k^1$ in 
\reff{sf-ex2-4} is related to the Radon measure
\begin{equation*}
  \sif^{k1}:=\f\nu^{\tilde\dom_k}\reme{\cH^1}{\Gamma_k^1}\,.
\end{equation*}
Since the sum over $k\in\N$ is also a Radon measure, we can take
it with the opposite sign also as part of~$\sif$. 
It remains to consider the integrals on 
$\Gamma_k$ in \reff{sf-ex2-4}. Here we have to realize that
the sum of the Radon measures $\f\reme{\cH^1}{\Gamma_k}$
is not bounded (cf. above) and, thus, not a Radon measure.
Therefore we have to transform these integrals on
$\Gamma_k$ in a suitable way
(it would be sufficient to take finitely many of these measures with the
correct sign as part of $\sif$ and to transform merely the rest, which would 
lead to measures $\mf$ that differ from that derived below).
With
\begin{equation*}
  g(x,y):=x^{-\frac14}y\,, \quad 
  G^0(x,y) :=  \begin{pmatrix} 0 & -x^{-\frac14} y \\ 0 & 0
             \end{pmatrix}\,,
\end{equation*}
integration by parts gives
\begin{eqnarray}
  \I{\Gamma_{k}}{\f\tf_1}{\cH^1} 
&=& 
  \int_0^{y_k}g_y(x_k,y)\tf_1(x_k,y)\,dy  \nonumber \\
&=&
  - \int_0^{y_k}g(x_k,y)\tf_{1,y}(x_k,y)\,dy
  + \big[ g(x_k,y)\tf_1(x_k,y) \big]_0^{y_k}  \nonumber \\
&=&
  - x_k^{-\frac14} \int_0^{y_k} y \tf_{1,y}(x_k,y)\,dy
  + g(x_k,y_k) \tf_1(x_k,y_k)  \nonumber \\
&=&
  \int_0^{y_k} G^0(x_k,y): D\tf(x_k,y) \,dy 
  + \tfrac1k \tf_1(x_k,y_k)  \nonumber \\
&=& \label{sf-ex2-5}
   \I{\Gamma_k}{G^0:D\tf}{\cH^1} + \tfrac1k \tf_1(x_k,y_k) 
\end{eqnarray}
(here $:$ denotes the scalar product of matrices).
Now, taken with the correct sign, 
the measures $G^0\reme{\cH^1}{\Gamma_k}$ with total variation 
\begin{equation*}
  \big|G^0\reme{\cH^1}{\Gamma_k}\big|(\Gamma_k) =
  \int_0^{y_k}x_k^{-\frac{1}{4}}y\,dy = 
  \tfrac{1}{2}\big(\tfrac{1}{k}\big)^{\frac{9}{4}}
\end{equation*}
could contribute to $\mu_f$ 
and the Dirac measures $\frac1k\delta_{(x_k,y_k)}$ might contribute to
$\sif$. While the sum of the $G^0\reme{\cH^1}{\Gamma_k}$ gives a finite
measure, the sum of the Dirac measures is not finite and needs some further
transformation. We therefore consider (cf. \reff{sf-ex2-4})
\begin{eqnarray*}
&&
\hspace*{-20mm}
  \I{\Gamma_{2k}}{\f\tf_1}{\cH^1} - \I{\Gamma_{2k+1}}{\f\tf_1}{\cH^1}
  \\
&=&
  \I{\Gamma_{2k}}{G^0:D\tf}{\cH^1} - 
  \I{\Gamma_{2k+1}}{G^0:D\tf}{\cH^1}  \\[1mm]
&& 
 + \, \tfrac{1}{2k}\tf_1(x_{2k},y_{2k}) 
 - \tfrac{1}{2k+1}\tf_1(x_{2k+1},y_{2k+1})\,.
\end{eqnarray*}
Using $\frac1k=\frac{1}{k+1}+\frac{1}{k(k+1)}$ and
\begin{equation*}
  v_k:=\frac{(\delt{x}_k, \delt{y}_k)}{\big|(\delt{x}_k, \delt{y}_k)\big|}
  \qmq{with}
  (\delt{x}_k, \delt{y}_k) := 
  (x_{2k},y_{2k})-(x_{2k+1},y_{2k+1})
\end{equation*}
we get
\begin{eqnarray*}
&& \hspace*{-20mm}
  \tfrac{1}{2k}\tf_1(x_{2k},y_{2k}) - \tfrac{1}{2k+1}\tf_1(x_{2k+1},y_{2k+1}) \\
&=&
  \tfrac{1}{2k+1} \big(\tf_1(x_{2k},y_{2k}) - \tf_1(x_{2k+1},y_{2k+1})\big)
  + \tfrac{1}{2k(2k+1)} \tf_1(x_{2k},y_{2k})  \\
&=&
  \tfrac{1}{2k+1} \int_0^1
  D\tf_1(x_{2k+1}+t\!\delt{x}_k,y_{2k+1}+t\!\delt{y}_k) \cdot 
  (\delt{x}_k, \delt{y}_k) \, dt \\
&&
  + \, \tfrac{1}{2k(2k+1)} \tf_1(x_{2k},y_{2k})  \\
&=&
  \tfrac{1}{2k+1} \I{\Gamma_k^1}
    {v_k\cdot D\tf_1}{\cH^1} + \tfrac{1}{2k(2k+1)} \tf_1(x_{2k},y_{2k}) \,.
\end{eqnarray*}
Here the sum of the Dirac measures $\frac{1}{2k(2k+1)}\delta_{(x_{2k},y_{2k})}$
is finite and, thus, it can contribute to $\sif$. The measures 
$G_k^1\reme{\cH^1}{\Gamma_k^1}$ with
\begin{equation*}
  G_k^1(x,y) :=  \tfrac{1}{2k+1}
  \begin{pmatrix} v_{k,1} & v_{k,2} \\ 0 & 0
             \end{pmatrix}
\end{equation*}
have total variation
\begin{equation*}
  \tfrac{1}{2k+1}\cH^1(\Gamma_k^1) \le 
  \tfrac{1}{2k+1}\big(\tfrac{2}{2k}-\tfrac{2}{2k+1}\big) 
  = \tfrac{1}{k(2k+1)^2} \,.
\end{equation*}
Thus the sum of these measures is finite and can 
contribute to $\mf$.
Summarizing we obtain \reff{sf-s1-1} for $\tf\in C^1(\R^2,\R^2)$
with Radon measures 
\begin{equation*}
  \sif := \sif^0 -
  \sum_{k=1}^\infty  \big( \f\nu^{\tilde\dom_k}\reme{\cH^1}{\Gamma_k^1} + 
           \tfrac{1}{2k(2k+1)}\delta_{(x_{2k},y_{2k})} \big)
\end{equation*}
\begin{equation*}
   \mf := \sum_{k=1}^\infty G^0\reme{\cH^1}{\Gamma_{2k}}
             - G^0\reme {\cH^1}{\Gamma_{2k+1}}   
             + G_k^1\reme{\cH^1}{\Gamma_k^1} \,.
\end{equation*}

Let us now consider the extension of $\mu_f$ to 
$\tf\in\cW^{1,\infty}(\dom,\R^2)$. Since an extension by Hahn-Banach will not
be unique in general, we have to argue more carefully. 
Instead of an approximation by smooth functions as in Example~\ref{dt-ex5}
we want to present a more direct argument. First we consider 
$\Gamma_k$ with odd index $k$ and set 
\begin{equation*}
  \Gamma_{kt} := \big\{ (x_k+t,y) \;\big|\; y\in[0,y_k] \big\}\,, \quad 
  t\in\R\,.
\end{equation*}
Since $\f$, $\tf$ are continuous on a neighborhood of $\Gamma_k$ intersected
with $\dom$ and since 
$\tf(x_k+t,\cdot)$ is absolutely continuous on $\Gamma_{kt}$ 
for almost all small $t>0$, we argue similar as for \reff{sf-ex2-5} to get 
\begin{eqnarray*}
  \I{\Gamma_{k}}{\f\tf_1}{\cH^1} 
&=& 
  \lim_{t\downarrow 0}\frac1t \int_0^t\I{\Gamma_{k\tau}}{\f\tf_1}{\cH^1} \, d\tau \\
&=&
  \lim_{t\downarrow 0}\frac1t \int_0^t \Big(
  \I{\Gamma_{k\tau}}{G^0:D\tf}{\cH^1} + 
  \frac{y_k}{(x_k+\tau)^{\frac14}}\tf_1(x_k+\tau,y_k) \Big)\, d\tau \\
&=&
  \lim_{t\downarrow 0}\frac1t \int_0^t \I{\Gamma_{k\tau}}{G^0:D\tf}{\cH^1}\, d\tau
  + \tfrac1k\tf_1(x_k,y_k) \,.
\end{eqnarray*}
Obviously 
\begin{equation*}
  \Phi \to \lim_{t\downarrow 0}\frac1t 
  \int_0^t \I{\Gamma_{k\tau}}{\Phi:G^0}{\cH^1}\, d\tau
\end{equation*}
is a linear continuous functional on $\cL^\infty(\dom,\R^{2\times 2})$
depending on values of $\Phi$ near $\Gamma_k$.
Hence there are matrix-valued measures
$\mu_k^0\in\op{ba}(\dom,\cB(\dom),\cL^2)^{2\times 2}$   
with core $\Gamma_k$ such that
\begin{equation*}
  \sI{\Gamma_{k}}{\Phi}{\mu_k^0} =
  \lim_{t\downarrow 0}\frac1t \int_0^t \I{\Gamma_{k\tau}}{G^0:\Phi}{\cH^1}\, d\tau
\end{equation*}
for all $\Phi\in\cL^\infty(\dom,\R^{2\times 2})$. By the special form of $G^0$
the only non-vanishing component of the measure $\mu_k^0$ is $(\mu_k^0)_{12}$.
More precisely, with $\mu_{\Gamma_k}$ from Proposition~\ref{tt-s3} where
$\alse=\dom$ and 
$\gamma(\delta)=\delta$, we obtain from \reff{tt-s3-1} that
$(\mu_k^0)_{12}= \big(\!-x^{-\frac14} y\big)\mu_{\Gamma_k}$, i.e. a kind of
weighted density measure. Clearly,
\begin{equation*}
  \I{\Gamma_{k}}{\f\tf_1}{\cH^1} = 
  \sI{\Gamma_{k}}{D\tf_1}{\mu_k^0} + \tfrac1k\tf_1(x_k,y_k)
\end{equation*}
for all $\tf\in\cW^{1,\infty}(\dom,\R^2)$ and odd $k$. Analogously we get
measures $\mu_k^0$ for even $k$. 
Since the $\mu_k^0$ have the same total
variation as the $G^0\reme {\cH^1}{\Gamma_k}$, their sum 
is again a measure and we can replace the Radon measures 
$G^0\reme{\cH^1}{\Gamma_k}$ with the $\mu_k^0$ to handle  
$\tf\in\cW^{1,\infty}(\dom,\R^2)$. 
By analogous arguments we can construct measures 
$\mu_k^1\in\op{ba}(\dom,\cB(\dom),\cL^2)^{2\times 2}$ with core $\Gamma_k^1$
that can replace the Radon measures $G_k^1\reme {\cH^1}{\Gamma_k^1}$.
This way we can finally replace the former $\mu_f$ with
\begin{equation*}
   \mu_f := \sum_{k=1}^\infty \mu_{2k+1}^0 - \mu_{2k}^0 -\mu_k^1 
\end{equation*}
to get a Gauss-Green formula \reff{sf-s1-1} for all 
$\tf\in\cW^{1,\infty}(\dom,\R^2)$.

Notice that we didn't use for the derivation of the Gauss-Green formula
that \reff{bv-monster} is satisfied. Since $\sif$ and $\mf$ do not depend on
$\delta$, we could thus also apply
Proposition~\ref{sf-s2} after this derivation to get \reff{bv-monster}.
Recall that the choice of $\sif$ and $\mf$ is not unique, since an alternative
version with merely terms for $k>k_0$ in $\mf$ is possible for any $k_0\in\N$.
\end{example}

\medskip

Let us now discuss normal measures for Sobolev and BV functions. 
We start with a scalar variant of Proposition~\ref{nm-s5}.
\begin{proposition} \label{sf-s6}   
Let $\edom\subset\R^n$ be open, bounded, let $\dom\in\bor{\edom}$, 
let $f\in\cL^1(\dom)$, 
let $\nu$ be a normal measure related to a good approximation $\chi$
for $\chi_\dom$ and with approximating sequence 
$\{\chi_k\}$ satisfying
$|D\chi_k|\le\gamma k$ $\lem$-a.e. on $\dom$ for some $\gamma>0$, and let
$\au\subset\edom$ be an aura of $\nu$ as in \reff{bd-s5-0}.
If there is some $\tilde\delta>0$ such that 
\begin{equation}\label{sf-s6-1}
  \frac{1}{\delta}\I{(\bd\dom)_\delta\cap\au}{|f|}{\lem} 
  \qmq{is uniformly bounded for} 0<\delta<\tilde\delta \,,
\end{equation}
then $\tf f$ is $\nu$-integrable on $\edom$ for all 
$\tf\in\cL^\infty(\edom,\R^n)$. If, in addition, 
\begin{equation}\label{sf-s6-2}
  \lim_{k\to\infty}
  \frac{1}{\delta}\I{(\bd\dom)_\delta\cap\au\cap\{|f|\ge k\}}{|f|}{\lem} 
  = 0 \qmq{uniformly for $\delta\in(0,\tilde\delta)$\,,}
\end{equation}
then for each $\tf\in\cL^\infty(\edom,\R^n)$ 
there is a subsequence $\{\chi_{k'}\}$
such that
\begin{equation} \label{sf-s6-3}
   \lim_{k'\to\infty} \I{\edom}{f\tf\cdot D\chi_{k'}}{\lem} = 
  - \,\sI{\bd\dom}{f\tf}{\nu} \,.
\end{equation}
\end{proposition}
\noi
For the proof we apply Proposition~\ref{nm-s5} to vector fields $F_k$ as in
the proof of Theorem~\ref{sf-s1}. Analogously we obtain the subsequent results
from Theorem~\ref{nm-s7} and Proposition~\ref{nm-s9}.
\begin{theorem} \label{sf-s7}    
Let $\edom\subset\R^n$ be open and bounded, let $\dom\in\bor{\edom}$, let
$\nu$ be a normal measure of $\dom$ related to a good approximation $\chi$ for
$\chi_\dom$ with approximating sequence $\{\chi_k\}$, 
let $f\in\cB\cV(\edom)$ be $\nu$-integrable such that 
\reff{sf-s6-3} is satisfied, and let
$\chi_k\to\chi$ $Df$-a.e. on $\bd\dom$.
Then we have for all $\tf\in\woinf{\edom,\R^n}$ that
\begin{equation} \label{sf-s7-1}
  \I{\bd\dom}{\chi}{\divv(f\tf)} + \divv(f\tf)(\Int\dom) =
  \sI{\bd\dom}{f\tf}{\nu} \,.
\end{equation}
If $f\in\cB\cV(\edom)\cap\cL^\infty(\edom)$, then $f$ is $\nu$-integrable,
satisfies \reff{sf-s6-3}, and we have \reff{sf-s7-1} for all 
$\tf\in\woinf{\edom,\R^n}$.
\end{theorem}

\begin{proposition} \label{sf-s9}   
Let $\edom\subset\R^n$ be open and bounded, let $\dom\in\bor{\edom}$, let
$\nu^*$ be a normal measure where $*$ stands for {\rm int}, 
{\rm intc}, {\rm ext}, or
$\canon$, let $\chi^*$, $\chi^*_k$, and $\delta_k$ be related to $\nu^*$ 
as in the corresponding examples above, let  
$\dens^*_{\bd\dom}\in\bawl{\edom}$ be the related measure as in 
Proposition~\ref{nm-s9}, and let $\nu^\dom$ be the normal field from
\reff{nm-normal}. 
If $f\in\cB\cV(\edom)$ is $\nu^*$-integrable such that  
\reff{sf-s6-3} is satisfied and that $\chi_k^*\to\chi^*$ $Df$-a.e. on
$\bd\dom$, then we have for all 
$\tf\in\woinf{\edom,\R^n}$
\begin{eqnarray*} 
  \divv(f\tf)(\Int\dom) 
&=& 
  \sI{\bd\dom}{f \tf \cdot\nu^\dom}{\dens^{\rm int}_{\bd\dom}}  \,, \\
\divv(f\tf)(\Int\dom) 
&=& 
  \sI{\bd\dom}{f \tf \cdot\nu^\dom}{\dens^{\rm intc}_{\bd\dom}}  \,, \\
\divv(f\tf)(\ol\dom) 
&=& 
  \sI{\bd\dom}{f\tf \cdot\nu^\dom}{\dens^{\rm ext}_{\bd\dom}} \,, \\
  \tfrac{1}{2}\divv(f\tf)(\bd\dom) + \divv(f\tf)(\Int\dom) 
&=& 
  \sI{\bd\dom}{f\tf\cdot\nu^\dom}{\dens^\canon_{\bd\dom}} .
\end{eqnarray*}
\end{proposition}

\bigskip

For bounded functions on sets of finite perimeter we can transfer
Proposition~\ref{nm-s10a} (cf. also \cite[p.~25]{chen_2020}
and use \cite[p.~171, 177]{ambrosio} for \reff{sf-s14-7}).

\begin{proposition} \label{sf-s14}
Let $\edom\subset\R^n$ be open and bounded, let 
$\dom\subset\edom$ have finite perimeter, 
let $f\in\cB\cV(\edom)\cap\cL^\infty(\edom)$, and let
$\tilde\nu^{\rm int}$, $\tilde\nu^{\rm ext}$ be the normal measures from 
Examples~\ref{nm-int} and \ref{nm-ext}.

\bgl
\item
If $\dom\csubset\edom$, then there are vector-valued 
functions $F^{\rm int},F^{\rm ext}\in\cL^\infty(\mbd\dom,\hm)^n$ 
with $\|F^{\rm int}\|_\infty,\,\|F^{\rm ext}\|_\infty\le\|f\|_{\bd\dom}$ such that
\begin{equation} \label{sf-s14-1}
  \divv(\tf f)(\mint\dom) = \sI{\bd\dom}{f \tf}{\tilde\nu^{\rm int}}
  = \I{\mbd\dom}{\tf F^{\rm int}}{\hm}     \,,
\end{equation}
\begin{equation} \label{sf-s14-2}
  \divv(\tf f)(\mbd\dom\cup\mint\dom) =
   \sI{\bd\dom}{f\tf}{\tilde\nu^{\rm ext}}  =
  \I{\mbd\dom}{\tf F^{\rm ext}}{\hm}
\end{equation}
for all $\tf\in\woinf{\edom,\R^n}$
(cf. \reff{semi-norm} for $\|\cdot\|_{\bd\dom}$). If $\dom$ is open we also
have
\begin{equation} \label{sf-s14-3}
  \divv(\tf F)(\dom) = \I{\mbd\dom}{\tf F^{\rm int}}{\hm} -
 \I{\mint\dom\cap\bd\dom}{\tf}{Df}
\end{equation}
and if $\dom$ is closed
\begin{equation} \label{sf-s14-4}
  \divv(\tf F)(\dom) = \I{\mbd\dom}{\tf F^{\rm ext}}{\hm} +
 \I{\mext\dom\cap\bd\dom}{\tf}{Df}
\end{equation}
for all $\tf\in\woinf{\edom,\R^n}$. 

\item
Let $\dom\subset\edom$ be open with $\hm(\bd\dom\cap\mint\dom)<\infty$. Then 
$\hat f\in \cB\cV(\R^n)\cap\cL^\infty(\R^n)$ for the extension $\hat f$ of $f$
with zero outside $\dom$. 
There is also some normal measure $\nu\in\bawl{U}^n$ such that 
$(\bd\dom)_\delta\cap\dom$ is an aura for all $\delta>0$ and
\begin{equation} \label{sf-s14-5}
  \divv(\tf f)(\dom) = \I{\bd\dom}{f\tf}{\nu} 
\end{equation}
for all $\tf\in\woinf{\edom,\R^n}$. Moreover there is 
$F\in\cL^\infty\big(\bd\dom\setminus\mext\dom,\hm\big)^n$ 
with $\|F\|_\infty\le c\|f\|_{\bd\dom}$ for some $c>0$ depending merely on
$\dom$ such that
\begin{equation} \label{sf-s14-6}
  \divv(\tf f)(\dom) = \I{\bd\dom\setminus\mext\dom}{\tf F}{\hm}
\end{equation}
for all $\tf\in C^1(\R^n,\R^n)$ where
\begin{equation*}
  \reme{F\hm}{(\bd\dom\cap\mint\dom)} = 
  \reme{-Df}{(\bd\dom\cap\mint\dom)} \,.
\end{equation*}
\begin{equation*}
  \reme{F\hm}{(\mbd\dom)} = 
  \reme{f\nu^\dom\hm}{(\mbd\dom)}  \qmq{with}
\end{equation*}
\begin{equation} \label{sf-s14-7}
  f(x) = \lim_{r\downarrow 0} \mI{B_r(x)\cap\dom}{f}{\lem}
  \qmq{$\hm$-a.e. on $\mbd\dom$.}
\end{equation}
If $\hat f$ is continuous on a neighborhood of $\dom$, then 
$(Df)(\bd\dom\cap\mint\dom)=0$ and
\begin{equation} \label{nm-s14-8}
  \divv(\tf f)(\mint\dom) =
  \divv(\tf f)(\dom) = \I{\mbd\dom}{f\tf \cdot \nu^\dom}{\hm}
\end{equation}
for all $\tf\in C^1(\R^n,\R^n)$.
\el
\end{proposition}

\bigskip

Let us now consider the case where $\edom=\dom\subset\R^n$ has Lipschitz
boundary and let $f\in\cB\cV(\dom)$. By Corollary~\ref{cor:cases} we have case
(C) for $\K=\bd\dom$ and all $\delta>0$ and by Proposition~\ref{sf-s5} we have
\reff{bv-monster}. Hence Theorem~\ref{sf-s1} implies 
\begin{equation}\label{sf-e1}
   \divv(\tf\f)(\dom) = \I{\bd\dom}{\tf}{\sif} + 
            \sI{\bd\dom}{D\tf}{\mf} \qmq{for all}
  \tf\in\woinf{\dom,\R^n}
\end{equation}
where $\sif$ is a Radon measure supported on $\bd\dom$ and 
$\cor{\mf}\subset\bd\dom$. With the precise representative $\pr{f}$ 
according to Remark~\ref{pm-rem9}, we get from the literature that
\begin{equation}\label{sf-e2}
   \divv(\tf\f)(\dom) = \I{\bd\dom}{\pr{f}\, \tf\cdot\nu}{\cH^{n-1}} 
  \qmq{for all} \tf\in \woinf{\dom,\R^n}  
\end{equation}
(cf. \cite[p.~168]{pfeffer}, \cite[p.~177]{evans}).
This implies \reff{bv-special} for all $\delta>0$ and we can choose $\mf=0$
in \reff{sf-e1} by Proposition~\ref{sf-s2}. In this case we have 
\begin{equation*}
  \sif = \reme{\pr{f}\nu\cH^{n-1}}{\bd\dom}
\end{equation*}
since all $\tf\in C(\ol\dom,\R^n)$ can be uniformly approximated by 
functions in $\woinf{\dom,\R^n}$. Notice that this version of the Gauss-Green
formula with a $\sigma$-measure $\sif$ supported on the boundary of $\dom$ 
requires a pointwise trace function $\pr{f}$ on $\bd\dom$.
Let us now provide an alternative version where only the values of $f$ on
$\dom$ are used. For that we verify \reff{sf-s7-1} with 
$\chi=\chi_{\Int\dom}$ and $\nu=\nu^{\rm int}$.

\begin{theorem}\label{sf-s8}
Let $\edom=\dom\subset\R^n$ be open and bounded with Lipschitz boundary, let
$f\in\cB\cV(\dom)$, let $\nu^{\rm int}$ be the interior normal measure 
from Example~\ref{nm-int}, let $\dens^{\rm int}_{\bd\dom}$ be the measure 
according to Proposition~\ref{nm-s9}, and let $\nu^\dom$ be the normal field as
in \reff{nm-normal}. Then $f$ is $\nu^{\rm int}$-integrable and 
\begin{eqnarray}
  \divv(\tf\f)(\dom) 
&=& 
  \I{\dom}{\f\divv\tf}{\lem} + \I{\dom}{\tf}{D\f} \nonumber\\
&=&  \label{sf-s8-1}
  \sI{\bd\dom}{f\tf}{\nu^{\rm int}} 
\:=\: 
  \sI{\bd\dom}{f\tf\cdot\nu^\dom}{\dens^{\rm int}_{\bd\dom}}
\end{eqnarray}
for all $\tf\in \woinf{\dom,\R^n}$.
\end{theorem}

By $\cW^{1,1}(\dom)\subset\cB\cV(\dom)$ we have that \reff{sf-s8-1} is also 
valid for all Sobolev functions~$f$ (cf. Remark~\ref{tt-s6}).
For the proof 
we do not directly apply
Theorem~\ref{sf-s7} that is based on the technical condition
\reff{sf-s6-3}. We rather show the 
assertion directly by using an approximating sequence of $f$. 
Nevertheless, for the approximating sequence $\chi_k^{\rm int}$ of 
$\chi_{\Int\dom}$ related to $\nu^{\rm int}$ 
according to Example~\ref{nm-int}, we readily get
\reff{sf-s6-3} at the end of the proof. 
Notice that we can use $\chi_k^{\rm int}$ from \reff{nm-int-2}, since 
$\dom=\Int\dom$ and, thus, 
$\divv(\tf f)(\dom)=0$ if merely $\tf=0$ on $\bd\dom$ (cf.
Proposition~\ref{tt-s5}). Alternatively we can use
$\chi_k^{\rm intc}$ from
Example~\ref{nm-int} that might give a slightly different normal measure 
$\nu^{\rm intc}$, but the related integral in \reff{sf-s8-1} would give
the same values (roughly speaking, the integral performs a slightly different
averaging near $\bd\dom$ that doesn't change the result for functions entering
\reff{sf-s8-1}; cf. also Remark~\ref{nm-s8} (2)).  
In the proof we use arguments that are similar to those in the usual proof 
about traces (cf. \cite[p.~177-181]{evans}), however we have to work them out
in much more detail. But, before proving the theorem, let us 
still formulate a simple consequence.

\begin{corollary} \label{sf-s15}
 Let $\edom\subset\R^n$ be open and bounded, let $\dom\csubset\edom$ be open
 with Lipschitz boundary, let $f\in\cB\cV(\edom)$, 
let $\nu^{\rm ext}$ be the exterior normal measure from Example~\ref{nm-ext}, 
let $\dens^{\rm int}_{\bd\dom}$ be as in Proposition~\ref{nm-s9}, and let
the normal field $\nu^\dom$ be as in \reff{nm-normal}. Then $f$ is
$\nu^{\rm ext}$-integrable and  
\begin{eqnarray}
  \divv(\tf\f)(\ol\dom) 
&=& 
  \I{\ol\dom}{\f\divv\tf}{\lem} + \I{\ol\dom}{\tf}{D\f} \nonumber\\
&=&  \label{sf-s9-1}
  \sI{\bd\dom}{f\tf}{\nu^{\rm ext}} 
\:=\:
   \sI{\bd\dom}{f\tf \cdot\nu^\dom}{\dens^{\rm ext}_{\bd\dom}}
\end{eqnarray}
for all $\tf\in \woinf{\dom,\R^n}$. 
\end{corollary}

\noi
For the use of $\nu^{\rm ext}$ the function $f$ has to be given in a small
neighborhood of $\dom$. The results in \reff{sf-s8-1} and \reff{sf-s9-1} will
differ if $|Df|(\bd\dom)\ne 0$. But we have $|Df|(\bd\dom)=0$ 
for Sobolev functions and, thus, both formulas give the same in that case. 

\begin{proof}  
Let $\dom_0$ be open with Lipschitz boundary such that
$\dom\csubset\dom_0\csubset\edom$. We change $f$ to be zero on
$\edom\setminus\dom_0$ and still have $f\in\cB\cV(\edom)$
(cf. \cite[p.~183]{evans}). Moreover we can
assume that $\edom$ has Lipschitz boundary. From \reff{sf-s8-1} we get
\begin{equation*}
  \divv(\tf f)(\edom) = 0\,, \quad 
  \divv(\tf f)(\edom\setminus\ol\dom)=\sI{\bd\dom}{f\tf}{\nu^{\rm int}_{\dom^c}}
\end{equation*}
where $\nu^{\rm int}_{\dom^c}$ is the interior normal measure 
of $\dom^c$. By construction we readily see that 
$\nu^{\rm int}_{\dom^c}=-\nu^{\rm ext}$. Consequently
\begin{equation*}
  \divv(\tf f)(\ol\dom) = 
  \divv(\tf f)(\edom) - \divv(\tf f)(\edom\setminus\ol\dom) =
  \sI{\bd\dom}{f\tf}{\nu^{\rm ext}} 
\end{equation*}
which gives the first assertion. For the second one we use that
$\dens^{\rm ext}_{\bd\dom}=\dens^{\rm int}_{\bd(\dom^c)}$ by \reff{nm-s9-2} 
and that $\nu^\dom=-\nu^{\dom^c}$ by \reff{nm-normal}.
\end{proof}

\bigskip

\begin{proof+}{ of Theorem~\ref{sf-s8}}     
Let us fix some $\tf\in \woinf{\dom,\R^n}$ and recall Example~\ref{nm-int}.
Since $\dom$ has Lipschitz boundary, we have \reff{nm-int-1} for any sequence
$\delta_m\le\frac{1}{m}$. For the approximating sequence 
$\chi_m=\chi_m^{\rm int}$ of $\chi_{\op{int}\dom}$ according to
\reff{nm-int-2}, we set
$\psi_m:=1-\chi_m$. Notice that $\psi_m\tf\in\woinf{\dom,\R^n}$ for all $m$.
Since $\dom=\Int\dom$ and since $\psi_m$ equals $1$ on $\bd\dom$, we get from
Proposition~\ref{tt-s5} for all $g\in\cB\cV(\dom)$ and all $m\in\N$
\begin{eqnarray}
  \divv(\tf g)(\dom) 
&=&
  \divv(\psi_m\tf g)(\dom) \nonumber \\
&=& 
  \I{\dom}{g\divv(\psi_m\tf)}{\lem} + 
  \I{\dom}{\psi_m\tf}{Dg}  \nonumber \\
&=& \label{sf-s8-5}
  \I{\dom}{g\psi_m \divv\tf}{\lem} + 
  \I{\dom}{g\tf D\psi_m}{\lem} + 
  \I{\dom}{\psi_m\tf}{Dg} \,.
\end{eqnarray}
The first and the last integral in \reff{sf-s8-5} tend to zero for
$m\to\infty$ and, thus, 
\begin{equation} \label{sf-s8-5a}
  \lim_{m\to\infty}\I{\dom}{g\tf\cdot D\psi_m}{\lem} =
  \divv(\tf g)(\dom) \,.
\end{equation}
Choose now an approximating sequence $f_k\in\cB\cV(\dom)\cap C^\infty(\dom)$ 
for $f\in\cB\cV(\dom)$ with   
\begin{equation} \label{sf-s8-5b}
  f_k\to f \zmz{in} \cL^1(\dom)\,, \quad 
  |Df_k|(\dom)\to |Df|(\dom)\,, \quad
  Df_k\wto Df 
\end{equation}
where the last convergence denotes the weak$^*$ convergence in the sense of 
Radon measures with $(Df_k)(B)=\I{B}{Df_k}{\lem}$ for $B\in\bor{\dom}$
(cf. \cite[p.~54, 172, 175]{evans}). Then
\begin{eqnarray*}
  \lim_{k\to\infty} \divv(\tf f_k)(\dom) 
&=&
  \lim_{k\to\infty} \I{\dom}{f_k\divv\tf}{\lem} + \I{\dom}{\tf}{Df_k} \\
&=&
  \divv(\tf f)(\dom) \,.
\end{eqnarray*}
By Lemma~\ref{sf-s10} and Lemma~\ref{sf-s11} below we have that $f$ is
$\nu^{\rm int}$-integrable, that 
\begin{equation} \label{sf-s8-6}
  \lim_{k\to\infty}\sI{\bd\dom}{f_k\tf}{\nu^{\rm int}} =
  \sI{\bd\dom}{f\tf}{\nu^{\rm int}}\,, 
\end{equation}
and that 
\begin{equation} \label{sf-s8-7}
  \lim_{m\to\infty}\I{\dom}{f_k\tf\cdot D\psi_m}{\lem} =
  \sI{\bd\dom}{f_k\tf}{\nu^{\rm int}}  \qmz{for all} k\in\N
\end{equation}
(notice that Proposition~\ref{nm-s5} implies \reff{sf-s8-7} 
merely for a subsequence of $\{\psi_m\}$ that depends on $f_k$).
From \reff{sf-s8-5a} for $g=f_k$ and from \reff{sf-s8-7} we get
\begin{equation*}
   \divv(\tf f_k)(\dom) = \sI{\bd\dom}{f_k\tf}{\nu^{\rm int}} \qmz{for all}
   k\in\N\,. 
\end{equation*}
Consequently
\begin{eqnarray*}
&&
  \hspace*{-15mm}
  \Big| \divv(\tf\f)(\dom) - \sI{\bd\dom}{f\tf}{\nu^{\rm int}} \Big| \\
&\le&
  \big| \divv(\tf\f)(\dom) - \divv(\tf f_k)(\dom) \big| +
  \Big|\: \sI{\bd\dom}{f_k\tf}{\nu^{\rm int}} - 
  \sI{\bd\dom}{f\tf}{\nu^{\rm int}}\Big| \,.
\end{eqnarray*}
Since the right hand side tends to zero as $k\to\infty$, we get
\begin{equation*}
  \divv(\tf\f)(\dom) = \sI{\bd\dom}{f\tf}{\nu^{\rm int}}
\end{equation*}
which verifies the first equality in \reff{sf-s8-1}. From \reff{sf-s8-5a}
we get
\begin{equation*} 
  \lim_{m\to\infty}\I{\dom}{f\tf\cdot D\psi_m}{\lem} =
  \sI{\bd\dom}{f\tf}{\nu^{\rm int}} \,.
\end{equation*}
Thus we can apply Proposition~\ref{sf-s9} to get the second equality in 
\reff{sf-s8-1}.
\end{proof+}

\begin{lemma}\label{sf-s10}
Let $\dom\subset\R^n$ be open and bounded with Lipschitz boundary and let
$g\in\cB\cV(\dom)\cap C^\infty(\dom)$. Moreover let $\nu^{\rm int}$ be the
interior normal measure of $\dom$ with approximating sequence 
$\chi^{\rm int}_m$ 
according to \reff{nm-int-2}. Then 
\begin{equation} \label{sf-s10-1}
  \lim_{m\to\infty}\I{\dom}{g\tf\cdot D\chi^{\rm int}_m}{\lem} =
  -\, \sI{\bd\dom}{g\tf}{\nu^{\rm int}}  
\end{equation}
for all $\tf\in \woinf{\dom,\R^n}$.
\end{lemma}

\begin{proof}     
For $x\in\R^n$ we use the notation
\begin{equation*}
  x=(x_1,\dots,x_n)=(x',x_n)\in\R^n \qmq{with} 
  x'=(x_1,\dots,x_{n-1})\in\R^{n-1}\,. 
\end{equation*}
Since $\dom$ has Lipschitz boundary, for each $x\in\bd\dom$ there is a cylinder
\begin{equation*}
  C(x,r,h) := \{(y',y_n)\:|\: |y'-x'|<r\,,\; |y_n-x_n|< 2h \} 
\end{equation*}
and a Lipschitz continuous function $\gamma$ on $B_r(x')\subset\R^{n-1}$ such
that after a suitable rotation of the coordinate system
$|\gamma(y')-x_n|<h$ on $B_r(x')$ and
\begin{equation*}
  \dom\cap C(x,r,h) = 
  \{ y\in\R^n\:|\: |y'-x'|<r\,,\; \gamma(y') < y_n < x_n+2h \} \,.
\end{equation*}
Since $\bd\dom$ can be covered by finitely many such cylinders, it is
sufficient to show \reff{sf-s10-1} only for the case where the integrals 
are restricted to a cylinder $C:=C(\bar x,r,h)$ and to work in
the related coordinate system. The general case then follows 
by a straightforward argument with a partition of unity subordinate to finitely
many such cylinders. 

Let us fix some $\tf\in \woinf{\dom,\R^n}$ and a
cylinder $C=C(\bar x,r,h)$ with $\bar x\in\bd\dom$ and let us define 
\begin{equation*}
  C^{t,s}:= \{ x\in C\:|\: \gamma(x')+t<x_n<\gamma(x')+s \}
  \qmq{for} 0\le t < s < h \,,
\end{equation*}
\begin{equation*}
  g^t(x) := g(x',\gamma(x')+t)\,, \quad 
  \tilde g^t(x) := g(x',x_n+t) \qmq{for} t>0\,.
\end{equation*}
$\tilde g^t$ is a shift of $g$ with 
$\tilde g^t\in C^\infty(\ol{\dom\cap C})$, while $g^t$ is constant in the last
coordinate and not necessarily smooth. 
Moreover we briefly write $\chi_m=\chi_m^{\rm int}$ and 
let $\delta_m>0$ be related to it according to
\reff{nm-int-2}. Since $\bd\dom$ is Lipschitz, there is some $\tilde c>0$ such
that 
\begin{equation} \label{sf-s10-4}
  (\supp D\chi_m)\cap C = \dom_{\delta_m}\cap C \,\subset\,
  C^{0,\beta_m} \qmq{for}  \beta_m := \tilde c\delta_m \,.
\end{equation}

\medskip

(a) For $0<s<t<h$ we now have 
\begin{eqnarray*}
  |g^t(x)-g^s(x)| 
&\le& 
  \int_s^t\Big|\frac{\partial g}{\partial x_n}(x',\gamma(x')+\tau)\Big|\,d\tau\\
&\le&
  \int_s^t |Dg(x',\gamma(x')+\tau)|\,d\tau  \,.
\end{eqnarray*}
By the coarea formula, 
\begin{equation} \label{sf-s10-5}
  \I{\bd\dom\cap C}{|g^t-g^s|}{\hm} = \I{C^{s,t}}{|Dg|}{\lem}
  \le  \I{C^{0,t}}{|Dg|}{\lem}\,.
\end{equation}
The right hand side tends to zero as $t\to 0$ and, thus, there is some $g^0$
with 
\begin{equation} \label{sf-s10-6}
  \lim_{t\to 0} g^t = g^0 \qmq{in} \cL^1(\bd\dom\cap C,\hm) \,.
\end{equation}
We can extend $g^0$ on $C$ such that $g^0(x)=g^0(x',\gamma(x'))$. 
By \reff{sf-s10-5} and the integrability of $|Dg|$,  
for any $\eps>0$ there is some $\delta>0$ such that 
\begin{equation}  \label{sf-s10-7}
  \I{\bd\dom\cap C}{|g^t-g^s|}{\hm}<\eps \qmq{whenever} |t-s|<\delta
\end{equation}
(cf. \cite[p.~1016]{zeidler_IIB}). Below we show for $t>0$ that
\begin{equation} \label{sf-s10-10}
  \lim_{m\to\infty} \I{\dom\cap C}{\tilde g^t\tf\cdot D\chi_m}{\lem} =
  - \I{\bd\dom\cap C}{g^t\tf\cdot\nu^\dom}{\hm}  \,,
\end{equation}
\begin{equation} \label{sf-s10-11}
  \sI{\bd\dom\cap C}{g^t\tf}{\nu^{\rm int}} =
  \I{\bd\dom\cap C}{g^t\tf\cdot\nu^\dom}{\hm} \,,
\end{equation}
\begin{equation} \label{sf-s10-12}
  \lim_{t\to 0} \I{\dom\cap C}{\tilde g^t\tf\cdot D\chi_m}{\lem} =
  \I{\dom\cap C}{g\tf\cdot D\chi_m}{\lem} \qmz{uniformly for} m\in\N\,,
\end{equation}
\begin{equation} \label{sf-s10-13}
  \lim_{t\to 0}\, \sI{\bd\dom\cap C}{g^t\tf}{\nu^{\rm int}} =
  \sI{\bd\dom\cap C}{g\tf}{\nu^{\rm int}} \,.
\end{equation}
Consequently, for $\eps>0$ there is some $t_0>0$ and some $m_0\in\N$ such that
\begin{eqnarray*}
&&
  \hspace{-20mm}
  \Big| \I{\dom}{g\tf\cdot D\chi_m}{\lem} + 
  \sI{\bd\dom}{g\tf}{\nu^{\rm int}} \Big| \\ 
&\le&
  \Big| \I{\dom}{g\tf\cdot D\chi_m}{\lem} -
  \I{\dom\cap C}{\tilde g^{t_0}\tf\cdot D\chi_m}{\lem} \Big| \\
&&
  + \: \Big| \I{\dom\cap C}{\tilde g^{t_0}\tf\cdot D\chi_m}{\lem} +
  \I{\bd\dom\cap C}{g^{t_0}\tf\cdot\nu^\dom}{\hm} \Big| \\
&&
  +\:  \Big|\, \sI{\bd\dom\cap C}{g\tf}{\nu^{\rm int}} -
  \sI{\bd\dom\cap C}{g^{t_0}\tf}{\nu^{\rm int}} \Big| \\
&\le&
  3\eps \qmz{for all} m>m_0\,.
\end{eqnarray*}
But this implies the assertion \reff{sf-s10-1} and it remains to show
\reff{sf-s10-10}-\reff{sf-s10-13}.

(b) Let us show \reff{sf-s10-10}. We have that 
$\tilde g^t,\chi_m\in\woinf{\dom}$
and that $\dom\cap C$ has Lipschitz boundary. Thus, 
integration by parts gives for $t>0$
\begin{eqnarray*}
&&
  \hspace{-30mm}
  - \I{\dom\cap C}{\tilde g^t\tf\cdot D\chi_m}{\lem} \\
&=& 
  \I{\dom\cap C}{\tilde g^t\tf\cdot D(1-\chi_m)}{\lem} \\
&=&
  - \I{\dom\cap C}{(1-\chi_m)\divv(\tilde g^t\tf)}{\lem} \\
&& 
  + \I{\bd(\dom\cap C)}{(1-\chi_m)\,\tilde g^t\tf\cdot\nu^{\dom\cap C}}{\hm} \\
&\overset{m\to\infty}{\longrightarrow}&
  \I{\bd\dom\cap C}{\tilde g^t\tf\cdot\nu^\dom}{\hm}
\:=\:
  \I{\bd\dom\cap C}{g^t\tf\cdot\nu^\dom}{\hm}  \,.
\end{eqnarray*}
But this is \reff{sf-s10-10}.

(c) For \reff{sf-s10-11} we observe that $g^t\in\cL^\infty(\dom)$. Then, 
by Theorem~\ref{bd-s5}, there is a subsequence $\{\chi_{m'}\}$ with
\begin{equation*}
  \lim_{m'\to\infty}\I{\dom\cap C}{g^t\tf\cdot D\chi_{m'}}{\lem} =
  -\: \sI{\bd\dom\cap C}{g^t\tf}{\nu^{\rm int}} \,.
\end{equation*}
With \reff{sf-s10-10} we get \reff{sf-s10-11}.

(d) Let us verify \reff{sf-s10-12}.
For $\eps>0$ we choose $\delta>0$ as in \reff{sf-s10-7} and obtain 
for all $m\in\N$
\begin{eqnarray*}
&&
  \hspace{-25mm}
  \I{\dom\cap C}{|\tilde g^t-g|\,\tf\cdot D\chi_m}{\lem} \\
&\le&
  \frac{\|\tf\|_\infty}{\delta_m} \I{C^{0,\beta_m}}{|g^t-g|}{\lem} \\
&=&
  \frac{\|\tf\|_\infty}{\delta_m}
  \int_0^{\beta_m} \I{\bd\dom\cap C}{|g^{s+t}-g^s|}{\hm}\, ds  \\
&\le&
  \frac{\|\tf\|_\infty\beta_m\eps}{\delta_m} 
\:=\: \|\tf\|_\infty\,\tilde c\, \eps 
  \quad\qmz{for all} 0<t<\delta,\;m\in\N\,. 
\end{eqnarray*}
This gives \reff{sf-s10-12}.

(e) As preparation for the proof of \reff{sf-s10-13} we first show that
$g^t\overset{\nu^{\rm int}}{\longrightarrow}g$. Let us fix some $\eps>0$
and notice that $g^t\overset{\hm}{\longrightarrow}g^0$ on $\bd\dom$
by \reff{sf-s10-6}. Therefore
\begin{equation} \label{sf-s10-15}
  \hm(B^t) \overset{t\to 0}{\longrightarrow} 0 \qmq{where}
  B^t:=\big\{x\in\bd\dom\cap C\:\big|\: |g^t-g^0|>\eps \big\} \,.
\end{equation}
Obviously,
\begin{eqnarray*}
  B^{t,s}
&:=&
  \big\{x\in\bd\dom\cap C\:\big|\: |g^t-g^s|>2\eps \big\} \\
&\subset&
  \big\{x\in\bd\dom\cap C\:\big|\: |g^t-g^0|+|g^s-g^0| > 2\eps \big\} 
\:\subset\:  B^t \cup B^s \,.
\end{eqnarray*}
For $\delta>0$ and 
\begin{equation*}
  \tilde B^t := \big\{ x\in\dom\cap C \:\big|\: |g^t-g|>2\eps \big\} \,,
\end{equation*}
we have that
\begin{equation*}
  \lem(\tilde B^t\cap C^{0,\delta}) = \int_0^\delta \hm(B^{t,s})\, ds
  \le \int_0^\delta  \hm(B^t) + \hm(B^{s})\, ds
\end{equation*}
By Theorem~\ref{bd-s5},
\begin{eqnarray*}
  |\nu^{\rm int}|(\tilde B^t) 
&\le& 
  \limsup_{m\to\infty}\|D\chi_m\|_{\cL^1(\tilde B^t)} =
  \limsup_{m\to\infty} \I{\tilde B^t}{|D\chi_m|}{\lem} \\
&\le&
  \limsup_{m\to\infty}\tfrac{1}{\delta_m}
  \I{\tilde B^t\cap C^{0,\beta_m}}{}{\lem} \\
&\le&
  \limsup_{m\to\infty}\tfrac{1}{\delta_m}
  \int_0^{\beta_m} \hm(B^t) + \hm(B^{s})\, ds  \\
&\le&
  \limsup_{m\to\infty}\tfrac{2\beta_m}{\delta_m}
  \sup_{s\in(0,t)}  \hm(B^s)  \qmz{(use $\beta_m<t$ for $m$ large)}   \\
&=&
2\tilde c \sup_{s\in(0,t)}  \hm(B^s) 
\end{eqnarray*}
By \reff{sf-s10-15} the right hand side tends to zero as $t\to 0$.
Since $\eps>0$ was arbitrary, we obtain 
$g^t\overset{\nu^{\rm int}}{\longrightarrow}g$.
For $\tf g^t$ we have
\begin{equation*}
  \tilde B^t_\tf :=
  \big\{ x\in\dom\cap C \:\big|\: |\tf g^t-\tf g|>\eps \big\} \subset
  \big\{ x\in\dom\cap C \:\big|\: \|\tf\|_\infty|g^t-g|>\eps \big\} \,.
\end{equation*}
Hence, $|\nu^{\rm int}|(\tilde B^t)\to 0$ for all $\eps>0$ implies that
$|\nu^{\rm int}|(\tilde B^t_\tf)\to 0$ for all $\eps>0$. Thus we also have
$\tf g^t\overset{\nu^{\rm int}}{\longrightarrow}\tf g$.
 
(f) We now proof \reff{sf-s10-13}. Let us fix some $\eps>0$. For fixed
$s,t$ we consider the measure $\nu:=\|\tf\|_\infty|g^t-g^s|\nu^{\rm int}$. By 
\reff{pm-tva} there is some $\tilde\tf\in\cL^\infty(\dom,\R^n)$ 
with $\|\tilde\tf\|_\infty\le 1$ and by Theorem~\ref{bd-s5} there is some
subsequence $\chi_{m'}$ such that
\begin{eqnarray*}
&&
  \hspace{-25mm}
  \I{\dom\cap C}{\|\tf\|_\infty|g^t-g^s|}{|\nu^{\rm int}|} - \eps \\
&\le& 
  \|\tf\|_\infty \I{\dom\cap C}{|g^t-g^s|\tilde\tf}{\nu^{\rm int}} \\
&=&
  \lim_{m'\to\infty} \|\tf\|_\infty
  \I{\dom\cap C}{|g^t-g^s|\tilde\tf\cdot D\chi_{m'}}{\lem} \\
&\le&
  \limsup_{m\to\infty} \:\frac{ \|\tf\|_\infty}{\delta_{m}} \,
  \I{\dom_{\delta_{m}}\cap C}{|g^t-g^s|}{\lem} \\ 
&\le&
  \limsup_{m\to\infty} \:\frac{ \|\tf\|_\infty}{\delta_{m}}\,
  \I{C^{0,\beta_{m}}}{|g^t-g^s|}{\lem} \\
&=&
  \limsup_{m\to\infty} \:\frac{\beta_m \|\tf\|_\infty}{\delta_{m}}\,
  \I{\bd\dom\cap C}{|g^t-g^s|}{\hm} \\
&=& \tilde c \|\tf\|_\infty \I{\bd\dom\cap C}{|g^t-g^s|}{\hm} \,.
\end{eqnarray*}
Since $\eps>0$ is arbitrary, the estimate is also true without $\eps$.
By \reff{sf-s10-6}, the right hand side tends to zero if $t,s\to 0$. 
Hence $\tf g$ is $\nu^{\rm int}$-integrable and we have \reff{sf-s10-13}
(cf. \cite[p.~114]{rao}).
\end{proof}

\begin{lemma} \label{sf-s11}
Let $\dom\subset\R^n$ be an open bounded set with Lipschitz boundary, let
$f\in\cB\cV(\dom)$ and let $f_k\in\cB\cV(\dom)\cap C^\infty(\dom)$ be an
approximating sequence satisfying \reff{sf-s8-5b}.
Moreover let $\nu^{\rm int}$ be the interior
normal measure from Example~\ref{nm-int}. Then
$f$ is $\nu^{\rm int}$-integrable and
\begin{equation} \label{sf-s11-1}
  \lim_{k\to\infty}\sI{\bd\dom}{f_k\tf}{\nu^{\rm int}} =
  \sI{\bd\dom}{f\tf}{\nu^{\rm int}}
\end{equation}
for all $\tf\in \woinf{\dom,\R^n}$.
\end{lemma}

\smallskip

\begin{proof}     
We use the notation from the proof of Lemma~\ref{sf-s10} and, as there, 
it is sufficient to show \reff{sf-s11-1} for the case where the integrals
are restricted to an open cylinder $C:=C(\bar x,r,h)$ for some 
$\bar x\in\bd\dom$  and to work in the related coordinate system.
Let us also fix some $\tf\in \woinf{\dom,\R^n}$.

\newcommand{\av}{\tau} 

(a) We start with some preliminaries.
For $g\in\cB\cV(\dom)$ and $t,\av>0$ we set
\begin{equation*}
  g^x(t) := g^t(x) = g(x',\gamma(x')+t)\,, \quad 
  g^{(t,\av)}(x):=\frac{1}{\av}\int_t^{t+\av} g^s(x)\, ds \,.
\end{equation*}
Then there is some $\K\subset\bd\dom\cap C$ with 
$\hm\big((\bd\dom\cap C)\setminus\K\big)=0$ such that 
\begin{equation*}
  g^x\in \cB\cV\big((0,h)\big) \qmz{for all} x\in\K
\end{equation*}
(cf. \cite[p.~217, 220]{evans}). These $g^x$ agree $\cL^1$-a.e.
with their right continuous 
representative. Thus we can identify $g$ 
with a representative where 
\begin{equation*}
  \tx{all $g^x$ with $x\in\K$ are continuous from the right}
\end{equation*}
(cf. \cite[p.~136]{ambrosio}). 
With the distributional derivative $Dg^x$, that is a Radon measure on $(0,h)$,
we then have 
\begin{equation*}
  g^x(s) = g^x(t) + Dg^x\big((t,s]\big) \qmz{for all} t<s\,,\; x\in\K
\end{equation*}
(cf. \cite[p.~136, 139]{ambrosio}). Since $|Dg^x|\big((t,s]\big)$ is the
total variation of $g^x$ on $(t,s]$, 
\begin{equation} \label{sf-s11-5}
   |g^x(s) - g^x(t)| \le |Dg^x|\big((t,s]\big) \qmz{for all} x\in\K \,.
\end{equation}
The distributional derivative $D_ng$ with respect to $x_n$
is a Radon measure on $\dom\cap C$, since we have for all $0\le t<s$ 
\begin{eqnarray}
  |D_ng|(C^{t,s}) 
&=&
  \sup \Big\{ 
  \I{C^{t,s}}{g\frac{\partial\psi}{\partial x_n}}{\lem}\:\Big|\: 
  \psi\in C^1_c(C^{t,s})\,,\;\|\psi\|_\infty\le 1 \Big\} \nonumber\\
&\le&
  \sup \Big\{ 
  \I{C^{t,s}}{g\divv\tf}{\lem}\:\Big|\: 
  \tf\in C^1_c(C^{t,s},\R^n)\,,\;\|\tf\|_\infty\le 1 \Big\} \nonumber\\
&=&  \label{sf-s11-7}
  |Dg|(C^{t,s})
\end{eqnarray}
(cf. \cite[p.~194, 195]{ambrosio} and take 
$\tf=(0,\dots,0,\psi)$ to see the inequality). Therefore
\begin{equation}  \label{sf-s11-6}
  |D_ng|(C^{t,s}) = \I{\bd\dom\cap C}{|Dg^x|\big((t,s)\big)}{\hm} < \infty 
  \qmz{for all} 0\le t<s
\end{equation}
(cf. \cite[p.~195]{ambrosio}, \cite[p.~220]{evans}). Since $|Dg|$ is a Radon
measure on the open set $\dom\cap C$,
\begin{equation}  \label{sf-s11-8}
  \lim_{t\to 0} |Dg|(\ol{C^{0,t}}) = 0 \,.
\end{equation}
Using \reff{sf-s11-5}-\reff{sf-s11-7}, we get for $t,\av>0$ small,  
\begin{eqnarray}
  \I{\bd\dom\cap C}{|g^t-g^{(t,\av)}|}{\hm} 
&=&
  \I{\bd\dom\cap C}{\Big|\frac{1}{\av}\int_t^{t+\av}g^t-g^s\, ds\Big|}{\hm} 
  \nonumber\\
&\le&
  \frac{1}{\av} \int_t^{t+\av} \I{\bd\dom\cap C}{|g^t-g^s|}{\hm} \, ds 
  \nonumber\\
&\le&
  \frac{1}{\av} \int_t^{t+\av} 
  \I{\bd\dom\cap C}{|Dg^x|\big( (t,s+\av)\big)}{\hm} \, ds  
  \nonumber\\
&=&
  \frac{1}{\av} \int_t^{t+\av}  |D_ng|(C^{t,s+\av})  \, ds  
  \nonumber\\
&\le& \label{sf-s11-8a}
 |D_ng|(C^{t,t+2\av}) 
\:\le\:
  |Dg|(C^{0,t+2\av}) \,.
\end{eqnarray}

(b) We fix some $\tilde\eps>0$ and show that there is some $t_0>0$ and some
$k_0\in\N$ such that
\begin{equation} \label{sf-s11-9}
  \I{\bd\dom\cap C}{|f_k^t-f^t|}{\hm} \le \tilde\eps
  \qmz{for all} k>k_0\,,\; 0<t<t_0\,.
\end{equation}
By \reff{sf-s11-8} we can choose some $t_0>0$ such that
\begin{equation*}
  |Df|(\ol{C^{0,3t_0}}) < \frac{\tilde\eps}{4} \,.
\end{equation*}
Since
\begin{equation*}
  \limsup_{k\to\infty}|Df_k|(\ol{C^{0,3t_0}}) \le |Df|(\ol{C^{0,3t_0}})
\end{equation*}
(cf. \cite[p.~93]{pfeffer}), there is some $k_0\in\N$ with
\begin{equation*}
  |Df_k|(\ol{C^{0,3t_0}}) \le |Df|(\ol{C^{0,3t_0}}) + \frac{\tilde\eps}{4}
  \qmz{for all} k>k_0\,.
\end{equation*}
Let us fix some $\av<t_0$. Then, by $f_k\to f$ in $\cL^1(\dom)$, we can assume
that $k_0$ is so large that
\begin{equation*}
  \frac{1}{\av} \I{C^{0,2t_0}}{|f_k-f|}{\lem} < \frac{\tilde\eps}{4}
  \qmz{for all} k>k_0\,.
\end{equation*}
Consequently, using \reff{sf-s11-8a}, we obtain for $0<t<t_0$ and $k>k_0$ 
\begin{eqnarray*}
&& \hspace{-20mm}
  \I{\bd\dom\cap C}{|f_k^t-f^t|}{\hm}  \\
&\le&
  \I{\bd\dom\cap C}{|f_k^t-f_k^{(t,\av)}|}{\hm} +
  \I{\bd\dom\cap C}{|f_k^{(t,\av)}-f^{(t,\av)}|}{\hm}  \\
&& 
  +\: \I{\bd\dom\cap C}{|f^{(t,\av)}-f^t|}{\hm}  \\
&\le&
  |Df_k|(C^{0,t+2\av}) +  |Df|(C^{0,t+2\av}) \\
&&
  + \I{\bd\dom\cap C}{\Big|\frac{1}{\av}
  \int_t^{t+\av}f_k^s-f^s\,ds\Big|}{\hm}  \\
&\le&
  |Df_k|(\ol{C^{0,3t_0}}) + |Df|(\ol{C^{0,3t_0}}) +
  \frac{1}{\av}\I{C^{t,t+\av}}{|f_k-f|}{\lem}  \\
&\le&
  2 |Df|(\ol{C^{0,3t_0}}) + \frac{\tilde\eps}{4} +
  \frac{1}{\av} \I{C^{0,2t_0}}{|f_k-f|}{\lem} 
\:\le\:  \tilde\eps 
\end{eqnarray*}
which verifies \reff{sf-s11-9}.

(c) We show that $\tf f_k\overset{\nu}{\to}\tf f$. For $\eps>0$ we define
\begin{equation*}
  B_k:= \big\{ x\in\dom\cap C\:\big|\: |f_k-f|>\eps \big\}\,, 
\end{equation*}
\begin{equation*}
  B_k^t := \big\{ x\in\bd\dom\cap C \:\big|\: |f_k^t-f^t|>\eps \big\} \,.
\end{equation*}
Let us also fix some $\tilde\eps>0$ and let $t_0>0$ and $k_0\in\N$ be related
to $\tilde\eps$ according to \reff{sf-s11-9}.
The  Chebyshev inequality and \reff{sf-s11-9} imply
\begin{equation*}
  \hm(B^t_k) \le \frac{1}{\eps} \I{B^t_k}{|f_k^t-f^t|}{\hm}
  \le \frac{\tilde\eps}{\eps}
  \qmz{for all} k>k_0\,,\; 0<t<t_0 \,.
\end{equation*}
With \reff{bd-s5-0a}, \reff{sf-s10-4}, and $\beta_m\to 0$, we get
\begin{eqnarray*}
  |\nu^{\rm int}|(B_k) 
&\le& 
  \limsup_{m\to\infty}\|D\chi_m\|_{\cL^1(B_k)} =
  \limsup_{m\to\infty} \I{B_k}{|D\chi_m|}{\lem} \\
&\le&
  \limsup_{m\to\infty}\tfrac{1}{\delta_m}
  \I{B_k\cap C^{0,\beta_m}}{}{\lem} \\
&\le&
  \limsup_{m\to\infty}\tfrac{1}{\delta_m}
  \int_0^{\beta_m} \hm(B_k^t)\, dt \\
&\le&
  \limsup_{m\to\infty}\tfrac{\tilde\eps\beta_m}{\eps\delta_m} 
\:=\:
  \frac{\tilde c\tilde\eps}{\eps} \qmz{for all} k>k_0\,.
\end{eqnarray*}
Therefore $|\nu^{\rm int}|(B_k)\to 0$ for all $\eps>0$ and, hence,
$f_k\overset{\nu^{\rm int}}{\longrightarrow}f$.
For $\tf f_k\overset{\nu^{\rm int}}{\longrightarrow}\tf f$
we argue as in part (e) of the proof of Lemma~\ref{sf-s10}.

(d) Let us finally show the assertion \reff{sf-s11-1}.
For that we fix $\eps>0$ and let $t_0>0$, $k_0\in\N$ be related to
$\tilde\eps>0$ as in \reff{sf-s11-9}.
By \reff{pm-tva} there is some 
$\tilde\tf\in\cL^\infty(\dom,\R^n)$ with $\|\tilde\tf\|_\infty\le 1$ 
and by Theorem~\ref{bd-s5} there is a subsequence $\{\chi_{m'}\}$ such
that for $k,l>k_0$  
\begin{eqnarray*}
&&
  \hspace{-25mm}
  \I{\dom\cap C}{\|\tf\|_\infty|f_k-f_l|}{|\nu^{\rm int}|} -\eps  \\
&\le& 
  \I{\dom\cap C}{\|\tf\|_\infty|f_k-f_l|\,\tilde\tf}{\nu^{\rm int}} \\
&=&
  \|\tf\|_\infty \lim_{m'\to\infty} 
  \I{\dom\cap C}{|f_k-f_l|\,\tilde\tf\cdot D\chi_{m'}}{\lem} \\
&\le&
  \|\tf\|_\infty \limsup_{m\to\infty} \frac{1}{\delta_{m}} 
  \I{C^{0,\beta_m}}{|f_k-f_l|}{\lem} \\ 
&=&
  \|\tf\|_\infty \limsup_{m\to\infty} \frac{1}{\delta_{m}} \int_0^{\beta_m}\!\!
  \I{\bd\dom\cap C}{|f_k^t-f_l^t|}{\hm} \, dt \\
&\le&
 \|\tf\|_\infty \limsup_{m\to\infty} \frac{1}{\delta_{m}} \int_0^{\beta_m}\!\!
  \I{\bd\dom\cap C}{\big(|f_k^t-f^t|+|f^t-f_l^t|\big)}{\hm} \, dt \\
&\le&
  \|\tf\|_\infty \limsup_{m\to\infty} \frac{1}{\delta_{m}} \int_0^{\beta_m}
  2\tilde\eps \, ds \quad\qmq{(since $t<t_0$ for $m$ large)}\\
&=& 
  2\tilde c\tilde\eps\|\tf\|_\infty  \,. 
\end{eqnarray*}
This is true without $\eps>0$, since it is arbitrary.
Since $\tilde\eps>0$ is arbitrary, the left hand side tends to zero as
$k,l\to\infty$. Therefore $\tf f$ is $\nu^{\rm int}$-integrable and 
\reff{sf-s11-1} follows (cf. \cite[p.~114]{rao}).
\end{proof}

As application of the introduced theory we finally consider a general boundary
value problem for the $p$-Laplace operator. Let $\dom\subset\R^n$ be an open 
bounded set and let $1<p<\infty$. The trace operator 
$T:\cW^{1,1}(\dom)\to\woinf{\dom}^*$ from Proposition~\ref{tt-s5} (with
$\edom=\dom$) is also a
linear continuous operator on $\cW^{1,p}(\dom)$ by the continuous embedding
$\cW^{1,p}(\edom))\hookrightarrow\cW^{1,1}(\edom)$. For given
$g\in\cL^{p'}(\dom)$ 
and $f_b\in\cW^{1,p}(\dom)$ we call $f\in\cW^{1,p}(\dom)$
weak solution of the boundary value problem 
\begin{equation} \label{sf-e3}
  -\divv{\big( |Df|^{p-2}Df \big)} = g \qmz{on} \dom\,, \quad
  f=f_b \qmz{on} \bd\dom
\end{equation}
if we have that
\begin{equation*}
  \I{\dom}{|Df|^{p-2}DfD\tf-g\tf}{\lem} = 0 
  \qmz{for all} \tf\in C_c^\infty(\dom) \qmq{and}
\end{equation*}
\begin{equation} \label{sf-e5}
  \df{T(f-f_b)}{\tf} =0 \qmz{for all} \tf\in\woinf{\dom,\R^n} \,.
\end{equation}
We show that this problem has always a solution
without any regularity assumption on the boundary $\bd\dom$. 
Before let us discuss the
boundary condition \reff{sf-e5} for $\dom$ having Lipschitz boundary. 
From \reff{sf-e2} we get 
\begin{equation*}
  \df{T(f-f_b)}{\tf} = \divv((f-f_b)\tf)(\dom)
  = \I{\bd\dom}{(\pr{f}-\pr{f_b})\tf\cdot\nu}{\hm} = 0\, 
\end{equation*}
for all $\tf\in\woinf{\dom,\R^n}$. By approximation, the most right 
equality is even valid for all $\tf\in C(\dom,\R^n)$. Hence 
$\reme{(\pr{f}-\pr{f_b})\nu\hm}{\bd\dom}$ has to be the zero
measure. Consequently  
\begin{equation*}
  \pr{f}=\pr{f_b} \qmz{$\hm$-a.e. on $\bd\dom$\,,}
\end{equation*}
which is the usual pointwise boundary condition. Let us still point out that 
the trace $Tf$ is uniquely defined though its representation according to
Theorem~\ref{sf-s1} is not. 

\begin{theorem}\label{plap}
Let $\dom\subset\R^n$ be open and bounded, let $1<p<\infty$, let 
$g\in \cL^{p'}(\dom)$ (where $\frac{1}{p}+\frac{1}{p'}=1$), 
let $f_b\in\cW^{1,p}(\dom)$, and let 
$T$ be the trace operator from Proposition~\ref{tt-s5}.
Then there is a weak solution $f\in\cW^{1,p}(\dom)$ of the boundary value
problem \reff{sf-e3}. 
\end{theorem}
\noi
Notice that we obviously have that
\begin{equation*}
  Tf=0 \qmq{for all} f\in C_c^\infty(\dom)\,.
\end{equation*}
Thus, by the continuity of $T$ on $\cW^{1,p}(\dom)$,
\begin{equation} \label{sf-e4}
  \cW^{1,p}_0(\dom) \subset \{f\in\cW^{1,p}(\dom)\mid Tf=0\} \,.
\end{equation}
The drawback of the set on the right hand side is that the Poicar\'e
inequality might not be true for all functions. But it turns 
out to be sufficient for the theorem to study a
variational problem on $f_b+\cW^{1,p}_0(\dom)$. 

\begin{proof+}{ of Theorem~\ref{plap}}  
We consider the minimization problem 
\begin{equation*}
  E(f):=\I{\dom}{|Df|^p-fg}{\lem} \; \to\; \op{Min!}\,, \quad
  f\in\cW^{1,p}(\dom)
\end{equation*}
subject to
\begin{equation*}
  M:=\big\{ f\in\cW^{1,p}(\dom) \:\big|\: f=f_b+f_0\,,\;
  f_0\in\cW^{1,p}_0(\dom) \big \} .
\end{equation*}
Let $f_k\in\cW^{1,p}(\dom)$ be a minimizing sequence $f_k\in\cW^{1,p}(\dom)$.
Then, by the Poincar\'e inequality, there is some $c>0$ such that
\begin{eqnarray*}
  \|f_k\|_{\cL^p} 
&\le&
  \|f_k-f_b\|_{\cL^p} + \|f_b\|_{\cL^p} \\
&\le&
  c\|Df_k-Df_b\|_{\cL^p} + \|f_b\|_{\cL^p} \\
&\le&
  c\|Df_k\|_{\cL^p} + c\|Df_b\|_{\cL^p} + \|f_b\|_{\cL^p}  
\end{eqnarray*}
Consequently, for some $\tilde c>0$,
\begin{eqnarray*}
  E(f_k) 
&\ge&
   \|Df_k\|^p_{\cL^p} - \|g\|_{\cL^{p'}}\|f_k\|_{\cL^p} \\
&\ge&
  \|Df_k\|^p_{\cL^p} - 
  \tilde c \big( \|Df_k\|_{\cL^p} + \|Df_b\|_{\cL^p} + \|f_b\|_{\cL^p} \big) \\
&=&
  \|Df_k\|_{\cL^p}\big( \|Df_k\|^{p-1}_{\cL^p}-\tilde c \big) -
  \tilde c \big(\|Df_b\|_{\cL^p} + \|f_b\|_{\cL^p} \big) 
\end{eqnarray*}
Combining both estimates we get that the $f_k$
must be bounded in $\cW^{1,p}(\dom)$. Thus there is a weakly convergent
subsequence, denoted the same way, with $f_k\wto:f$. Since 
$M$ is a closed affine subspace of
$\cW^{1,p}(\dom)$, it is also weakly closed. Therefore $f\in M$.
As convex and continuous function, $E$ is weakly lower semicontinuous
(cf. \cite[p.~49 or 74]{dacorogna}). 
This implies that $f$ solves the minimization problem. 
Obviously, $f=f_b+f_0$ for some $f_0\in\cW^{1,p}_0(\dom)$. Hence
\begin{equation*}
  \df{T(f-f_b)}{\tf} = \df{Tf_0}{\tf} = 0  
  \qmz{for all} \tf\in\woinf{\dom,\R^n}\,
\end{equation*}
and, thus, $f$ satisfies the boundary condition. 

Now we decompose $E=E_1-E_2$ in the obvious way where $E_1$ is convex and
$E_2$ is linear and continuous. Clearly,
\begin{equation*}
  E_2'(f,\tf) = \I{\dom}{g\tf}{\lem} \qmz{for all} \tf\in\cW^{1,p}_0(\dom)\,.
\end{equation*}
Moreover $E_1$ is G\^ ateaux differentiable on $M$ with
\begin{equation*}
  E_1'(f,\tf) = \I{\dom}{|Df|^{p-2}DfD\tf}{\lem}
  \qmz{for all} \tf\in\cW^{1,p}_0(\dom)
\end{equation*}
(cf. \cite[p.~89]{dacorogna}). Since $f$ minimizes $E$ on $M$,
\begin{equation*}
  E'(f,\tf) =  \I{\dom}{|Df|^{p-2}DfD\tf - g\tf}{\lem} = 0
  \qmz{for all} \tf\in\cW^{1,p}_0(\dom) \,.
\end{equation*}
Consequently, $f$ is a weak solution of 
\reff{sf-e3} and the proof is complete.
\end{proof+}

\bigskip

\bibliographystyle{plain}

\vspace{5mm}

\noi
Friedemann Schuricht (corresponding author)\\
TU Dresden \\
Fakultät Mathematik \\
01062 Dresden \\
Germany\\
email: {\it friedemann.schuricht@tu-dresden.de}

\bigskip

\noi
Moritz Schönherr\\
TU Dresden \\
Fakultät Mathematik \\
01062 Dresden \\
Germany \\
email: {\it moritz.schoenherr@posteo.de}

\end{document}